%% file: pap.tex
\documentclass{elsarticle}
  
\input{preamble.pap.tex}

\input{preamble.macros.tex}

\input{preamble.tikz.tex}

\usepackage{setspace}

\newbool{fastcompile}
\setbool{fastcompile}{false}

\newcommand{\cdens}{\rho}
\newcommand{\ddens}{\rhobold}
\newcommand{\cparam}{\psi}
\newcommand{\dparam}{\psibold}

\newcommand{\dstvcL}{\ubm}

\newcommand{\resLu}{\rbm_\dstvcL}
\newcommand{\stiffmatL}{\Kbm}
\newcommand{\rhsvecL}{\fbm}
\newcommand{\asmblmatL}{\Pbm}

\newcommand{\dstvcH}{\phibold}
\newcommand{\resH}{\Rbm_\dstvcH}
\newcommand{\stiffmatH}{\Hbm}
\newcommand{\rhsvecH}{\bbm}
\newcommand{\asmblmatH}{\Qbm}

\begin{document}
\title{A globally convergent method to accelerate topology optimization using on-the-fly model reduction}

\author[rvt1]{Masayuki Yano\fnref{fn1}}
\ead{myano@utias.utoronto.ca}

\author[rvt2]{Tianci Huang\fnref{fn2}}
\ead{thuang5@nd.edu}

\author[rvt2]{Matthew J. Zahr\fnref{fn3}\corref{cor1}}
\ead{mzahr@nd.edu}

\address[rvt1]{Institute for Aerospace Studies, University of Toronto,
               Toronto, ON, M3H 5T6, Canada}
\address[rvt2]{Department of Aerospace and Mechanical Engineering, University
               of Notre Dame, Notre Dame, IN 46556, United States}
\cortext[cor1]{Corresponding author}

\fntext[fn1]{Assistant Professor, Institute for Aerospace Studies,
             University of Toronto}
\fntext[fn2]{Graduate Student, Department of Aerospace and Mechanical
             Engineering, University of Notre Dame}
\fntext[fn3]{Assistant Professor, Department of Aerospace and Mechanical
             Engineering, University of Notre Dame}


\begin{abstract}
  We present a globally convergent method to accelerate density-based topology optimization using projection-based reduced-order models (ROMs) and trust-region methods.
  To accelerate topology optimization, we replace the large-scale finite element simulation, which dominates the computational cost, with ROMs that reduce the cost of objective function and gradient evaluations by orders of magnitude.
To guarantee convergence,
we first introduce a trust-region method that employs generalized trust-region constraints and prove it is globally convergent. 
We then devise a class of globally convergent ROM-accelerated topology
optimization methods informed by two theories: the aforementioned trust-region theory, which identifies the ROM accuracy conditions required to guarantee the method converges to a critical point of the original topology optimization problem; \textit{a posteriori} error estimation theory for projection-based ROMs, which informs ROM construction procedure to meet the accuracy conditions.  This leads to trust-region methods that construct and update the ROM on-the-fly during optimization; the methods are guaranteed to converge to a critical point of the original, unreduced topology optimization problem, regardless of starting point.
Numerical experiments on three different
structural topology optimization problems demonstrate the proposed reduced
topology optimization methods accelerate convergence to the optimal design
by up to an order of magnitude.
\end{abstract}


\maketitle


\input{intro}

\input{topopt}

\input{rom}

\input{tr}
\input{numexp.tex}
\input{conclusion}

\appendix
\input{adjoint}
\input{tr_proofs}

\input{rom_proofs}

\section*{Acknowledgments}
This material is based upon work supported by the Air Force Office of
Scientific Research (AFOSR) under award number FA9550-20-1-0236.
The content of this publication does not necessarily reflect the position
or policy of any of these supporters, and no official endorsement
should be inferred.


\bibliographystyle{plain}
\bibliography{biblio}

\end{document}

%% file: preamble.pap.tex
\usepackage[margin=1in]{geometry}

\usepackage{graphicx}
\usepackage{pstool}
\usepackage{wrapfig}
\usepackage{caption}
\usepackage{subcaption}
\usepackage[section]{placeins} 

\usepackage{fancyhdr} 
\usepackage{lastpage} 
\usepackage{extramarks} 

\usepackage{xcolor} 
\usepackage{enumerate} 
\usepackage{paralist} 
\usepackage{amsmath, amsthm, amssymb, mathtools}
\usepackage{mathabx, pifont, stmaryrd} 
\usepackage[explicit]{titlesec} 
\usepackage{etoolbox} 
\usepackage{bibentry} 
\makeatletter\let\saved@bibitem\@bibitem\makeatother 
\usepackage[colorlinks, bookmarksopen, bookmarksnumbered,
            citecolor=red,urlcolor=red]{hyperref} 
\makeatletter\let\@bibitem\saved@bibitem\makeatother 

\usepackage{algorithm}
\usepackage{algorithmic}

\newtheorem{theorem}{Theorem}
\newtheorem{lemma}{Lemma}

\newtheorem{remark}{Remark}

\theoremstyle{definition}
\newtheorem{assume}{Assumption}

%% file: preamble.macros.tex
\usepackage{bm} 
\usepackage{amsfonts} 

\newcommand{\ds}[1]{\ensuremath{\displaystyle{#1}}}
\newcommand{\func}[3]{\ensuremath{#1 : #2 \rightarrow #3}}

\newcommand{\norm}[1]{\ensuremath{\left\| #1 \right\|}}

\newcommand{\suchthat}{\mathrel{}\middle|\mathrel{}}

\newcommand{\optunc}[2]{\underset{#1}{\text{minimize}} ~~ #2}

\newcommand{\optconOne}[3]{
\begin{aligned}
& \underset{#1}{\text{minimize}}
& & #2 \\
& \text{subject to} & & #3
\end{aligned}}

\newcommand{\pder}[2]{\ensuremath{\frac{\partial #1}{\partial #2}}}

\newcommand{\oder}[2]{\ensuremath{D_{#2} #1}}

\newcommand{\Bcal}{\ensuremath{\mathcal{B}}}
\newcommand{\Ccal}{\ensuremath{\mathcal{C}}}
\newcommand{\Dcal}{\ensuremath{\mathcal{D}}}

\newcommand{\Ocal}{\ensuremath{\mathcal{O}}}

\newcommand{\Scal}{\ensuremath{\mathcal{S}}}

\newcommand{\Nbb}{\ensuremath{\mathbb{N} }}

\newcommand{\Qbb}{\ensuremath{\mathbb{Q} }}
\newcommand{\Rbb}{\ensuremath{\mathbb{R} }}

\newcommand\Abm{{\ensuremath{\bm{A}}}}
\newcommand\Bbm{{\ensuremath{\bm{B}}}}
\newcommand\Cbm{{\ensuremath{\bm{C}}}}
\newcommand\Dbm{{\ensuremath{\bm{D}}}}

\newcommand\Hbm{{\ensuremath{\bm{H}}}}

\newcommand\Kbm{{\ensuremath{\bm{K}}}}

\newcommand\Pbm{{\ensuremath{\bm{P}}}}
\newcommand\Qbm{{\ensuremath{\bm{Q}}}}
\newcommand\Rbm{{\ensuremath{\bm{R}}}}

\newcommand\Ubm{{\ensuremath{\bm{U}}}}

\newcommand\bbm{{\ensuremath{\bm{b}}}}

\newcommand\dbm{{\ensuremath{\bm{d}}}}

\newcommand\fbm{{\ensuremath{\bm{f}}}}
\newcommand\gbm{{\ensuremath{\bm{g}}}}

\newcommand\pbm{{\ensuremath{\bm{p}}}}
\newcommand\qbm{{\ensuremath{\bm{q}}}}
\newcommand\rbm{{\ensuremath{\bm{r}}}}

\newcommand\ubm{{\ensuremath{\bm{u}}}}
\newcommand\vbm{{\ensuremath{\bm{v}}}}
\newcommand\wbm{{\ensuremath{\bm{w}}}}
\newcommand\xbm{{\ensuremath{\bm{x}}}}
\newcommand\ybm{{\ensuremath{\bm{y}}}}

\newcommand\Pbold{\ensuremath{\mathbf{P}}}

\newcommand\lambdabold{{\ensuremath{\boldsymbol{\lambda}}}}

\newcommand\mubold{{\ensuremath{\boldsymbol{\mu}}}}

\newcommand\rhobold{{\ensuremath{\boldsymbol{\rho}}}}

\newcommand\phibold{{\ensuremath{\boldsymbol{\phi}}}}

\newcommand\psibold{{\ensuremath{\boldsymbol{\psi}}}}

\newcommand\xibold{{\ensuremath{\boldsymbol{\xi}}}}

\newcommand\Phibold{{\ensuremath{\boldsymbol{\Phi}}}}

\newcommand\Lambdabold{{\ensuremath{\boldsymbol{\Lambda}}}}
\newcommand\Psibold{{\ensuremath{\boldsymbol{\Psi}}}}

\newcommand\zerobold{\ensuremath{\mathbf{0}}}
\newcommand\onebold{\ensuremath{\mathbf{1}}}

%% file: preamble.tikz.tex
\usepackage{tikz}
\usepackage{pgfplots}
\usepackage{pgfplotstable, filecontents, booktabs}
\pgfplotsset{compat=1.9}

\usetikzlibrary{pgfplots.groupplots}
\usepgfplotslibrary{fillbetween}
\usetikzlibrary{calc,fit,matrix,arrows,automata,positioning,shapes}
\usetikzlibrary{arrows.meta}

\pgfplotsset{select coords between index/.style 2 args={
    x filter/.code={
        \ifnum\coordindex<#1\fi
        \ifnum\coordindex>#2\fi
    }
}}

\tikzset{
 invisible/.style={opacity=0},
 visible on/.style={alt={#1{}{invisible}}},
 alt/.code args={<#1>#2#3}{%
   \alt<#1>{\pgfkeysalso{#2}}{\pgfkeysalso{#3}}
 },
}


%% file: intro.tex
\section{Introduction}
\label{sec:intro}

Topology optimization enables the discovery of new, intricate, and non-intuitive structural designs, which can now be manufactured due to recent advances in additive manufacturing techniques~\cite{zegard_bridging_2016}. However, topology optimization remains computationally expensive; for instance, density-based methods
require billions of computational voxels to accurately represent complex geometries, and optimization requires hundreds of thousands of
computing hours \cite{aage_giga-voxel_2017}. The computational cost increases further for robust optimization that accounts for uncertainties in loading conditions or material properties, often rendering it intractable in practical engineering settings.

A typical topology optimization approach uses first-order optimization methods and requires hundreds of design iterations for convergence.  The dominant cost in each design iteration is the solution of the finite element problem associated with the (linear) elasticity equations.  Hence, one class of approaches to improve the efficiency of topology optimization
targets the (iterative) solver for the linear(ized) finite element
system. Amir~et~al.~\cite{amir_efficient_2010} and Choi~et~al.~\cite{choi_accelerating_2019} use inexact linear solves to
efficiently compute approximate finite element solutions, while
others have developed improved preconditioners, e.g., based on
substructuring~\cite{evgrafov_large-scale_2008},
multigrid~\cite{amir_multigrid-cg_2014}, and Krylov methods with
subspace recycling~\cite{wang_large-scale_2007}. Alternatively,
Rojas-Labanda~et~al.~\cite{rojas-labanda_efficient_2016} investigate
the use of second-order optimization methods to reduce the
number of design iterations. Nguyen~et~al.~\cite{nguyen_computational_2010}
introduce a multi-resolution approach to topology optimization that uses two different meshes for finite element analysis and geometry representation to reduce the analysis cost.  

Another class of approaches~\cite{yoon_structural_2010, gogu_improving_2015,
      nguyen_three-dimensional_2019, choi_accelerating_2019, xiao_fly_2020} to reduce the computational cost of topology optimization
is based on the concept of reanalysis~\cite{kirsch_structural_2001, zheng_approximate_2017,
      sun_efficient_2018, senne_approximate_2019}; i.e.,
 they recognize that the optimization problem is a ``many-query'' problem that requires the analysis of many closely related designs along the optimization path and use previous finite element solutions to reduce the cost of subsequent solutions.  Most of these methods are based on
projection-based reduced-order models (ROMs), which approximate the solution to the finite element problem using a low-dimensional reduced basis.
Yoon~\cite{yoon_structural_2010} uses ROMs in the
frequency domain, with a reduced basis consisting of eigenmodes, to reduce the cost of
optimizing the frequency response of structures.
Gogu~\cite{gogu_improving_2015} constructs ROMs on-the-fly,
i.e., during the topology optimization procedure, using previously computed
finite element solutions; the reduced basis is adapted when the finite
element residual exceeds a predefined tolerance.
Choi et~al.~\cite{choi_accelerating_2019} use a similar approach that adapts the
residual tolerance based on the first-order optimality criteria to ensure
more stringent accuracy requirements on the ROM are used near convergence;
they also investigate its efficient implementation for large-scale
problems. More recently, Xiao~et~al.~\cite{xiao_fly_2020} introduced a
similar method to \cite{gogu_improving_2015} that constructs a reduced basis
on-the-fly using proper orthogonal decomposition (POD)~\cite{sirovich_turbulence_1987} rather than Gram-Schmidt
orthogonalization. A commonality of these approaches is they have been
shown to effectively reduce the overall cost of topology optimization;
however, to our knowledge, they make no attempt to rigorously assess the impact of ROM approximation on the convergence or to design algorithms that guarantee convergence to an optimal solution of the original (unreduced) problem. 
This work addresses the need for rigorous convergence theories and algorithms to accelerate topology optimization using ROMs constructed on-the-fly.

The first contribution of this work is the development of a model
management framework \cite{alexandrov_approximation_2001}
and global convergence theory for optimization
problems with convex constraints based on generalized, error-aware
trust regions \cite{zahr_phd_2016}. Similar to (traditional) trust-region methods,
error-aware trust-region methods (approximately) solve an inexpensive
optimization problem associated with an approximate model within a trust
region to advance the trust-region
center toward the solution of the original optimization problem. The
key difference is that error-aware trust-region methods consider
a trust region based on the sublevel sets of
an error indicator for the approximation model, whereas
traditional trust regions are based on a notion of distance in the
design space \cite{conn_trust-region_2000}.
In this work, we generalize the error-aware trust-region method
in \cite{zahr_phd_2016}, developed for unconstrained optimization,
to problems with inexpensive convex constraints, which commonly
arise in topology optimization, e.g., box constraints on the design
variables and volume fraction bounds. We prove global convergence of the
method assuming (i) the value and gradient of the
approximation model at any iteration match the optimization
objective function at the corresponding trust-region center and (ii) the
trust-region constraint is the sublevel set of an asymptotic error
bound for the approximation model. While assumptions (i) on the accuracy
of the approximation model at the trust-region center can be weakened
considerably by appealing to more general trust-region theory
\cite{carter_global_1991,
      heinkenschloss_analysis_2002,
      kouri_trust-region_2013,
      heinkenschloss_matrix-free_2014,
      kouri_inexact_2014},
this turns out to be unnecessary; the assumptions are easy to satisfy if projection-based ROMs are used as the approximation model. 
This error-aware trust-region method provides a general framework
to develop globally convergent solvers for optimization
problems constrained by partial differential equations using
ROMs constructed on-the-fly, which leads to the second contribution
of this work.

The second key contribution of this work is the development of an
efficient, globally convergent topology optimization method that
uses ROMs as the approximation model
in the proposed error-aware trust-region framework.
At a given iteration, the ROM is
updated on-the-fly based on the solution and adjoint snapshots at all trust-region centers encountered and a POD compression procedure that guarantees the ROM satisfies the aforementioned accuracy requirements (i) and (ii) for the trust-region approximation model.
We will consider
two trust-region constraints: the traditional trust region
based on distance from the center and the sublevel sets
of the residual-based error indicator of the ROM.
The latter choice follows
the work in \cite{zahr_progressive_2015, zahr_phd_2016}
for general PDE-constrained optimization problems as well as the
work in \cite{gogu_improving_2015, choi_accelerating_2019, xiao_fly_2020}
that use the residual as an indicator to update the ROM.
We appeal to \textit{a posteriori} error estimation theory for projection-based ROMs (see, e.g., review paper~\cite{Rozza_2008_RB_Review} and textbook~\cite{Hesthaven_2016_RB_Book}) to prove that the procedure satisfies the criteria (i) and (ii) for global convergence in the error-aware trust-region framework; therefore, the method is guaranteed to converge to a local
minimum of the original optimization problem from an
arbitrary starting point. We use three benchmark topology
optimization problems to demonstrate global convergence of the
method and demonstrate speedups up to a factor of ten
relative to a standard topology optimization approach that
applies the method of moving asymptotes \cite{svanberg_method_1987}
to the original (unreduced) topology optimization problem.
This is not the first work to
embed on-the-fly ROMs in a trust-region setting, e.g.,
\cite{arian_trust-region_2000, arian_managing_2002, yue_accelerating_2013,
      agarwal_trust-region_2013, sachs_adaptive_2014,
      zahr_progressive_2015, zahr_phd_2016, qian_certified_2017},
or to use error-aware trust-regions
\cite{yue_accelerating_2013, zahr_progressive_2015,
      zahr_phd_2016, qian_certified_2017};
however, to our knowledge it is the first to do so in the context of
topology optimization and incorporate convex side constraints into
the error-aware trust region setting.


The remainder of the paper is organized as follows.
Section~\ref{sec:topopt} introduces a density-based topology
optimization formulation with Helmholtz filtering and the
adjoint method to compute gradients with respect to the
design variables.
Section~\ref{sec:rom} reviews projection-based reduced-order
models for the linear elasticity finite element system
and provides the associated error analysis
that is required to establish global convergence of the
proposed ROM-based topology optimization method.
Section~\ref{sec:tr} introduces the error-aware trust-region
method that utilizes generalized trust-region constraints and
establishes its global convergence theory. Then we use projection-based
reduced-order models and the error analysis from Section~\ref{sec:rom},
and introduce a procedure to construct a reduced basis from
state and adjoint snapshots, to define a class of globally convergent
ROM-based topology optimization methods. These new methods are
tested on a suite of benchmark topology optimization problems
in Section~\ref{sec:numexp}. Finally, Section~\ref{sec:concl}
offers conclusions.

%% file: topopt.tex
\section{Density-based topology optimization with Helmholtz filtering}
\label{sec:topopt}
In this section, we present a computational procedure for topology optimization where we seek a structure that minimizes a given objective function (e.g., compliance) under a volume constraint.  Our presentation of topology optimization follows~\cite{sigmund_99_2001,andreassen_efficient_2011}.  We review a linear elasticity formulation (Section~\ref{sec:topopt:linelast}), review a Helmholtz filtering technique (Section~\ref{sec:topopt:helm}), state the topology optimization problem (Section~\ref{sec:topopt:opt}), derive the adjoint method to efficiently evaluate the objective function gradient (Section~\ref{sec:topopt:deriv}, \ref{sec:adjoint}), and analyze the cost of the optimization procedure (Section~\ref{sec:topopt:comp}).

\subsection{Linear elasticity}
\label{sec:topopt:linelast}
We first review the finite element approximation of the linear elasticity equation for density-based topology optimization. We represent the geometry of the structure by a density distribution over a $d$-dimensional design domain, where the (near) zero and unity density indicate the absence and presence, respectively, of the material in a given region. To this end, we partition our design domain $\Omega$ into $N_e$ non-overlapping polygonal elements, $\{ \Omega_e \}_{e=1}^{N_e}$.  This partition, or mesh, is delineated by $N_v$ vertices. The element-wise discontinuous density is represented as a vector $\ddens \in P \subset \Rbb^{N_e}$ with entries $\ddens = (\rho_1, \dots, \rho_{N_e})^T$, where $\rho_e$ is the density of element $e$ and $P \coloneqq [\rho_l,1]^{N_e} \subset \Rbb^{N_e}$ is the space of admissible densities for a lower bound $\rho_l = 10^{-3}$. The volume constraint is expressed as 
\begin{equation}
  \label{eq:vol_const}
 \sum_{e=1}^{N_e} \rho_e |\Omega_e| \leq V,
\end{equation}
where $| \Omega_e |$ is the volume of element $e$.

We model the response of the structure under the specified load condition using linear elasticity.
To this end, we introduce a density-dependent stiffness matrix $\Kbm(\rhobold) \in \Rbb^{N_\dstvcL \times N_\dstvcL}$ and a (fixed) load vector $\fbm\in\Rbb^{N_\dstvcL}$ associated with a bi-linear (in two dimensions) finite element approximation of the linear elasticity equations, where $N_\dstvcL$ is the number of global finite element degrees of freedom. 
To show the explicit dependence of the stiffness matrix on the density, we introduce the assembly operator $\asmblmatL_e \in \Rbb^{N_\dstvcL \times N_\dstvcL^e}$ such that $\vbm_e = \asmblmatL_e^T \vbm$, where $\vbm_e \in \Rbb^{N_\dstvcL^e}$ is the element displacement vector and $\vbm \in \Rbb^{N_\dstvcL}$ is the global displacement vector, for element $e = 1,\dots,N_e$. 
The stiffness matrix and load vector are represented in
terms of their element contributions as
\begin{equation} \label{eqn:linelast-res1}
 \stiffmatL(\ddens) = \sum_{e=1}^{N_e} \alpha(\cdens_e)
                      \asmblmatL_e \stiffmatL_e\asmblmatL_e^T,
 \qquad
 \rhsvecL = \sum_{e=1}^{N_e} \asmblmatL_e\rhsvecL_e,
\end{equation}
where $\alpha(\cdens_e) \coloneqq \rho_l + (1-\rho_l) \cdens_e^p$ is the material property model, $\stiffmatL_e\in\Rbb^{N_\dstvcL^e \times N_\dstvcL^e}$ is the element stiffness matrix for unity density, and $\rhsvecL_e\in\Rbb^{N_\dstvcL^e}$ is the element load vector. Both $\stiffmatL_e$ and $\rhsvecL_e$ are independent of $\rhobold$. The exponent $p$ is the penalization parameter; we choose $p=3$ which ensures good convergence to 0-1 solutions~\cite{Sigmund_2013_TO_Review}. If all
elements are the same size and shape, all element stiffness matrices are identical; this is the case for typical density-based topology optimization problems~\cite{Sigmund_2013_TO_Review}. 

Given element-wise density $\rhobold \in P$, the finite element solution $\ubm^\ast(\rhobold) \in \Rbb^{N_\dstvcL}$ must satisfy
\begin{equation}
  \label{eqn:linelastsoln}
  \resLu(\dstvcL^\ast(\rhobold); \ddens) = 0;
\end{equation}
here $\resLu: \Rbb^{N_\dstvcL} \times P \to \Rbb^{N_\dstvcL}$  is the residual operator associated with the force-equilibrium condition given by
\begin{equation} \label{eqn:linelastU}
 \resLu(\dstvcL; \ddens) \coloneqq \stiffmatL(\ddens)\dstvcL - \rhsvecL.
\end{equation}
While not explicitly stated, the degrees of freedom associated with homogeneous Dirichlet boundary conditions are eliminated in the definition of the residual, following the standard finite element formulation.

\subsection{Helmholtz density filter}
\label{sec:topopt:helm}
As noted in numerous works and summarized in a review paper~\cite{Sigmund_2013_TO_Review}, the topology optimization problem is ill-posed if we choose the density field as the design variable $\rhobold \in P$, which defines the elemental density of the stiffness matrix $\Kbm(\rhobold)$.  To stabilize the formulation, we choose a separate ``unfiltered'' density field $\psibold$ as the design variable, which implicitly defines the ``filtered'' density $\rhobold \in P$.  The unfiltered density $\psibold \in \Psi \subset \Rbb^{N_e}$ is expressed as $\psibold = (\psi_1, \dots, \psi_{N_e})^T$, where $\psi_e$ is the unfiltered density associated with element $e$ and $\Psi \coloneqq [0,1]^{N_e}$
is the space of admissible unfiltered densities. There exists many filtering procedures~\cite{Sigmund_2013_TO_Review}; in this work we consider the so-called Helmholtz filter~\cite{lazarov_filters_2011} projected onto the space of element-wise constant functions to preserve the structure of the stiffness matrix in (\ref{eqn:linelast-res1}).

We provide a brief description of the Helmholtz filter; we refer to~\cite{lazarov_filters_2011} for details.  The main idea behind the Helmholtz filter is to relate the (function representation of) the unfiltered density field $\psi$ and the filtered density field $\rho$ by the Helmholtz equation $- r^2 \Delta \phi + \phi = \psi$ with homogeneous Neumann boundary condition, where $r = 2\sqrt{3}R$ and $R$ is the characteristic radius of the filter. To this end, given an elemental unfiltered density vector $\psibold \in \Psi \subset \Rbb^{N_e}$, we first compute nodal filtered density $\phibold^\star(\psibold) \in \Rbb^{N_v}$ that satisfies a finite element approximation of the Helmholtz equation:
\begin{equation}
  \label{eqn:helm-soln}
  \resH(\dstvcH^\star(\dparam); \dparam) = 0;
\end{equation}
here $\resH : \Rbb^{N_v} \times \Psi \to \Rbb^{N_v}$ is the Helmholtz residual operator given by
\begin{equation} \label{eqn:helm-res0}
 \resH(\dstvcH; \dparam) \coloneqq \stiffmatH\dstvcH - \rhsvecH(\dparam),
\end{equation}
where $\Hbm \in \Rbb^{N_v \times N_v}$ is the assembled Helmholtz ``stiffness matrix'', and $\func{\bbm}{\Psi}{\Rbb^{N_v}}$ is an operator that maps
the unfiltered density vector to the Helmholtz ``load vector''. To show the
explicit dependence of the load vector on the density,
we introduce the Helmholtz assembly operator $\asmblmatH_e \in \Rbb^{N_v \times N_v^e}$ such that $\wbm_e = \asmblmatH_e^T \wbm$,  where $\wbm_e \in \Rbb^{N_v^e}$ is the element vector and $\wbm \in \Rbb^{N_v}$ is the global vector, for element $e = 1,\dots, N_e$.
 The stiffness matrix and load vector are represented in terms of their element contributions as
\begin{equation} \label{eqn:helm-res1}
 \stiffmatH = \sum_{e=1}^{N_e} \asmblmatH_e \stiffmatH_e\asmblmatH_e^T,
 \qquad
 \rhsvecH(\dparam) = \sum_{e=1}^{N_e} \cparam_e\asmblmatH_e\rhsvecH_e,
\end{equation}
where $\stiffmatH_e\in\Rbb^{N_v^e\times N_v^e}$ is the element Helmholtz stiffness
matrix and $\rhsvecH_e\in\Rbb^{N_v^e}$ is the element Helmholtz load vector,
both of which are independent of $\psibold$. If all
elements are the same size and shape, all element stiffness matrices are identical.  Given the nodal filtered density $\phibold \in \Rbb^{N_v}$, the elemental filtered density $\rhobold^\star(\phibold) \in P$ is given by an element-wise nodal averaging operator $\rhobold^\ast: \Rbb^{N_v} \to P$ defined as
\begin{equation}
  \label{eqn:phitorho}
  \rho_e^\ast(\phibold)
  \coloneqq
  \frac{1}{N_v^e} \onebold^T\phibold_e
 = \frac{1}{N_v^e} \onebold^T\asmblmatH_e^T\phibold , \qquad e = 1,\dots,N_e.
\end{equation}
We also introduce the filtering operator $\rhobold^\star: \Psi \to P$ from elemental unfiltered density $\psibold \in \Psi$ to elemental filtered density $\rho \in P$ given by
\begin{equation} \label{eqn:cdens-star}
  \cdens_e^\star(\psibold)
  \coloneqq
  \cdens_e^\ast(\phibold^\star(\psibold))
  =
  \frac{1}{N_v^e}\onebold^T\asmblmatH_e^T\dstvcH^\star(\psibold), \qquad e = 1,\dots,N_e. 
\end{equation}

Before we conclude this section, we recall that we wish to impose volume constraint \eqref{eq:vol_const} on the \emph{filtered} density $\rhobold$, which describes the physically relevant density distribution. 
Owing to the volume preservation property of the Green operator of the
Helmholtz equation~\cite{lazarov_filters_2011}, this volume constraint on the \emph{filtered} density field $\rhobold$ is satisfied as long as the \emph{unfiltered} density field $\psibold$ satisfies the identical volume constraint
\begin{equation*}
 \sum_{e=1}^{N_e} \psi_e |\Omega_e| \leq V.
\end{equation*}
Hence, the volume constraint can be imposed as a linear constraint on the \emph{unfiltered} density field.

\begin{remark}
    The filter used in this work differs from the original Helmholtz filter~\cite{lazarov_filters_2011} in that the nodal filtered density in $H^1(\Omega)$ is projected onto the space of element-wise constant functions in $L^2(\Omega)$. The modification is made to preserve the structure of the stiffness matrix in (\ref{eqn:linelast-res1}) and to facilitate the ROM matrix assembly.  However, the convergence theory for topology optimization in~\cite{lazarov_filters_2011}, which is an application of the theory in~\cite{bourdin_filters_2001}, does not apply to the present filter as not all regularity conditions are satisfied.  Nevertheless, in practice, we did not observe convergence issues in our numerical studies.
\end{remark}


\subsection{Optimization problem}
\label{sec:topopt:opt}
We now introduce an objective function $\func{j}{\Rbb^{N_\ubm} \times P}{\Rbb}$. Our formulation will treat general objective functions, but we assume that the function is twice continuously differentiable to develop convergence proofs.  One objective function that is of particular interest is the compliance output $\func{j_c}{\Rbb^{N_\ubm} \times P}{\Rbb}$ given by
\begin{equation*}
  j_c(\ubm; \rhobold) \coloneqq  \fbm^T\ubm ,
\end{equation*}
where $\fbm$ is the load vector given by~\eqref{eqn:linelast-res1}.

For notational convenience, we introduce the mapping $\func{\dstvcL^\star}{\Psi}{\Rbb^{N_\ubm}}$ from the unfiltered density $\psibold$ to the unconstrained elasticity displacements $\dstvcL^\star(\psibold)$ given by $\dstvcL^\star(\psibold) \coloneqq \dstvcL^\ast(\rhobold^\star(\psibold))$.  Note that $\dstvcL^\star(\psibold)$ satisfies $\resLu(\dstvcL^\star(\psibold),\ddens^\star(\dparam)) = \zerobold$ for any $\dparam \in \Psi$.  In the remainder, we will refer to the mapping $\ubm^\star$ as the (primal) high-dimensional model (HDM). We also introduce the topology optimization objective function $\func{J}{\Psi}{\Rbb}$ that maps the unfiltered density $\psibold \in \Psi$ to the output of the objective function $j : \Rbb^{N_\ubm} \times P \to \Rbb$ that properly accounts for the solution of the elasticity and Helmholtz equations, i.e.,
\begin{equation}
  \label{eq:Jfom}
  J(\dparam) \coloneqq j(\dstvcL^\star(\dparam), \ddens^\star(\dparam)).
\end{equation}
The final form of the topology optimization problem is as follows:
\begin{equation}\label{eqn:topopt0}
 \optconOne{\psibold\in \Psi \coloneqq [0,1]^{N_e}}
           {J(\psibold)}
           {\sum_{e=1}^{N_e} \psi_e |\Omega_e| \leq V.}
\end{equation}


\subsection{Derivative computations via the adjoint method}
\label{sec:topopt:deriv}

Given the large number of design variables in topology
optimization problems, its efficient solution requires gradient-based optimization methods. That is, in addition to the objective
function $J : \Psi \to \Rbb$, we also require its gradient
$\func{\nabla J}{\Psi}{\Rbb^{N_e}}$.  We here present an adjoint-based gradient expression; the derivation is provided in \ref{sec:adjoint}. (Throughout this work, we adhere to the convention where the partial derivative of scalar-valued and vector-valued functions yield a row vector and a matrix, respectively.)  The gradient can be expressed as
\begin{equation}  \label{eq:pap_sensitivity}
  \pder{J}{\psi_e} =
 \pder{\rhobold^\star}{\psi_e}^T
 \left(\pder{j}{\rhobold}^T-\pder{\rbm_\ubm}{\rhobold}^T
 \lambdabold^\star(\psibold) \right), \quad e = 1,\dots, N_e;
\end{equation}
here
\begin{equation}
  \pder{\rbm_\ubm}{\rho_e}^T  \lambdabold^\star(\psibold)
  = \alpha'(\rho_e)  \ubm^\star(\psibold)^T \Pbm_e \Kbm_e \Pbm_e^T  \lambdabold^\star(\psibold), \quad e = 1,\dots, N_e,
\end{equation}
where $\alpha'(\rho_e) = (1-\rho_l) p \rho_e^{p-1}$, and $\lambdabold^\star: \Rbb^{N_e} \to \Rbb^{N_\ubm}$ is the adjoint of the linear elasticity equation that satisfies
\begin{equation} \label{eqn:linelast-adj1}
  \Kbm(\rhobold^\star(\psibold))^T\lambdabold^\star(\psibold) =
  \pder{j}{\ubm}(\ubm(\psibold),\psibold)^T,
\end{equation}
and the application of $\pder{\rhobold^\star}{\psi_e}^T$ to any vector $\vbm \in \Rbb^{N_e}$ is given by
\begin{equation*}
  \pder{\rhobold^\star}{\psi_e}^T\vbm
  = -\bbm_e^T\Qbm_e^T\mubold(\vbm), \quad e = 1,\dots, N_e,
\end{equation*}
where $\mubold: \Rbb^{N_e} \to \Rbb^{N_v}$ is the adjoint of the Helmholtz equation that satisfies
\begin{equation} \label{eqn:helm-adj}
 \Hbm^T\mubold(\vbm) = \qbm(\vbm)
\end{equation}
for an averaging operator $\qbm : \Rbb^{N_e} \to \Rbb^{N_v}$ given by $\qbm(\vbm) = \sum_{e=1}^{N_e} \frac{v_e}{N_v^e} \Qbm_e\onebold$.  (We again note that the degrees of freedom associated with Dirichlet boundary conditions are eliminated in \eqref{eqn:linelast-adj1}.) The partial derivatives of
$\func{j}{\Rbb^{N_\ubm}\times\Rbb^{N_e}}{\Rbb}$ can be derived
analytically from the expression for $j(\ubm, \rhobold)$.  For instance, for the compliance output, we have
\begin{equation} \label{eqn:min-compl-pders}
 \pder{j_c}{\ubm}(\ubm, \rhobold)^T = \fbm, \qquad
 \pder{j_c}{\rhobold}(\ubm, \rhobold)^T = \zerobold.
\end{equation}
This shows that the primary operations required to compute the gradient
are an elasticity adjoint solve (to compute $\lambdabold^\star$)
and a Helmholtz adjoint solve (to apply $\ds{\pder{\rhobold^\star}{\psibold}}$
to a vector). We make two additional remarks.

\begin{remark}
 The linear elasticity and Helmholtz equations are self-adjoint and therefore
 their stiffness matrices are symmetric positive definite, which makes the
 transpose operations on $\Kbm$ and $\Hbm$ in
 \eqref{eqn:helm-adj} and \eqref{eqn:linelast-adj1}
 unnecessary.
\end{remark}

\begin{remark}
  \label{rem:topopt_compl}
 In the special case of compliance minimization, the gradient computation
 simplifies significantly. Due to the special form of the compliance
 function in \eqref{eqn:min-compl-pders} where
 $\ds{\pder{j}{\ubm}^T = \fbm}$ and the symmetry of the matrix $\Kbm$,
 the adjoint system in \eqref{eqn:linelast-adj1} is identical to the
 primal system in \eqref{eqn:linelastU} and therefore we have
 \begin{equation}
  \lambdabold^\star(\psibold) = \ubm^\star(\psibold).
 \end{equation}
 Due to this equivalence between the primal and adjoint elasticity state in
 this special case, the adjoint solve would be redundant. Furthermore, the
 compliance does not directly depend on the density field
 \eqref{eqn:min-compl-pders} and the gradient reduces to
 \begin{equation}
  \nabla J = -\pder{\rhobold^\star}{\psibold}^T
              \pder{\rbm_\ubm}{\rhobold}^T\ubm^\star.
 \end{equation}
 Thus the gradient computation only involves a Helmholtz adjoint solve
 to apply $\ds{\pder{\rhobold^\star}{\psibold}}$ to a vector. 
\end{remark}


\subsection{Computational cost}
\label{sec:topopt:comp}
To close this section, we analyze the cost of topology optimization.  For simplicity, we focus on compliance minimization problems for which $\lambdabold^\star(\psibold) = \ubm^\star(\psibold)$ and hence no explicit computation of the adjoint is required; see Remark~\ref{rem:topopt_compl}.  We decompose the mapping $\ubm^\star : \Psi \to \Rbb^{N_\ubm}$ from the (unfiltered) density $\psibold \in \Psi$ to the linear elasticity displacement $\ubm^\star(\psibold)$ into five steps, and analyze the cost from both the theoretical and practical perspectives; the practical assessment is based on the absolute and fractional run-time result shown in Figure~\ref{fig:timings_hdm} for a typical two-dimensional topology optimization problem (MBB beam; Section~\ref{sec:numexp:mbb}). 
\begin{enumerate}
\item The computation of $\bbm(\psibold) \in \Rbb^{N_v}$ by \eqref{eqn:helm-res1}.  This vector assembly requires $\Ocal(N_e)$ operations.  Figure~\ref{fig:timings_hdm} shows that in practice this cost is negligible.
\item The solution of the Helmholtz equation~\eqref{eqn:helm-soln} for $\phibold \in \Rbb^{N_v}$. Because the Helmholtz stiffness matrix $\Hbm \in \Rbb^{N_v \times N_v}$ does not depend on the design variable $\psibold \in \Psi$, the Cholesky factorization of the matrix can be computed for once and for all.  Given the Cholesky factors associated with a minimal-fill ordering, the marginal cost of each Helmholtz solve is $\Ocal(N_e \log(N_e))$ and $\Ocal(N_e^{4/3})$ in two- and three-dimensions, respectively. Figure~\ref{fig:timings_hdm} shows that in practice the cost grows superlinearly but is negligible.
\item The computation of $\rhobold^\ast(\phibold) \in P$ by \eqref{eqn:phitorho}. The vector assembly again requires $\Ocal(N_e)$ operations.  Figure~\ref{fig:timings_hdm} shows that in practice this cost is negligible.
\item The assembly of the linear elasticity stiffness matrix $\Kbm(\rhobold)$ by \eqref{eqn:linelast-res1}.  The assembly is done in two steps: the evaluation of the density-weighted elemental stiffness matrices $\alpha(\rho_e) \Kbm_e$; the assembly of the matrices into a sparse matrix. While both operations are $\Ocal(N_e)$ in theory, the latter operation dominates the run time on a modern computational platform as the operation requires unstructured memory access.  Figure~\ref{fig:timings_hdm} shows that this is one of the dominant costs of the map $\ubm^\star : \Psi \to \Rbb^{N_\ubm}$ in practice.
\item The solution of the linear elasticity system \eqref{eqn:linelastsoln} for $\ubm^\ast(\rhobold) \in \Rbb^{N_\ubm}$.  The operation is carried out using a sparse Cholesky solver. Using a minimal-fill ordering, the cost of the linear solve is $\Ocal(N_e^{3/2})$ and $\Ocal(N_e^2)$ in two- and three-dimensions, respectively.  Figure~\ref{fig:timings_hdm} shows that this is the dominant cost of the map $\ubm^\star : \Psi \to \Rbb^{N_\ubm}$ in practice, especially as the problem size increases.
\end{enumerate}
The cost analysis shows that the assembly and solution of the linear elasticity system accounts for roughly 98\% of the cost in computing $\ubm^\star : \Psi \to \Rbb^{N_\ubm}$ and the cost of Helmholtz filtering is negligible.  In Section~\ref{sec:rom} we introduced a surrogate model to accelerate the solution of the linear elasticity problem.


\begin{figure}
 \centering
 \input{py/timings_hdm.tikz}
 \caption{A decomposition of the run time to perform the mapping $\ubm^\star : \psibold \to \Rbb^{N_\ubm}$ for a typical two-dimensional problem:
   evaluation of $\rhsvecH(\dparam)$ (\ref{line:hdm:time_rhs_helm});
   solution of Helmholtz equation (\ref{line:hdm:time_solve_helm});
   evaluation of $\ddens^\star(\psibold)$ (\ref{line:hdm:time_rho});
   assembly of $\stiffmatL(\ddens)$ (\ref{line:hdm:time_stiff_linelast});
   solution of linear elasticity equation (\ref{line:hdm:time_solve_linelast}).}
 \label{fig:timings_hdm}
\end{figure}
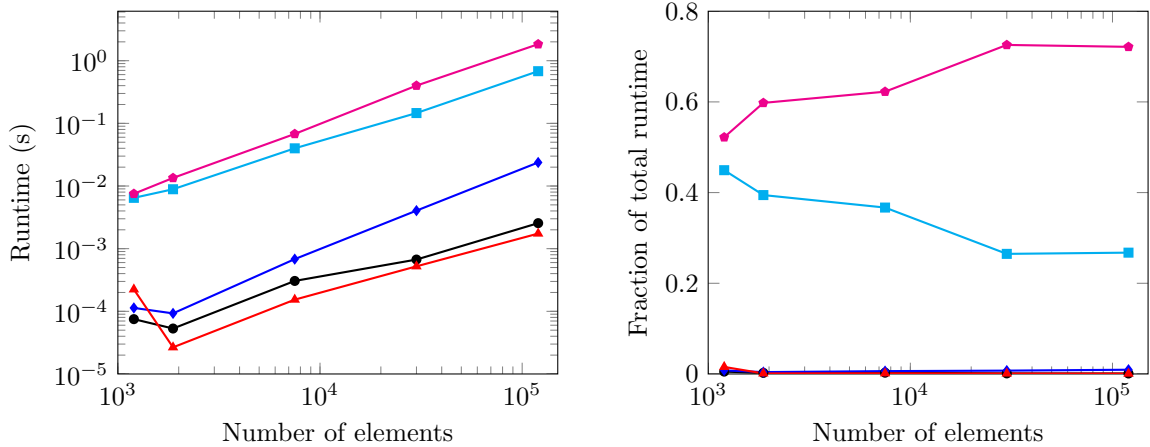

%% file: py/timings_hdm.tikz
\begin{tikzpicture}
\begin{groupplot} [
group style={group size = 2 by 1, horizontal sep = 2cm, vertical sep = 1.5cm}]
\nextgroupplot[width=0.45\textwidth, xlabel={Number of elements}, xmax=150000.0, ylabel={Runtime (s)}, xmin=1000.0, ymode=log, ymin=1e-05, xmode=log]
\addplot [solid, thick, color=black, mark=mark options={solid, thin}, mark=*, mark size=1.5, color=black]
coordinates {
( 1.20000000e+03,  7.50400000e-05)
( 1.87500000e+03,  5.31200000e-05)
( 7.50000000e+03,  3.05460000e-04)
( 3.00000000e+04,  6.70240000e-04)
( 1.20000000e+05,  2.55464000e-03)};\label{line:hdm:time_rhs_helm}

\addplot [solid, thick, color=blue, mark=mark options={solid, thin}, mark=diamond*, mark size=1.5, color=blue]
coordinates {
( 1.20000000e+03,  1.13340000e-04)
( 1.87500000e+03,  9.27000000e-05)
( 7.50000000e+03,  6.81160000e-04)
( 3.00000000e+04,  4.04590000e-03)
( 1.20000000e+05,  2.37543200e-02)};\label{line:hdm:time_solve_helm}

\addplot [solid, thick, color=red, mark=mark options={solid, thin}, mark=triangle*, mark size=1.5, color=red]
coordinates {
( 1.20000000e+03,  2.23660000e-04)
( 1.87500000e+03,  2.66200000e-05)
( 7.50000000e+03,  1.53540000e-04)
( 3.00000000e+04,  5.22200000e-04)
( 1.20000000e+05,  1.74060000e-03)};\label{line:hdm:time_rho}

\addplot [solid, thick, color=cyan, mark=mark options={solid, thin}, mark=square*, mark size=1.5, color=cyan]
coordinates {
( 1.20000000e+03,  6.47318000e-03)
( 1.87500000e+03,  8.87972000e-03)
( 7.50000000e+03,  3.99627800e-02)
( 3.00000000e+04,  1.46048800e-01)
( 1.20000000e+05,  6.80403620e-01)};\label{line:hdm:time_stiff_linelast}

\addplot [solid, thick, color=magenta, mark=mark options={solid, thin}, mark=pentagon*, mark size=1.5, color=magenta]
coordinates {
( 1.20000000e+03,  7.51846000e-03)
( 1.87500000e+03,  1.34559800e-02)
( 7.50000000e+03,  6.77943600e-02)
( 3.00000000e+04,  4.00154200e-01)
( 1.20000000e+05,  1.83468820e+00)};\label{line:hdm:time_solve_linelast}

\nextgroupplot[width=0.45\textwidth, xlabel={Number of elements}, ymax=0.8, xmax=150000.0, ylabel={Fraction of total runtime}, xmin=1000.0, ymin=0, xmode=log]
\addplot [solid, thick, color=black, mark=mark options={solid, thin}, mark=*, mark size=1.5, color=black, forget plot]
coordinates {
( 1.20000000e+03,  5.20977972e-03)
( 1.87500000e+03,  2.36003508e-03)
( 7.50000000e+03,  2.80502822e-03)
( 3.00000000e+04,  1.21543300e-03)
( 1.20000000e+05,  1.00452142e-03)};

\addplot [solid, thick, color=blue, mark=mark options={solid, thin}, mark=diamond*, mark size=1.5, color=blue, forget plot]
coordinates {
( 1.20000000e+03,  7.86882241e-03)
( 1.87500000e+03,  4.11851001e-03)
( 7.50000000e+03,  6.25506785e-03)
( 3.00000000e+04,  7.33695446e-03)
( 1.20000000e+05,  9.34054244e-03)};

\addplot [solid, thick, color=red, mark=mark options={solid, thin}, mark=triangle*, mark size=1.5, color=red, forget plot]
coordinates {
( 1.20000000e+03,  1.55279762e-02)
( 1.87500000e+03,  1.18268324e-03)
( 7.50000000e+03,  1.40995231e-03)
( 3.00000000e+04,  9.46972891e-04)
( 1.20000000e+05,  6.84429113e-04)};

\addplot [solid, thick, color=cyan, mark=mark options={solid, thin}, mark=square*, mark size=1.5, color=cyan, forget plot]
coordinates {
( 1.20000000e+03,  4.49411539e-01)
( 1.87500000e+03,  3.94511497e-01)
( 7.50000000e+03,  3.66976775e-01)
( 3.00000000e+04,  2.64849204e-01)
( 1.20000000e+05,  2.67544552e-01)};

\addplot [solid, thick, color=magenta, mark=mark options={solid, thin}, mark=pentagon*, mark size=1.5, color=magenta, forget plot]
coordinates {
( 1.20000000e+03,  5.21981882e-01)
( 1.87500000e+03,  5.97827275e-01)
( 7.50000000e+03,  6.22553176e-01)
( 3.00000000e+04,  7.25651436e-01)
( 1.20000000e+05,  7.21425955e-01)};

\end{groupplot}\end{tikzpicture}

%% file: rom.tex
\section{Projection-based model reduction for inexpensive approximation of
         linear elasticity equations}
\label{sec:rom}
In this section we introduce the projection-based reduced-order model (ROM) which will be used as the approximation model in our trust-region method. We present the Galerkin ROMs to approximate the solution $\ubm^\star(\psibold)$ and adjoint $\lambdabold^\star(\psibold)$ (Sections~\ref{sec:rom:gal} and \ref{sec:rom:adj}), residual-based error estimates for ROM solutions (Section~\ref{sec:rom:errest}), and the computational cost associated with the construction and evaluation of the ROM (Section~\ref{sec:rom:comp}).

\subsection{Galerkin reduced-order model for the primal solution}
\label{sec:rom:gal}
We first discuss the construction of the Galerkin ROM for the primal solution $\ubm^\star(\psibold)$. To this end, we introduce a reduced basis (matrix) $\Phibold_k \in \Rbb^{N_\dstvcL \times k}$; here, $k \leq N_\dstvcL$ formally, but $k \ll N_\dstvcL$ in practice.  (We will detail the construction of $\Phibold_k$ in Section~\ref{sec:tr:etr-topopt}; for now we assume it is given.)  We then represent the ROM solution $\ubm_k^\star(\psibold) \in \Rbb^{N_\dstvcL}$ as a linear combination of the reduced basis:  $\ubm_k^\star(\psibold) = \Phibold_k \hat \ubm_k^\star(\psibold)$.  The coefficients $\ubm_k^\star(\psibold)$ are given by the following residual statement: given $\psibold \in \Rbb^{N_e}$, find the ROM coefficients $\hat \ubm_k^\star(\psibold) \in \Rbb^k$ such that 
\begin{equation}
  \hat \rbm_k(\hat \ubm_k^\star(\psibold); \rhobold^\star(\psibold)) = 0 \quad \text{in } \Rbb^k,
  \label{eq:rom_soln}
\end{equation}
where the ROM residual $\hat \rbm_k: \Rbb^k \times N_e \to \Rbb^k$ is given by
\begin{equation*}
  \hat \rbm_k(\hat \wbm; \rhobold)
  \coloneqq \Phibold_k^T \rbm(\Phibold_k \hat \wbm; \rhobold)
  = \hat \Kbm_k(\rhobold) \hat \wbm - \hat \fbm
  \quad \forall \hat \wbm \in \Rbb^k, \ \forall \rhobold \in \Rbb^{N_e}, 
\end{equation*}
and the ROM stiffness matrix $\hat \Kbm_k: \Rbb^{N_e} \to \Rbb^{k \times k}$ and load vector $\hat \fbm_k \in \Rbb^{k}$ are given by
\begin{align}
  \hat \Kbm_k(\rhobold)
  &\coloneqq \Phibold_k^T \Kbm(\rhobold) \Phibold_k
  = \sum_{e=1}^{N_e} \alpha(\rho_e) (\Pbold_e^T \Phibold_k)^T \Kbm_e (\Pbold_e^T \Phibold_k) \quad \forall \rhobold \in \Rbb^{N_e}\label{eq:hatkbm} \\
  \hat \fbm_k
  &\coloneqq \Phibold_k^T \fbm
  = \sum_{e=1}^{N_e} (\Pbold_e^T \Phibold_k)^T \fbm_e.
  \notag
\end{align}
The solution to~\eqref{eq:rom_soln} satisfies a linear system: given $\psibold$, find $\hat \ubm_k^\star(\psibold)$ such that
\begin{equation*}
  \hat \Kbm_k(\rhobold^\star(\psibold)) \hat \ubm_k^\star(\psibold)
  =
  \hat \fbm_k \quad \text{in } \Rbb^k.
\end{equation*}
Finally, the ROM objective function $J_k: \Psi \to \Rbb$ is given by
\begin{equation}
  \label{eq:Jrom}
  J_k(\psibold) \coloneqq j(\ubm_k^\star(\psibold), \rhobold^\star(\psibold)).
\end{equation}
We make a few remarks:
\begin{remark}
  \label{rem:rom_wellposed}
  Assuming the columns of $\Phibold_k$ are linearly independent, the ROM stiffness matrix $\hat \Kbm_k(\rhobold)$ is symmetric positive definite and the ROM problem is well-posed.
\end{remark}
\begin{remark}
  \label{rem:rom_opt}
  The Galerkin approximation is optimal in the energy norm in the sense that, for any given $\psibold \in \Rbb^{N_e}$, 
  \begin{equation*}
    \| \ubm^\star(\psibold) - \ubm_k^\star(\psibold) \|_{\Kbm(\rhobold^\star(\psibold))}
    = \min_{\wbm_k \in \mathrm{Img}(\Phibold_k)} \| \ubm^\star(\psibold) - \wbm_k \|_{\Kbm(\rhobold^\star(\psibold))},
  \end{equation*}
  where the energy norm, which depends on $\rhobold \in \Rbb^{N_e}$, is defined by $\| \wbm \|_{\Kbm(\rhobold)} \coloneqq \sqrt{\wbm^T \Kbm(\rhobold) \wbm}$ $\forall \wbm \in \Rbb^{N_\dstvcL}$, $\forall \rhobold \in \Rbb^{N_e}$.
\end{remark}
\begin{remark}
  \label{rem:rom_refine}
  Suppose we are given a basis matrix $\Phibold_k \in \Rbb^{N_\dstvcL \times k}$, and we construct a new basis matrix $\Phibold_{k+1} = [\Phibold_k, \ybm] \in \Rbb^{N_\dstvcL \times (k+1)}$ by augmenting $\Phibold_k$ by a vector $\ybm \in \Rbb^{N_\dstvcL}$. If $\ubm_k^\star(\psibold)$ and $\ubm_{k+1}^\star(\psibold)$ are the ROM approximations associated with $\Phibold_k$ and $\Phibold_{k+1}$ respectively, then
  \begin{equation*}
    \| \ubm^\star(\psibold) - \ubm_{k+1}^\star(\psibold) \|_{\Kbm(\rhobold^\star(\psibold))} \leq \| \ubm^\star(\psibold) - \ubm_k^\star(\psibold) \|_{\Kbm(\rhobold^\star(\psibold))};
  \end{equation*}
  the energy norm of the error is a non-increasing function under basis enrichment.
\end{remark}
\begin{remark}
  \label{rem:rom_exact}
  Suppose $\Phibold_k$ is chosen such that $\ubm^\star(\psibold) \in \mathrm{Img}(\Phibold_k)$.  Then the ROM reproduces the solution $\ubm^\star(\psibold)$; i.e., $\ubm_k^\star(\psibold) = \ubm^\star(\psibold)$.
\end{remark}
\begin{remark}
  \label{rem:rom_assemble}
  The expression $\hat \Kbm_k(\rhobold) = \sum_{e=1}^{N_e} \alpha(\rho_e) (\Pbold_e^T \Phibold_k)^T \Kbm_e (\Pbold_e^T \Phibold_k)$ emphasizes that the ROM stiffness matrix $\hat \Kbm_k(\rhobold)$ can be assembled using only dense operations.  The elemental reduced basis $\Pbold_e^T \Phibold_k \in \Rbb^{N_\dstvcL^e \times k}$ comprises the rows of $\Phibold_k$ associated with the element $e$; we then perform the dense operation $(\Pbold_e^T \Phibold_k)^T \Kbm_e (\Pbold_e^T \Phibold_k)$ for each element.  The use of dense operations enables an efficient assembly of the ROM residual using BLAS operations.  The same remark applies to the construction of $\hat \fbm_k$.
\end{remark}

\subsection{Galerkin reduced-order model for the adjoint solution}
\label{sec:rom:adj}
We now discuss the construction of the ROM for the approximation of the adjoint.  We again assume that the reduced basis (matrix) $\Phibold_k \in \Rbb^{N_\dstvcL \times k}$ for $k \ll N_\dstvcL$ is given.  As before, given $\psibold \in \Rbb^{N_e}$, we construct an approximation of the form $\lambdabold_k^\star(\psibold) = \Phibold_k \hat \lambdabold_k^\star(\psibold) \in \Rbb^{N_\dstvcL}$ to the solution $\lambdabold^\star(\psibold) \in \Rbb^{N_\dstvcL}$ using Galerkin projection.  The coefficients $\hat \lambdabold_k^\star(\psibold) \in \Rbb^k$ are given by the following adjoint problem: given $\psibold \in \Rbb^{N_e}$, find the ROM coefficients $\hat \lambdabold_k^\star(\psibold) \in \Rbb^k$ such that
\begin{equation*}
  \hat \Kbm_k(\rhobold^\star(\psibold)) \hat \lambdabold_k^\star(\psibold) = \hat \gbm_k(\ubm_k^\star(\psibold),\rhobold^\star(\psibold)) \quad \text{in } \Rbb^k,
\end{equation*}
where the ROM stiffness matrix $\hat \Kbm_k: \Rbb^{N_e} \to \Rbb^{k \times k}$ is given by \eqref{eq:hatkbm} and the adjoint load vector
$\func{\hat\gbm_k}{\Rbb^k\times\Rbb^{N_e}}{\Rbb^k}$ is given by
\begin{equation*}
  \hat \gbm_k(\wbm_k,\rhobold)
  \coloneqq \Phibold_k^T \pder{j}{\ubm}(\wbm_k,\rhobold)^T
  = \sum_{e=1}^{N_e} (\Pbold_e^T \Phibold_k)^T \pder{j}{\ubm_e}(\Pbold_e^T \wbm_k,\rho_e)^T \quad \forall \wbm_k \in \Rbb^k, \ \forall \rhobold \in \Rbb^{N_e}.
\end{equation*}
For the compliance output, the right hand side simplifies to $\hat \gbm_k(\ubm_k^\star(\psibold),\rhobold^\star(\psibold)) = \hat \fbm_k$ and hence $\lambdabold_k^\star(\psibold) = \ubm_k^\star(\psibold)$.
\begin{remark}
  \label{rem:rom_adjoint}
  Remarks~\ref{rem:rom_opt}-\ref{rem:rom_assemble} also hold for the adjoint system.  Namely, the Galerkin approximation is optimal in the energy norm in the sense that, for any given $\psibold \in \Rbb^{N_e}$, 
  \begin{equation*}
    \| \lambdabold^\star(\psibold) - \lambdabold_k^\star(\psibold) \|_{\Kbm(\rhobold^\star(\psibold))} = \min_{\vbm_k \in \mathrm{Img}(\Phibold_k)} \| \lambdabold^\star(\psibold) - \vbm_k \|_{\Kbm(\rhobold^\star(\psibold))},
  \end{equation*}
  which is the counterpart of Remark~\ref{rem:rom_opt}.
  The optimality implies that the energy norm of the error in the adjoint is a non-increasing function under basis enrichment (Remark~\ref{rem:rom_refine}).  In addition, if $\Phibold_k$ is chosen such that $\ubm^\star(\psibold) \in \mathrm{Img}(\Phibold_k)$ \emph{and} $\lambdabold^\star(\psibold) \in \mathrm{Img}(\Phibold_k)$, then the ROM reproduces the adjoint $\lambdabold^\star(\psibold)$; i.e., $\lambdabold_k^\star(\psibold) = \lambdabold^\star(\psibold)$ (Remark~\ref{rem:rom_exact}).  Note that we require both $\ubm^\star(\psibold)$ and $\lambdabold^\star(\psibold)$ to be in $\mathrm{Img}(\Phibold_k)$, because the adjoint equation itself depends on $\ubm_k^\star(\psibold)$ and the adjoint would not be exact if the primal solution $\ubm_k^\star(\psibold)$ is not exact.  Finally, the adjoint system can be efficiently assembled using dense linear algebra operations (Remark~\ref{rem:rom_assemble}).
\end{remark}

\subsection{Error estimation}
\label{sec:rom:errest}
We now introduce error estimates for our ROM approximation of the solution, adjoint, objective function, and objective function gradient.  To begin, we introduce the key ingredients of our (\textit{a posteriori}) error estimates.  The (primal) residual, which has been defined earlier, is given by
\begin{equation*}
  \rbm(\wbm; \rhobold^\star(\psibold))
  \coloneqq
  \Kbm(\rhobold^\star(\psibold)) \wbm - \fbm,
  \quad \forall \wbm \in \Rbb^{N_\dstvcL}, \ \forall \psibold \in \Rbb^{N_e}.
\end{equation*}
The adjoint residual is given by
\begin{equation*}
  \rbm^{\mathrm{adj}}(\vbm; \rhobold^\star(\psibold))
  \coloneqq
  \Kbm(\rhobold^\star(\psibold)) \vbm - \pder{j}{\ubm}(\ubm^\star_k(\psibold) ; \rhobold^\star(\psibold)  )
  \quad \forall \vbm \in \Rbb^{N_\dstvcL}, \  \forall \psibold \in \Rbb^{N_e};
\end{equation*}
the adjoint residual depends on both the adjoint approximation $\vbm$ and the linearization point of the ROM adjoint problem $\ubm^\star_k(\psibold)$.  In addition, given any arbitrary matrix $\Abm \in \Rbb^{m \times n}$, we denote its maximum and minimum singular values by $\sigma_{\mathrm{max}}(\Abm)$ and $\sigma_{\mathrm{min}}(\Abm)$, respectively.  Finally, given a symmetric positive definite matrix $\Abm \in \Rbb^{m \times m}$, the symmetric positive definite matrix $\Abm^{1/2}$ is the matrix that satisfies $\Abm = \Abm^{1/2} \Abm^{1/2}$.

We now state a series of theorems that relates various errors to the residuals; the associated proofs are provided in~\ref{sec:rom_proofs}.
\begin{theorem}
  \label{thm:rom_primal_error}
  For any given $\psibold \in \Rbb^{N_e}$, the energy norm of the error in the solution is bounded by
\begin{equation*}
  \| \ubm^\star(\psibold) - \ubm_k^\star(\psibold) \|_{\Kbm(\rhobold^\star(\psibold))}
  \leq \sigma_{\mathrm{min}}(\Kbm(\rhobold^\star(\psibold)))^{-1/2} \| \rbm(\ubm_k^\star(\psibold); \rhobold^\star(\psibold)) \|_2.
\end{equation*}
\end{theorem}
\begin{theorem}
  \label{thm:rom_adjoint_error}
  For any given $\psibold \in \Rbb^{N_e}$, the energy norm of the error in the adjoint is bounded by
  \begin{equation*}
    \| \lambdabold^\star(\psibold) - \lambdabold_k^\star(\psibold) \|_{\Kbm(\rhobold^\star(\psibold))}
    \leq \sigma_{\mathrm{min}}(\Kbm(\rhobold^\star(\psibold)))^{-1/2}  \| \rbm^{\mathrm{adj}}(\lambdabold_k^\star(\psibold); \rhobold^\star(\psibold)) \|_2
    + \sigma_{\mathrm{max}}(\Abm)\| \rbm(\ubm_k^\star(\psibold); \rhobold^\star(\psibold)) \|_2,
  \end{equation*}
  where 
  \begin{equation}
    \Abm
    \coloneqq
    \Kbm^{-1/2}  \left[ \int_{\theta = 0}^1 \pder{^2j}{\ubm^2}(\theta \ubm^\star + (1- \theta) \ubm^\star_k) d\theta \right] \Kbm^{-1},
    \label{eq:rom_Abm}
  \end{equation}
  with all variables and forms evaluated about $\psibold$ and $\rhobold^\star(\psibold)$.
\end{theorem}
\begin{theorem}
  \label{thm:rom_output_error}
  For any given $\psibold \in \Rbb^{N_e}$, the error in the objective function is bounded by
\begin{align*}
  | J(\psibold) - J_k(\psibold) |
  &\leq
  \sigma_{\rm min}(\Kbm(\rhobold^\star(\psibold)))^{-1} \| \rbm(\ubm_k^\star(\psibold); \rhobold^\star(\psibold)) \|_2\| \rbm^{\mathrm{adj}}(\lambdabold_k^\star(\psibold); \rhobold^\star(\psibold)) \|_2
  \\
  &\quad + \sigma_{\rm max}(\Bbm) \| \rbm(\ubm_k^\star(\psibold); \rhobold^\star(\psibold)) \|_2^2,
\end{align*}
where 
\begin{equation}
    \Bbm \coloneqq \Kbm^{-1} \left[ \int_{\theta = 0}^1 \pder{^2j}{\ubm^2}(\theta \ubm^\star + (1- \theta) \ubm^\star_k) \theta d\theta \right] \Kbm^{-1}
    \label{eq:rom_Bbm}
\end{equation}
  with all variables and forms evaluated about $\psibold$ and $\rhobold^\star(\psibold)$.
\end{theorem}
\begin{theorem}
  \label{thm:rom_sensitivity_error}
  For any given $\psibold \in \Rbb^{N_e}$, the error in the objective function gradient is bounded by
  \begin{align*}
    \| \nabla J(\psibold) - \nabla J_k(\psibold) \|_2
    \leq \sigma_{\mathrm{max}}(\Cbm) \| \rbm(\ubm_k^\star(\psibold); \rhobold^\star(\psibold)) \|_2
    + \sigma_{\mathrm{max}}(\Dbm) \| \rbm^{\mathrm{adj}}(\lambdabold_k^\star(\psibold); \rhobold^\star(\psibold)) \|_2,
  \end{align*}
  where the entries of matrices $\Cbm \in \Rbb^{N_e \times N_\dstvcL}$ and $\Dbm \in \Rbb^{N_e \times N_\dstvcL}$ are given by
  \begin{align}
    \Cbm_{pm}
    &=
    - \pder{\rhobold^\star_q}{\psibold_p}
    \left[
    \int_{\theta=0}^1 \pder{^2j}{\rhobold_q \partial \ubm_s}(\theta \ubm^\star + (1-\theta)\ubm^\star_k) d\theta
    - \lambdabold^\star_{k,i} \int_{\theta=0}^1 \pder{^2\rbm_i}{\rhobold_q \partial \ubm_s}(\theta \ubm^\star + (1-\theta)\ubm^\star_k) d\theta \right.
    \\
    &\qquad
    \left.
    - \pder{\rbm_i}{\rhobold_q}(\ubm^\star) (\Kbm^{-1})_{il}
    \int_{\theta=0}^1 \pder{^2j}{\ubm_l \partial \ubm_s}(\theta \ubm^\star + (1-\theta)\ubm^\star_k) d\theta
    \right]
    \Kbm^{-1}_{sm} 
    , \quad p = 1,\dots,N_e, \ m = 1,\dots,N_\dstvcL, \label{eq:rom_Cbm}
    \\
    \Dbm_{pm}
    &=
    \pder{\rhobold^\star_q}{\psibold_p}
    \pder{\rbm_i}{\rhobold_q}(\ubm^\star) (\Kbm^{-1})_{im} ,
    \quad p = 1,\dots,N_e, \ m = 1,\dots,N_\dstvcL, 
    \label{eq:rom_Dbm}
  \end{align}
  where all variables and forms evaluated about $\psibold$ and $\rhobold^\star(\psibold)$, and the summation on repeated indices is implied.
\end{theorem}
\begin{remark}
  If the objective function $j: \Rbb^{N_\dstvcL} \to \Rbb$ is linear, then the matrix $\Abm$ given by \eqref{eq:rom_Abm} and consequently $\Bbm$ vanishes.  As a result, the adjoint and output error estimates in Theorems~\ref{thm:rom_adjoint_error} and \ref{thm:rom_output_error} simplify to 
  \begin{align*}
  \| \lambdabold^\star(\psibold) - \lambdabold_k^\star(\psibold) \|_{\Kbm(\rhobold^\star(\psibold))}
  &\leq \sigma_{\rm min}(\Kbm(\rhobold^\star(\psibold)))^{-1}\| \rbm^{\mathrm{adj}}(\lambdabold_k^\star(\psibold); \rhobold^\star(\psibold)) \|_2,
  \\
  | J(\psibold) - J_k(\psibold) |
  &\leq \sigma_{\rm min}(\Kbm(\rhobold^\star(\psibold)))^{-1} \| \rbm(\ubm_k^\star(\psibold); \rhobold^\star(\psibold)) \|_2\| \rbm^{\mathrm{adj}}(\lambdabold_k^\star(\psibold); \rhobold^\star(\psibold)) \|_2.
  \end{align*}
  In addition, in Theorem \ref{thm:rom_sensitivity_error}, the last term of $\Cbm$ in \eqref{eq:rom_Cbm} vanishes.
\end{remark}
\begin{remark}
  For the compliance output, the objective value and gradient error estimates further simplify to
  \begin{align*}
  | J(\psibold) - J_k(\psibold) |
  &\leq \sigma_{\rm min}(\Kbm(\rhobold^\star(\psibold)))^{-1} \| \rbm(\ubm_k^\star(\psibold); \rhobold^\star(\psibold)) \|_2^2, \\
  \| \nabla J(\psibold) - \nabla J_k(\psibold) \|_2
  &\leq
  (\sigma_{\mathrm{max}}(\Cbm) + \sigma_{\mathrm{max}}(\Dbm) ) \| \rbm(\ubm_k^\star(\psibold); \rhobold^\star(\psibold)) \|_2
  \end{align*}
  because $\lambdabold_k^\star(\psibold) = \ubm_k^\star(\psibold)$ and $\rbm(\cdot; \cdot) = \rbm^{\mathrm{adj}}(\cdot; \cdot)$.
\end{remark}
We may also obtain a ``looser-version'' of Theorem~\ref{thm:rom_output_error}, which does not involve the adjoint residual:
\begin{theorem}
  \label{thm:rom_output_error_simple}
  For any given $\psibold \in \Rbb^{N_e}$, the error in the objective function is bounded by
  \begin{align*}
    | J(\psibold) - J_k(\psibold) |
    \leq \| \lambdabold^\star \|_2 \| \rbm(\ubm_k^\star(\psibold); \rhobold^\star(\psibold)) \|_2 + \sigma_{\rm max}(\Bbm) \| \rbm(\ubm_k^\star(\psibold); \rhobold^\star(\psibold)) \|_2^2,
  \end{align*}
  where $\Bbm$ is given by \eqref{eq:rom_Bbm}.
\end{theorem}
Theorems \ref{thm:rom_primal_error}--\ref{thm:rom_output_error_simple} show that all quantities that are relevant in topology optimization---and in particular the objective function value and the associated gradient ---are well approximated by the ROM as long as the primal and adjoint residuals are small.  In other words, the primal and adjoint residuals serve as an indicator of the errors in the ROM approximations (up to a constant).  We will leverage this observation to devise residual-aware ROM-accelerated topology optimization strategies in Section~\ref{sec:tr}.

\subsection{Computational cost}
\label{sec:rom:comp}
We now assess the (online) computational cost of the ROM solution and the associated residual.

\emph{Solution evaluation.} We first analyze the cost of ROM solution evaluation: $\rhobold^\star(\psibold) \mapsto \hat \ubm_k(\psibold)$.  We consider the evaluation of the primal solution; the cost for the adjoint solution can be analyzed in a similar manner.  Given the density distribution $\rhobold^\star(\psibold) \in \Rbb^{N_e}$ and a reduced basis matrix $\Phibold_k \in \Rbb^{N_\dstvcL \times k}$, we decompose the computation of the ROM solution $\ubm_k(\rhobold) \in \Rbb^{N_\dstvcL}$ into three steps and assess the associated costs:
\begin{enumerate}
\item Assembly of the ROM stiffness matrix and load vector: given $\rhobold^\star(\psibold) \in \Rbb^{N_e}$, assemble
  \begin{equation*}
  \hat \Kbm_k(\rhobold^\star(\psibold))
  = \sum_{e=1}^{N_e} \alpha(\rho_e^{\star}(\psibold)) (\Pbold_e^T \Phibold_k)^T \Kbm_e (\Pbold_e^T \Phibold_k) \quad \mathrm{and} \quad
  \hat \fbm_k
  = \sum_{e=1}^{N_e} (\Pbold_e^T \Phibold_k)^T \fbm_e.
  \end{equation*}
  The computation of $\hat \Kbm_k(\rhobold^\star(\psibold))$ dominates the cost of this step.  The operation count is $N_e( 2 k (N_\dstvcL^e)^2 + 2 k^2 N_\dstvcL^e)$; assuming $k \gtrsim N_\dstvcL^e$, the cost is hence $\approx 2 N_\dstvcL^e k^2 N_e$. For $\Qbb^1$ finite elements in two and three dimensions, the cost is $\approx 16 k^2 N_e$ and $48 k^2 N_e$, respectively.  The cost of this step scales linearly with the number of finite elements $N_e$ and quadratically with the dimension of the reduced basis space $k$.  We also note that $\hat \fbm_k$ is independent of $\rhobold^\star(\psibold)$, and hence $\hat \fbm_k$ needs to be recomputed only when the reduced basis $\Phibold_k$ is updated; this is unlike $\hat \Kbm_k(\rhobold^\star(\psibold))$ which must be recomputed for every new density $\rhobold^\star(\psibold)$ encountered during the optimization procedure.
\item Solution of the ROM linear system: find $\hat \ubm_k(\psibold) \in \Rbb^k$ such that
  \begin{equation*}
    \hat \Kbm_k(\rhobold^\star(\psibold)) \hat \ubm_k^\star(\psibold) = \hat \fbm_k \quad \text{in } \Rbb^k.
  \end{equation*}
  The solution of the linear system requires $\approx k^3/3$ operations using Cholesky factorization.
\item Representation of the solution in the original vector space
  \begin{equation*}
    \ubm_k^\star(\psibold) = \Phibold_k \hat \ubm_k^\star(\psibold).
  \end{equation*}
  The multiplication requires $2 N_\dstvcL k$ operations.  For $\Qbb^1$ finite element in two and three dimensions, $N_\dstvcL \approx dN_e$ and hence the cost of this step is $\approx 4 k N_e$ and $6 k N_e$, respectively.  The cost of this step scales linearly with both $N_e$ and $k$.
\end{enumerate}
The first step dominates the overall computational cost, and hence the cost of finding the ROM solution is approximately $2 (N_\dstvcL^e) k^2 N_e$ (assuming $k \gtrsim N_\dstvcL^e)$; for $\Qbb^1$ finite element in two and dimensions, the cost is $\approx 16 k^2 N_e$ and $48 k^2 N_e$, respectively.

\emph{Residual evaluation.} We next analyze the cost of evaluating the 2-norm of the residual: $\hat \ubm_k^\star(\psibold) \times \rhobold^\star(\psibold) \mapsto \| \rbm(\Phibold_k \hat \ubm^\star_k(\psibold); \rhobold^\star(\psibold))\|_2$.  In particular, our interest is in the \emph{marginal} cost of the residual evaluation for many different densities $\rho^\star(\psibold)$ and ROM coefficients $\hat \ubm_k(\psibold)$ for a fixed reduced basis $\Phibold_k$.  To minimize the \emph{marginal} cost, we first precompute the density-independent quantity
\begin{equation*}
  \Abm_e = \Kbm_e (\Pbm_e^T \Phibold_k) \quad \text{in } \Rbb^{N_\dstvcL^e \times k}
\end{equation*}
for each $e = 1,\dots,N_e$.  This computation requires $(N_\dstvcL^e)^2 k N_e$ operations, but can be performed once and for all for a fixed $\Phibold_k$.  Given the precomputed quantity $\Abm_e$, the computation of the residual for each $\rhobold^\star(\psibold)$ and the associated solution $\hat \ubm_k(\psibold)$ is performed in two steps:
\begin{enumerate}
\item Construction of the residual vector $\rbm(\ubm^\star_k(\psibold); \rhobold^\star(\psibold))$: compute
  \begin{equation*}
    \rbm(\ubm^\star_k(\psibold); \rhobold^\star(\psibold)) =
    \sum_{e = 1}^{N_e} \Pbm_e [ \alpha(\rho_e^\star(\psibold)) \Abm_e \hat \ubm_k^\star(\psibold) ]  - \fbm
        \quad \text{in } \Rbb^{N_\dstvcL}.
  \end{equation*}
  The two major costs of the evaluation are associated with (i) the computation of $[  \alpha(\rho_e^\star(\psibold)) \Abm_e \ubm_k(\psibold) ] $, which requires $\approx N_\dstvcL^e k N_e$ operations, and (ii) the assembly $\sum_{e=1}^{N_e} \Pbm_e [\cdot]$, which is $\Ocal(N_\dstvcL^e N_e)$ (though the actual cost is highly dependent on the implementation and computer architecture).  For a sufficiently large $k$, the overall cost is $\approx N_\dstvcL^e k N_e$; for $\Qbb^1$ finite element in two and three dimensions, the cost is $\approx 8 k N_e$ and $24 k N_e$.
\item Evaluation of the 2-norm $\| \rbm(\ubm^\star_k(\psibold); \rhobold^\star(\psibold)) \|_2$.  This cost of the norm evaluation is $\approx 2 N_\dstvcL$.  For $\Qbb^1$ finite elements in two and three dimensions, the cost is $\approx 4 N_e$ and $6 N_e$, respectively.
\end{enumerate}
The first step dominates the overall computational cost, and hence the marginal cost of evaluating the residual is approximately $\approx N_\dstvcL^e k N_e$, which evaluates to $\approx 8 k N_e$ and $24 k N_e$ for $\Qbb^1$ finite element in two and three dimensions.  For a sufficiently large $k$, the marginal cost of residual evaluation is insignificant compared to the cost of the solution evaluation.

We compare the cost of the ROM analysis to the finite element analysis. For the finite element analysis, we consider the MBB beam configuration in Section~\ref{sec:numexp:mbb}, which has three times as many $\Qbb^1$ elements in the first dimension as in the second, and we employ a direct solver with an asymptotically optimal (i.e., nested dissection) ordering.  We recall from Section~\ref{sec:topopt:comp} that the computational cost for the finite element analysis in two and three dimensions are $\Ocal(N_e^{3/2})$ and $\Ocal(N_e^2)$, respectively.  (We also note that the storage requirement in two and three dimensions are $\Ocal(N_e \log(N_e))$ and $\Ocal(N_e^{4/3})$.)

To provide a concrete cost assessment of the ROM approximation in topology optimization, we show in Figure~\ref{fig:rom_cost} a breakdown of the ROM runtime as a fraction of the HDM runtime.  Consistent with the theory, for a sufficiently large $k$ the dominant cost of the ROM solve \emph{and} residual evaluation is associated with the construction of the ROM stiffness matrix; nevertheless, the runtime for the assembly is a small fraction of the HDM solve, varying from $\approx 0.3\%$ for $k = 5$ to $\approx 1\%$ for $k = 20$.  In other words, we can perform $\approx 100$ ROM solves and residual evaluations in the time it takes to complete a single HDM solve. 
\begin{figure}
 \centering
 \input{py/timings_rom.tikz}
 \caption{A decomposition of the runtime to perform the mappings
          $\ubm_k^\ast : \rhobold \mapsto \Rbb^{N_\dstvcL}$ and
          $(\ubm, \rhobold) \mapsto \norm{\rbm(\ubm,\rhobold)}_2$ for
          a typical two-dimensional problem as a function of reduced basis
          size for a fixed mesh with $600\times 200$ $\Qbb^1$ elements
          (\textit{left}) and as a function of mesh size for
          reduced basis sizes of $k\in [5, 20]$ (\textit{right}):
          evaluation and assembly of $\hat\Kbm_k(\rhobold)$
          (\ref{line:rom:time_stiff}); solution of reduced elasticity
          equations (\ref{line:rom:time_solve}); reconstruction of
          $\ubm_k = \Phibold_k \hat\ubm_k$ (\ref{line:rom:time_recon});
          and evaluation of $\norm{\rbm(\ubm,\rhobold)}_2$
          (\ref{line:rom:time_resnm}). In the right plot,
          for a fixed color, the lower line corresponds to a reduced basis
          of size $k=5$, the upper line corresponds to $k=20$, and the shaded
          region corresponds to $k\in(5,20)$.}
 \label{fig:rom_cost}
\end{figure}
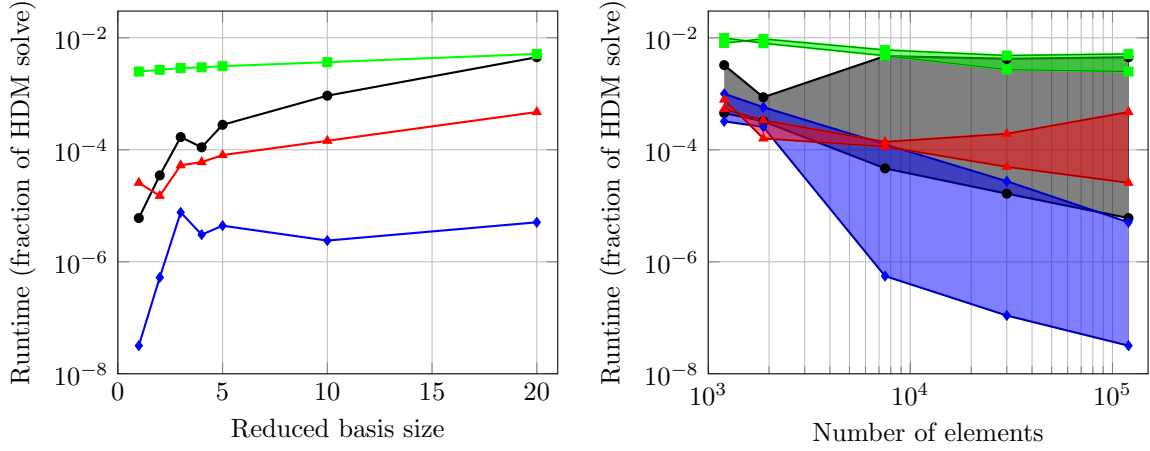

%% file: py/timings_rom.tikz
\begin{tikzpicture}
\begin{groupplot} [
group style={group size = 2 by 1, horizontal sep = 2cm, vertical sep = 1.5cm}]
\nextgroupplot[width=0.45\textwidth, ytick={1e-8, 1e-6, 1e-4, 1e-2}, xlabel={Reduced basis size}, ymax=0.03, xmax=21.0, ylabel={Runtime (fraction of HDM solve)}, xmin=0.0, grid=both, ymin=1e-08, ymode=log]
\addplot [solid, thick, color=black, mark=mark options={solid, thin}, mark=*, mark size=1.5, color=black]
coordinates {
( 1.00000000e+00,  6.03556494e-06)
( 2.00000000e+00,  3.49092623e-05)
( 3.00000000e+00,  1.68956058e-04)
( 4.00000000e+00,  1.10826968e-04)
( 5.00000000e+00,  2.79393378e-04)
( 1.00000000e+01,  9.25254490e-04)
( 2.00000000e+01,  4.49566092e-03)};\label{line:rom:time_stiff}

\addplot [solid, thick, color=blue, mark=mark options={solid, thin}, mark=diamond*, mark size=1.5, color=blue]
coordinates {
( 1.00000000e+00,  3.18079839e-08)
( 2.00000000e+00,  5.24831734e-07)
( 3.00000000e+00,  7.64186812e-06)
( 4.00000000e+00,  3.05356645e-06)
( 5.00000000e+00,  4.41335776e-06)
( 1.00000000e+01,  2.39355078e-06)
( 2.00000000e+01,  5.07337342e-06)};\label{line:rom:time_solve}

\addplot [solid, thick, color=red, mark=mark options={solid, thin}, mark=triangle*, mark size=1.5, color=red]
coordinates {
( 1.00000000e+00,  2.57644669e-05)
( 2.00000000e+00,  1.50849363e-05)
( 3.00000000e+00,  5.28887252e-05)
( 4.00000000e+00,  6.02443214e-05)
( 5.00000000e+00,  8.04741991e-05)
( 1.00000000e+01,  1.44837654e-04)
( 2.00000000e+01,  4.72213376e-04)};\label{line:rom:time_recon}

\addplot [solid, thick, color=green, mark=mark options={solid, thin}, mark=square*, mark size=1.5, color=green]
coordinates {
( 1.00000000e+00,  2.49741180e-03)
( 2.00000000e+00,  2.67783464e-03)
( 3.00000000e+00,  2.88709937e-03)
( 4.00000000e+00,  2.96227754e-03)
( 5.00000000e+00,  3.11381872e-03)
( 1.00000000e+01,  3.66843068e-03)
( 2.00000000e+01,  5.14389968e-03)};\label{line:rom:time_resnm}

\nextgroupplot[ymax=0.03, xmode=log, ymode=log, width=0.45\textwidth, ytick={1e-8, 1e-6, 1e-4, 1e-2}, xlabel={Number of elements}, xmax=150000.0, ylabel={Runtime (fraction of HDM solve)}, xmin=1000.0, grid=both, ymin=1e-08]
\addplot [solid, thick, color=black, mark=mark options={solid, thin}, mark=*, mark size=1.5, color=black, forget plot]
coordinates {
( 1.20000000e+03,  4.53127725e-04)
( 1.87500000e+03,  3.22353900e-04)
( 7.50000000e+03,  4.65862401e-05)
( 3.00000000e+04,  1.64773903e-05)
( 1.20000000e+05,  6.03556494e-06)};

\addplot [solid, thick, color=black, mark=mark options={solid, thin}, mark=*, mark size=1.5, color=black, forget plot]
coordinates {
( 1.20000000e+03,  3.25337130e-03)
( 1.87500000e+03,  8.67669247e-04)
( 7.50000000e+03,  4.72005846e-03)
( 3.00000000e+04,  4.21978642e-03)
( 1.20000000e+05,  4.49566092e-03)};

\addplot [opacity=0.5, fill=black, forget plot]
coordinates {
( 1.20000000e+03,  4.53127725e-04)
( 1.87500000e+03,  3.22353900e-04)
( 7.50000000e+03,  4.65862401e-05)
( 3.00000000e+04,  1.64773903e-05)
( 1.20000000e+05,  6.03556494e-06)
( 1.20000000e+05,  4.49566092e-03)
( 3.00000000e+04,  4.21978642e-03)
( 7.50000000e+03,  4.72005846e-03)
( 1.87500000e+03,  8.67669247e-04)
( 1.20000000e+03,  3.25337130e-03)};

\addplot [solid, thick, color=blue, mark=mark options={solid, thin}, mark=diamond*, mark size=1.5, color=blue, forget plot]
coordinates {
( 1.20000000e+03,  3.21620625e-04)
( 1.87500000e+03,  2.56092265e-04)
( 7.50000000e+03,  5.56807651e-07)
( 3.00000000e+04,  1.09849268e-07)
( 1.20000000e+05,  3.18079839e-08)};

\addplot [solid, thick, color=blue, mark=mark options={solid, thin}, mark=diamond*, mark size=1.5, color=blue, forget plot]
coordinates {
( 1.20000000e+03,  9.99168075e-04)
( 1.87500000e+03,  5.75759882e-04)
( 7.50000000e+03,  1.25096119e-04)
( 3.00000000e+04,  2.74623171e-05)
( 1.20000000e+05,  5.07337342e-06)};

\addplot [opacity=0.5, fill=blue, forget plot]
coordinates {
( 1.20000000e+03,  3.21620625e-04)
( 1.87500000e+03,  2.56092265e-04)
( 7.50000000e+03,  5.56807651e-07)
( 3.00000000e+04,  1.09849268e-07)
( 1.20000000e+05,  3.18079839e-08)
( 1.20000000e+05,  5.07337342e-06)
( 3.00000000e+04,  2.74623171e-05)
( 7.50000000e+03,  1.25096119e-04)
( 1.87500000e+03,  5.75759882e-04)
( 1.20000000e+03,  9.99168075e-04)};

\addplot [solid, thick, color=red, mark=mark options={solid, thin}, mark=triangle*, mark size=1.5, color=red, forget plot]
coordinates {
( 1.20000000e+03,  7.87613175e-04)
( 1.87500000e+03,  1.59386095e-04)
( 7.50000000e+03,  1.14331171e-04)
( 3.00000000e+04,  4.97251022e-05)
( 1.20000000e+05,  2.57644669e-05)};

\addplot [solid, thick, color=red, mark=mark options={solid, thin}, mark=triangle*, mark size=1.5, color=red, forget plot]
coordinates {
( 1.20000000e+03,  5.31746100e-04)
( 1.87500000e+03,  3.37576167e-04)
( 7.50000000e+03,  1.38645105e-04)
( 3.00000000e+04,  1.93407945e-04)
( 1.20000000e+05,  4.72213376e-04)};

\addplot [opacity=0.5, fill=red, forget plot]
coordinates {
( 1.20000000e+03,  7.87613175e-04)
( 1.87500000e+03,  1.59386095e-04)
( 7.50000000e+03,  1.14331171e-04)
( 3.00000000e+04,  4.97251022e-05)
( 1.20000000e+05,  2.57644669e-05)
( 1.20000000e+05,  4.72213376e-04)
( 3.00000000e+04,  1.93407945e-04)
( 7.50000000e+03,  1.38645105e-04)
( 1.87500000e+03,  3.37576167e-04)
( 1.20000000e+03,  5.31746100e-04)};

\addplot [solid, thick, color=green, mark=mark options={solid, thin}, mark=square*, mark size=1.5, color=green, forget plot]
coordinates {
( 1.20000000e+03,  9.91735065e-03)
( 1.87500000e+03,  8.00691270e-03)
( 7.50000000e+03,  4.78576176e-03)
( 3.00000000e+04,  2.70525794e-03)
( 1.20000000e+05,  2.49741180e-03)};

\addplot [solid, thick, color=green, mark=mark options={solid, thin}, mark=square*, mark size=1.5, color=green, forget plot]
coordinates {
( 1.20000000e+03,  8.05052160e-03)
( 1.87500000e+03,  9.53451201e-03)
( 7.50000000e+03,  6.05602561e-03)
( 3.00000000e+04,  4.82150409e-03)
( 1.20000000e+05,  5.14389968e-03)};

\addplot [opacity=0.5, fill=green, forget plot]
coordinates {
( 1.20000000e+03,  9.91735065e-03)
( 1.87500000e+03,  8.00691270e-03)
( 7.50000000e+03,  4.78576176e-03)
( 3.00000000e+04,  2.70525794e-03)
( 1.20000000e+05,  2.49741180e-03)
( 1.20000000e+05,  5.14389968e-03)
( 3.00000000e+04,  4.82150409e-03)
( 7.50000000e+03,  6.05602561e-03)
( 1.87500000e+03,  9.53451201e-03)
( 1.20000000e+03,  8.05052160e-03)};

\end{groupplot}\end{tikzpicture}

%% file: tr.tex
\section{Globally convergent, error-aware trust-region method for
         efficient topology optimization}
\label{sec:tr}
In this section, we introduce an extension of the error-aware trust-region
method developed in \cite{zahr_phd_2016} for unconstrained
problems to problems with convex constraints. We specialize
the method to the present topology optimization setting and use the
projection-based ROM introduced in Section~\ref{sec:rom} as the trust-region
model and the residual as the error-aware trust-region constraint.
Global convergence theory for the proposed method is provided in
\ref{sec:tr-proofs}.

\subsection{Error-aware trust-region method}
\label{sec:tr:etr}
We begin by introducing the proposed error-aware trust-region method
as a general method to solve an optimization problem with convex
constraints. To this end, consider the following abstract optimization problem on a convex set $\Ccal \subset \Rbb^N$:
\begin{equation} \label{eqn:optim-gen}
 \optunc{\xbm\in\Ccal}{F(\xbm),}
\end{equation}
where $\func{F}{\Ccal}{\Rbb}$ and $\Ccal$ satisfy the following assumptions:
\begin{assume}[Constraints]~
  \label{assume:con}
  The constrained set $\Ccal \subset \Rbb^N$ satisfies the following conditions:
\begin{enumerate}[1)]
 \item\label{assume:con:c2} $\Ccal = \cap_{i=1}^m \Ccal_i$, where $\Ccal_i=\{\xbm\in\Rbb^N\mid c_i(\xbm)\geq 0\}$ and each $\func{c_i}{\Rbb^N}{\Rbb}$ is twice continuously differentiable on $\Rbb^N$;
 \item\label{assume:con:cvx} $\Ccal$ is nonempty, closed, convex; and
 \item\label{assume:con:qual} a first-order constraint qualification holds at any critical point of (\ref{eqn:optim-gen}).
\end{enumerate}
\end{assume}

\begin{assume}[Objective function]~
  \label{assume:obj}
  For a domain $\Ccal \subset \Rbb^N$ that satisfies Assumption~\ref{assume:con}, the objective function $F: \Ccal \to \Rbb$ satisfies the following conditions:
\begin{enumerate}[1)]
 \item\label{assume:obj:c2} $F$ is twice continuously differentiable on $\Ccal$;
 \item\label{assume:obj:bnd} $F$ is bounded from below on $\Ccal$; and
 \item\label{assume:obj:hess} $\|\nabla^2 F\|_2$ is bounded on $\Ccal$. 
\end{enumerate}
\end{assume}

To solve \eqref{eqn:optim-gen}, we introduce Algorithm~\ref{alg:etr}
that produces a sequence of points $\{\xbm_k\}_{k=1}^\infty$,
which we call \textit{trust-region centers}, that converges to a
first-order critical point. We use the following criticality
measure
\begin{equation}
  \label{eq:crt_meas}
 \chi(\xbm) \coloneqq
 \left|\min_{\substack{\xbm+\dbm\in\Ccal\\\norm{\dbm}_2\leq1}}
 \langle\nabla F(\xbm),\dbm\rangle\right|,
\end{equation}
where $\chi(\xbm^*) = 0$ implies $\xbm^*$ is a first-order critical
point of \eqref{eqn:optim-gen} \cite{conn_trust-region_2000}. At each trust-region center $\xbm_k$, we introduce an
inexpensive model $\func{m_k}{\Ccal}{\Rbb}$ intended to approximate
the objective function $F$ in the error-aware trust region
\begin{equation}
  \label{eq:tr:Bcal}
 \Bcal_k \coloneqq
 \left\{\xbm\in\Ccal \mid \vartheta_k(\xbm)\leq \Delta_k\right\},
\end{equation}
where $\func{\vartheta_k}{\Ccal}{\Rbb_{\geq 0}}$ is the trust-region
constraint and $\Delta_k\in\Rbb_{>0}$ is the trust-region radius. Standard
trust-region algorithms take $\vartheta_k(\xbm) = \norm{\xbm-\xbm_k}_2$;
however, we generalize the notion of a trust-region constraint to allow
for trust regions that take into account the \textit{error} in the
approximation model. We require the approximation models and trust-region constraints to satisfy the following assumptions:
\begin{assume}[Approximation model]~
  \label{assume:model}
   For a domain $\Ccal \subset \Rbb^N$ that satisfies Assumption~\ref{assume:con}, approximation models $m_k: \Ccal \to \Rbb$ associated with trust-region centers $\xbm_k \in \Ccal$, $k \in \Nbb$, satisfy the following conditions:
\begin{enumerate}[1)]
 \item\label{assume:model:c2} $m_k$ is twice continuously differentiable on $\Ccal$;
 \item\label{assume:model:val} $m_k(\xbm_k) = F(\xbm_k)$;
 \item\label{assume:model:grad} $\nabla m_k(\xbm_k) = \nabla F(\xbm_k)$;
 \item\label{assume:model:hess} there exists $\beta>0$ (independent of $k$)
 such that $\beta_k\leq\beta$ for all $k \in \Nbb$, where
 \begin{equation}
   \label{eq:assume4_betak}
  \beta_k \coloneqq 1 + \max_{\xbm\in\Ccal} \norm{\nabla^2 m_k(\xbm)}_2.
\end{equation}
\end{enumerate}
\end{assume}

\begin{assume}[Trust-region constraint]~
  \label{assume:trcon}
  For a domain $\Ccal \subset \Rbb^N$ that satisfies Assumption~\ref{assume:con}, the trust-region constraints $\func{\vartheta_k}{\Ccal}{\Rbb_{\geq 0}}$ associated with the trust-region centers $\xbm_k\in\Ccal$, $k\in\Nbb$, satisfy the following conditions:
\begin{enumerate}[1)]
 \item\label{assume:trcon:c2} $\vartheta_k$ is twice continuously differentiable on $\Ccal$;
 \item\label{assume:trcon:grad} there exists $\kappa_{\nabla\vartheta}>0$ (independent of $k$) such that $\max_{\xbm\in\Ccal} \norm{\nabla\vartheta_k(\xbm)}\leq\kappa_{\nabla\vartheta}$ for all $k \in \Nbb$;
 \item\label{assume:trcon:val} $\vartheta_k(\xbm_k) = 0$;
 \item\label{assume:trcon:err} there exist $\zeta > 0$ and $\nu > 1$ (independent of $k$) such that $|F(\xbm)-m_k(\xbm)|\leq \zeta \vartheta_k(\xbm)^\nu$ for all $\xbm\in\Bcal_k$. 
\end{enumerate}
\end{assume}

\begin{algorithm}
 \caption{Error-aware trust-region method with convex constraints}
 \label{alg:etr}
 \begin{algorithmic}[1]
  \STATE \textbf{Initialization}: Given
 \begin{center}
  $\xbm_0\in\Ccal$, $\Delta_0$, $\Delta_\text{max}$,
  $0<\gamma_1<\gamma_2<1$,
  $0 < \eta_1 < \eta_2 < 1$
 \end{center}
   \STATE \textbf{Model and constraint update}:
 Choose a model $m_k: \Ccal \to \Rbb$ and trust-region constraint $\vartheta_k: \Ccal \to \Rbb$
 that satisfy Assumptions~\ref{assume:model} and \ref{assume:trcon}.
\STATE \textbf{Step computation}: Approximately solve the trust-region
 subproblem
 \begin{equation*}
  \underset{\xbm\in\Bcal_k}{\min}~m_k(\xbm)
 \end{equation*}
 for a candidate step $\hat\xbm_k$ that satisfies the fraction of Cauchy decrease condition \eqref{eqn:fcd}.
\STATE \textbf{Actual-to-predicted reduction}:
Compute actual-to-predicted reduction ratio approximation
\begin{equation*}
 \varrho_k \coloneqq \frac{F(\xbm_k)-F(\hat\xbm_k)}
               {m_k(\xbm_k)-m_k(\hat\xbm_k)}.
\end{equation*}
\STATE \textbf{Step acceptance}:
\begin{equation*}
 \textbf{if} \qquad \varrho_k \geq \eta_1 \qquad \textbf{then}
  \qquad \xbm_{k+1} = \hat\xbm_k \qquad \textbf{else}
  \qquad \xbm_{k+1} = \xbm_k \qquad \textbf{end if}
\end{equation*}
\STATE \textbf{Trust region update}:
\begin{equation*}
\begin{aligned}
\textbf{if} \qquad &\varrho_k < \eta_1 \qquad &\textbf{then} \qquad
&\Delta_{k+1} \in [\gamma_1\Delta_k, \gamma_2\Delta_k] \quad
&\textbf{end if}\\
\textbf{if} \qquad &\varrho_k \in [\eta_1, \eta_2)\qquad &\textbf{then}\qquad
 &\Delta_{k+1} \in [\gamma_2\Delta_k, \Delta_k] \qquad
 &\textbf{end if}\\
\textbf{if} \qquad &\varrho_k \geq \eta_2 \qquad &\textbf{then} \qquad
&\Delta_{k+1} \in [\Delta_k, \Delta_\text{max}] \qquad &\textbf{end if}\\
\end{aligned}
\end{equation*}
\end{algorithmic}
\end{algorithm}

The error-aware trust-region algorithm is described in Algorithm~\ref{alg:etr}. The algorithm is initialized from
an initial guess for the optimization solution $\xbm_0\in\Ccal$,
an initial trust-region radius $\Delta_0$, and a number of
other (standard) constants \cite{conn_trust-region_2000}.
If the starting point is not feasible (i.e., $\xbm_0\notin\Ccal$),
an auxiliary feasibility restoration problem is solved to obtain
a feasible initial guess.
At iteration $k$, an approximation
model and trust-region constraint are constructed that satisfy
Assumptions~\ref{assume:model} and \ref{assume:trcon} to generate
the trust-region subproblem
\begin{equation} \label{eqn:trsub}
 \underset{\xbm\in\Bcal_k}{\min}~m_k(\xbm).
\end{equation}
We (approximately) solve this subproblem whose solution,
denoted $\hat\xbm_k$, is a \textit{candidate} for the next trust-region
center $\xbm_{k+1}$. It is important to note that the trust-region subproblem
does not need to be solved exactly; it only needs to find a point that
satisfies the \textit{fraction of Cauchy decrease condition}:
  \begin{equation} \label{eqn:fcd}
    m_k(\xbm_k) - m_k(\hat\xbm_k) \geq
    \kappa \chi(\xbm_k)
    \min\left[\frac{\chi(\xbm_k)}{\beta_k},\kappa'\Delta_k,1\right],
  \end{equation}
  where $\kappa\in(0,1)$ and $\kappa'>0$ are constants (independent of $k$), and $\beta_k$ is given by \eqref{eq:assume4_betak}. 
Theorem~\ref{thm:fcd} establishes the existence of such a point within the
trust region.
\begin{theorem} \label{thm:fcd}
  Suppose Assumptions~\ref{assume:con}.\ref{assume:con:c2}--\ref{assume:con}.\ref{assume:con:qual},~\ref{assume:obj}.\ref{assume:obj:c2},~\ref{assume:model}.\ref{assume:model:c2}--\ref{assume:model}.\ref{assume:model:hess},~\ref{assume:trcon}.\ref{assume:trcon:c2}--\ref{assume:trcon}.\ref{assume:trcon:val} hold. Then there exists a point $\xbm\in\Bcal_k$ that satisfies the fraction of (generalized) Cauchy decrease condition~\eqref{eqn:fcd} for $\kappa' = \kappa_{\nabla\vartheta}^{-1}$: i.e.,
  \begin{equation}
m_k(\xbm_k) - m_k(\xbm) \geq
\kappa \chi(\xbm_k)
\min\left[\frac{\chi(\xbm_k)}{\beta_k}, \kappa_{\nabla\vartheta}^{-1}\Delta_k, 1\right].
  \end{equation}
  \begin{proof}
    See \ref{sec:tr-proofs}.
  \end{proof}
\end{theorem}

Finally, we compute the actual-to-predicted reduction,
\begin{equation}
 \varrho_k \coloneqq \frac{F(\xbm_k)-F(\hat\xbm_k)}
                          {m_k(\xbm_k)-m_k(\hat\xbm_k)},
\end{equation}
and use it to assess whether to accept or reject the candidate step
as the next trust-region center. If $\rho_k > \eta_1$ for
some constant $\eta_1\in(0,1)$, the candidate is accepted; otherwise
it is rejected. In addition, the actual-to-predicted reduction is
used to modify the trust-region radius for the next iteration. If
$\rho_k < \eta_1$, the step is called \textit{unsuccessful} and
the radius is decreased to $\Delta_{k+1}\in[\gamma_1\Delta_k,\gamma_2\Delta_k]$,
where $0<\gamma_1<\gamma_2<1$. If $\rho_k\in[\eta_1,\eta_2)$, where
$\eta_2\in(\eta_1,1)$ is a constant, the step is considered
\textit{successful} and the radius is updated to lie in the range
$\Delta_{k+1} = [\gamma_2\Delta_k,\Delta_k]$ (usually $\Delta_{k+1}=\Delta_k$).
Finally, if $\rho_k \geq \eta_2$, the step is called \textit{very successful}
and the radius is increased or kept the same.
We terminate Algorithm~\ref{alg:etr}
when $\norm{\xbm_k-P_\Ccal(\xbm_k-\nabla F(\xbm_k))}$ falls below a predefined
tolerance, where $P_\Ccal$ is the projection operator onto the feasible set
$\Ccal$. 

This error-aware trust-region algorithm (Algorithm~\ref{alg:etr}) is \textit{globally convergent}: for any $\xbm_0 \in \Ccal$ and $\Delta_0 \in \Rbb_{>0}$, the algorithm generates a sequence $\{\xbm_k\}_{k=1}^\infty$ such that any convergent subsequence converges to a first-order critical point of (\ref{eqn:optim-gen}).  The result is summarized in the following theorem:
\begin{theorem} \label{thm:globconv}
Suppose Assumptions~\ref{assume:con}.\ref{assume:con:c2}--\ref{assume:con}.\ref{assume:con:qual},~\ref{assume:obj}.\ref{assume:obj:c2}--\ref{assume:obj}.\ref{assume:obj:hess},~\ref{assume:model}.\ref{assume:model:c2}--\ref{assume:model}.\ref{assume:model:hess},~\ref{assume:trcon}.\ref{assume:trcon:c2}--\ref{assume:trcon}.\ref{assume:trcon:err} hold. Then
\begin{equation}
 \liminf_{k\rightarrow\infty}~\chi(\xbm_k) = 0,
\end{equation}
where $\chi: \Ccal \to \Rbb$ is the criticality measure~\eqref{eq:crt_meas}.
\begin{proof}
  See \ref{sec:tr-proofs}.
\end{proof}
\end{theorem}

\subsection{Error-aware trust-region method for topology optimization}
\label{sec:tr:etr-topopt}
We propose an error-aware trust-region method (Section~\ref{sec:tr:etr})
to efficiently solve the topology optimization problem (\ref{eqn:topopt0})
using projection-based ROMs introduced in Section~\ref{sec:rom} as the approximation model.
The topology optimization problem \eqref{eqn:topopt0} exactly fits
the form of the general optimization problem \eqref{eqn:optim-gen},
where the optimization variables are the unfiltered element densities
$\xbm\coloneqq \psibold \in \Rbb^{N \coloneq N_e}$, the objective function
is the topology optimization objective function~\eqref{eq:Jfom} 
\begin{equation} \label{eqn:obj-topopt}
 F(\psibold) \coloneqq J(\psibold) \coloneqq
 j(\ubm_u^\star(\psibold),\rhobold^\star(\psibold)),
\end{equation}
and the constraint set is the intersection of the simple bounds and
volume constraint 
\begin{equation} \label{eqn:feaset-topopt}
 \Ccal \coloneqq
 \left\{
 \psibold\in \Psi \coloneqq [0,1]^{N_e} \suchthat
 \sum_{e=1}^{N_e} \psi_e|\Omega_e| \leq V
   \right\}.
\end{equation}
As we will see in Theorem~\ref{thm:globconv2}, the objective function and constraint satisfy Assumptions~\ref{assume:con} and \ref{assume:obj}.
\begin{remark}
Assumptions~\ref{assume:con} and \ref{assume:obj} include a range of
practical topology optimization problems; however, they preclude
problems with nonlinear, non-convex constraints that involve the
structural state (e.g., stress bounds). Such problems would require a
trust-region method that explicitly accounts for state-dependent
constraints, e.g., sequential quadratic programming.
\end{remark}

We now outline the procedure to construct an approximation model based using projection-based reduced-order models (Section~\ref{sec:rom}) that satisfies the requirements for global convergence in Assumption~\ref{assume:model}. Suppose $k$ trust-region iterations have been completed and we collected the density fields defining the trust-region centers
\begin{equation*}
  \Psibold_k \coloneqq ( \psibold^{(0)}, \dots, \psibold^{(k)} ).
\end{equation*}
We collect the associated primal and adjoint solutions in matrices of the form
\begin{align*}
  \Ubm_k &\coloneqq [ \ubm^\star(\psibold^{(0)}), \cdots, \ubm^\star(\psibold^{(k)}) ] \in \Rbb^{N_{\ubm} \times (k+1)}, \\
  \Lambdabold_k &\coloneqq [ \lambdabold^\star(\psibold^{(0)}), \cdots, \lambdabold^\star(\psibold^{(k)}) ] \in \Rbb^{N_{\ubm} \times (k+1)},
\end{align*}
and apply POD to the first $k$ snapshots to obtain ``compressed'' matrices
\begin{align*}
  \Phibold_{n_k}^{\Ubm_{k-1}, \rm POD} &\coloneqq \textrm{POD}_{n_k}(\Ubm_{k-1}) \in \Rbb^{N_{\ubm} \times n_k}, \\
  \Phibold_{n_k}^{\Lambdabold_{k-1}, \rm POD} &\coloneqq \textrm{POD}_{n_k}(\Lambdabold_{k-1}) \in \Rbb^{N_{\ubm} \times n_k},
\end{align*}
where $\textrm{POD}_n : \Rbb^{N_{\ubm} \times (k-1)} \to \Rbb^{N_{\ubm} \times n}$ applies the singular value decomposition to the input matrix $\Ubm$ (resp.~$\Lambdabold$) and extract the first $n$ left singular vectors to form $\Phibold_n^{\Ubm, \rm POD}$ (resp.~$\Phibold_n^{\Lambdabold, \rm POD}$).  We combine the POD bases and the primal and adjoint solution at the trust-region center, $\ubm^\star(\psibold^{(k)})$ and $\lambdabold^\star(\psibold^{(k)})$, and orthonormalize the resulting matrix using the Gram-Schmidt procedure to obtain the reduced basis (matrix):
\begin{equation} \label{eqn:rob-optim}
  \Phibold_{j_k} = \mathrm{GramSchmidt}([\Phibold_{n_k}^{\Ubm_{k-1}, \rm POD},\Phibold_{n_k}^{\Lambdabold_{k-1}, \rm POD}, \ubm^\star(\psibold^{(k)}), \lambdabold^\star(\psibold^{(k)})  ]) \in \Rbb^{N_{\ubm} \times j_k},
\end{equation}
where $n_k\in\{0,1,\dots,k-1\}$ is the number of vectors retained after
POD is applied to the primal and adjoint snapshot matrices and
$j_k\coloneqq 2(n_k+1)$ is the size of the reduced basis.

Then the approximation model is the ROM objective function~\eqref{eq:Jrom}:
\begin{equation} \label{eqn:model-topopt}
  m_k(\psibold)\coloneqq J_{j_k}(\psibold) \coloneqq
  j(\ubm_{j_k}^\star(\psibold), \rhobold^\star(\psibold)),
\end{equation}
where $\ubm_{j_k}^\star(\psibold)$ is the solution of the ROM associated with the basis $\Phibold_{j_k}$. 
The choice of basis in (\ref{eqn:rob-optim}) guarantees that
\begin{equation*}
 \ubm^\star(\psibold^{(k)}), \lambdabold^\star(\psibold^{(k)}) \in
 \mathrm{Img}(\Phibold_{j_k})
\end{equation*}
\textit{regardless of $n_k$}.  Hence by Remarks~\ref{rem:rom_exact} and \ref{rem:rom_adjoint}, the primal and adjoint solution of the ROM will be
exact at the trust-region center $\psibold^{(k)}$: 
\begin{equation*}
 \ubm^\star(\psibold^{(k)}) = \ubm_{j_k}^\star(\psibold^{(k)}), \quad
 \lambdabold^\star(\psibold^{(k)}) = \lambdabold_{j_k}^\star(\psibold^{(k)}).
\end{equation*}
As a result, the approximation model and its gradient agree with
the original topology optimization objective at the trust-region center
(Theorem~\ref{thm:rom_output_error}), which establishes
Assumptions~\ref{assume:model}.\ref{assume:model:val}--\ref{assume:model}.\ref{assume:model:grad}, regardless of $n_k$.
\begin{remark}
By defining a single ROM from a basis constructed from primal and dual
snapshots, it is guaranteed that the ROM objective and gradient evaluations
will be consistent with each other \cite{zahr_phd_2016}, which is important
for convergence of the trust-region subproblem, particularly when using
black-box solvers.
\end{remark}
\begin{remark}
  For compliance output, the primal and adjoint solutions are identical: $\lambdabold^\star(\psibold^{(i)}) = \ubm^\star(\psibold^{(i)})$, $i = 0,\dots,k$.  Hence, we need to construct only the primal snapshot matrix $\Ubm_{k-1}$ and the associated POD matrix $\Phibold_{n_k}^{\Ubm_{k-1}, \rm POD}$.  The reduced basis (matrix) is then given by
  \begin{equation*}
    \Phibold_{j_k} = \mathrm{GramSchmidt}([\Phibold_{n_k}^{\Ubm_{k-1}, \rm POD}, \ubm^\star(\psibold^{(k)}) ]) \in \Rbb^{N_\ubm \times j_k},
  \end{equation*}
where $j_k = n_k+1$.
  Note that, because the columns of $\Phibold_{n_k}^{\Ubm_{k-1},\rm POD}$ are orthonormal, we need to perform just one step of Gram-Schmidt to orthonormalize $\ubm^\star(\psibold^{(k)})$ with respect to $\Phibold_{n_k}^{\Ubm_{k-1}, \rm POD}$.
\end{remark}
In this work, we consider two instances of the trust-region constraint. The
first is the traditional trust-region constraint of the form
\begin{equation} \label{eqn:trcon-topopt-dist}
 \vartheta_k(\psibold) \coloneqq \norm{\psibold-\psibold^{(k)}}_2.
\end{equation}
This constraint is simple to implement, efficient to evaluate, and
will lead to a globally convergent method \cite{conn_trust-region_2000};
however, it does not account for the error in the approximation model.
The other trust-region constraint we consider is the norm of the primal
residual evaluated at the reconstructed ROM state
\begin{equation} \label{eqn:trcon-topopt}
  \vartheta_k(\psibold) \coloneqq\norm{\rbm(\ubm_{j_k}^\star(\psibold);\rhobold^\star(\psibold))}_2
\end{equation}
since it bounds the error in the approximation model
(Theorem~\ref{thm:rom_output_error}) and provides a notion of an error-aware
trust region. Theorem~\ref{thm:globconv2} will verify (\ref{eqn:trcon-topopt})
satisfies Assumption~\ref{assume:trcon}, thus leading to a globally
convergent method. Furthermore, this choice of trust-region constraint is
relevant given that the residual norm has proven to be a popular
indicator to trigger adaptation to a reduced basis constructed
on-the-fly, both in the context of topology optimization
\cite{gogu_improving_2015, choi_accelerating_2019, xiao_fly_2020}
and more generally
\cite{zahr_progressive_2015}; however, these methods are heuristic and
cannot guarantee global convergence. The proposed trust-region framework
provides a rigorous setting that can be used to adapt existing methods such
that they are guaranteed to converge to a local minima from any starting point.


\begin{algorithm}
 \caption{Error-aware trust-region method for efficient topology optimization}
 \label{alg:etr:topopt}
 \begin{algorithmic}[1]
  \STATE \textbf{Initialization}: Given
 \begin{center}
  $\psibold^{(0)}\in\Ccal$, $\Ubm_0 = \Lambdabold_0 = \emptyset$,
  $\Delta_0$, $\Delta_\text{max}$,
  $0<\gamma_1<\gamma_2<1$,
  $0 < \eta_1 < \eta_2 < 1$
 \end{center}
   \STATE\label{step:modelUpdate} \textbf{Model and constraint update}: Compute
 HDM primal/adjoint solutions and define reduced basis as
\begin{equation*}
 \Phibold_{j_k} = \mathrm{GramSchmidt}([\mathrm{POD}_{n_k}(\Ubm_{k-1}), \mathrm{POD}_{n_k}(\Lambdabold_{k-1}), \ubm^\star(\psibold^{(k)}), \lambdabold^\star(\psibold^{(k)})])
\end{equation*}
where $0\leq n_k\leq k-1$ and $j_k = 2(n_k+1)$. The approximation model
and trust-region constraint are taken as
\begin{equation*}
  m_k(\psibold)\coloneqq J_{j_k}(\psibold)
  \coloneqq  j(\ubm_{j_k}^\star(\psibold), \rhobold^\star(\psibold))
  , \qquad
 \vartheta_k(\psibold) \coloneqq\norm{\rbm(\ubm_{j_k}^\star(\psibold);\rhobold^\star(\psibold))}_2
\end{equation*}
and the snapshot matrices are updated for the next iteration
\begin{equation*}
 \Ubm_k = \begin{bmatrix} \Ubm_{k-1} & \ubm^\star(\psibold^{(k)})\end{bmatrix},
 \quad
 \Lambdabold_k = \begin{bmatrix} \Lambdabold_{k-1} & \lambdabold^\star(\psibold^{(k)})\end{bmatrix}.
\end{equation*}
\STATE \textbf{Step computation}: Approximately solve the trust-region
 subproblem
 \begin{equation*}
  \underset{\psibold\in\Bcal_k}{\min}~m_k(\psibold)
 \end{equation*}
 for a candidate step $\hat\psibold^{(k)}$ that satisfies (\ref{eqn:fcd}).
\STATE\label{step:actualToPredict} \textbf{Actual-to-predicted reduction}:
Compute actual-to-predicted reduction ratio approximation
\begin{equation*}
 \varrho_k \coloneqq \frac{J(\psibold^{(k)})-J(\hat\psibold^{(k)})}
                          {J_k(\psibold^{(k)})-J_k(\hat\psibold^{(k)})}.
\end{equation*}
\STATE \textbf{Step acceptance}:
\begin{equation*}
 \textbf{if} \qquad \varrho_k \geq \eta_1 \qquad \textbf{then}
  \qquad \psibold^{(k+1)} = \hat\psibold^{(k)} \qquad \textbf{else}
  \qquad \psibold^{(k+1)} = \psibold^{(k)} \qquad \textbf{end if}
\end{equation*}
\STATE \textbf{Trust region update}:
\begin{equation*}
\begin{aligned}
\textbf{if} \qquad &\varrho_k < \eta_1 \qquad &\textbf{then} \qquad
&\Delta_{k+1} \in [\gamma_1\Delta_k, \gamma_2\Delta_k] \quad
&\textbf{end if}\\
\textbf{if} \qquad &\varrho_k \in [\eta_1, \eta_2)\qquad &\textbf{then}\qquad
 &\Delta_{k+1} \in [\gamma_2\Delta_k, \Delta_k] \qquad
 &\textbf{end if}\\
\textbf{if} \qquad &\varrho_k \geq \eta_2 \qquad &\textbf{then} \qquad
&\Delta_{k+1} \in [\Delta_k, \Delta_\text{max}] \qquad &\textbf{end if}\\
\end{aligned}
\end{equation*}
\end{algorithmic}
\end{algorithm}

The final error-aware trust-region method based on ROMs
for efficient topology optimization is summarized in
Algorithm~\ref{alg:etr:topopt}.
\begin{remark}
The ROM-based topology optimization method that uses traditional trust regions,
i.e., the trust-region constraint in
(\ref{eqn:trcon-topopt-dist}), is given by Algorithm~\ref{alg:etr:topopt}
with the definition of $\vartheta_k$ in Step~\ref{step:modelUpdate}
replaced with (\ref{eqn:trcon-topopt-dist}).
\end{remark}
\begin{remark}
Each iteration in Algorithm~\ref{alg:etr:topopt} involves a number of
primal and adjoint ROM solutions to solve the trust-region subproblem and
a single primal and adjoint HDM solution to evaluate the actual-to-predicted
reduction ratio (primal only) and update the reduced basis (both primal and
adjoint).
\end{remark}
The algorithm produces a sequence
of trust-region centers that converges (liminf sense) to a
critical point of (\ref{eqn:topopt0}); see Theorem~\ref{thm:globconv2}. 
\begin{lemma}
  \label{lem:assume_ver}
  Suppose the output functional $j : \Rbb^{N_\ubm} \times P \to \Rbb$ is twice continuously differentiable, and we choose $\vartheta_k(\psibold) \coloneqq\norm{\rbm(\ubm_{j_k}^\star(\psibold);\rhobold^\star(\psibold))}_2^{1-\epsilon}$ for $\epsilon \in (0,1)$. The error-aware trust-region method based on ROMs for topology optimization satisfy Assumptions~\ref{assume:con}--\ref{assume:trcon}.  For $\epsilon = 0$, all conditions except the fourth condition of Assumption~\ref{assume:trcon} are satisfied.
  \begin{proof}
    See Lemmas~\ref{lem:assume:con}--\ref{lem:assume:trcon} in~\ref{sec:assume_proofs}.
  \end{proof}
\end{lemma}
\begin{remark}
  In practice we set $\epsilon = 0$ in Lemma~\ref{lem:assume_ver}.  As a result, the fourth condition of Assumption~\ref{assume:trcon} is not satisfied.
\end{remark}

\begin{theorem} \label{thm:globconv2}
Consider the optimization problem in \eqref{eqn:optim-gen} with the
objective function \eqref{eqn:obj-topopt} where
$\func{j}{\Rbb^{N_\ubm} \times P}{\Rbb}$ is twice continuously
differentiable with respect to both of its arguments
and the feasible set is defined in (\ref{eqn:feaset-topopt}) with
$V\geq\sum_{e=1}^{N_e} \rho_l|\Omega_e|$. Then the sequence
of trust-region centers produced by Algorithm~\ref{alg:etr:topopt} satisfies
\begin{equation}
 \liminf_{k\rightarrow\infty} \chi(\psibold^{(k)}) = 0
\end{equation}
regardless of $n_k$.
\begin{proof}
  The result is a direct consequence of Lemma~\ref{lem:assume_ver} and Theorem~\ref{thm:globconv}.
\end{proof}
\end{theorem}

To close, we discuss two pertinent details for the implementation of
Algorithm~\ref{alg:etr:topopt} in practice. First, the global convergence
theory is independent of $n_k$; however, the behavior of the algorithm
does depend on $n_k$. For example, if we take $n_k = 0$, each ROM solve
will be extremely fast (since $j_k = 2$ for all $k$), but the ROM will
have little predictive capability and a large number of trust-region
center updates will be required. On the other end of the spectrum, choosing
$n_k=k-1$ will lead to expensive ROM solves when $k$ becomes large; however,
fewer HDM solves will be required because the ROM basis will be as rich as
possible. The approach we take is to choose
\begin{equation}
 n_k = \min\{k-1, n_\text{max}\},
\end{equation}
where $n_\text{max}$ is the maximum truncation size for the primal/adjoint
reduced-order bases, which guarantees the ROM size will be bounded
$j_k \leq 2 (n_\text{max}+1)$. Thus $n_\text{max}$ should be chosen as
large as possible such that individual ROM solves are sufficiently fast.
\begin{remark}
In practice, we only retain the $L$ most recent
trust-region centers in $\Ubm_{k-1}$ and $\Lambdabold_{k-1}$
and require $n_\text{max}\leq L$, rather than
all previous centers. This ensures that POD is applied to at most $L$ snapshots and hence controls the POD cost. Furthermore, since the matrix
$\Ubm_{k-1}$ (resp. $\Lambdabold_{k-1}$) is a low-rank update to
$\Ubm_{k-2}$ (resp. $\Lambdabold_{k-2}$), efficient algorithms for
updating the singular value decomposition
\cite{brand_fast_2006, choi_accelerating_2019},
adapted to the truncated singular value decomposition (POD) in
\cite{amsallem_fast_2015, washabaugh_use_2016, oxberry_limited-memory_2017},
can be used to ensure the POD cost is negligible in comparison to a HDM solve.
In the case where $k<L$, the algorithm in
\cite{amsallem_fast_2015, washabaugh_use_2016,
      oxberry_limited-memory_2017, choi_accelerating_2019}
is used to update the POD of $\Ubm_{k-2}$ with the new snapshot to obtain
the POD of $\Ubm_{k-1}$ (rank-one update), while in the case where $k=L$,
the POD update also removes the oldest snapshot (rank-two update).
\end{remark}

The other practical issue that must be addressed is how to efficiently
solve the trust-region subproblem to ensure the fraction of Cauchy decrease condition
holds. Given Theorem~\ref{thm:fcd} that establishes the existence of a point in
the trust region that satisfies (\ref{eqn:fcd}), the (global) minimizer
of (\ref{eqn:trsub}) is guaranteed to satisfy (\ref{eqn:fcd}).
Recall from Figure~\ref{fig:rom_cost} that ROM evaluations are
$\approx 100\times$ cheaper than
HDM evaluations; however, they are not \textit{free} so trust-region
subproblem solvers that require hundreds or thousands of ROM evaluations
will not be competitive. This immediately eliminates the use of
``global'' optimization methods due to the large number of objective
evaluations required. In addition, our experiments using gradient-based
optimization methods to find a local solution to (\ref{eqn:trsub}) also
required several hundred iterations to even converge to relatively loose
optimality tolerances. The approach that proved most effective in practice
was to use a gradient-based optimization procedure (Method of Moving Asymptotes
\cite{svanberg_method_1987} in this work) to solve
\begin{equation*}
 \underset{\psibold\in\Ccal}{\min}~m_k(\psibold),
\end{equation*}
i.e., the trust-region subproblem \textit{without} the trust-region
constraint, and \textit{terminate} the iterations once the trust-region
constraint was violated $\vartheta_k(\psibold) > \Delta_k$.
While this is not guaranteed to produce a step that satisfies the
fraction of Cauchy decrease, it works well in practice as we show
in the next section.

We conclude the section with a remark on the extension of the proposed framework to other PDEs.
  \begin{remark}
    While the application focus of this work is density-based topology optimization governed by linear elasticity equations, the proposed ROM-based trust-region framework, at least at the abstract level, can be applied to other PDEs.  However, in general an extension would (i) require additional technical ingredients and/or (ii) result in the loss of certain theoretical properties.  For instance, consider nonlinear elasticity problems; the construction of an online-efficient ROM for nonlinear PDEs requires so-called hyperreduction by, for example, the empirical interpolation method~\cite{Barrault_2004_EIM}.  As another example, suppose the underlying HDM is not finite-element based; the algebraic formulation of projection-based ROMs in Section~\ref{sec:rom} would extend to other discretizations (e.g., finite difference), but both our density-based topology optimization formulation and the ROM error analysis build on the variational and finite-element formulation.  Finally, suppose the PDE of interest is non-symmetric and/or non-coercive.  In these problems, Galerkin ROMs lose the energy optimality property (Remark~\ref{rem:rom_opt}) and the formulation is not guaranteed to be well-posed (Remark~\ref{rem:rom_wellposed}); it can be advantageous to consider minimum-residual ROMs, whose formulation and analysis are provided in e.g.~\cite{Maday_2002_RB_Noncoercive}, for these classes of problems.
  \end{remark}  

%% file: numexp.tex
\section{Numerical experiments}
\label{sec:numexp}
In this section, we study the performance of the proposed ROM-based
trust-region methods for topology optimization on three standard benchmark
problems: MBB beam, cantilever beam, and simply supported beam. We will consider two instances of the method,
one where the trust region is taken as the sublevel sets of the HDM
residual \eqref{eqn:trcon-topopt}, subsequently called ROM-TR-RES,
and one with traditional trust regions
$\vartheta_k(\psibold)=\norm{\psibold-\psibold^{(k)}}_2$
\cite{conn_trust-region_2000}, 
to be called ROM-TR-DIST;
both instances use the same trust-region model $m_k(\psibold)$
based on the reduced-order model. Given its popularity, we use MMA
applied to the original (unreduced) topology optimization, referred to as
HDM-MMA, as the standard for comparison. In the remainder of this
section, we use the terminology \textit{major iteration}, to refer
to an update from $\psibold^{(k)}$ to $\psibold^{(k+1)}$,
including rejected steps where $\psibold^{(k+1)}=\psibold^{(k)}$.
In all methods considered, a major iteration involves a single HDM evaluation;
the ROM-based methods additionally require a number of ROM evaluations.

We use the computational cost required to drive the objective function
to within a prescribed tolerance of its optimal value to quantify the
performance of the ROM-based optimization algorithms. 
For all benchmark problems, we take
the optimal value of the objective function, denoted $J^*$, to be the $2000$th
iteration of the HDM-MMA algorithm; the algorithm is cutoff at the smallest
$n\in\Nbb$ such that
\begin{equation}
 |J(\psibold^{(n)}) - J^*| < \epsilon|J^*|,
\end{equation}
where $\epsilon > 0$ is the (relative) cutoff tolerance. The
computational cost required to converge to a tolerance of $\epsilon$,
denoted $C_\epsilon$, will be measured in units of
\textit{equivalent HDM solves}, i.e.,
\begin{equation}
 C_\epsilon = \frac{T_\epsilon}{t_\text{HDM}},
\end{equation}
where $T_\epsilon$ is the time required for the method under consideration
to converge the topology optimization problem to the tolerance $\epsilon$,
and $t_\text{HDM}$ is the time required for a single HDM solve. Because
the dominant cost of topology optimization comes from the HDM and ROM
solves, $T_\epsilon$ can be expanded as
\begin{equation}
 T_\epsilon \approx
              N_\text{HDM}^\epsilon t_\text{HDM} +
              N_\text{ROM}^\epsilon t_\text{ROM},
\end{equation}
where $N_\text{HDM}^\epsilon$ and $N_\text{ROM}^\epsilon$ are the number of
HDM and ROM solves, respectively, required for the method to converge the
objective function to a tolerance $\epsilon$, and $t_\text{ROM}$ is the
average time required for a single ROM solve. From
Figure~\ref{fig:rom_cost},
$t_\text{ROM}$ depends on the size of the reduced basis. In our numerical
experiments, we fix the maximum basis size to $20$ ($n_\text{max}=19$) and
therefore use the conservative estimate of $t_\text{ROM} = \nu t_\text{HDM}$,
where $\nu = 0.01$ from Figure~\ref{fig:rom_cost}. Therefore, the computational
cost takes the form
\begin{equation} \label{eqn:cost}
 C_\epsilon = N_\text{HDM}^\epsilon + \nu N_\text{ROM}^\epsilon.
\end{equation}

There are a number of user-defined parameters required by the
ROM-based trust-region method (Algorithm~\ref{alg:etr:topopt}); however, most of these are
standard trust-region parameters, and we fix them at reasonable
values: $\gamma_1 = 0.5$, $\gamma_2 = 1$, $\eta_1 = 0.1$,
$\eta_2 = 0.75$, $\Delta_\text{max} = 100\Delta_0$. The remaining parameter,
$\Delta_0$, is more delicate and potentially problem-specific so we study
its impact on the ROM-based trust-region algorithms in the following sections.
To ensure $\Delta_0$
is appropriately scaled, we consider initial trust-region radii of the
form
\begin{equation} \label{eqn:trrad0}
 \Delta_0 = \tau \norm{\rbm_\ubm(\zerobold, \rhobold^\star(\psibold^{(0)}))}_2,
 \qquad
 \Delta_0 = \tau \norm{\psibold^{(0)}}_2,
\end{equation}
for the TR-ROM-RES and TR-ROM-DIST methods, respectively, where
$\psibold^{(0)}$ is the (feasible) initial design and $\tau\in\Rbb_{>0}$ 
controls the (scaled) initial radius.
\begin{remark}
We compare the proposed methods (ROM-TR-RES and ROM-TR-DIST) to MMA rather
than a method with the same algorithmic structure, e.g., a trust-region method
with a quadratic model problem constructed from the HDM, due to the popularity
of MMA in the topology optimization community. Furthermore, HDM-based
trust-region methods require the Hessian of the Lagrangian, which is
not practical to compute even in a medium-scale topology optimization
setting. For large-scale problems, we could use an approximate subproblem
solver, such as the Steihaug-Toint truncated conjugate gradient method
\cite{conn_trust-region_2000}, that only requires Hessian-vector products,
which can be approximated via first-order finite differences at the cost of
a single HDM adjoint solve per subproblem iteration; however, the HDM adjoint
solves render each subproblem solve relatively expensive. In our numerical
experiments, such methods were not competitive with MMA. We hence do not
include these methods in the comparisons that follow.
\end{remark}

\subsection{MBB beam}
\label{sec:numexp:mbb}
We begin with the most canonical topology optimization problem: the MBB beam
(Figure~\ref{fig:mbb0_setup}). We use a finite element mesh consisting of
$180\times 60$ bilinear quadrilateral elements, minimum filtering length scale of $R = 0.12$, and maximum volume $V = \frac{1}{2} |\Omega|$.
The optimal design is shown in Figure~\ref{fig:mbb0_setup}
($J^* = 19.96$).
For the studies in this section, all algorithms are initialized from the
same feasible design: $\psibold^{(0)}=0.5\cdot\onebold$.
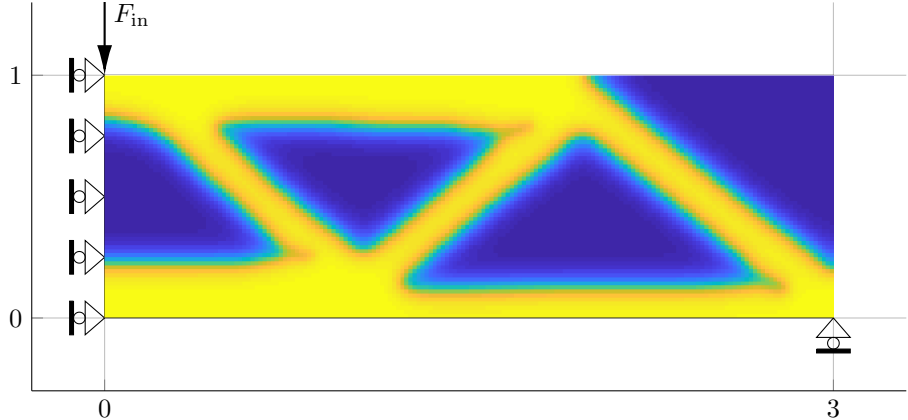
\begin{figure}
\centering
\input{py/mbb0_geom.tikz}
\caption{Setup and optimal design of MBB beam. The point load
         $F_\text{in} = 0.3$ is implemented as a distributed load of magnitude
         $q_\text{in} = 1$ applied to a segment of length $0.3$.}
\label{fig:mbb0_setup}
\end{figure}

First, we demonstrate the significance of embedding the reduced
topology optimization problem in the adaptive trust-region framework. To this end, we compare two ROM-accelerated methods (with HDM-MMA as the benchmark): ROM-TR-RES, the residual-based trust-region method introduced in Section~\ref{sec:tr:etr-topopt}; ROM-FIX-RES, an alternative that uses a fixed residual-based ``trust-region'' radius and does not scrutinize the trust-region step, i.e., accepts any candidate point produced by the trust-region subproblem. 
ROM-TR-RES
($\tau=0.1$) requires fewer HDM solves to drive the objective function to a fixed tolerance of the optimal value and converges
to $\epsilon = 0.01$ in $23$ major iterations; HDM-MMA requires
$32$ iterations (Table~\ref{tab:mbb0_nel180x60_rmin0p12}). Furthermore,
ROM-TR-RES is rather insensitive to the initial trust-region radius; we
consider three orders of magnitude variation in $\tau \in \{0.01, 0.1, 1, 10\}$
and performance degradation only emerges for the largest values $\tau \in\{1, 10\}$.
The fastest convergence is obtained with $\tau = 0.1$. In contrast,
without the trust region adaptivity and candidate scrutiny,
ROM-FIX-RES convergence is usually slower than the HDM-MMA algorithm,
severely degrades as $\tau$ increases, and, in some cases, may not be achieved (Figure~\ref{fig:mbb0_nel180x60_rmin0p12_romset0_majit0}).
\begin{figure}
\centering
\input{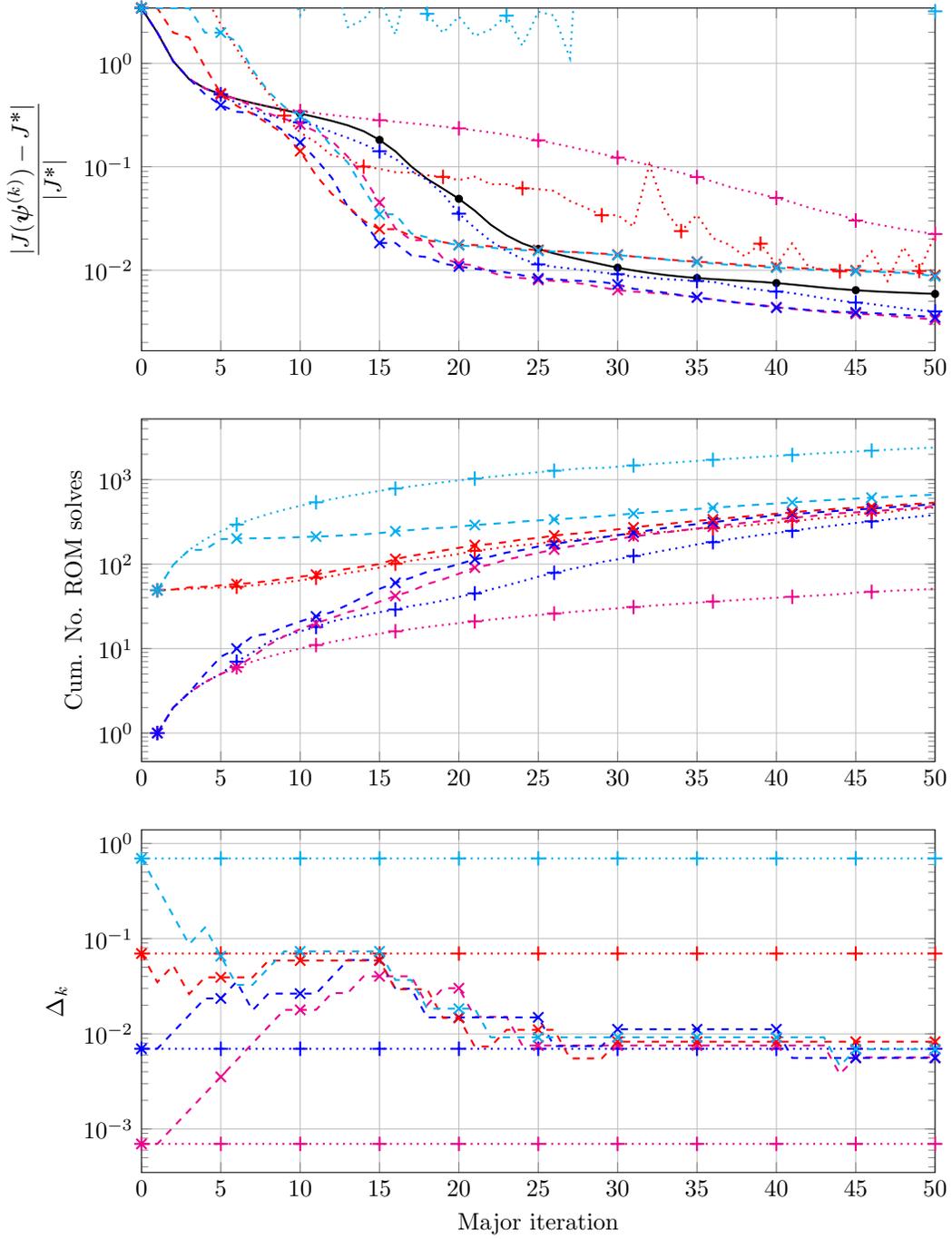}
\caption{Performance of
         HDM-MMA (\ref{line:mbb0_nel180x60_rmin0p12_romset0_majit0:hdm}) vs.
         ROM-TR-RES [$\tau = 0.01$: (\ref{line:mbb0_nel180x60_rmin0p12_romset0_majit0:romB0}), $\tau = 0.1$: (\ref{line:mbb0_nel180x60_rmin0p12_romset0_majit0:romB1}), $\tau = 1$: (\ref{line:mbb0_nel180x60_rmin0p12_romset0_majit0:romB2}), $\tau = 10$: (\ref{line:mbb0_nel180x60_rmin0p12_romset0_majit0:romB3})]
         vs.
         ROM-FIX-RES [$\tau = 0.01$: (\ref{line:mbb0_nel180x60_rmin0p12_romset0_majit0:romA0}), $\tau = 0.1$: (\ref{line:mbb0_nel180x60_rmin0p12_romset0_majit0:romA1}), $\tau = 1$: (\ref{line:mbb0_nel180x60_rmin0p12_romset0_majit0:romA2}), $\tau = 10$: (\ref{line:mbb0_nel180x60_rmin0p12_romset0_majit0:romA3})]
         applied to compliance minimization of the MBB beam: the convergence
         of the objective function to its optimal value (\textit{top row}),
         the cumulative number of ROM solves required (\textit{middle row}),
         and the evolution of the trust-region radius (\textit{bottom row}).}
\label{fig:mbb0_nel180x60_rmin0p12_romset0_majit0}
\end{figure}

To assess the overall cost $C_\epsilon$ of the ROM-based optimization methods,
we account for the cost of the ROM solves as well as the
HDM solves according to \eqref{eqn:cost}. We note only $\Ocal(10^2)$ ROM
evaluations are required to drive the objective function to within the relative error of $\epsilon = 0.01$ for reasonable values of $\tau \in \{0.01, 0.1\}$. As a result, the ROM evaluations only add an equivalent of $1$ to $2$ HDM evaluations, and the ROM-based trust-region methods reduces the computational cost by nearly $33\%$ relative HDM-MMA (Table~\ref{tab:mbb0_nel180x60_rmin0p12}). Finally, Figure~\ref{fig:mbb0_nel180x60_rmin0p12_romset0_majit0}
shows the evolution of the trust-region radius; for the non-adaptive
methods, the radius is fixed at its initial value, whereas the
trust-region method adapts it. It is interesting to note that, for
the adaptive method, the trust-region radius converges to nearly the
same value, regardless of its starting point.

\begin{remark}
We repeated this experiment for a number of mesh resolutions, both
coarser and finer than the $180\times 60$ quadrilateral mesh used
here, and did not observe meaningful differences in the relative
performance of the methods.  
\end{remark}

To highlight the progress of the ROM-based trust-region algorithm (ROM-TR-RES) relative
to the HDM-MMA algorithm over the first $20$ major iterations, we compare the evolution of the filtered density field $\rhobold$  in Figure~\ref{fig:mbb0_viz}. After only $20$ major
iterations, the ROM-based method has visually converged to the optimal solution,
whereas the HDM-MMA method still has volume to remove and boundaries to
sharpen. Furthermore, by comparing the two algorithms at iteration $12$,
it is clear the ROM algorithm makes substantial progress toward the optimal
design in the early iterations.
\begin{figure}
\centering
\includegraphics[width=0.45\textwidth]{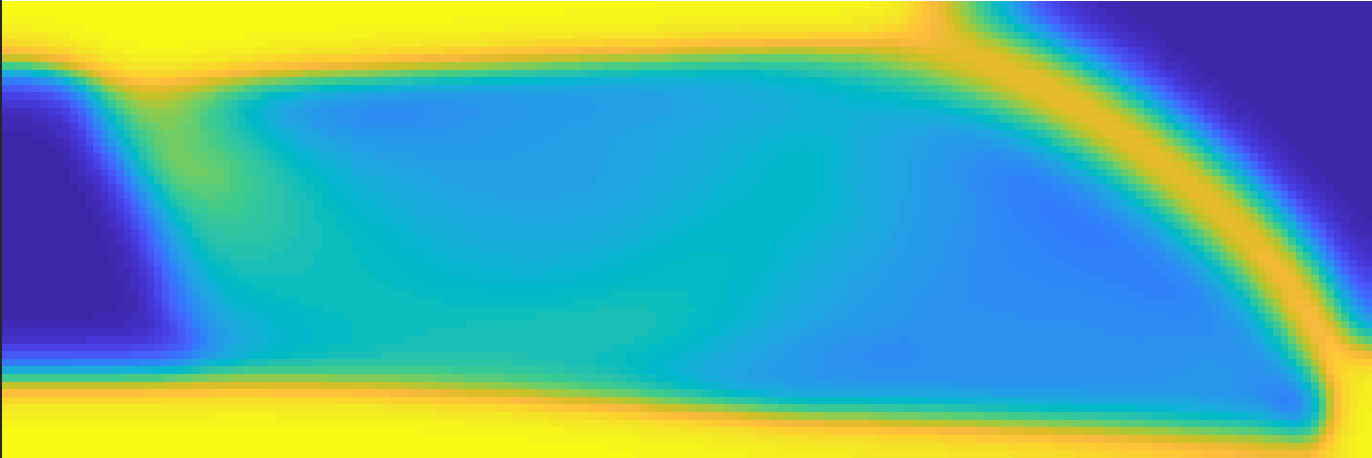} \quad
\includegraphics[width=0.45\textwidth]{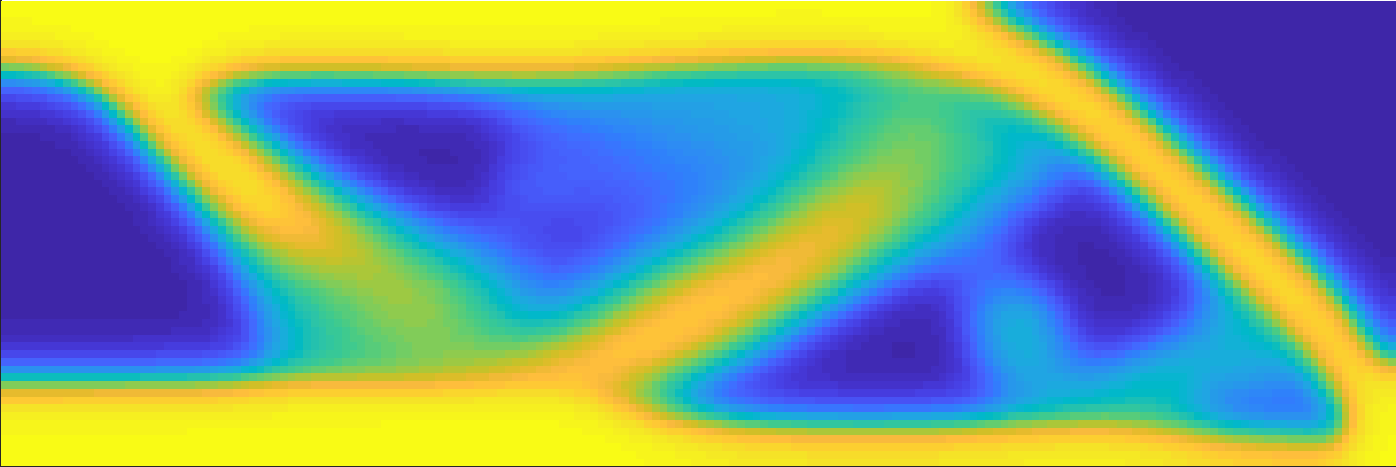} \\\vspace{2mm}
\includegraphics[width=0.45\textwidth]{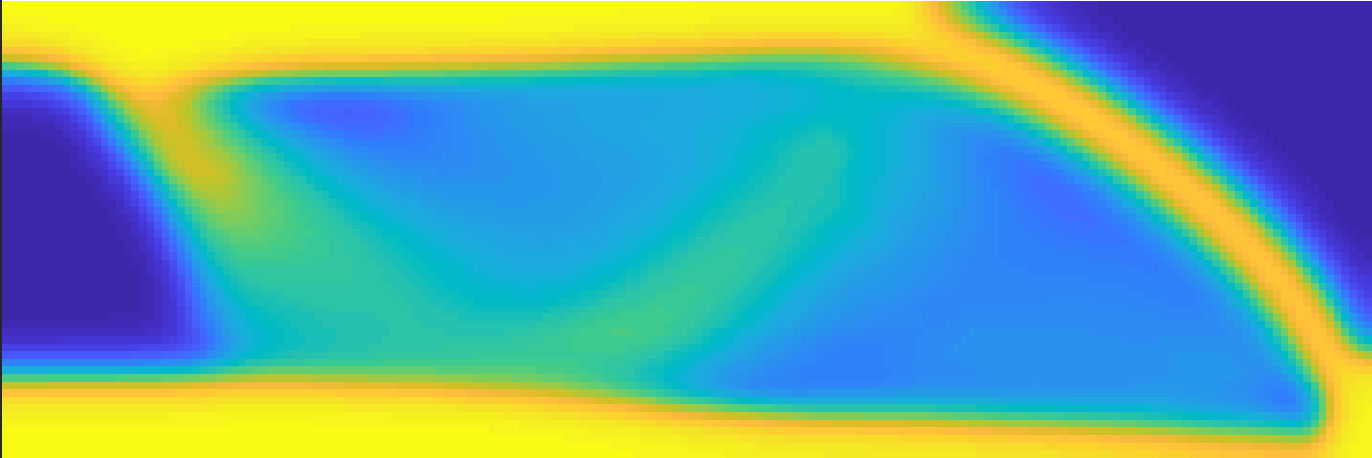} \quad
\includegraphics[width=0.45\textwidth]{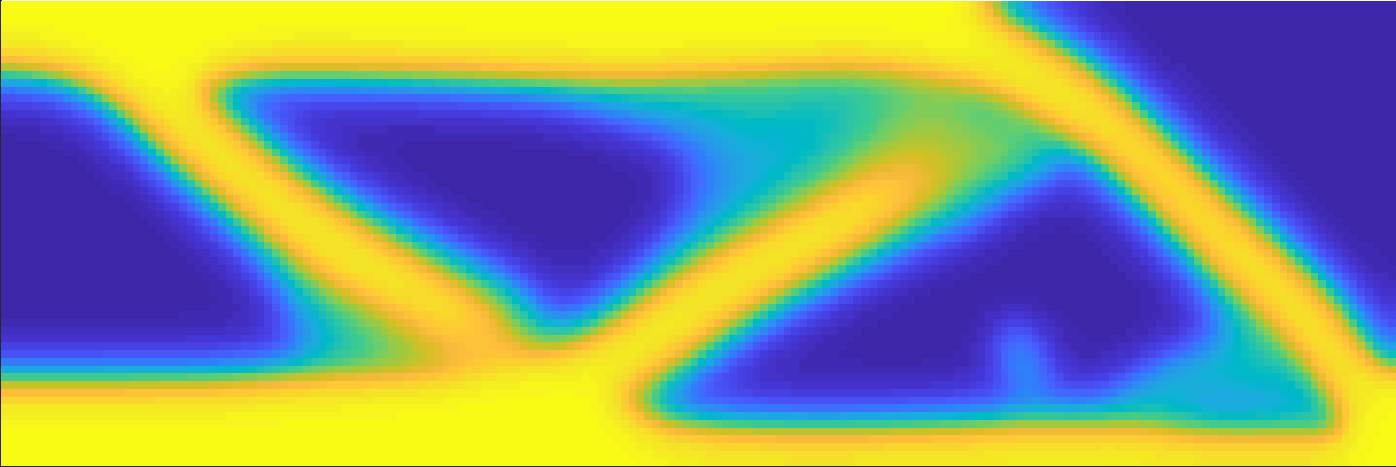} \\\vspace{2mm}
\includegraphics[width=0.45\textwidth]{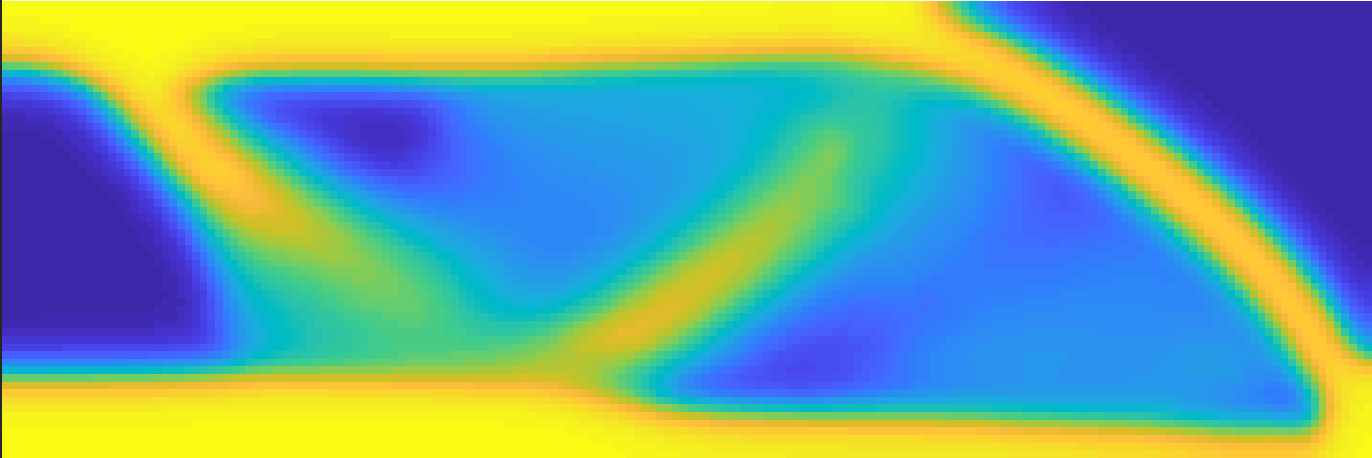} \quad
\includegraphics[width=0.45\textwidth]{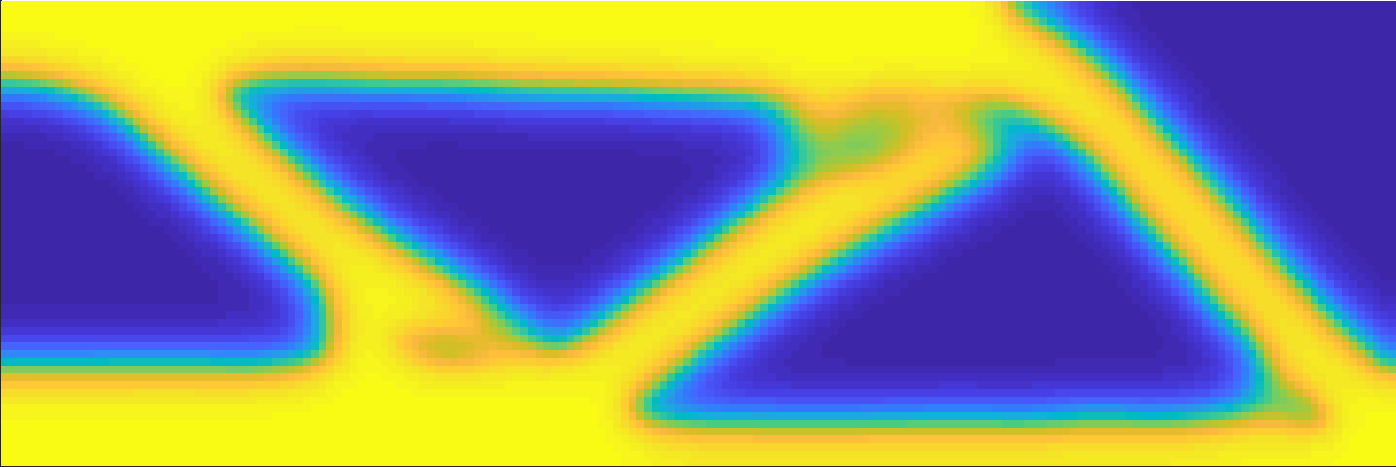} \\\vspace{2mm}
\includegraphics[width=0.45\textwidth]{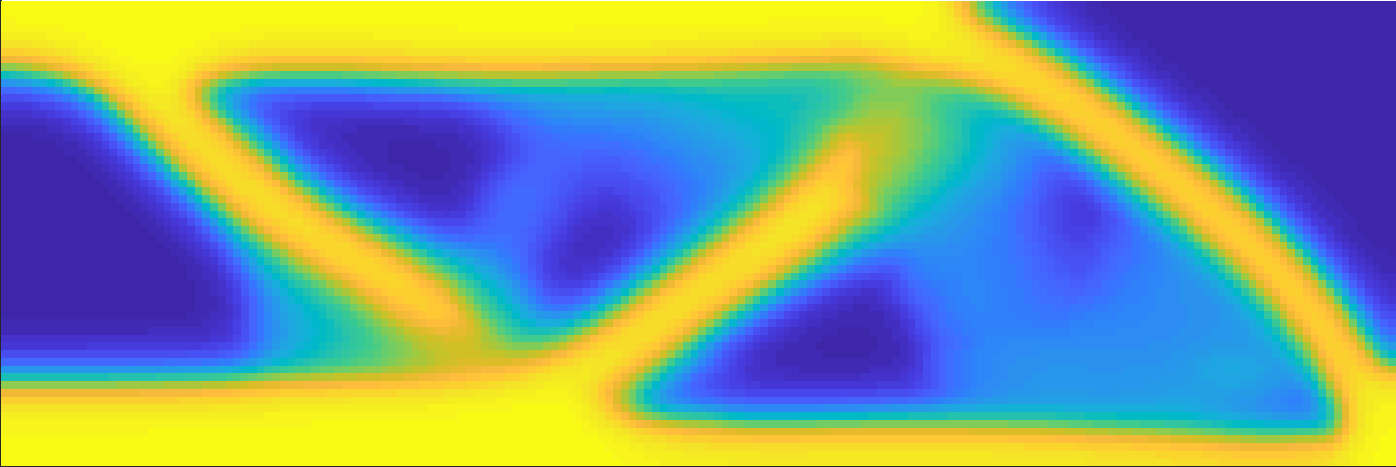} \quad
\includegraphics[width=0.45\textwidth]{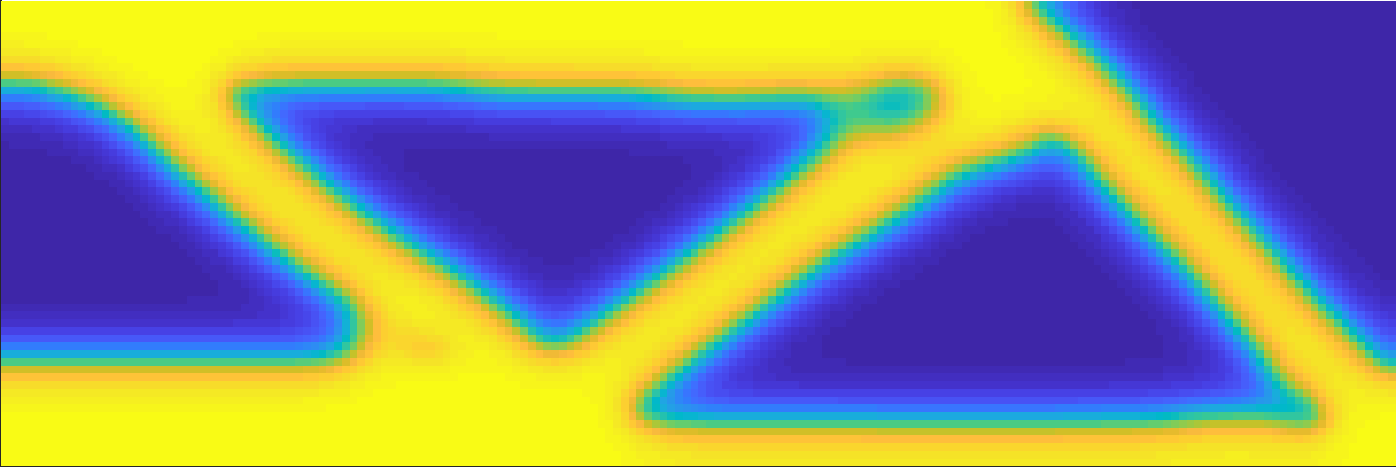} \\\vspace{2mm}
\includegraphics[width=0.45\textwidth]{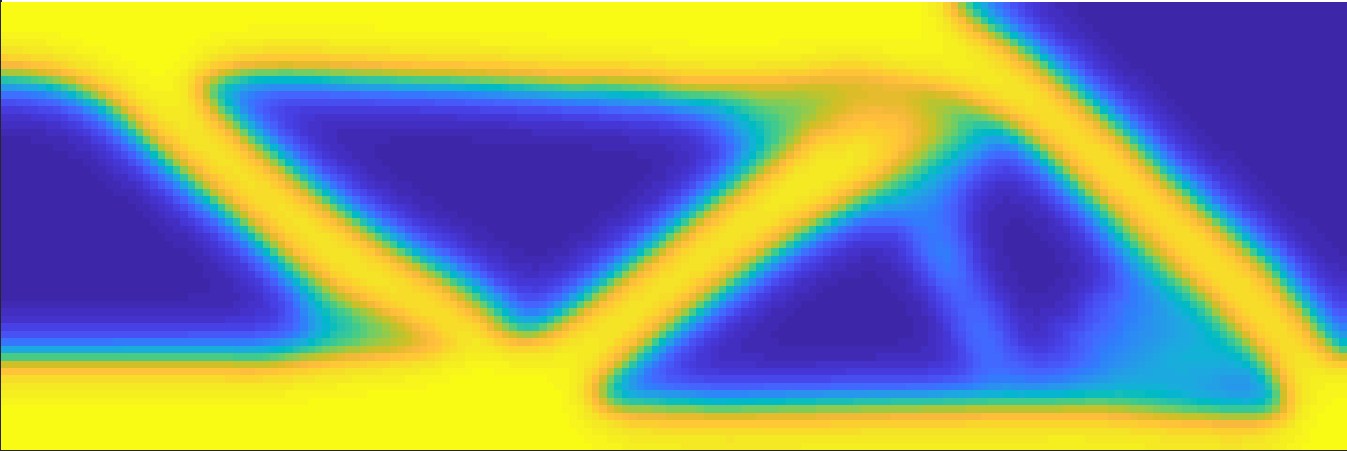} \quad
\includegraphics[width=0.45\textwidth]{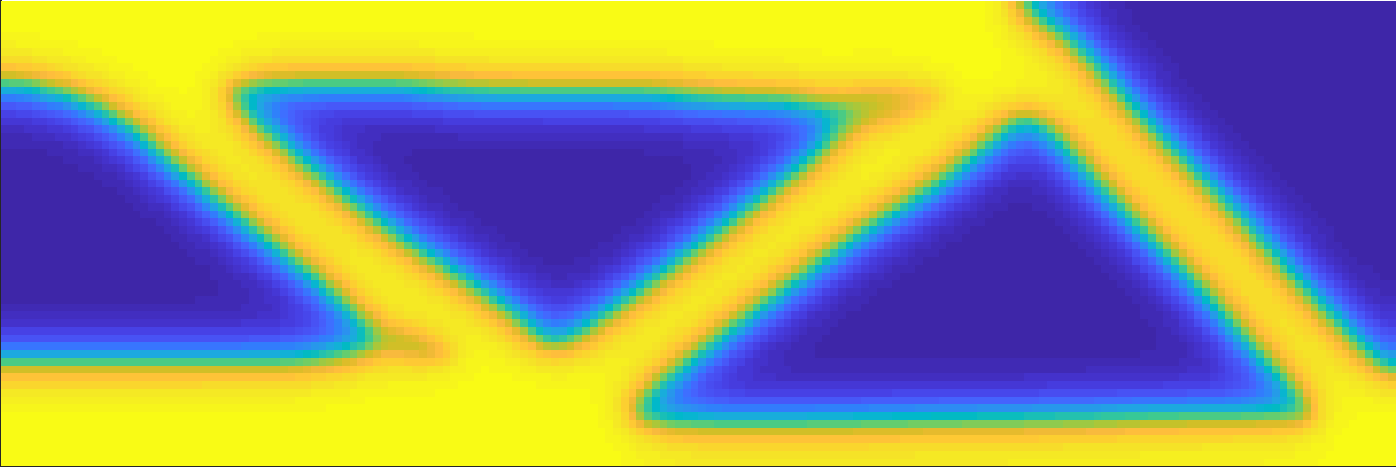}
\caption{The filtered density fields for the MBB beam obtained using the HDM-MMA
         (\textit{left}) and ROM-TR-RES (\textit{right}) methods at major
  iterations (\textit{top-to-bottom}): 10, 12, 14, 16, 20.}
\label{fig:mbb0_viz}
\end{figure}

We close this section with a comparison of the two ROM-based trust-region methods: ROM-TR-RES and ROM-TR-DIST. Both algorithms perform similarly and are
competitive relative to the HDM-MMA approach
(Figure~\ref{fig:mbb0_nel180x60_rmin0p12_romset1_majit0},
 Table~\ref{tab:mbb0_nel180x60_rmin0p12}) with the ROM-TR-DIST method
converging slightly faster than ROM-TR-RES. 
\begin{figure}
\centering
\input{py/mbb0_nel180x60_rmin0p12_romset1_majit0_matinterp1.tikz}
\caption{Performance of
         HDM-MMA (\ref{line:mbb0_nel180x60_rmin0p12_romset1_majit0:hdm}) vs.
         ROM-TR-RES [$\tau = 0.01$: (\ref{line:mbb0_nel180x60_rmin0p12_romset1_majit0:romB0}), $\tau = 0.1$: (\ref{line:mbb0_nel180x60_rmin0p12_romset1_majit0:romB1})]
         vs.
         ROM-TR-DIST [$\tau = 0.01$: (\ref{line:mbb0_nel180x60_rmin0p12_romset1_majit0:romA0}), $\tau = 0.1$: (\ref{line:mbb0_nel180x60_rmin0p12_romset1_majit0:romA1})]
         applied to compliance minimization of the MBB beam: the convergence
         of the objective function to its optimal value (\textit{top row}),
         the cumulative number of ROM solves required (\textit{middle row}),
         and the evolution of the trust-region radius (\textit{bottom row}).}
\label{fig:mbb0_nel180x60_rmin0p12_romset1_majit0}
\end{figure}
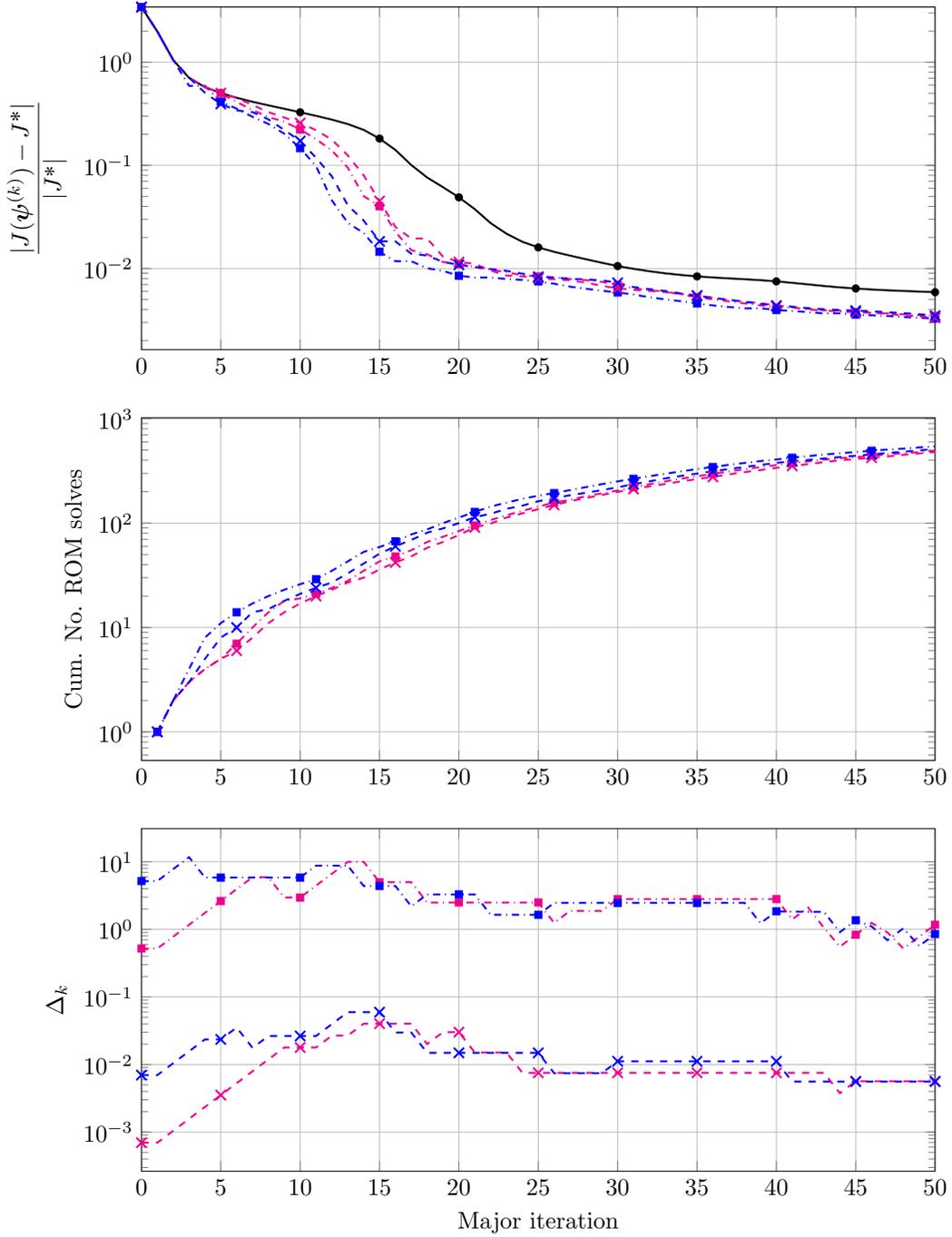
\begin{remark}
We have extensively tested the ROM-based trust-region methods where the
trust-region subproblem is solved exactly using nonlinear programming methods
for this benchmark. We found a large number of ROM evaluations ($\Ocal(10^3)$) are required to solve each subproblem, at least for the suite of optimization methods tested. Therefore, these methods are not competitive in terms of the overall cost. See also the discussion at the end of Section~\ref{sec:tr:etr-topopt}.
\end{remark}
\begin{table}
\caption{Performance of HDM-MMA vs.\ ROM-TR-RES vs.\ ROM-TR-DIST applied to
         compliance minimization of the MBB beam as a function
         of convergence tolerance. The reference value of the optimal objective is $J^* = 19.96$.}
\label{tab:mbb0_nel180x60_rmin0p12}
\begin{tabular}{r|cc|c|ccc}
 ROM-TR-DIST& $\epsilon$ & $\tau$ & final objective & \# HDM solves & \# ROM solves & cost ($C_\epsilon$) \\\hline
 \input{py/mbb0_nel180x60_rmin0p12_majit0_matinterp1.dat}
\end{tabular}
\end{table}

\subsection{Cantilever beam}
Next, we consider compliance minimization of a cantilever beam
(Figure~\ref{fig:cbeam0_setup}) using a finite element mesh consisting
of $160\times100$ bilinear quadrilateral elements, minimum filtering length scale of $R = 2$, and maximum volume of
$V = \frac{1}{2} |\Omega|$.  The optimal design is shown in Figure~\ref{fig:cbeam0_setup} ($J^* = 394.71$).
All algorithms studied in this section are initialized from the same
feasible design: $\psibold^{(0)} = 0.5\cdot\onebold$.
\begin{figure}
\centering
\input{py/cbeam0_geom.tikz}
\caption{Setup and optimal design of cantilever beam. The point load
         $F_\text{in} = 3$ is implemented as a distributed load of magnitude
         $q_\text{in} = 1$ applied to a segment of length $3$.}
\label{fig:cbeam0_setup}
\end{figure}
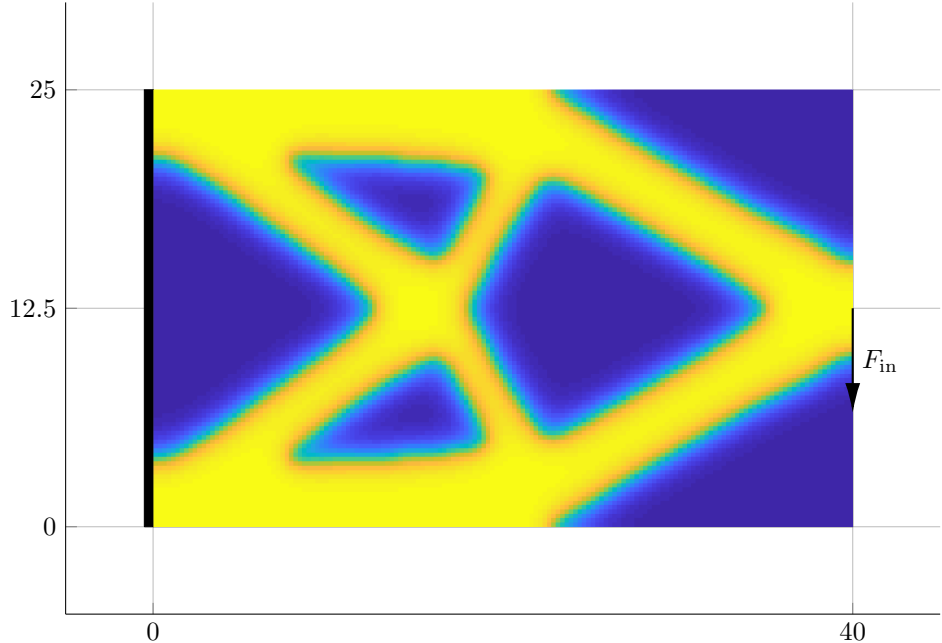

Having established the superiority of the (adaptive) trust-region methods
over the fixed-radius methods for the MBB beam in Section~\ref{sec:numexp:mbb},
we consider only the (adaptive) trust-region methods for the cantilever beam (Figure~\ref{fig:cbeam0_nel160x100_rmin2_romset1_majit0},
 Table~\ref{tab:cbeam0_nel160x100_rmin2}).
The HDM-MMA algorithm requires $27$ iterations to converge to a
relative objective function tolerance of $\epsilon = 0.01$, and significantly more iterations
to converge to tighter tolerances; e.g., $468$ iterations for $\epsilon=0.001$.
In contrast, ROM-TR-DIST with an initial trust-region radius multiplier of $\tau = 0.1$ requires only $12$ HDM evaluations and $62$ ROM
evaluations to converge to $\epsilon = 0.01$. As the cost for all ROM evaluations is less than a single HDM evaluation, the overall cost is reduced by more than a factor of two. Even
larger reductions are observed for tighter tolerances; e.g., to achieve $\epsilon = 0.001$, ROM-TR-DIST requires $58$ HDM evaluations and $934$ ROM evaluations (cost of $9-10$ HDM evaluations) and reduces the overall computational cost by a factor of $\approx 7$.
\begin{figure}
\centering
\input{py/cbeam0_nel160x100_rmin2_romset1_majit0_matinterp1.tikz}
\caption{Performance of
         HDM-MMA (\ref{line:cbeam0_nel160x100_rmin2_romset1_majit0:hdm}) vs.\
         ROM-TR-RES [$\tau = 0.1$: (\ref{line:cbeam0_nel160x100_rmin2_romset1_majit0:romB1}), $\tau = 1$: (\ref{line:cbeam0_nel160x100_rmin2_romset1_majit0:romB2})]
         vs.\
         ROM-TR-DIST [$\tau = 0.1$: (\ref{line:cbeam0_nel160x100_rmin2_romset1_majit0:romA1}), $\tau = 1$: (\ref{line:cbeam0_nel160x100_rmin2_romset1_majit0:romA2})]
         applied to compliance minimization of the cantilever beam:
         the convergence of the objective function to its optimal value
         (\textit{top row}) and the cumulative number of ROM solves required
         (\textit{bottom row}).}
\label{fig:cbeam0_nel160x100_rmin2_romset1_majit0}
\end{figure}
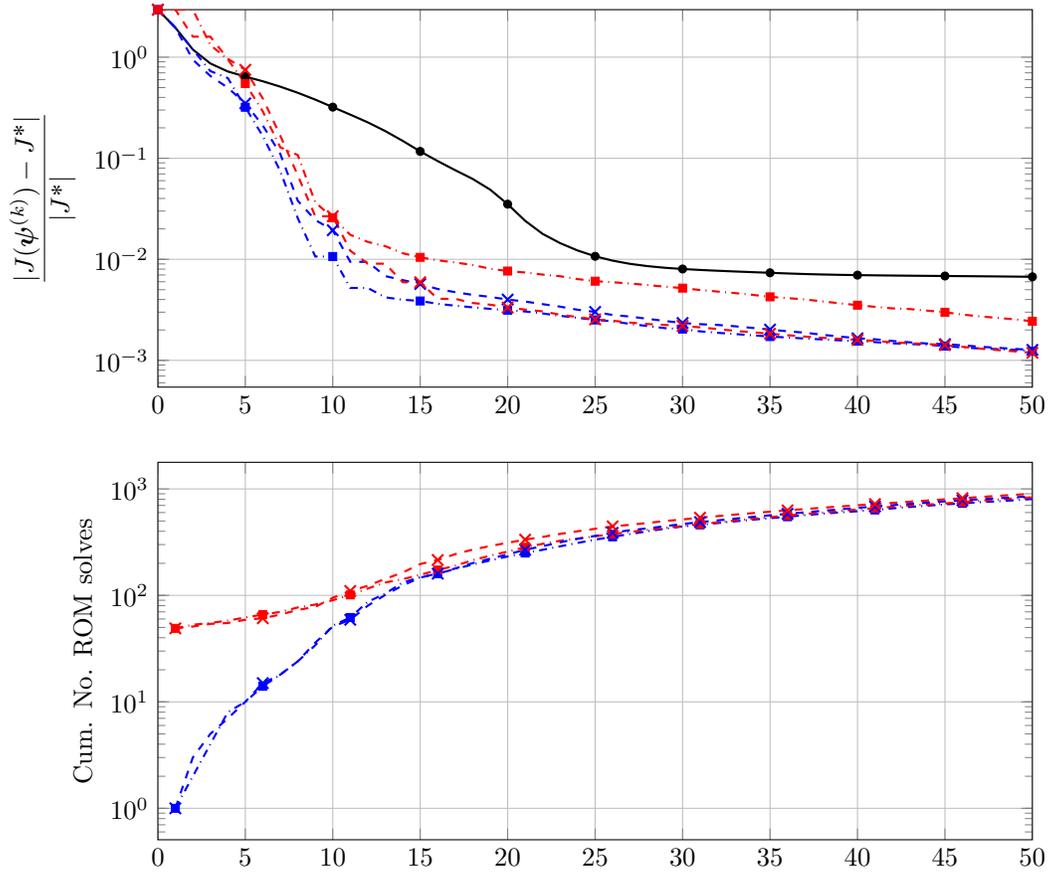
\begin{table}
\caption{Performance of HDM-MMA vs. ROM-TR-RES vs. ROM-TR-DIST applied to
         compliance minimization of the cantilever beam as a function
         of convergence tolerance. The reference value of the optimal objective is $J^* = 394.71$.}
\label{tab:cbeam0_nel160x100_rmin2}
\begin{tabular}{r|cc|c|ccc}
 & $\epsilon$ & $\tau$ & final objective & \# HDM solves & \# ROM solves & cost ($C_\epsilon$) \\\hline
 \input{py/cbeam0_nel160x100_rmin2_majit0_matinterp1.dat}
\end{tabular}
\end{table}

Figure~\ref{fig:cbeam0_viz} shows the filtered density fields for the ROM-based trust-region method (ROM-TR-RES) and HDM-MMA at select major iterations.
The comparison shows that ROM-TR-RES obtains a good approximation
to the optimal design after only $14$ iterations; the HDM-MMA design contains suboptimal artifacts even after $20$ iterations.  Furthermore, after
only $10$ iterations, the ROM algorithm has essentially discovered the
optimal topology and uses the next several iterations to refine it;
the internal structure has not emerged at this point from the HDM-MMA
method.
\begin{figure}
\centering
\includegraphics[width=0.45\textwidth]{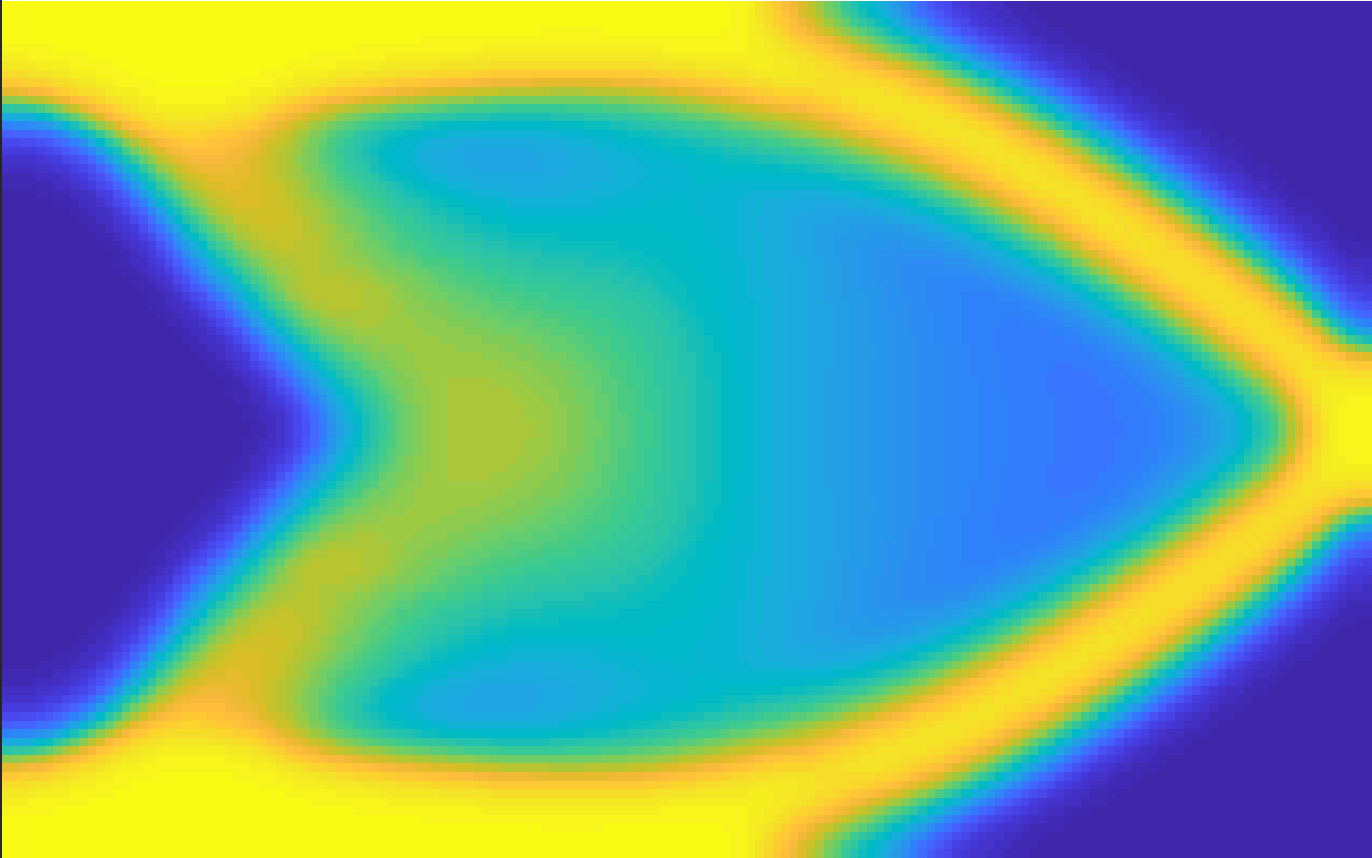} \quad
\includegraphics[width=0.45\textwidth]{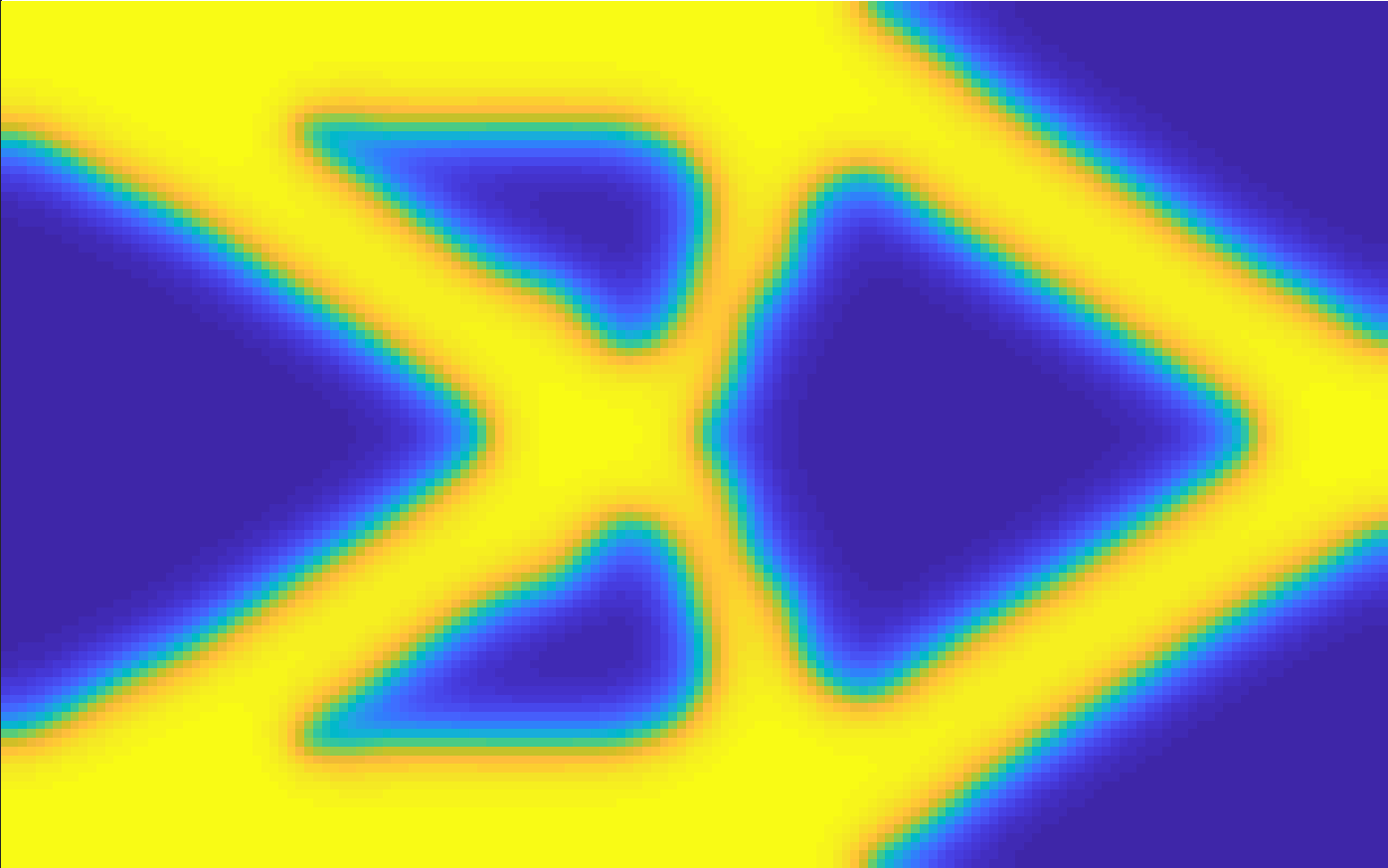} \\\vspace{2mm}
\includegraphics[width=0.45\textwidth]{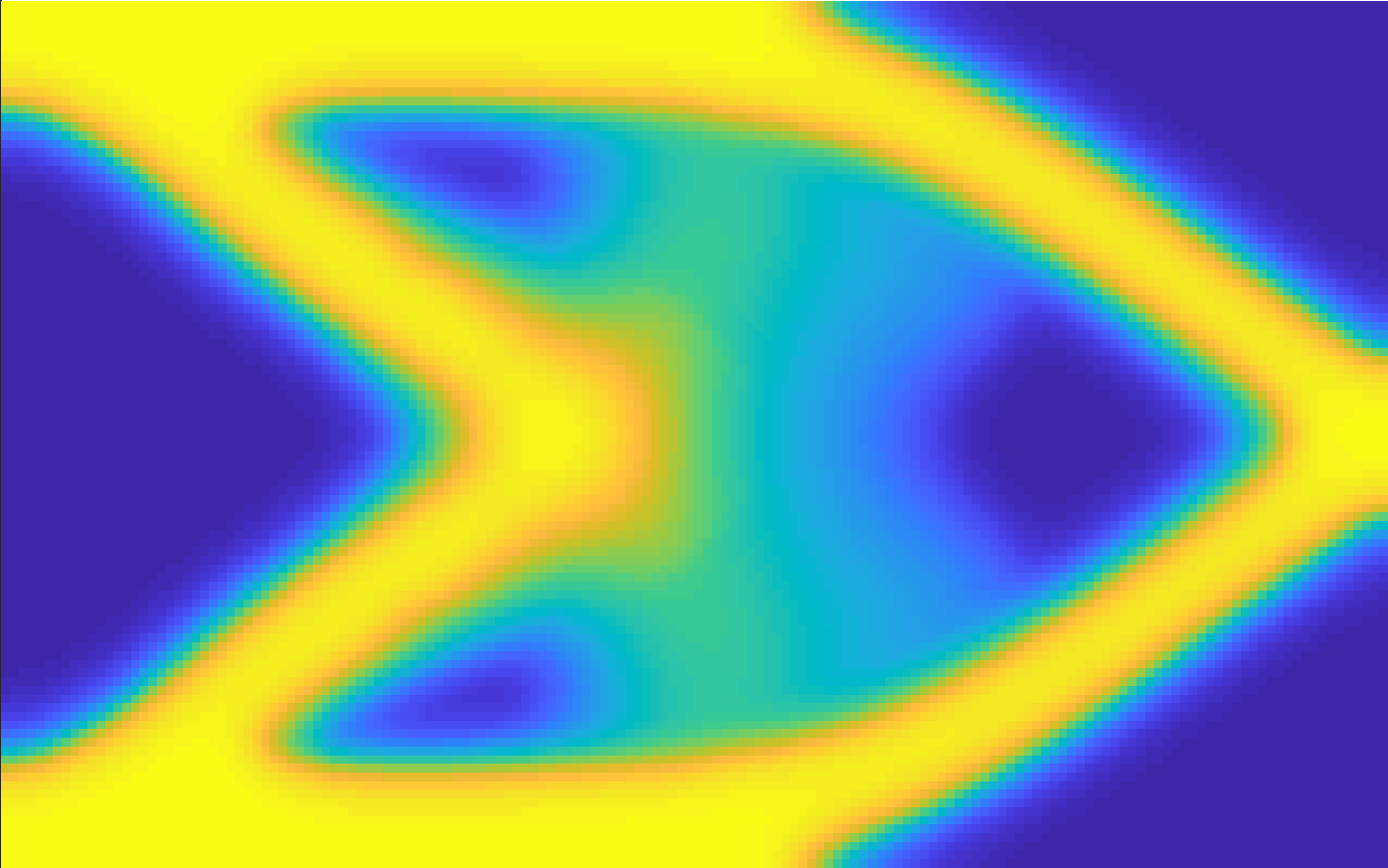} \quad
\includegraphics[width=0.45\textwidth]{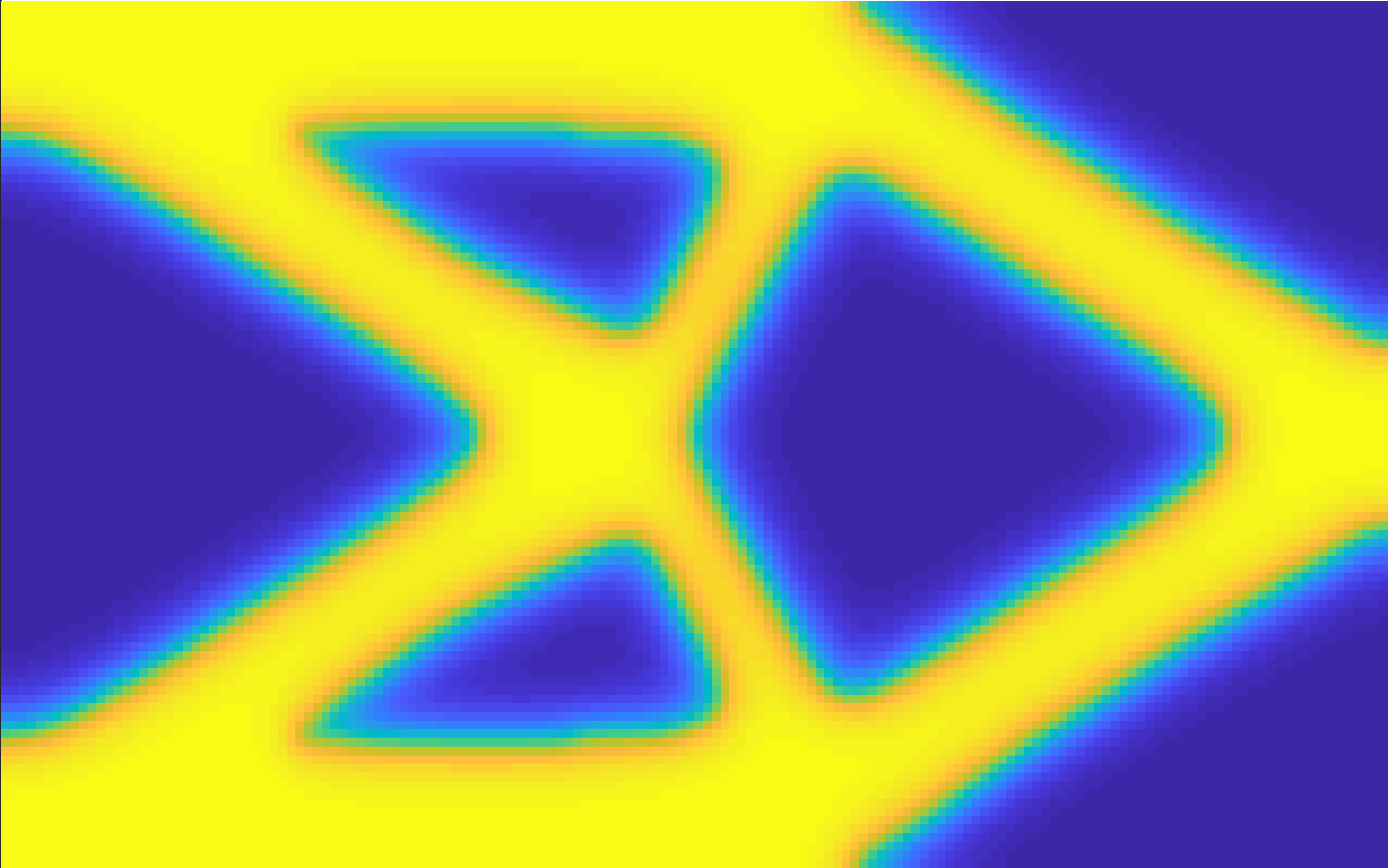} \\\vspace{2mm}
\includegraphics[width=0.45\textwidth]{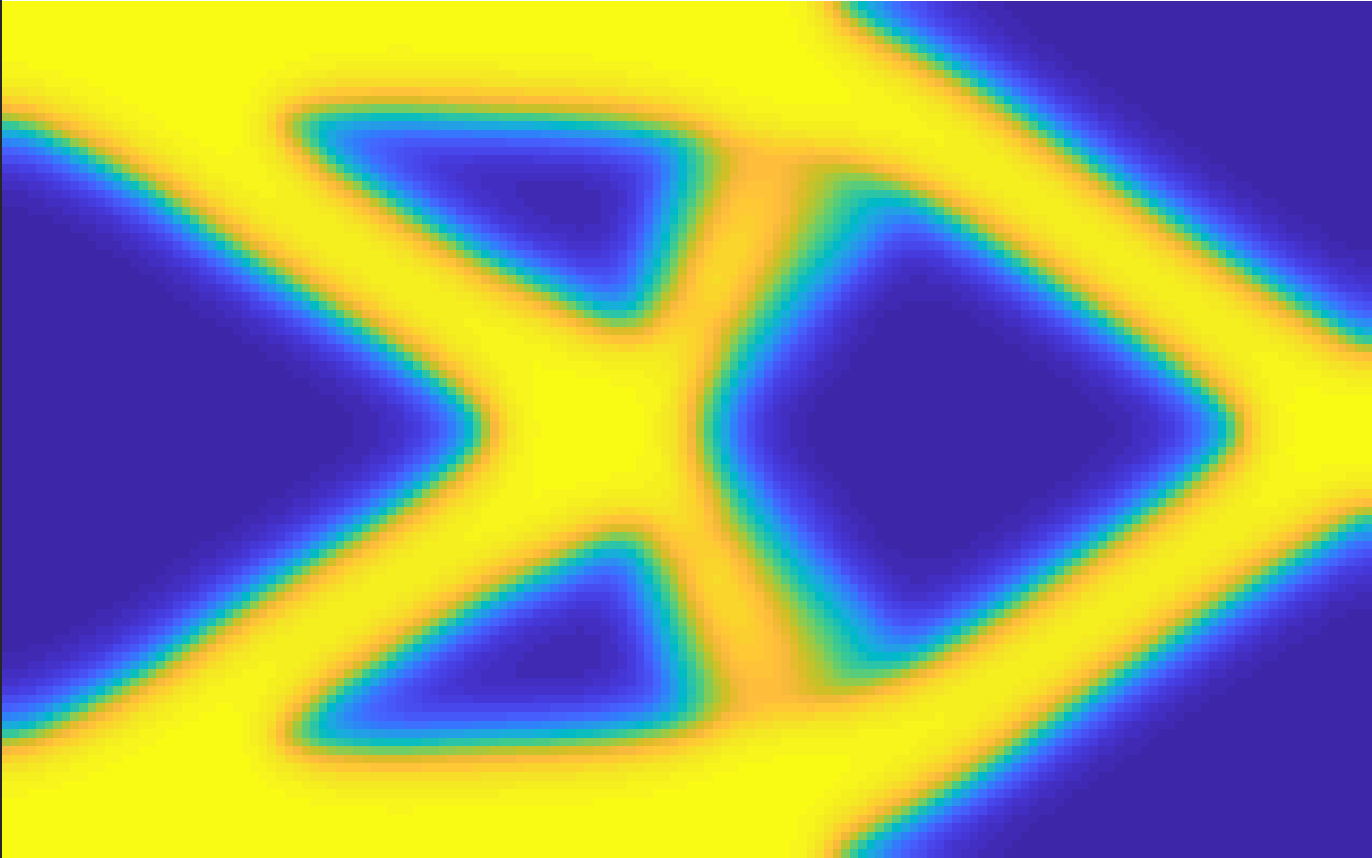} \quad
\includegraphics[width=0.45\textwidth]{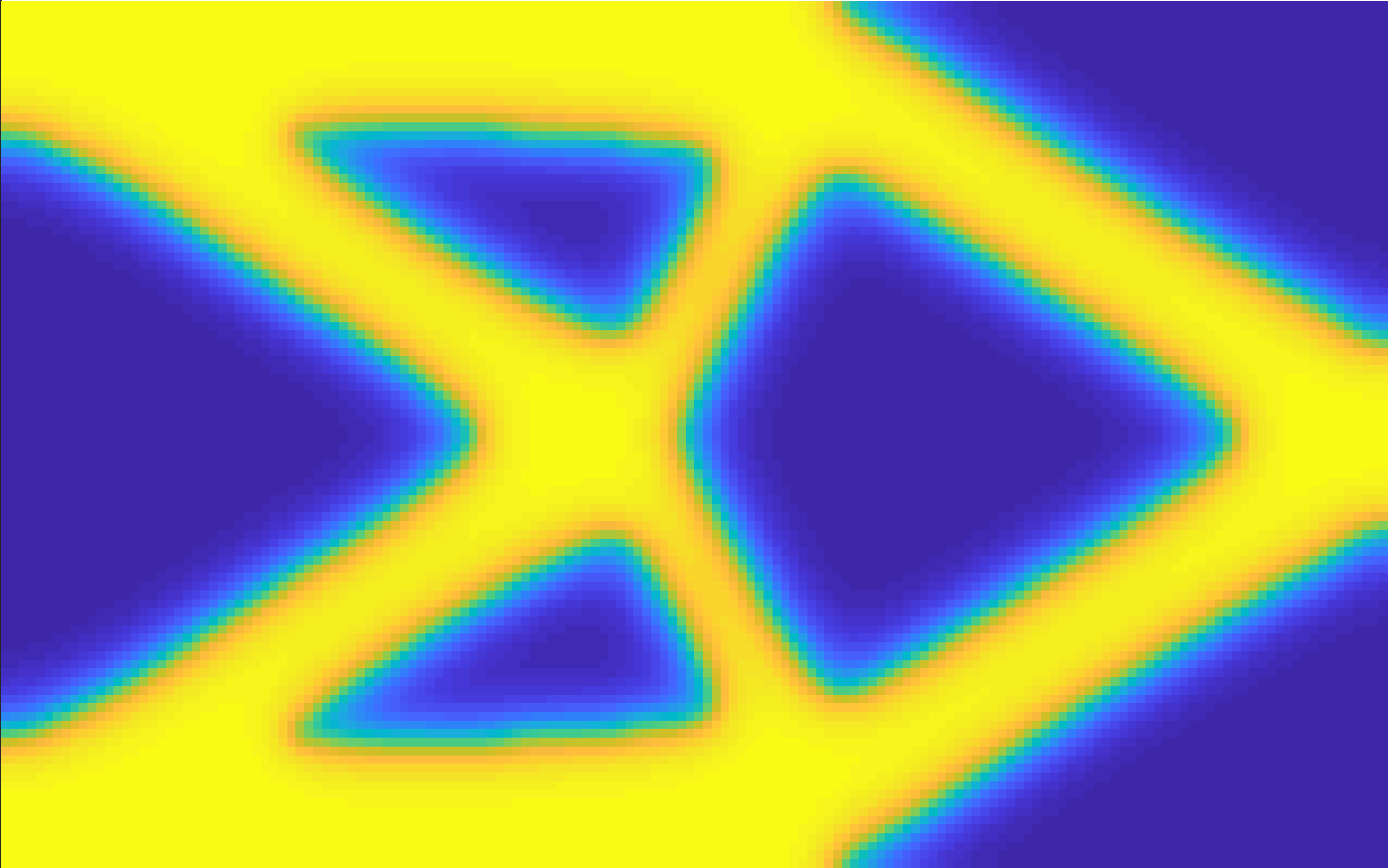}
\caption{Evolution of the design of the cantilever beam using the HDM-MMA
         (\textit{left}) and ROM-TR-RES (\textit{right}) methods at major
  iterations (\textit{top-to-bottom}): 10, 14, 20.}
\label{fig:cbeam0_viz}
\end{figure}

\subsection{Simply supported beam}
Finally, we consider compliance minimization of a simply supported beam
(Figure~\ref{fig:ssbeam0_setup}) using a finite element mesh consisting
of $180\times 90$ bilinear quadrilateral elements, minimum length scale of $R = 0.5$, and maximum volume of $V = \frac{2}{5}|\Omega|$.
The optimal design is shown in Figure~\ref{fig:ssbeam0_setup}
($J^* = 153.92$). All algorithms studied
in this section are initialized from the same feasible design:
$\psibold^{(0)} = 0.4 \cdot \onebold$.
\begin{figure}
\centering
\input{py/ssbeam0_geom.tikz}
\caption{Setup and optimal design of a simply supported beam. The point load
         $F_\text{in} = 3$ is implemented as a distributed load of magnitude
         $q_\text{in} = 1$ applied to a segment of length $3$.}
\label{fig:ssbeam0_setup}
\end{figure}
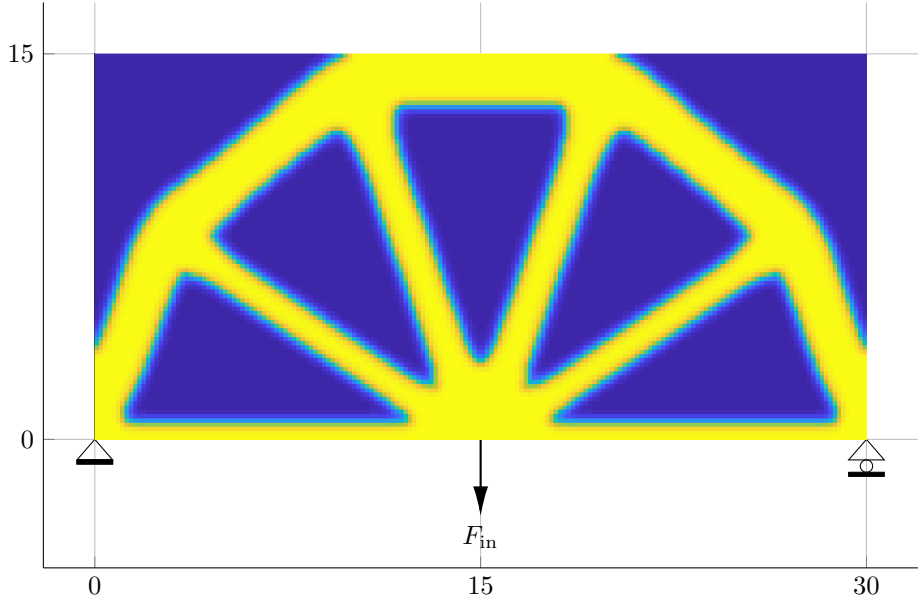

We consider the performance of the ROM-based trust-region methods
relative to the HDM-MMA algorithm
(Figure~\ref{fig:ssbeam0_nel180x90_rmin0p5_romset1_majit0},
 Table~\ref{tab:ssbeam0_nel180x90_rmin0p5}). Unlike the previous
problems, HDM-MMA converges to relatively tight tolerances rather
quickly, e.g., $36$ iterations for $\epsilon = 0.01$ and
$207$ iterations for $\epsilon=0.001$. However, the ROM-based trust-region
methods are faster; the relative speedup is similar to the previous problems.
For example, ROM-TR-DIST ($\tau=0.1$) requires only $14$ HDM evaluations
and $54$ ROM evaluations to converge to $\epsilon=0.01$, which translates to a speedup of greater
than two. For the tighter tolerance $\epsilon=0.001$, ROM-TR-DIST only requires
$19$ HDM evaluations and $105$ ROM evaluations, which translates to an order of magnitude speedup.
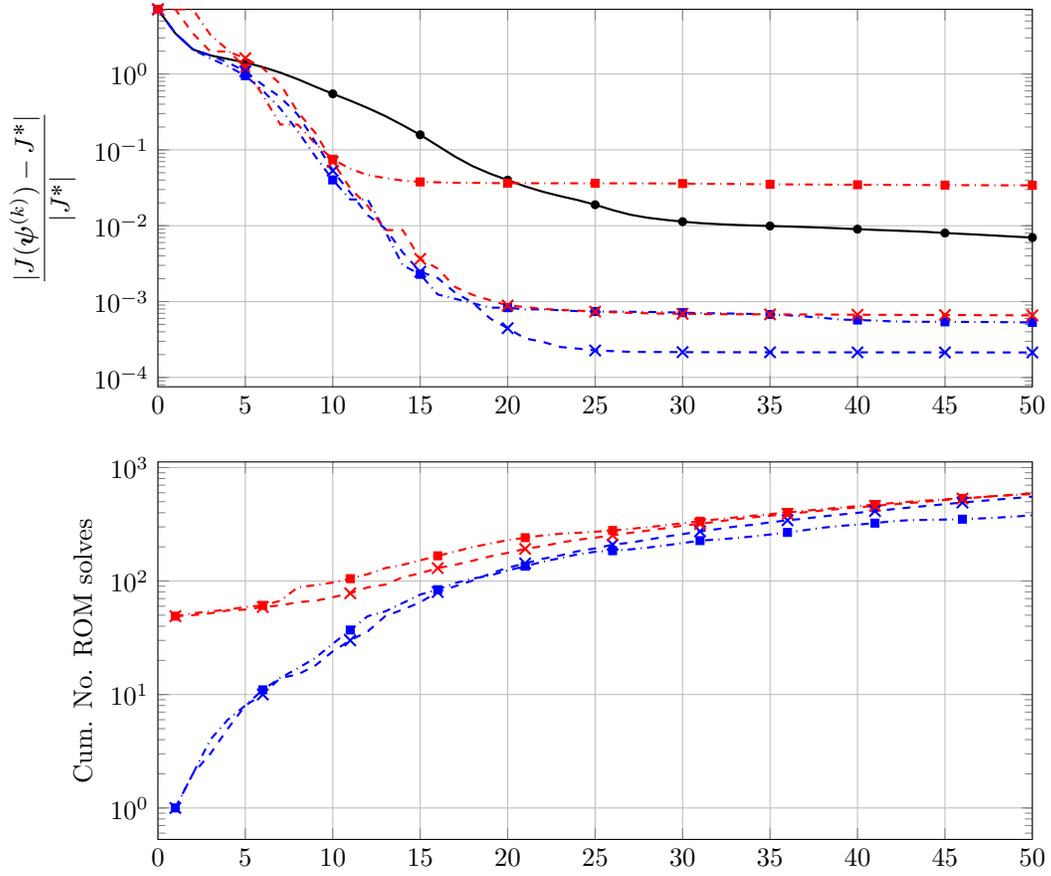
\begin{figure}
\centering
\input{py/ssbeam0_nel180x90_rmin0p5_romset1_majit0_matinterp1.tikz}
\caption{Performance of
         HDM-MMA (\ref{line:ssbeam0_nel180x90_rmin0p5_romset1_majit0:hdm}) vs.
         ROM-TR-RES [$\tau = 0.1$: (\ref{line:ssbeam0_nel180x90_rmin0p5_romset1_majit0:romB1}), $\tau = 1$: (\ref{line:ssbeam0_nel180x90_rmin0p5_romset1_majit0:romB2})]
         vs.
         ROM-TR-DIST [$\tau = 0.1$: (\ref{line:ssbeam0_nel180x90_rmin0p5_romset1_majit0:romA1}), $\tau = 1$: (\ref{line:ssbeam0_nel180x90_rmin0p5_romset1_majit0:romA2})]
         applied to compliance minimization of the simply supported beam:
         the convergence of the objective function to its optimal value
         (\textit{top row}) and the cumulative number of ROM solves required
         (\textit{bottom row}).}
\label{fig:ssbeam0_nel180x90_rmin0p5_romset1_majit0}
\end{figure}

Figure~\ref{fig:ssbeam0_viz} shows the filtered density fields for the ROM-based trust-region method (ROM-TR-RES) and HDM-MMA at select major iterations. The comparison reveals ROM-TR-RES obtains a good approximation
to the optimal design after only $10$ iterations; the HDM-MMA design contains suboptimal artifacts even after $30$ iterations.
\begin{figure}
\centering
\includegraphics[width=0.45\textwidth]{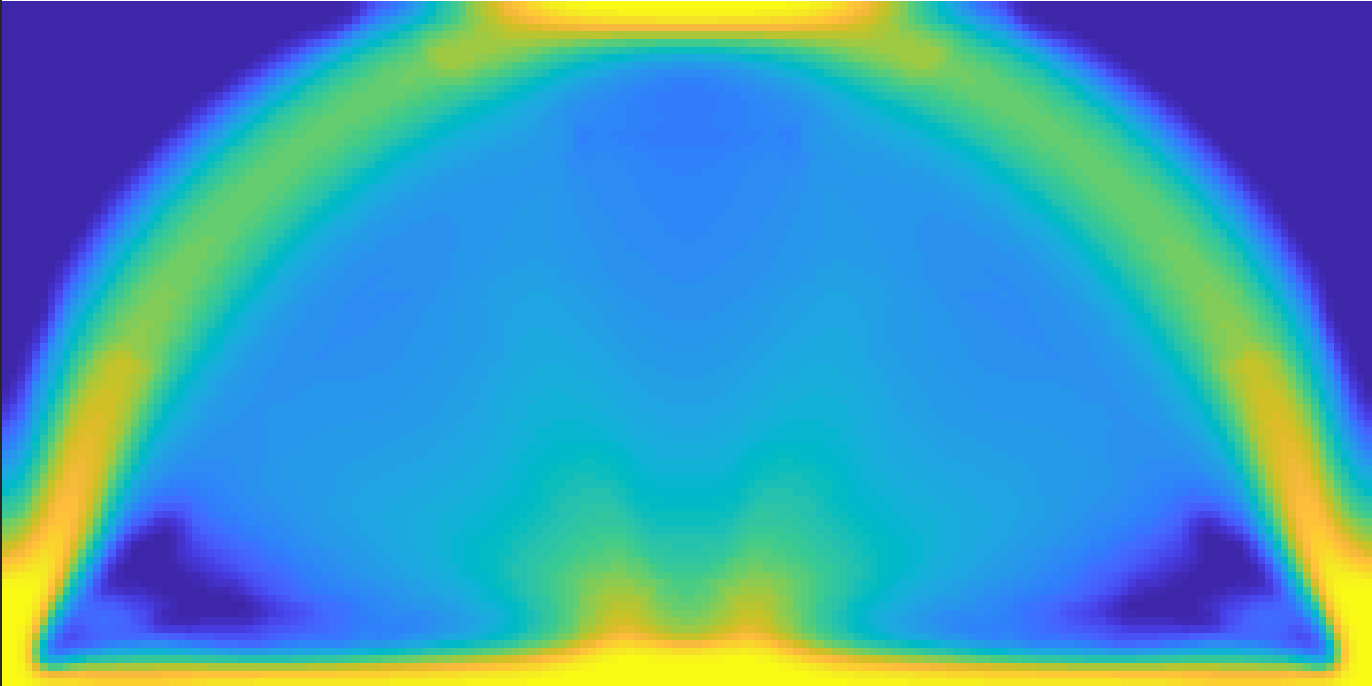} \quad
\includegraphics[width=0.45\textwidth]{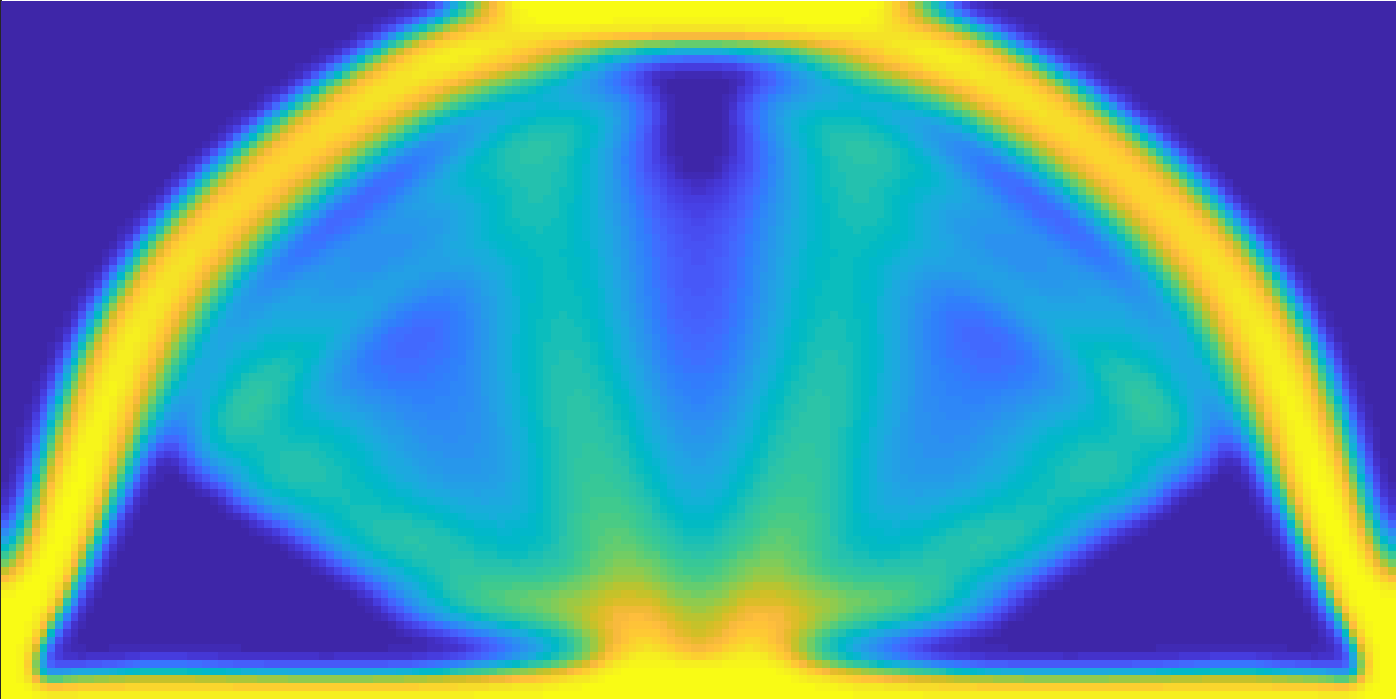} \\\vspace{2mm}
\includegraphics[width=0.45\textwidth]{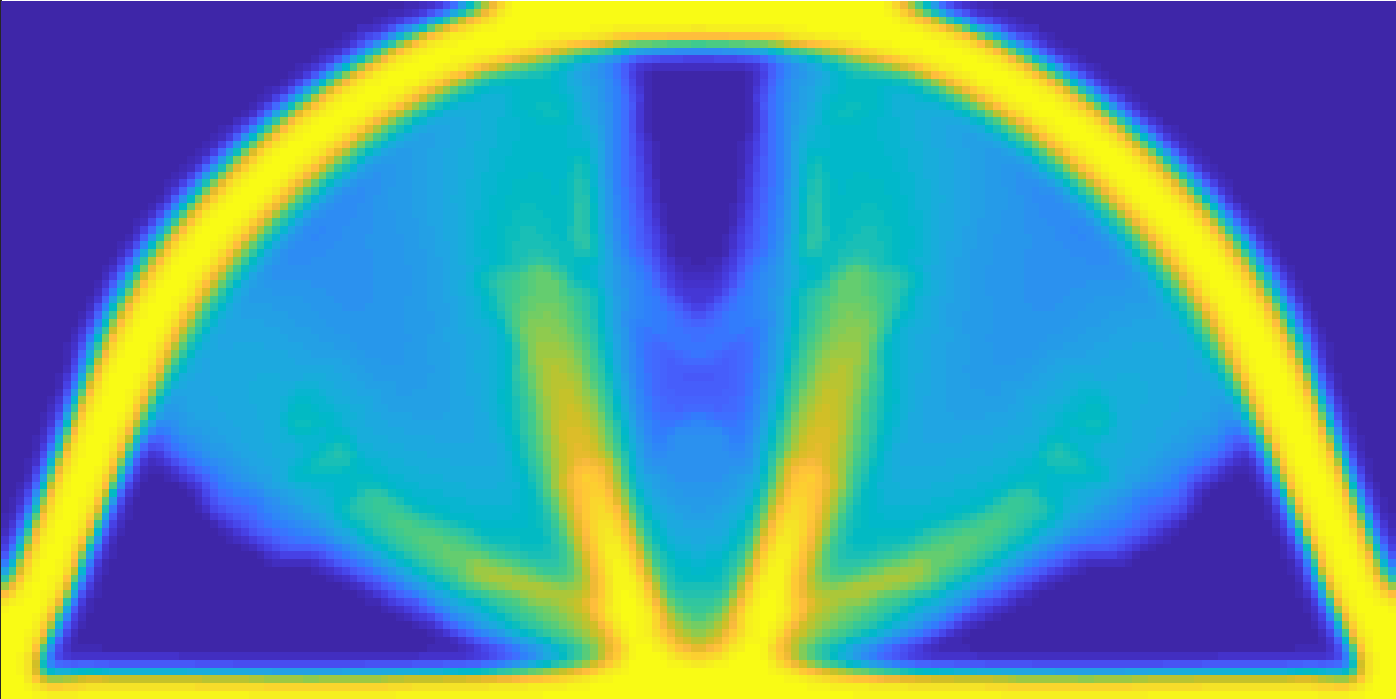} \quad
\includegraphics[width=0.45\textwidth]{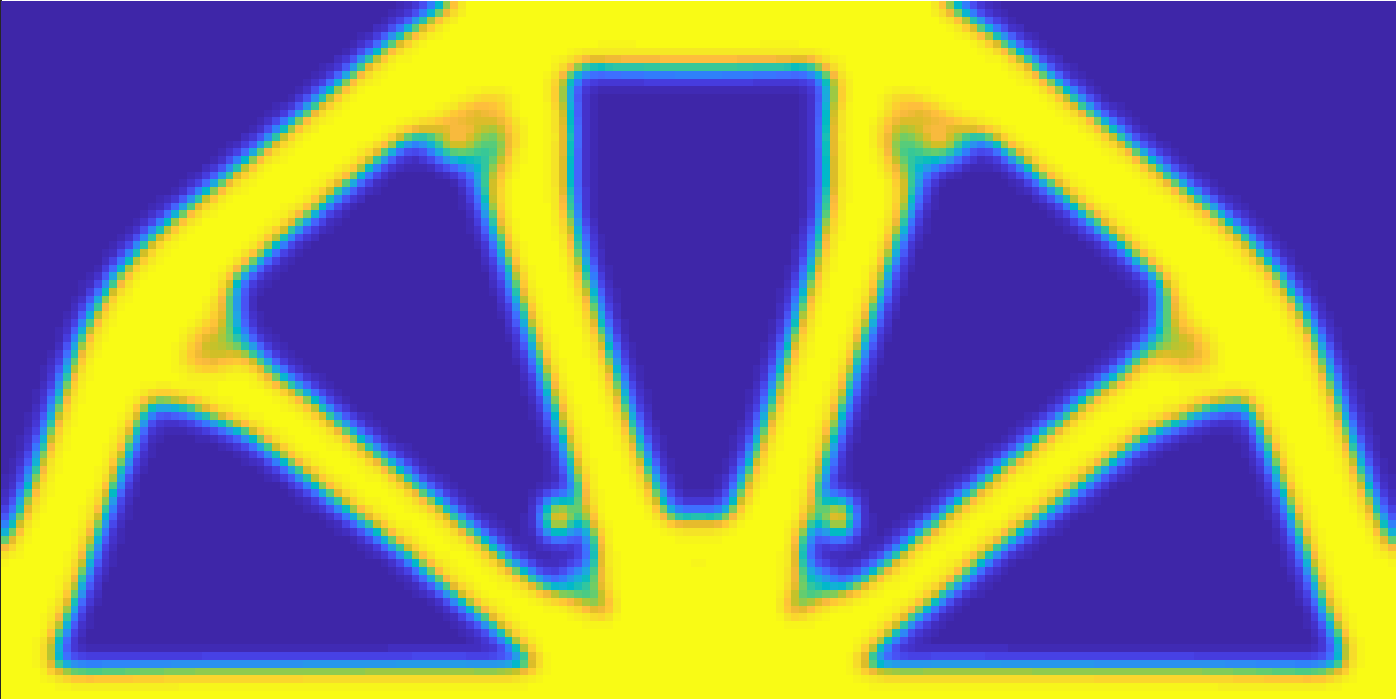} \\\vspace{2mm}
\includegraphics[width=0.45\textwidth]{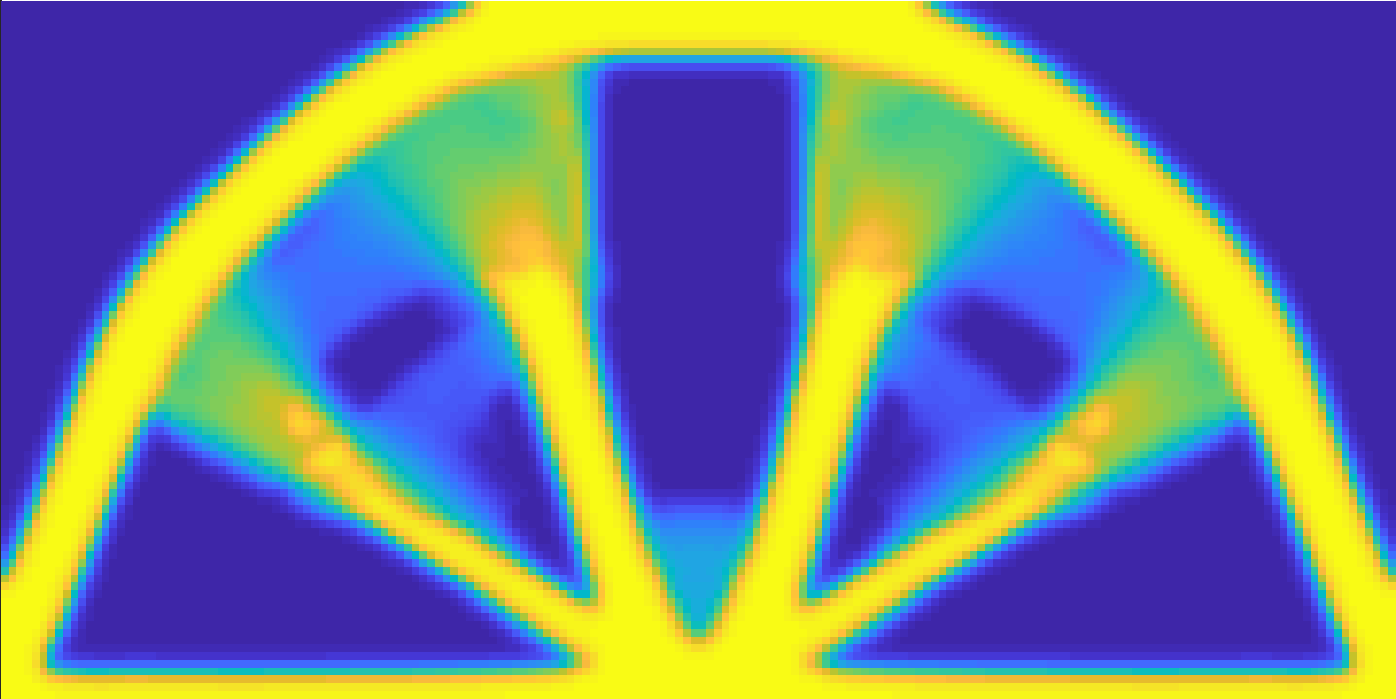} \quad
\includegraphics[width=0.45\textwidth]{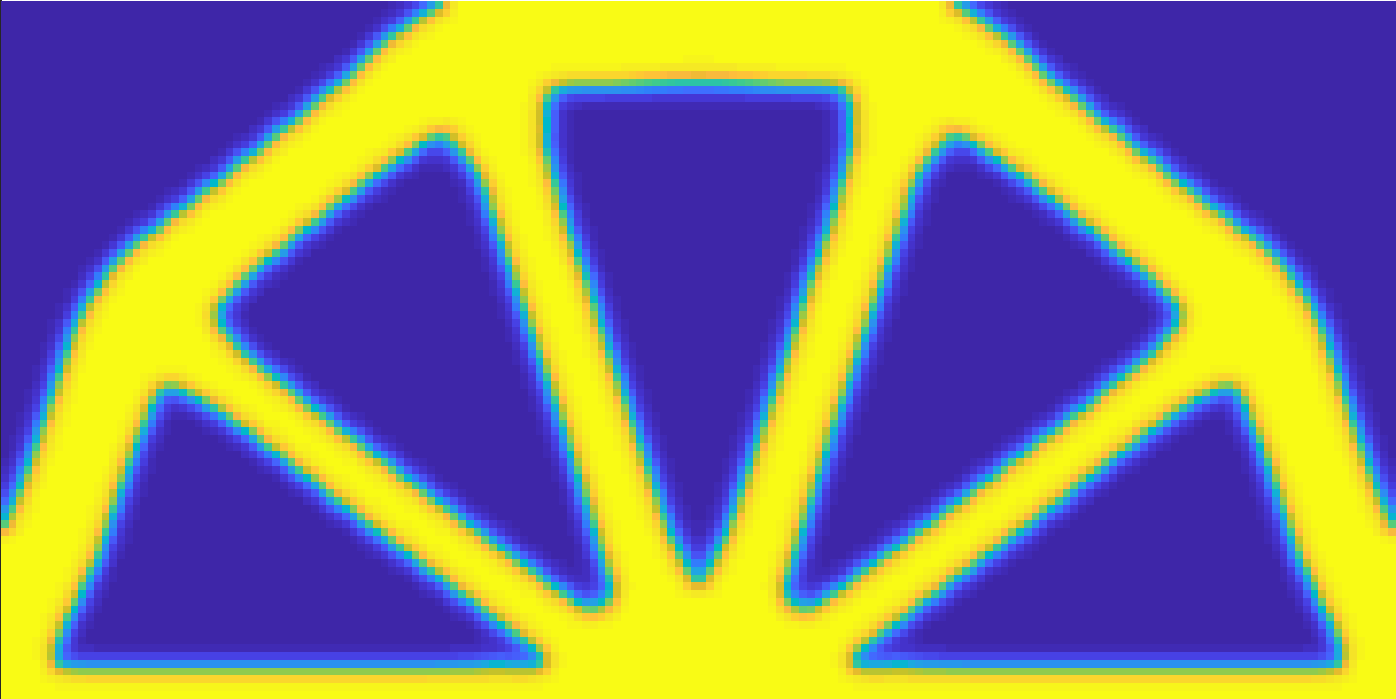} \\\vspace{2mm}
\includegraphics[width=0.45\textwidth]{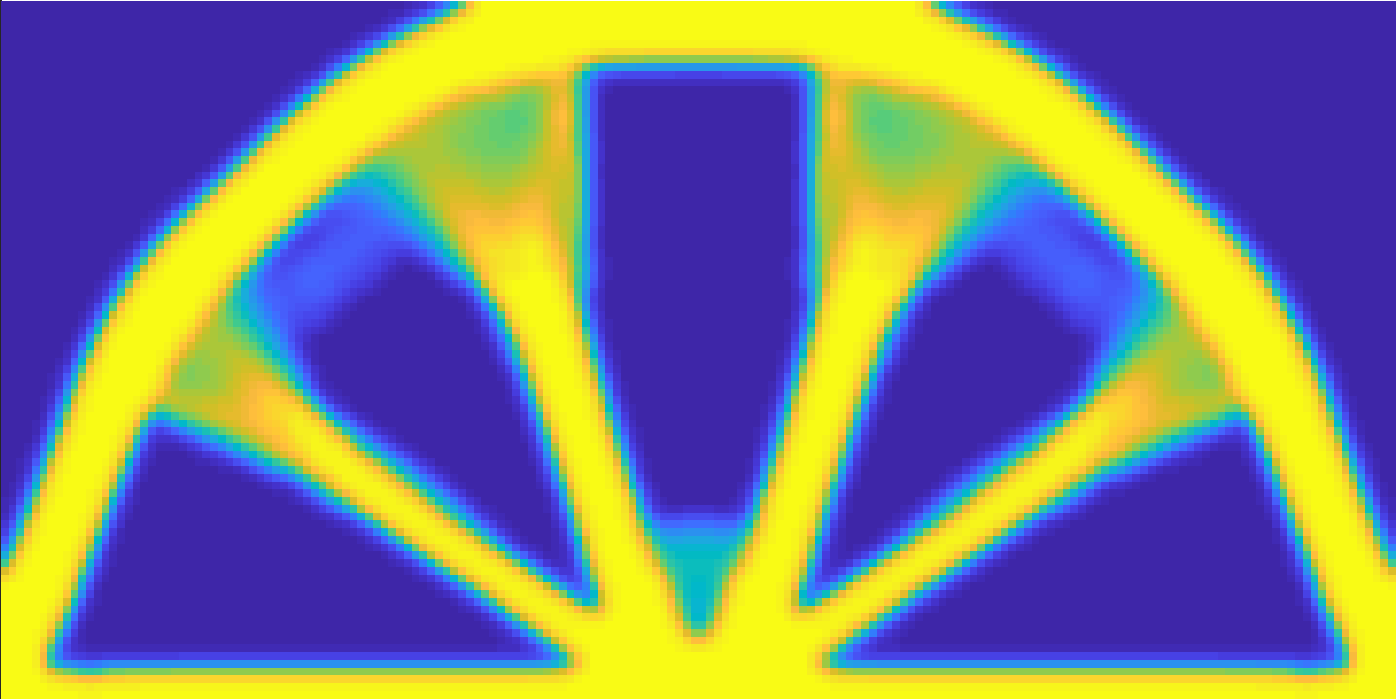} \quad
\includegraphics[width=0.45\textwidth]{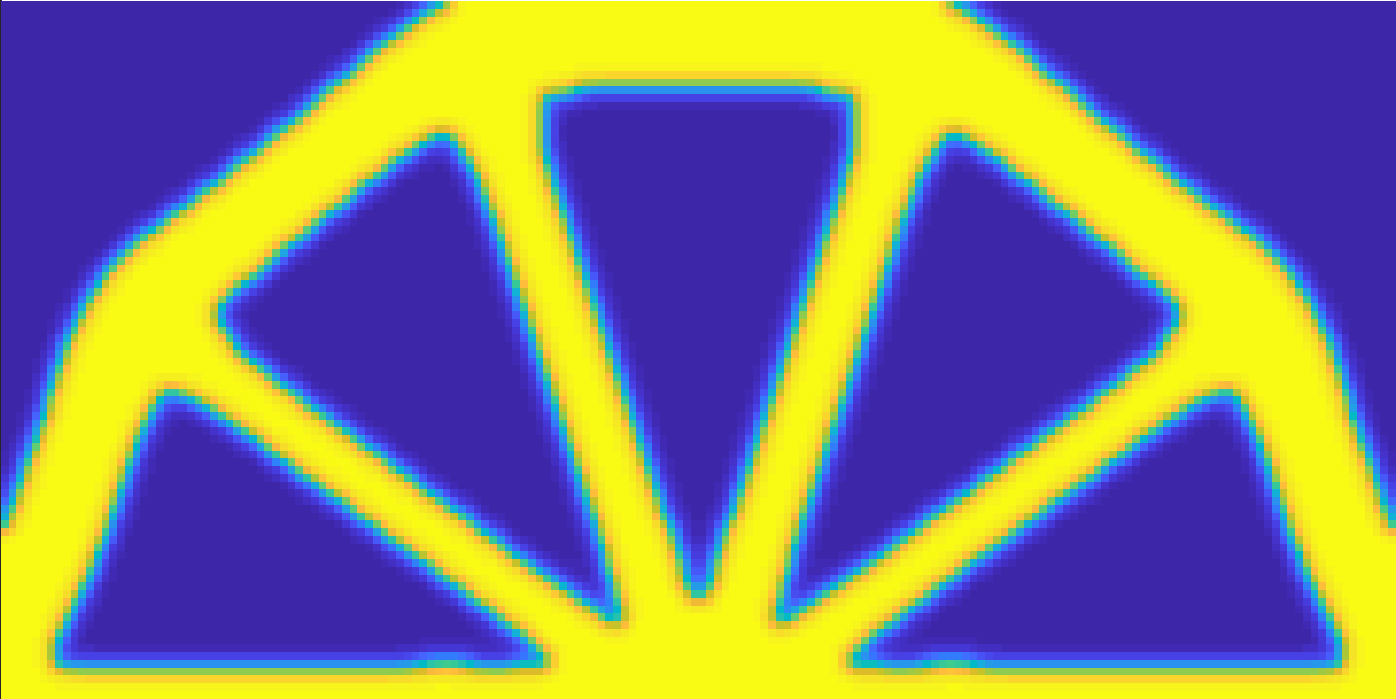} \\\vspace{2mm}
\includegraphics[width=0.45\textwidth]{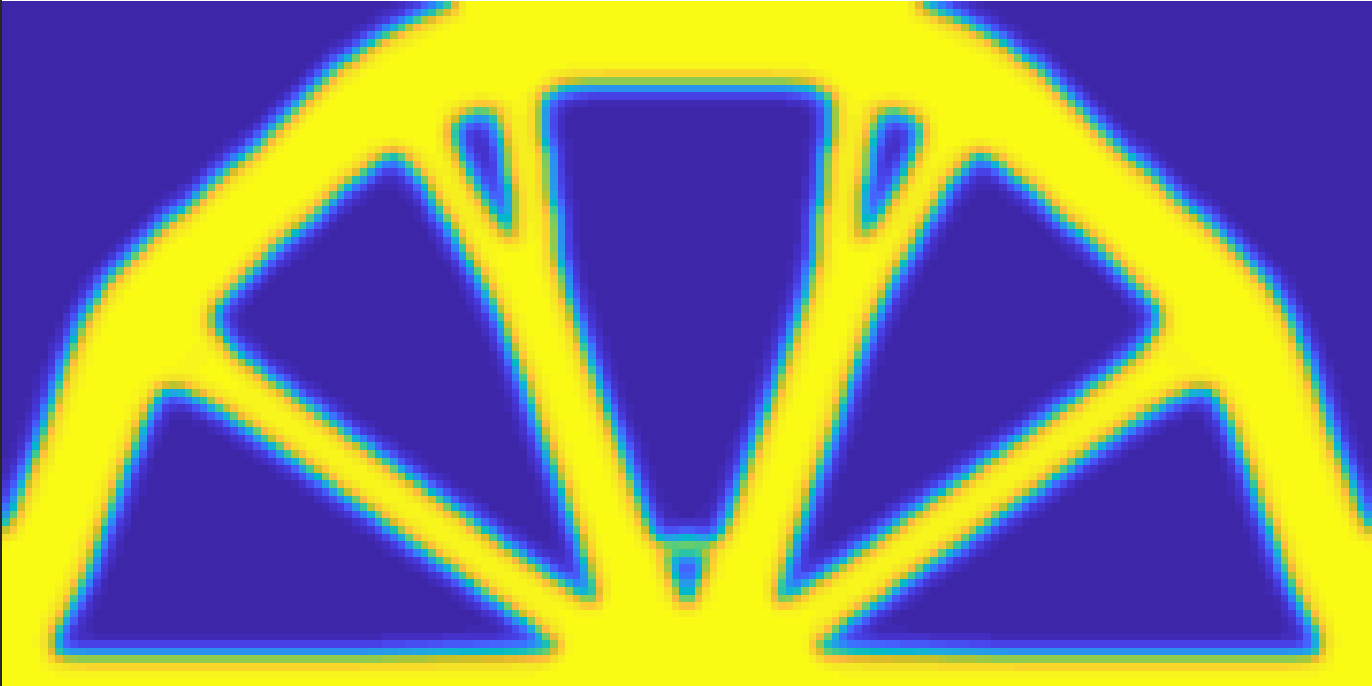} \quad
\includegraphics[width=0.45\textwidth]{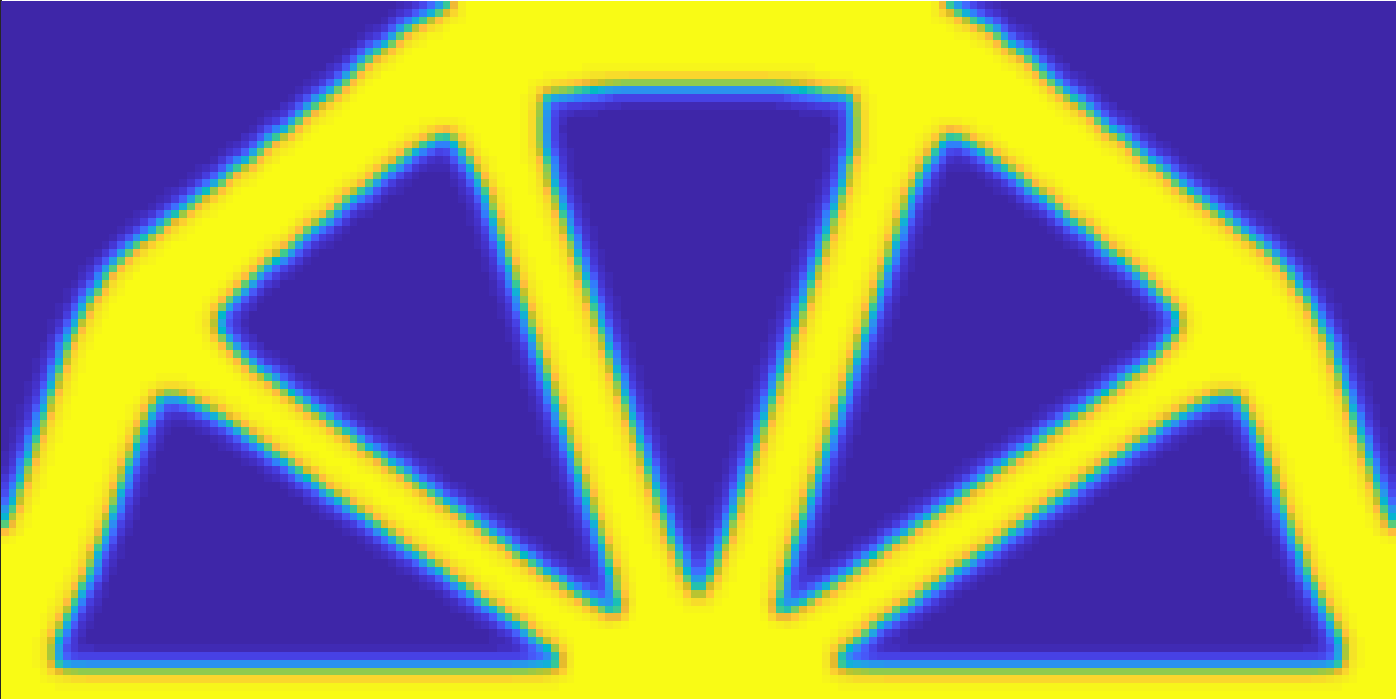}
\caption{Evolution of the design of the simply supported beam using the HDM-MMA
         (\textit{left}) and ROM-TR-RES (\textit{right}) methods at major
  iterations (\textit{top-to-bottom}): 6, 10, 14, 16, 30.}
\label{fig:ssbeam0_viz}
\end{figure}

\begin{table}
\caption{Performance of HDM-MMA vs. ROM-TR-RES vs. ROM-TR-DIST applied to
         compliance minimization of the simply supported beam as a function
         of convergence tolerance. The reference value of the optimal objective is $J^* = 153.92$.}
\label{tab:ssbeam0_nel180x90_rmin0p5}
\begin{tabular}{r|cc|c|ccc}
 & $\epsilon$ & $\tau$ & final objective & \# HDM solves & \# ROM solves & cost ($C_\epsilon$) \\\hline
 \input{py/ssbeam0_nel180x90_rmin0p5_majit0_matinterp1.dat}
\end{tabular}
\end{table}



%% file: py/mbb0_geom.tikz
\begin{tikzpicture}
\begin{axis}[
axis equal image,
axis x line*=bottom,
axis y line*=left,
width=0.8\textwidth,
xtick={0, 3},
ytick={0, 1},
grid=both,
ymax=1.3,
xmax=3.3,
xmin=-0.3,
ymin=-0.3]
\addplot [black, solid, opacity=0.6, fill=lightgray, forget plot]
coordinates {
( 0.00000000e+00,  0.00000000e+00)
( 3.00000000e+00,  0.00000000e+00)
( 3.00000000e+00,  1.00000000e+00)
( 0.00000000e+00,  1.00000000e+00)
( 0.00000000e+00,  0.00000000e+00)};

\draw [-{Latex[width=2mm,length=4mm]}, thick] plot [] coordinates {(axis cs:0.0, 1.5) (axis cs:0.0, 1.0)
};

\node[right]    at    (axis cs:0, 1.25) {$F_\text{in}$};
\addplot [black, solid, forget plot]
coordinates {
( 0.00000000e+00,  0.00000000e+00)
(-8.00000000e-02,  6.92820323e-02)
(-8.00000000e-02, -6.92820323e-02)
( 0.00000000e+00,  0.00000000e+00)};

\addplot [black, solid, forget plot]
coordinates {
(-1.04000000e-01, -2.40000000e-02)
(-1.00930948e-01, -2.38029603e-02)
(-9.79122900e-02, -2.32150767e-02)
(-9.49935919e-02, -2.22460022e-02)
(-9.22227788e-02, -2.09116489e-02)
(-8.96453473e-02, -1.92339269e-02)
(-8.73036188e-02, -1.72403844e-02)
(-8.52360444e-02, -1.49637552e-02)
(-8.34765737e-02, -1.24414216e-02)
(-8.20540970e-02, -9.71480023e-03)
(-8.09919715e-02, -6.82866208e-03)
(-8.03076372e-02, -3.83039748e-03)
(-8.00123308e-02, -7.69237862e-04)
(-8.01109013e-02,  2.30455262e-03)
(-8.06017301e-02,  5.34050241e-03)
(-8.14767579e-02,  8.28876131e-03)
(-8.27216166e-02,  1.11009190e-02)
(-8.43158659e-02,  1.37307998e-02)
(-8.62333281e-02,  1.61352214e-02)
(-8.84425185e-02,  1.82747030e-02)
(-9.09071624e-02,  2.01141145e-02)
(-9.35867903e-02,  2.16232528e-02)
(-9.64374028e-02,  2.27773379e-02)
(-9.94121929e-02,  2.35574198e-02)
(-1.02462315e-01,  2.39506894e-02)
(-1.05537685e-01,  2.39506894e-02)
(-1.08587807e-01,  2.35574198e-02)
(-1.11562597e-01,  2.27773379e-02)
(-1.14413210e-01,  2.16232528e-02)
(-1.17092838e-01,  2.01141145e-02)
(-1.19557481e-01,  1.82747030e-02)
(-1.21766672e-01,  1.61352214e-02)
(-1.23684134e-01,  1.37307998e-02)
(-1.25278383e-01,  1.11009190e-02)
(-1.26523242e-01,  8.28876131e-03)
(-1.27398270e-01,  5.34050241e-03)
(-1.27889099e-01,  2.30455262e-03)
(-1.27987669e-01, -7.69237862e-04)
(-1.27692363e-01, -3.83039748e-03)
(-1.27008028e-01, -6.82866208e-03)
(-1.25945903e-01, -9.71480023e-03)
(-1.24523426e-01, -1.24414216e-02)
(-1.22763956e-01, -1.49637552e-02)
(-1.20696381e-01, -1.72403844e-02)
(-1.18354653e-01, -1.92339269e-02)
(-1.15777221e-01, -2.09116489e-02)
(-1.13006408e-01, -2.22460022e-02)
(-1.10087710e-01, -2.32150767e-02)
(-1.07069052e-01, -2.38029603e-02)
(-1.04000000e-01, -2.40000000e-02)};

\addplot [black, solid, fill=black, forget plot]
coordinates {
(-1.28000000e-01,  6.92820323e-02)
(-1.28000000e-01, -6.92820323e-02)
(-1.44000000e-01, -6.92820323e-02)
(-1.44000000e-01,  6.92820323e-02)
(-1.28000000e-01,  6.92820323e-02)};

\addplot [black, solid, forget plot]
coordinates {
( 0.00000000e+00,  2.50000000e-01)
(-8.00000000e-02,  3.19282032e-01)
(-8.00000000e-02,  1.80717968e-01)
( 0.00000000e+00,  2.50000000e-01)};

\addplot [black, solid, forget plot]
coordinates {
(-1.04000000e-01,  2.26000000e-01)
(-1.00930948e-01,  2.26197040e-01)
(-9.79122900e-02,  2.26784923e-01)
(-9.49935919e-02,  2.27753998e-01)
(-9.22227788e-02,  2.29088351e-01)
(-8.96453473e-02,  2.30766073e-01)
(-8.73036188e-02,  2.32759616e-01)
(-8.52360444e-02,  2.35036245e-01)
(-8.34765737e-02,  2.37558578e-01)
(-8.20540970e-02,  2.40285200e-01)
(-8.09919715e-02,  2.43171338e-01)
(-8.03076372e-02,  2.46169603e-01)
(-8.00123308e-02,  2.49230762e-01)
(-8.01109013e-02,  2.52304553e-01)
(-8.06017301e-02,  2.55340502e-01)
(-8.14767579e-02,  2.58288761e-01)
(-8.27216166e-02,  2.61100919e-01)
(-8.43158659e-02,  2.63730800e-01)
(-8.62333281e-02,  2.66135221e-01)
(-8.84425185e-02,  2.68274703e-01)
(-9.09071624e-02,  2.70114115e-01)
(-9.35867903e-02,  2.71623253e-01)
(-9.64374028e-02,  2.72777338e-01)
(-9.94121929e-02,  2.73557420e-01)
(-1.02462315e-01,  2.73950689e-01)
(-1.05537685e-01,  2.73950689e-01)
(-1.08587807e-01,  2.73557420e-01)
(-1.11562597e-01,  2.72777338e-01)
(-1.14413210e-01,  2.71623253e-01)
(-1.17092838e-01,  2.70114115e-01)
(-1.19557481e-01,  2.68274703e-01)
(-1.21766672e-01,  2.66135221e-01)
(-1.23684134e-01,  2.63730800e-01)
(-1.25278383e-01,  2.61100919e-01)
(-1.26523242e-01,  2.58288761e-01)
(-1.27398270e-01,  2.55340502e-01)
(-1.27889099e-01,  2.52304553e-01)
(-1.27987669e-01,  2.49230762e-01)
(-1.27692363e-01,  2.46169603e-01)
(-1.27008028e-01,  2.43171338e-01)
(-1.25945903e-01,  2.40285200e-01)
(-1.24523426e-01,  2.37558578e-01)
(-1.22763956e-01,  2.35036245e-01)
(-1.20696381e-01,  2.32759616e-01)
(-1.18354653e-01,  2.30766073e-01)
(-1.15777221e-01,  2.29088351e-01)
(-1.13006408e-01,  2.27753998e-01)
(-1.10087710e-01,  2.26784923e-01)
(-1.07069052e-01,  2.26197040e-01)
(-1.04000000e-01,  2.26000000e-01)};

\addplot [black, solid, fill=black, forget plot]
coordinates {
(-1.28000000e-01,  3.19282032e-01)
(-1.28000000e-01,  1.80717968e-01)
(-1.44000000e-01,  1.80717968e-01)
(-1.44000000e-01,  3.19282032e-01)
(-1.28000000e-01,  3.19282032e-01)};

\addplot [black, solid, forget plot]
coordinates {
( 0.00000000e+00,  5.00000000e-01)
(-8.00000000e-02,  5.69282032e-01)
(-8.00000000e-02,  4.30717968e-01)
( 0.00000000e+00,  5.00000000e-01)};

\addplot [black, solid, forget plot]
coordinates {
(-1.04000000e-01,  4.76000000e-01)
(-1.00930948e-01,  4.76197040e-01)
(-9.79122900e-02,  4.76784923e-01)
(-9.49935919e-02,  4.77753998e-01)
(-9.22227788e-02,  4.79088351e-01)
(-8.96453473e-02,  4.80766073e-01)
(-8.73036188e-02,  4.82759616e-01)
(-8.52360444e-02,  4.85036245e-01)
(-8.34765737e-02,  4.87558578e-01)
(-8.20540970e-02,  4.90285200e-01)
(-8.09919715e-02,  4.93171338e-01)
(-8.03076372e-02,  4.96169603e-01)
(-8.00123308e-02,  4.99230762e-01)
(-8.01109013e-02,  5.02304553e-01)
(-8.06017301e-02,  5.05340502e-01)
(-8.14767579e-02,  5.08288761e-01)
(-8.27216166e-02,  5.11100919e-01)
(-8.43158659e-02,  5.13730800e-01)
(-8.62333281e-02,  5.16135221e-01)
(-8.84425185e-02,  5.18274703e-01)
(-9.09071624e-02,  5.20114115e-01)
(-9.35867903e-02,  5.21623253e-01)
(-9.64374028e-02,  5.22777338e-01)
(-9.94121929e-02,  5.23557420e-01)
(-1.02462315e-01,  5.23950689e-01)
(-1.05537685e-01,  5.23950689e-01)
(-1.08587807e-01,  5.23557420e-01)
(-1.11562597e-01,  5.22777338e-01)
(-1.14413210e-01,  5.21623253e-01)
(-1.17092838e-01,  5.20114115e-01)
(-1.19557481e-01,  5.18274703e-01)
(-1.21766672e-01,  5.16135221e-01)
(-1.23684134e-01,  5.13730800e-01)
(-1.25278383e-01,  5.11100919e-01)
(-1.26523242e-01,  5.08288761e-01)
(-1.27398270e-01,  5.05340502e-01)
(-1.27889099e-01,  5.02304553e-01)
(-1.27987669e-01,  4.99230762e-01)
(-1.27692363e-01,  4.96169603e-01)
(-1.27008028e-01,  4.93171338e-01)
(-1.25945903e-01,  4.90285200e-01)
(-1.24523426e-01,  4.87558578e-01)
(-1.22763956e-01,  4.85036245e-01)
(-1.20696381e-01,  4.82759616e-01)
(-1.18354653e-01,  4.80766073e-01)
(-1.15777221e-01,  4.79088351e-01)
(-1.13006408e-01,  4.77753998e-01)
(-1.10087710e-01,  4.76784923e-01)
(-1.07069052e-01,  4.76197040e-01)
(-1.04000000e-01,  4.76000000e-01)};

\addplot [black, solid, fill=black, forget plot]
coordinates {
(-1.28000000e-01,  5.69282032e-01)
(-1.28000000e-01,  4.30717968e-01)
(-1.44000000e-01,  4.30717968e-01)
(-1.44000000e-01,  5.69282032e-01)
(-1.28000000e-01,  5.69282032e-01)};

\addplot [black, solid, forget plot]
coordinates {
( 0.00000000e+00,  7.50000000e-01)
(-8.00000000e-02,  8.19282032e-01)
(-8.00000000e-02,  6.80717968e-01)
( 0.00000000e+00,  7.50000000e-01)};

\addplot [black, solid, forget plot]
coordinates {
(-1.04000000e-01,  7.26000000e-01)
(-1.00930948e-01,  7.26197040e-01)
(-9.79122900e-02,  7.26784923e-01)
(-9.49935919e-02,  7.27753998e-01)
(-9.22227788e-02,  7.29088351e-01)
(-8.96453473e-02,  7.30766073e-01)
(-8.73036188e-02,  7.32759616e-01)
(-8.52360444e-02,  7.35036245e-01)
(-8.34765737e-02,  7.37558578e-01)
(-8.20540970e-02,  7.40285200e-01)
(-8.09919715e-02,  7.43171338e-01)
(-8.03076372e-02,  7.46169603e-01)
(-8.00123308e-02,  7.49230762e-01)
(-8.01109013e-02,  7.52304553e-01)
(-8.06017301e-02,  7.55340502e-01)
(-8.14767579e-02,  7.58288761e-01)
(-8.27216166e-02,  7.61100919e-01)
(-8.43158659e-02,  7.63730800e-01)
(-8.62333281e-02,  7.66135221e-01)
(-8.84425185e-02,  7.68274703e-01)
(-9.09071624e-02,  7.70114115e-01)
(-9.35867903e-02,  7.71623253e-01)
(-9.64374028e-02,  7.72777338e-01)
(-9.94121929e-02,  7.73557420e-01)
(-1.02462315e-01,  7.73950689e-01)
(-1.05537685e-01,  7.73950689e-01)
(-1.08587807e-01,  7.73557420e-01)
(-1.11562597e-01,  7.72777338e-01)
(-1.14413210e-01,  7.71623253e-01)
(-1.17092838e-01,  7.70114115e-01)
(-1.19557481e-01,  7.68274703e-01)
(-1.21766672e-01,  7.66135221e-01)
(-1.23684134e-01,  7.63730800e-01)
(-1.25278383e-01,  7.61100919e-01)
(-1.26523242e-01,  7.58288761e-01)
(-1.27398270e-01,  7.55340502e-01)
(-1.27889099e-01,  7.52304553e-01)
(-1.27987669e-01,  7.49230762e-01)
(-1.27692363e-01,  7.46169603e-01)
(-1.27008028e-01,  7.43171338e-01)
(-1.25945903e-01,  7.40285200e-01)
(-1.24523426e-01,  7.37558578e-01)
(-1.22763956e-01,  7.35036245e-01)
(-1.20696381e-01,  7.32759616e-01)
(-1.18354653e-01,  7.30766073e-01)
(-1.15777221e-01,  7.29088351e-01)
(-1.13006408e-01,  7.27753998e-01)
(-1.10087710e-01,  7.26784923e-01)
(-1.07069052e-01,  7.26197040e-01)
(-1.04000000e-01,  7.26000000e-01)};

\addplot [black, solid, fill=black, forget plot]
coordinates {
(-1.28000000e-01,  8.19282032e-01)
(-1.28000000e-01,  6.80717968e-01)
(-1.44000000e-01,  6.80717968e-01)
(-1.44000000e-01,  8.19282032e-01)
(-1.28000000e-01,  8.19282032e-01)};

\addplot [black, solid, forget plot]
coordinates {
( 0.00000000e+00,  1.00000000e+00)
(-8.00000000e-02,  1.06928203e+00)
(-8.00000000e-02,  9.30717968e-01)
( 0.00000000e+00,  1.00000000e+00)};

\addplot [black, solid, forget plot]
coordinates {
(-1.04000000e-01,  9.76000000e-01)
(-1.00930948e-01,  9.76197040e-01)
(-9.79122900e-02,  9.76784923e-01)
(-9.49935919e-02,  9.77753998e-01)
(-9.22227788e-02,  9.79088351e-01)
(-8.96453473e-02,  9.80766073e-01)
(-8.73036188e-02,  9.82759616e-01)
(-8.52360444e-02,  9.85036245e-01)
(-8.34765737e-02,  9.87558578e-01)
(-8.20540970e-02,  9.90285200e-01)
(-8.09919715e-02,  9.93171338e-01)
(-8.03076372e-02,  9.96169603e-01)
(-8.00123308e-02,  9.99230762e-01)
(-8.01109013e-02,  1.00230455e+00)
(-8.06017301e-02,  1.00534050e+00)
(-8.14767579e-02,  1.00828876e+00)
(-8.27216166e-02,  1.01110092e+00)
(-8.43158659e-02,  1.01373080e+00)
(-8.62333281e-02,  1.01613522e+00)
(-8.84425185e-02,  1.01827470e+00)
(-9.09071624e-02,  1.02011411e+00)
(-9.35867903e-02,  1.02162325e+00)
(-9.64374028e-02,  1.02277734e+00)
(-9.94121929e-02,  1.02355742e+00)
(-1.02462315e-01,  1.02395069e+00)
(-1.05537685e-01,  1.02395069e+00)
(-1.08587807e-01,  1.02355742e+00)
(-1.11562597e-01,  1.02277734e+00)
(-1.14413210e-01,  1.02162325e+00)
(-1.17092838e-01,  1.02011411e+00)
(-1.19557481e-01,  1.01827470e+00)
(-1.21766672e-01,  1.01613522e+00)
(-1.23684134e-01,  1.01373080e+00)
(-1.25278383e-01,  1.01110092e+00)
(-1.26523242e-01,  1.00828876e+00)
(-1.27398270e-01,  1.00534050e+00)
(-1.27889099e-01,  1.00230455e+00)
(-1.27987669e-01,  9.99230762e-01)
(-1.27692363e-01,  9.96169603e-01)
(-1.27008028e-01,  9.93171338e-01)
(-1.25945903e-01,  9.90285200e-01)
(-1.24523426e-01,  9.87558578e-01)
(-1.22763956e-01,  9.85036245e-01)
(-1.20696381e-01,  9.82759616e-01)
(-1.18354653e-01,  9.80766073e-01)
(-1.15777221e-01,  9.79088351e-01)
(-1.13006408e-01,  9.77753998e-01)
(-1.10087710e-01,  9.76784923e-01)
(-1.07069052e-01,  9.76197040e-01)
(-1.04000000e-01,  9.76000000e-01)};

\addplot [black, solid, fill=black, forget plot]
coordinates {
(-1.28000000e-01,  1.06928203e+00)
(-1.28000000e-01,  9.30717968e-01)
(-1.44000000e-01,  9.30717968e-01)
(-1.44000000e-01,  1.06928203e+00)
(-1.28000000e-01,  1.06928203e+00)};

\addplot [black, solid, forget plot]
coordinates {
( 3.00000000e+00,  0.00000000e+00)
( 2.93071797e+00, -8.00000000e-02)
( 3.06928203e+00, -8.00000000e-02)
( 3.00000000e+00,  0.00000000e+00)};

\addplot [black, solid, forget plot]
coordinates {
( 3.02400000e+00, -1.04000000e-01)
( 3.02380296e+00, -1.00930948e-01)
( 3.02321508e+00, -9.79122900e-02)
( 3.02224600e+00, -9.49935919e-02)
( 3.02091165e+00, -9.22227788e-02)
( 3.01923393e+00, -8.96453473e-02)
( 3.01724038e+00, -8.73036188e-02)
( 3.01496376e+00, -8.52360444e-02)
( 3.01244142e+00, -8.34765737e-02)
( 3.00971480e+00, -8.20540970e-02)
( 3.00682866e+00, -8.09919715e-02)
( 3.00383040e+00, -8.03076372e-02)
( 3.00076924e+00, -8.00123308e-02)
( 2.99769545e+00, -8.01109013e-02)
( 2.99465950e+00, -8.06017301e-02)
( 2.99171124e+00, -8.14767579e-02)
( 2.98889908e+00, -8.27216166e-02)
( 2.98626920e+00, -8.43158659e-02)
( 2.98386478e+00, -8.62333281e-02)
( 2.98172530e+00, -8.84425185e-02)
( 2.97988589e+00, -9.09071624e-02)
( 2.97837675e+00, -9.35867903e-02)
( 2.97722266e+00, -9.64374028e-02)
( 2.97644258e+00, -9.94121929e-02)
( 2.97604931e+00, -1.02462315e-01)
( 2.97604931e+00, -1.05537685e-01)
( 2.97644258e+00, -1.08587807e-01)
( 2.97722266e+00, -1.11562597e-01)
( 2.97837675e+00, -1.14413210e-01)
( 2.97988589e+00, -1.17092838e-01)
( 2.98172530e+00, -1.19557481e-01)
( 2.98386478e+00, -1.21766672e-01)
( 2.98626920e+00, -1.23684134e-01)
( 2.98889908e+00, -1.25278383e-01)
( 2.99171124e+00, -1.26523242e-01)
( 2.99465950e+00, -1.27398270e-01)
( 2.99769545e+00, -1.27889099e-01)
( 3.00076924e+00, -1.27987669e-01)
( 3.00383040e+00, -1.27692363e-01)
( 3.00682866e+00, -1.27008028e-01)
( 3.00971480e+00, -1.25945903e-01)
( 3.01244142e+00, -1.24523426e-01)
( 3.01496376e+00, -1.22763956e-01)
( 3.01724038e+00, -1.20696381e-01)
( 3.01923393e+00, -1.18354653e-01)
( 3.02091165e+00, -1.15777221e-01)
( 3.02224600e+00, -1.13006408e-01)
( 3.02321508e+00, -1.10087710e-01)
( 3.02380296e+00, -1.07069052e-01)
( 3.02400000e+00, -1.04000000e-01)};

\addplot [black, solid, fill=black, forget plot]
coordinates {
( 2.93071797e+00, -1.28000000e-01)
( 3.06928203e+00, -1.28000000e-01)
( 3.06928203e+00, -1.44000000e-01)
( 2.93071797e+00, -1.44000000e-01)
( 2.93071797e+00, -1.28000000e-01)};

\addplot []
graphics [xmin=0,xmax=3,ymin=0,ymax=1] { 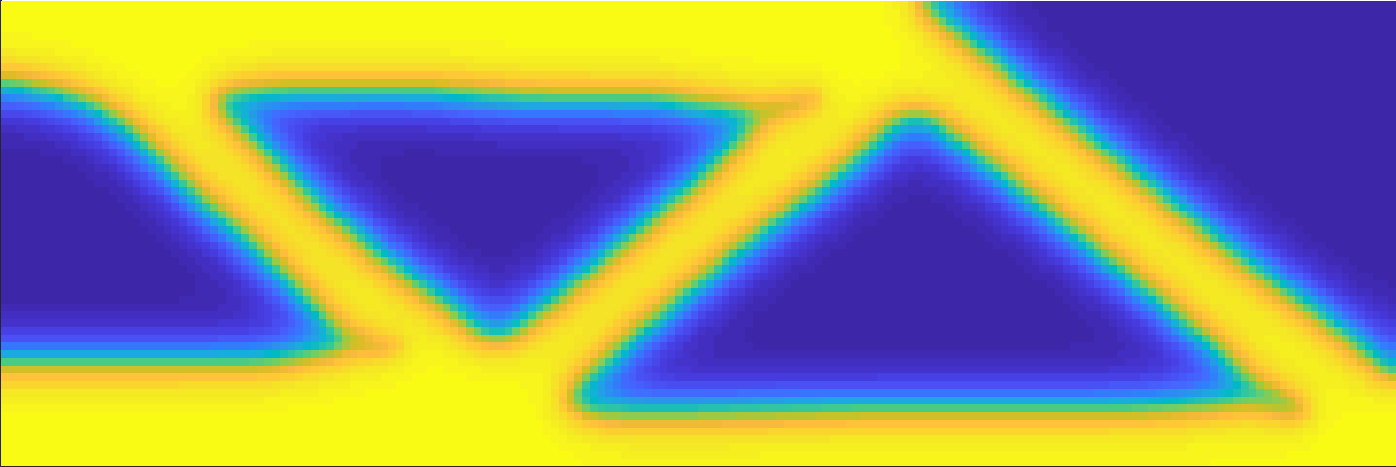};

\end{axis}
\end{tikzpicture}

%% file: py/mbb0_nel180x60_rmin0p12_romset1_majit0_matinterp1.tikz
\begin{tikzpicture}
\begin{groupplot} [
group style={group size = 1 by 3, horizontal sep = 2.0cm, vertical sep = 1.0cm}]
\nextgroupplot[width=0.8\textwidth, grid=major, ymax=3.422783273303775, xmax=50, ylabel={$\displaystyle{\frac{|J(\psibold^{(k)})-J^*|}{|J^*|}}$}, xmin=0, ymode=log, height=0.4\textwidth]
\addplot [solid, thick, mark options={solid, thin}, mark=*, mark size=1.5, color=black, label=line:hdm, mark repeat=5]
coordinates {
( 0.00000000e+00,  3.42278327e+00)
( 1.00000000e+00,  1.98323690e+00)
( 2.00000000e+00,  1.04972077e+00)
( 3.00000000e+00,  7.04036868e-01)
( 4.00000000e+00,  5.72946557e-01)
( 5.00000000e+00,  5.02187501e-01)
( 6.00000000e+00,  4.51853963e-01)
( 7.00000000e+00,  4.13257472e-01)
( 8.00000000e+00,  3.81624657e-01)
( 9.00000000e+00,  3.53462598e-01)
( 1.00000000e+01,  3.27081372e-01)
( 1.10000000e+01,  3.01816376e-01)
( 1.20000000e+01,  2.77481388e-01)
( 1.30000000e+01,  2.51362651e-01)
( 1.40000000e+01,  2.20691283e-01)
( 1.50000000e+01,  1.81704933e-01)
( 1.60000000e+01,  1.41031224e-01)
( 1.70000000e+01,  1.00663985e-01)
( 1.80000000e+01,  7.62457922e-02)
( 1.90000000e+01,  6.13724946e-02)
( 2.00000000e+01,  4.89544807e-02)
( 2.10000000e+01,  3.76326730e-02)
( 2.20000000e+01,  2.75436634e-02)
( 2.30000000e+01,  2.17381892e-02)
( 2.40000000e+01,  1.81608130e-02)
( 2.50000000e+01,  1.60282475e-02)
( 2.60000000e+01,  1.45127513e-02)
( 2.70000000e+01,  1.33365816e-02)
( 2.80000000e+01,  1.23099170e-02)
( 2.90000000e+01,  1.13760202e-02)
( 3.00000000e+01,  1.05652492e-02)
( 3.10000000e+01,  9.93492572e-03)
( 3.20000000e+01,  9.41254261e-03)
( 3.30000000e+01,  8.98101156e-03)
( 3.40000000e+01,  8.65351484e-03)
( 3.50000000e+01,  8.39043796e-03)
( 3.60000000e+01,  8.19118905e-03)
( 3.70000000e+01,  8.02456473e-03)
( 3.80000000e+01,  7.87271899e-03)
( 3.90000000e+01,  7.71005479e-03)
( 4.00000000e+01,  7.50361849e-03)
( 4.10000000e+01,  7.26305381e-03)
( 4.20000000e+01,  7.01646666e-03)
( 4.30000000e+01,  6.77122001e-03)
( 4.40000000e+01,  6.56313549e-03)
( 4.50000000e+01,  6.40386012e-03)
( 4.60000000e+01,  6.26113358e-03)
( 4.70000000e+01,  6.14822300e-03)
( 4.80000000e+01,  6.06189018e-03)
( 4.90000000e+01,  5.97627292e-03)
( 5.00000000e+01,  5.89359120e-03)
( 5.10000000e+01,  5.81088316e-03)
( 5.20000000e+01,  5.72445528e-03)
( 5.30000000e+01,  5.63631932e-03)
( 5.40000000e+01,  5.54636712e-03)
( 5.50000000e+01,  5.44494792e-03)
( 5.60000000e+01,  5.33133365e-03)
( 5.70000000e+01,  5.22853143e-03)
( 5.80000000e+01,  5.13131709e-03)
( 5.90000000e+01,  5.03987814e-03)
( 6.00000000e+01,  4.95732190e-03)
( 6.10000000e+01,  4.88067599e-03)
( 6.20000000e+01,  4.81662462e-03)
( 6.30000000e+01,  4.76178533e-03)
( 6.40000000e+01,  4.70742388e-03)
( 6.50000000e+01,  4.65363886e-03)
( 6.60000000e+01,  4.60362825e-03)
( 6.70000000e+01,  4.55170876e-03)
( 6.80000000e+01,  4.49607933e-03)
( 6.90000000e+01,  4.43812706e-03)
( 7.00000000e+01,  4.37638282e-03)
( 7.10000000e+01,  4.31109383e-03)
( 7.20000000e+01,  4.24688160e-03)
( 7.30000000e+01,  4.18657345e-03)
( 7.40000000e+01,  4.12597892e-03)
( 7.50000000e+01,  4.06839634e-03)
( 7.60000000e+01,  4.01900134e-03)
( 7.70000000e+01,  3.97784071e-03)
( 7.80000000e+01,  3.94125184e-03)
( 7.90000000e+01,  3.90609799e-03)
( 8.00000000e+01,  3.87408465e-03)
( 8.10000000e+01,  3.84411307e-03)
( 8.20000000e+01,  3.82051998e-03)
( 8.30000000e+01,  3.79817636e-03)
( 8.40000000e+01,  3.77574375e-03)
( 8.50000000e+01,  3.75404114e-03)
( 8.60000000e+01,  3.73014496e-03)
( 8.70000000e+01,  3.70445314e-03)
( 8.80000000e+01,  3.67853484e-03)
( 8.90000000e+01,  3.65136833e-03)
( 9.00000000e+01,  3.62614186e-03)
( 9.10000000e+01,  3.60137170e-03)
( 9.20000000e+01,  3.57412134e-03)
( 9.30000000e+01,  3.54500964e-03)
( 9.40000000e+01,  3.51478690e-03)
( 9.50000000e+01,  3.48027240e-03)
( 9.60000000e+01,  3.44313394e-03)
( 9.70000000e+01,  3.40531070e-03)
( 9.80000000e+01,  3.36930874e-03)
( 9.90000000e+01,  3.33869996e-03)};\label{line:mbb0_nel180x60_rmin0p12_romset1_majit0:hdm}

\addplot [dash dot, thick, mark options={solid, thin}, mark=square*, mark size=1.5, color=magenta, label=line:rom10, mark repeat=5]
coordinates {
( 0.00000000e+00,  3.42278327e+00)
( 1.00000000e+00,  1.98323690e+00)
( 2.00000000e+00,  1.04972077e+00)
( 3.00000000e+00,  7.09256038e-01)
( 4.00000000e+00,  5.74497223e-01)
( 5.00000000e+00,  5.01287161e-01)
( 6.00000000e+00,  4.12072548e-01)
( 7.00000000e+00,  3.43042364e-01)
( 8.00000000e+00,  2.94205875e-01)
( 9.00000000e+00,  2.64693191e-01)
( 1.00000000e+01,  2.22648000e-01)
( 1.10000000e+01,  1.79291967e-01)
( 1.20000000e+01,  1.39035543e-01)
( 1.30000000e+01,  9.48673406e-02)
( 1.40000000e+01,  5.02872393e-02)
( 1.50000000e+01,  4.01968674e-02)
( 1.60000000e+01,  2.26864460e-02)
( 1.70000000e+01,  1.50718021e-02)
( 1.80000000e+01,  1.37566907e-02)
( 1.90000000e+01,  1.15654876e-02)
( 2.00000000e+01,  1.10324456e-02)
( 2.10000000e+01,  1.04019948e-02)
( 2.20000000e+01,  9.78252239e-03)
( 2.30000000e+01,  9.20184268e-03)
( 2.40000000e+01,  8.59995285e-03)
( 2.50000000e+01,  8.38616254e-03)
( 2.60000000e+01,  8.21578622e-03)
( 2.70000000e+01,  7.92433894e-03)
( 2.80000000e+01,  7.71926172e-03)
( 2.90000000e+01,  7.42672557e-03)
( 3.00000000e+01,  6.92948278e-03)
( 3.10000000e+01,  6.43023699e-03)
( 3.20000000e+01,  6.14969585e-03)
( 3.30000000e+01,  5.98003838e-03)
( 3.40000000e+01,  5.59738522e-03)
( 3.50000000e+01,  5.27258710e-03)
( 3.60000000e+01,  5.02482579e-03)
( 3.70000000e+01,  4.81073730e-03)
( 3.80000000e+01,  4.58933219e-03)
( 3.90000000e+01,  4.39266201e-03)
( 4.00000000e+01,  4.18858402e-03)
( 4.10000000e+01,  4.18858402e-03)
( 4.20000000e+01,  3.99041059e-03)
( 4.30000000e+01,  3.92365725e-03)
( 4.40000000e+01,  3.82002273e-03)
( 4.50000000e+01,  3.76995303e-03)
( 4.60000000e+01,  3.71264336e-03)
( 4.70000000e+01,  3.66522760e-03)
( 4.80000000e+01,  3.59462499e-03)
( 4.90000000e+01,  3.53201797e-03)
( 5.00000000e+01,  3.46627369e-03)
( 5.10000000e+01,  3.38397741e-03)
( 5.20000000e+01,  3.32867739e-03)
( 5.30000000e+01,  3.27124904e-03)
( 5.40000000e+01,  3.20354877e-03)
( 5.50000000e+01,  3.15332857e-03)
( 5.60000000e+01,  3.08448228e-03)
( 5.70000000e+01,  3.01100150e-03)
( 5.80000000e+01,  2.94413537e-03)
( 5.90000000e+01,  2.90053458e-03)
( 6.00000000e+01,  2.84777509e-03)
( 6.10000000e+01,  2.76275789e-03)
( 6.20000000e+01,  2.71529040e-03)
( 6.30000000e+01,  2.60154902e-03)
( 6.40000000e+01,  2.51006315e-03)
( 6.50000000e+01,  2.42151638e-03)
( 6.60000000e+01,  2.36512255e-03)
( 6.70000000e+01,  2.31444556e-03)
( 6.80000000e+01,  2.24044309e-03)
( 6.90000000e+01,  2.19972219e-03)
( 7.00000000e+01,  2.17075301e-03)
( 7.10000000e+01,  2.15910934e-03)
( 7.20000000e+01,  2.14433011e-03)
( 7.30000000e+01,  2.12465226e-03)
( 7.40000000e+01,  2.09863077e-03)
( 7.50000000e+01,  2.07330423e-03)
( 7.60000000e+01,  2.03844711e-03)
( 7.70000000e+01,  2.00096508e-03)
( 7.80000000e+01,  1.97580905e-03)
( 7.90000000e+01,  1.96850100e-03)
( 8.00000000e+01,  1.95967281e-03)
( 8.10000000e+01,  1.94586758e-03)
( 8.20000000e+01,  1.92356652e-03)
( 8.30000000e+01,  1.89266456e-03)
( 8.40000000e+01,  1.85006667e-03)
( 8.50000000e+01,  1.80258443e-03)
( 8.60000000e+01,  1.76826363e-03)
( 8.70000000e+01,  1.75047489e-03)
( 8.80000000e+01,  1.73810450e-03)
( 8.90000000e+01,  1.72536371e-03)
( 9.00000000e+01,  1.70739843e-03)
( 9.10000000e+01,  1.68765410e-03)
( 9.20000000e+01,  1.67500799e-03)
( 9.30000000e+01,  1.65668458e-03)
( 9.40000000e+01,  1.63296350e-03)
( 9.50000000e+01,  1.60275569e-03)
( 9.60000000e+01,  1.57783333e-03)
( 9.70000000e+01,  1.55834436e-03)
( 9.80000000e+01,  1.53633480e-03)
( 9.90000000e+01,  1.52970050e-03)
( 1.00000000e+02,  1.52224685e-03)};\label{line:mbb0_nel180x60_rmin0p12_romset1_majit0:romA0}

\addplot [dash dot, thick, mark options={solid, thin}, mark=square*, mark size=1.5, color=blue, label=line:rom11, mark repeat=5]
coordinates {
( 0.00000000e+00,  3.42278327e+00)
( 1.00000000e+00,  1.98323690e+00)
( 2.00000000e+00,  1.04972077e+00)
( 3.00000000e+00,  5.89695544e-01)
( 4.00000000e+00,  5.89695544e-01)
( 5.00000000e+00,  4.14658998e-01)
( 6.00000000e+00,  3.54185817e-01)
( 7.00000000e+00,  2.95223003e-01)
( 8.00000000e+00,  2.51892666e-01)
( 9.00000000e+00,  2.02457030e-01)
( 1.00000000e+01,  1.46594227e-01)
( 1.10000000e+01,  9.82642410e-02)
( 1.20000000e+01,  4.56355444e-02)
( 1.30000000e+01,  2.77382020e-02)
( 1.40000000e+01,  2.21064854e-02)
( 1.50000000e+01,  1.45241493e-02)
( 1.60000000e+01,  1.17597777e-02)
( 1.70000000e+01,  1.17597777e-02)
( 1.80000000e+01,  1.00702257e-02)
( 1.90000000e+01,  9.46875020e-03)
( 2.00000000e+01,  8.51934979e-03)
( 2.10000000e+01,  8.17209243e-03)
( 2.20000000e+01,  8.17209243e-03)
( 2.30000000e+01,  7.91687875e-03)
( 2.40000000e+01,  7.76305306e-03)
( 2.50000000e+01,  7.48363379e-03)
( 2.60000000e+01,  7.12080960e-03)
( 2.70000000e+01,  6.66657337e-03)
( 2.80000000e+01,  6.36682350e-03)
( 2.90000000e+01,  6.09232945e-03)
( 3.00000000e+01,  5.85229538e-03)
( 3.10000000e+01,  5.59513140e-03)
( 3.20000000e+01,  5.26098091e-03)
( 3.30000000e+01,  5.02611710e-03)
( 3.40000000e+01,  4.81237375e-03)
( 3.50000000e+01,  4.58672917e-03)
( 3.60000000e+01,  4.38730014e-03)
( 3.70000000e+01,  4.23653127e-03)
( 3.80000000e+01,  4.11066833e-03)
( 3.90000000e+01,  4.11066833e-03)
( 4.00000000e+01,  3.96673773e-03)
( 4.10000000e+01,  3.87454725e-03)
( 4.20000000e+01,  3.77796581e-03)
( 4.30000000e+01,  3.67639009e-03)
( 4.40000000e+01,  3.67639009e-03)
( 4.50000000e+01,  3.56525227e-03)
( 4.60000000e+01,  3.51707211e-03)
( 4.70000000e+01,  3.46535445e-03)
( 4.80000000e+01,  3.38720148e-03)
( 4.90000000e+01,  3.32204643e-03)
( 5.00000000e+01,  3.27711572e-03)
( 5.10000000e+01,  3.22044591e-03)
( 5.20000000e+01,  3.17043296e-03)
( 5.30000000e+01,  3.09862619e-03)
( 5.40000000e+01,  3.06319808e-03)
( 5.50000000e+01,  3.01772060e-03)
( 5.60000000e+01,  2.95436490e-03)
( 5.70000000e+01,  2.95436490e-03)
( 5.80000000e+01,  2.90255849e-03)
( 5.90000000e+01,  2.83766056e-03)
( 6.00000000e+01,  2.77747016e-03)
( 6.10000000e+01,  2.72399527e-03)
( 6.20000000e+01,  2.65049973e-03)
( 6.30000000e+01,  2.56494839e-03)
( 6.40000000e+01,  2.46521134e-03)
( 6.50000000e+01,  2.37824528e-03)
( 6.60000000e+01,  2.31587433e-03)
( 6.70000000e+01,  2.31587433e-03)
( 6.80000000e+01,  2.25594851e-03)
( 6.90000000e+01,  2.20408167e-03)
( 7.00000000e+01,  2.17061636e-03)
( 7.10000000e+01,  2.14038986e-03)
( 7.20000000e+01,  2.13103202e-03)
( 7.30000000e+01,  2.11832069e-03)
( 7.40000000e+01,  2.09980827e-03)
( 7.50000000e+01,  2.07029555e-03)
( 7.60000000e+01,  2.05037242e-03)
( 7.70000000e+01,  2.02717201e-03)
( 7.80000000e+01,  1.99439090e-03)
( 7.90000000e+01,  1.95780406e-03)
( 8.00000000e+01,  1.92689047e-03)
( 8.10000000e+01,  1.91397425e-03)
( 8.20000000e+01,  1.89884046e-03)
( 8.30000000e+01,  1.87706490e-03)
( 8.40000000e+01,  1.84341874e-03)
( 8.50000000e+01,  1.80447759e-03)
( 8.60000000e+01,  1.77627478e-03)
( 8.70000000e+01,  1.76371596e-03)
( 8.80000000e+01,  1.75022873e-03)
( 8.90000000e+01,  1.73265381e-03)
( 9.00000000e+01,  1.72292773e-03)
( 9.10000000e+01,  1.69701814e-03)
( 9.20000000e+01,  1.67606390e-03)
( 9.30000000e+01,  1.66626632e-03)
( 9.40000000e+01,  1.65112992e-03)
( 9.50000000e+01,  1.62851257e-03)
( 9.60000000e+01,  1.59308417e-03)
( 9.70000000e+01,  1.57755176e-03)
( 9.80000000e+01,  1.54571880e-03)
( 9.90000000e+01,  1.53412906e-03)
( 1.00000000e+02,  1.53080349e-03)};\label{line:mbb0_nel180x60_rmin0p12_romset1_majit0:romA1}

\addplot [dashed, thick, mark options={solid, thick}, mark=x, mark size=3, color=magenta, label=line:trrom0, mark repeat=5]
coordinates {
( 0.00000000e+00,  3.42278327e+00)
( 1.00000000e+00,  1.98323690e+00)
( 2.00000000e+00,  1.04972077e+00)
( 3.00000000e+00,  7.09256038e-01)
( 4.00000000e+00,  5.74497223e-01)
( 5.00000000e+00,  5.01287161e-01)
( 6.00000000e+00,  4.49585547e-01)
( 7.00000000e+00,  3.84842152e-01)
( 8.00000000e+00,  3.27761944e-01)
( 9.00000000e+00,  2.90764154e-01)
( 1.00000000e+01,  2.57367897e-01)
( 1.10000000e+01,  2.19270292e-01)
( 1.20000000e+01,  1.77145072e-01)
( 1.30000000e+01,  1.24872771e-01)
( 1.40000000e+01,  7.99724666e-02)
( 1.50000000e+01,  4.51291351e-02)
( 1.60000000e+01,  2.52776931e-02)
( 1.70000000e+01,  1.95412187e-02)
( 1.80000000e+01,  1.95412187e-02)
( 1.90000000e+01,  1.30254497e-02)
( 2.00000000e+01,  1.15817756e-02)
( 2.10000000e+01,  1.10854691e-02)
( 2.20000000e+01,  9.56396835e-03)
( 2.30000000e+01,  8.53947380e-03)
( 2.40000000e+01,  8.40449310e-03)
( 2.50000000e+01,  8.04555758e-03)
( 2.60000000e+01,  7.87277798e-03)
( 2.70000000e+01,  7.68382093e-03)
( 2.80000000e+01,  7.29099244e-03)
( 2.90000000e+01,  6.79035360e-03)
( 3.00000000e+01,  6.43105525e-03)
( 3.10000000e+01,  6.19921717e-03)
( 3.20000000e+01,  6.05987988e-03)
( 3.30000000e+01,  5.90621634e-03)
( 3.40000000e+01,  5.68162207e-03)
( 3.50000000e+01,  5.46147586e-03)
( 3.60000000e+01,  5.21789408e-03)
( 3.70000000e+01,  4.94130229e-03)
( 3.80000000e+01,  4.77994282e-03)
( 3.90000000e+01,  4.61338137e-03)
( 4.00000000e+01,  4.38306527e-03)
( 4.10000000e+01,  4.21348082e-03)
( 4.20000000e+01,  4.07904075e-03)
( 4.30000000e+01,  3.94102310e-03)
( 4.40000000e+01,  3.89365437e-03)
( 4.50000000e+01,  3.77486583e-03)
( 4.60000000e+01,  3.73263212e-03)
( 4.70000000e+01,  3.63194885e-03)
( 4.80000000e+01,  3.51111269e-03)
( 4.90000000e+01,  3.42263265e-03)
( 5.00000000e+01,  3.32782403e-03)
( 5.10000000e+01,  3.24254152e-03)
( 5.20000000e+01,  3.14172297e-03)
( 5.30000000e+01,  3.05030983e-03)
( 5.40000000e+01,  2.98763723e-03)
( 5.50000000e+01,  2.91428526e-03)
( 5.60000000e+01,  2.83715236e-03)
( 5.70000000e+01,  2.77601074e-03)
( 5.80000000e+01,  2.73753060e-03)
( 5.90000000e+01,  2.66293708e-03)
( 6.00000000e+01,  2.56737904e-03)
( 6.10000000e+01,  2.47985223e-03)
( 6.20000000e+01,  2.39852886e-03)
( 6.30000000e+01,  2.34732796e-03)
( 6.40000000e+01,  2.30321498e-03)
( 6.50000000e+01,  2.27651330e-03)
( 6.60000000e+01,  2.24264648e-03)
( 6.70000000e+01,  2.20883325e-03)
( 6.80000000e+01,  2.19501160e-03)
( 6.90000000e+01,  2.16711946e-03)
( 7.00000000e+01,  2.13182716e-03)
( 7.10000000e+01,  2.10182320e-03)
( 7.20000000e+01,  2.05695603e-03)
( 7.30000000e+01,  2.01575380e-03)
( 7.40000000e+01,  1.95762618e-03)
( 7.50000000e+01,  1.91447114e-03)
( 7.60000000e+01,  1.86793114e-03)
( 7.70000000e+01,  1.82981742e-03)
( 7.80000000e+01,  1.81285058e-03)
( 7.90000000e+01,  1.78160630e-03)
( 8.00000000e+01,  1.75151877e-03)
( 8.10000000e+01,  1.73141468e-03)
( 8.20000000e+01,  1.70927013e-03)
( 8.30000000e+01,  1.68099911e-03)
( 8.40000000e+01,  1.65065893e-03)
( 8.50000000e+01,  1.62354068e-03)
( 8.60000000e+01,  1.59348790e-03)
( 8.70000000e+01,  1.56785687e-03)
( 8.80000000e+01,  1.54656571e-03)
( 8.90000000e+01,  1.53937241e-03)
( 9.00000000e+01,  1.51986588e-03)
( 9.10000000e+01,  1.49508110e-03)
( 9.20000000e+01,  1.47533204e-03)
( 9.30000000e+01,  1.45527166e-03)
( 9.40000000e+01,  1.43452862e-03)
( 9.50000000e+01,  1.40806289e-03)
( 9.60000000e+01,  1.38695736e-03)
( 9.70000000e+01,  1.36418458e-03)
( 9.80000000e+01,  1.34705436e-03)
( 9.90000000e+01,  1.33150365e-03)
( 1.00000000e+02,  1.31370697e-03)};\label{line:mbb0_nel180x60_rmin0p12_romset1_majit0:romB0}

\addplot [dashed, thick, mark options={solid, thick}, mark=x, mark size=3, color=blue, label=line:trrom1, mark repeat=5]
coordinates {
( 0.00000000e+00,  3.42278327e+00)
( 1.00000000e+00,  1.98323690e+00)
( 2.00000000e+00,  1.04972077e+00)
( 3.00000000e+00,  7.09256038e-01)
( 4.00000000e+00,  5.04630951e-01)
( 5.00000000e+00,  3.93686583e-01)
( 6.00000000e+00,  3.40898954e-01)
( 7.00000000e+00,  3.28014727e-01)
( 8.00000000e+00,  2.78957220e-01)
( 9.00000000e+00,  2.20058220e-01)
( 1.00000000e+01,  1.72103486e-01)
( 1.10000000e+01,  1.18401089e-01)
( 1.20000000e+01,  7.81264980e-02)
( 1.30000000e+01,  4.09615802e-02)
( 1.40000000e+01,  2.91440573e-02)
( 1.50000000e+01,  1.83358661e-02)
( 1.60000000e+01,  1.83358661e-02)
( 1.70000000e+01,  1.39016032e-02)
( 1.80000000e+01,  1.33664047e-02)
( 1.90000000e+01,  1.15863321e-02)
( 2.00000000e+01,  1.08606484e-02)
( 2.10000000e+01,  1.03336867e-02)
( 2.20000000e+01,  9.88048120e-03)
( 2.30000000e+01,  9.50387622e-03)
( 2.40000000e+01,  8.90518183e-03)
( 2.50000000e+01,  8.31303269e-03)
( 2.60000000e+01,  8.18118841e-03)
( 2.70000000e+01,  7.95375937e-03)
( 2.80000000e+01,  7.84546715e-03)
( 2.90000000e+01,  7.68828965e-03)
( 3.00000000e+01,  7.27872260e-03)
( 3.10000000e+01,  6.60884890e-03)
( 3.20000000e+01,  6.31217853e-03)
( 3.30000000e+01,  6.05963511e-03)
( 3.40000000e+01,  5.72406582e-03)
( 3.50000000e+01,  5.44931473e-03)
( 3.60000000e+01,  5.18313652e-03)
( 3.70000000e+01,  4.98398408e-03)
( 3.80000000e+01,  4.74390999e-03)
( 3.90000000e+01,  4.55464537e-03)
( 4.00000000e+01,  4.33415992e-03)
( 4.10000000e+01,  4.26877142e-03)
( 4.20000000e+01,  4.12496377e-03)
( 4.30000000e+01,  4.02693651e-03)
( 4.40000000e+01,  3.96403423e-03)
( 4.50000000e+01,  3.90726815e-03)
( 4.60000000e+01,  3.85368417e-03)
( 4.70000000e+01,  3.76382590e-03)
( 4.80000000e+01,  3.69288356e-03)
( 4.90000000e+01,  3.60172024e-03)
( 5.00000000e+01,  3.48881168e-03)
( 5.10000000e+01,  3.41747671e-03)
( 5.20000000e+01,  3.34242363e-03)
( 5.30000000e+01,  3.25109395e-03)
( 5.40000000e+01,  3.17305930e-03)
( 5.50000000e+01,  3.09421573e-03)
( 5.60000000e+01,  3.01968127e-03)
( 5.70000000e+01,  2.96948053e-03)
( 5.80000000e+01,  2.87564930e-03)
( 5.90000000e+01,  2.83331601e-03)
( 6.00000000e+01,  2.74670705e-03)
( 6.10000000e+01,  2.67150741e-03)
( 6.20000000e+01,  2.59532207e-03)
( 6.30000000e+01,  2.48187845e-03)
( 6.40000000e+01,  2.40552560e-03)
( 6.50000000e+01,  2.33057249e-03)
( 6.60000000e+01,  2.27094546e-03)
( 6.70000000e+01,  2.20849813e-03)
( 6.80000000e+01,  2.17555612e-03)
( 6.90000000e+01,  2.13518467e-03)
( 7.00000000e+01,  2.08504176e-03)
( 7.10000000e+01,  2.01295675e-03)
( 7.20000000e+01,  1.94798588e-03)
( 7.30000000e+01,  1.90955812e-03)
( 7.40000000e+01,  1.87461108e-03)
( 7.50000000e+01,  1.84368183e-03)
( 7.60000000e+01,  1.79370575e-03)
( 7.70000000e+01,  1.79370575e-03)
( 7.80000000e+01,  1.75295276e-03)
( 7.90000000e+01,  1.74213824e-03)
( 8.00000000e+01,  1.72984361e-03)
( 8.10000000e+01,  1.70457428e-03)
( 8.20000000e+01,  1.68839565e-03)
( 8.30000000e+01,  1.66140094e-03)
( 8.40000000e+01,  1.63953779e-03)
( 8.50000000e+01,  1.60606206e-03)
( 8.60000000e+01,  1.58168126e-03)
( 8.70000000e+01,  1.55571574e-03)
( 8.80000000e+01,  1.53264723e-03)
( 8.90000000e+01,  1.51795015e-03)
( 9.00000000e+01,  1.49291654e-03)
( 9.10000000e+01,  1.47411200e-03)
( 9.20000000e+01,  1.44943686e-03)
( 9.30000000e+01,  1.41286742e-03)
( 9.40000000e+01,  1.38320565e-03)
( 9.50000000e+01,  1.35845191e-03)
( 9.60000000e+01,  1.33706889e-03)
( 9.70000000e+01,  1.31253301e-03)
( 9.80000000e+01,  1.28485266e-03)
( 9.90000000e+01,  1.26029363e-03)
( 1.00000000e+02,  1.23954542e-03)};\label{line:mbb0_nel180x60_rmin0p12_romset1_majit0:romB1}

\nextgroupplot[width=0.8\textwidth, grid=major, xmax=50, ylabel={Cum. No. ROM solves}, xmin=0, ymode=log, height=0.4\textwidth]
\addplot [dash dot, thick, mark options={solid, thin}, mark=square*, mark size=1.5, color=magenta, label=line:rom10, mark repeat=5, forget plot]
coordinates {
( 1.00000000e+00,  1.00000000e+00)
( 2.00000000e+00,  2.00000000e+00)
( 3.00000000e+00,  3.00000000e+00)
( 4.00000000e+00,  4.00000000e+00)
( 5.00000000e+00,  5.00000000e+00)
( 6.00000000e+00,  7.00000000e+00)
( 7.00000000e+00,  1.00000000e+01)
( 8.00000000e+00,  1.40000000e+01)
( 9.00000000e+00,  1.80000000e+01)
( 1.00000000e+01,  1.90000000e+01)
( 1.10000000e+01,  2.10000000e+01)
( 1.20000000e+01,  2.40000000e+01)
( 1.30000000e+01,  2.80000000e+01)
( 1.40000000e+01,  3.50000000e+01)
( 1.50000000e+01,  4.30000000e+01)
( 1.60000000e+01,  4.80000000e+01)
( 1.70000000e+01,  5.50000000e+01)
( 1.80000000e+01,  6.60000000e+01)
( 1.90000000e+01,  7.40000000e+01)
( 2.00000000e+01,  8.40000000e+01)
( 2.10000000e+01,  9.50000000e+01)
( 2.20000000e+01,  1.07000000e+02)
( 2.30000000e+01,  1.19000000e+02)
( 2.40000000e+01,  1.31000000e+02)
( 2.50000000e+01,  1.44000000e+02)
( 2.60000000e+01,  1.57000000e+02)
( 2.70000000e+01,  1.67000000e+02)
( 2.80000000e+01,  1.81000000e+02)
( 2.90000000e+01,  1.93000000e+02)
( 3.00000000e+01,  2.05000000e+02)
( 3.10000000e+01,  2.20000000e+02)
( 3.20000000e+01,  2.36000000e+02)
( 3.30000000e+01,  2.51000000e+02)
( 3.40000000e+01,  2.66000000e+02)
( 3.50000000e+01,  2.81000000e+02)
( 3.60000000e+01,  2.96000000e+02)
( 3.70000000e+01,  3.12000000e+02)
( 3.80000000e+01,  3.28000000e+02)
( 3.90000000e+01,  3.44000000e+02)
( 4.00000000e+01,  3.60000000e+02)
( 4.10000000e+01,  3.76000000e+02)
( 4.20000000e+01,  3.90000000e+02)
( 4.30000000e+01,  4.06000000e+02)
( 4.40000000e+01,  4.18000000e+02)
( 4.50000000e+01,  4.29000000e+02)
( 4.60000000e+01,  4.42000000e+02)
( 4.70000000e+01,  4.56000000e+02)
( 4.80000000e+01,  4.67000000e+02)
( 4.90000000e+01,  4.77000000e+02)
( 5.00000000e+01,  4.90000000e+02)
( 5.10000000e+01,  5.05000000e+02)
( 5.20000000e+01,  5.18000000e+02)
( 5.30000000e+01,  5.30000000e+02)
( 5.40000000e+01,  5.45000000e+02)
( 5.50000000e+01,  5.61000000e+02)
( 5.60000000e+01,  5.77000000e+02)
( 5.70000000e+01,  5.91000000e+02)
( 5.80000000e+01,  6.05000000e+02)
( 5.90000000e+01,  6.17000000e+02)
( 6.00000000e+01,  6.32000000e+02)
( 6.10000000e+01,  6.49000000e+02)
( 6.20000000e+01,  6.66000000e+02)
( 6.30000000e+01,  6.79000000e+02)
( 6.40000000e+01,  6.96000000e+02)
( 6.50000000e+01,  7.13000000e+02)
( 6.60000000e+01,  7.29000000e+02)
( 6.70000000e+01,  7.45000000e+02)
( 6.80000000e+01,  7.61000000e+02)
( 6.90000000e+01,  7.77000000e+02)
( 7.00000000e+01,  7.91000000e+02)
( 7.10000000e+01,  8.01000000e+02)
( 7.20000000e+01,  8.13000000e+02)
( 7.30000000e+01,  8.27000000e+02)
( 7.40000000e+01,  8.42000000e+02)
( 7.50000000e+01,  8.53000000e+02)
( 7.60000000e+01,  8.67000000e+02)
( 7.70000000e+01,  8.82000000e+02)
( 7.80000000e+01,  8.93000000e+02)
( 7.90000000e+01,  8.99000000e+02)
( 8.00000000e+01,  9.07000000e+02)
( 8.10000000e+01,  9.17000000e+02)
( 8.20000000e+01,  9.29000000e+02)
( 8.30000000e+01,  9.43000000e+02)
( 8.40000000e+01,  9.59000000e+02)
( 8.50000000e+01,  9.77000000e+02)
( 8.60000000e+01,  9.93000000e+02)
( 8.70000000e+01,  1.00600000e+03)
( 8.80000000e+01,  1.01400000e+03)
( 8.90000000e+01,  1.02600000e+03)
( 9.00000000e+01,  1.04100000e+03)
( 9.10000000e+01,  1.05700000e+03)
( 9.20000000e+01,  1.06800000e+03)
( 9.30000000e+01,  1.08200000e+03)
( 9.40000000e+01,  1.09700000e+03)
( 9.50000000e+01,  1.11400000e+03)
( 9.60000000e+01,  1.12900000e+03)
( 9.70000000e+01,  1.14600000e+03)
( 9.80000000e+01,  1.15900000e+03)
( 9.90000000e+01,  1.16800000e+03)
( 1.00000000e+02,  1.17900000e+03)};

\addplot [dash dot, thick, mark options={solid, thin}, mark=square*, mark size=1.5, color=blue, label=line:rom11, mark repeat=5, forget plot]
coordinates {
( 1.00000000e+00,  1.00000000e+00)
( 2.00000000e+00,  2.00000000e+00)
( 3.00000000e+00,  4.00000000e+00)
( 4.00000000e+00,  8.00000000e+00)
( 5.00000000e+00,  1.10000000e+01)
( 6.00000000e+00,  1.40000000e+01)
( 7.00000000e+00,  1.70000000e+01)
( 8.00000000e+00,  2.00000000e+01)
( 9.00000000e+00,  2.30000000e+01)
( 1.00000000e+01,  2.60000000e+01)
( 1.10000000e+01,  2.90000000e+01)
( 1.20000000e+01,  3.50000000e+01)
( 1.30000000e+01,  4.30000000e+01)
( 1.40000000e+01,  5.30000000e+01)
( 1.50000000e+01,  5.90000000e+01)
( 1.60000000e+01,  6.70000000e+01)
( 1.70000000e+01,  7.80000000e+01)
( 1.80000000e+01,  8.70000000e+01)
( 1.90000000e+01,  1.00000000e+02)
( 2.00000000e+01,  1.13000000e+02)
( 2.10000000e+01,  1.28000000e+02)
( 2.20000000e+01,  1.44000000e+02)
( 2.30000000e+01,  1.57000000e+02)
( 2.40000000e+01,  1.71000000e+02)
( 2.50000000e+01,  1.83000000e+02)
( 2.60000000e+01,  1.94000000e+02)
( 2.70000000e+01,  2.07000000e+02)
( 2.80000000e+01,  2.21000000e+02)
( 2.90000000e+01,  2.36000000e+02)
( 3.00000000e+01,  2.51000000e+02)
( 3.10000000e+01,  2.66000000e+02)
( 3.20000000e+01,  2.81000000e+02)
( 3.30000000e+01,  2.97000000e+02)
( 3.40000000e+01,  3.13000000e+02)
( 3.50000000e+01,  3.29000000e+02)
( 3.60000000e+01,  3.45000000e+02)
( 3.70000000e+01,  3.61000000e+02)
( 3.80000000e+01,  3.77000000e+02)
( 3.90000000e+01,  3.93000000e+02)
( 4.00000000e+01,  4.07000000e+02)
( 4.10000000e+01,  4.23000000e+02)
( 4.20000000e+01,  4.38000000e+02)
( 4.30000000e+01,  4.54000000e+02)
( 4.40000000e+01,  4.67000000e+02)
( 4.50000000e+01,  4.79000000e+02)
( 4.60000000e+01,  4.94000000e+02)
( 4.70000000e+01,  5.06000000e+02)
( 4.80000000e+01,  5.17000000e+02)
( 4.90000000e+01,  5.31000000e+02)
( 5.00000000e+01,  5.43000000e+02)
( 5.10000000e+01,  5.57000000e+02)
( 5.20000000e+01,  5.72000000e+02)
( 5.30000000e+01,  5.85000000e+02)
( 5.40000000e+01,  5.96000000e+02)
( 5.50000000e+01,  6.10000000e+02)
( 5.60000000e+01,  6.26000000e+02)
( 5.70000000e+01,  6.43000000e+02)
( 5.80000000e+01,  6.58000000e+02)
( 5.90000000e+01,  6.75000000e+02)
( 6.00000000e+01,  6.90000000e+02)
( 6.10000000e+01,  7.03000000e+02)
( 6.20000000e+01,  7.19000000e+02)
( 6.30000000e+01,  7.37000000e+02)
( 6.40000000e+01,  7.54000000e+02)
( 6.50000000e+01,  7.71000000e+02)
( 6.60000000e+01,  7.88000000e+02)
( 6.70000000e+01,  8.03000000e+02)
( 6.80000000e+01,  8.17000000e+02)
( 6.90000000e+01,  8.35000000e+02)
( 7.00000000e+01,  8.51000000e+02)
( 7.10000000e+01,  8.64000000e+02)
( 7.20000000e+01,  8.73000000e+02)
( 7.30000000e+01,  8.84000000e+02)
( 7.40000000e+01,  8.97000000e+02)
( 7.50000000e+01,  9.11000000e+02)
( 7.60000000e+01,  9.21000000e+02)
( 7.70000000e+01,  9.33000000e+02)
( 7.80000000e+01,  9.47000000e+02)
( 7.90000000e+01,  9.63000000e+02)
( 8.00000000e+01,  9.76000000e+02)
( 8.10000000e+01,  9.85000000e+02)
( 8.20000000e+01,  9.96000000e+02)
( 8.30000000e+01,  1.00900000e+03)
( 8.40000000e+01,  1.02400000e+03)
( 8.50000000e+01,  1.04100000e+03)
( 8.60000000e+01,  1.05500000e+03)
( 8.70000000e+01,  1.06600000e+03)
( 8.80000000e+01,  1.07900000e+03)
( 8.90000000e+01,  1.09400000e+03)
( 9.00000000e+01,  1.11100000e+03)
( 9.10000000e+01,  1.12200000e+03)
( 9.20000000e+01,  1.13800000e+03)
( 9.30000000e+01,  1.14900000e+03)
( 9.40000000e+01,  1.16300000e+03)
( 9.50000000e+01,  1.17800000e+03)
( 9.60000000e+01,  1.19500000e+03)
( 9.70000000e+01,  1.21300000e+03)
( 9.80000000e+01,  1.22700000e+03)
( 9.90000000e+01,  1.23700000e+03)
( 1.00000000e+02,  1.24100000e+03)};

\addplot [dashed, thick, mark options={solid, thick}, mark=x, mark size=3, color=magenta, label=line:trrom0, mark repeat=5, forget plot]
coordinates {
( 1.00000000e+00,  1.00000000e+00)
( 2.00000000e+00,  2.00000000e+00)
( 3.00000000e+00,  3.00000000e+00)
( 4.00000000e+00,  4.00000000e+00)
( 5.00000000e+00,  5.00000000e+00)
( 6.00000000e+00,  6.00000000e+00)
( 7.00000000e+00,  8.00000000e+00)
( 8.00000000e+00,  1.10000000e+01)
( 9.00000000e+00,  1.40000000e+01)
( 1.00000000e+01,  1.70000000e+01)
( 1.10000000e+01,  2.00000000e+01)
( 1.20000000e+01,  2.30000000e+01)
( 1.30000000e+01,  2.70000000e+01)
( 1.40000000e+01,  3.00000000e+01)
( 1.50000000e+01,  3.60000000e+01)
( 1.60000000e+01,  4.20000000e+01)
( 1.70000000e+01,  4.80000000e+01)
( 1.80000000e+01,  5.80000000e+01)
( 1.90000000e+01,  6.60000000e+01)
( 2.00000000e+01,  7.70000000e+01)
( 2.10000000e+01,  9.10000000e+01)
( 2.20000000e+01,  9.90000000e+01)
( 2.30000000e+01,  1.13000000e+02)
( 2.40000000e+01,  1.24000000e+02)
( 2.50000000e+01,  1.36000000e+02)
( 2.60000000e+01,  1.49000000e+02)
( 2.70000000e+01,  1.62000000e+02)
( 2.80000000e+01,  1.74000000e+02)
( 2.90000000e+01,  1.86000000e+02)
( 3.00000000e+01,  1.99000000e+02)
( 3.10000000e+01,  2.12000000e+02)
( 3.20000000e+01,  2.25000000e+02)
( 3.30000000e+01,  2.37000000e+02)
( 3.40000000e+01,  2.51000000e+02)
( 3.50000000e+01,  2.64000000e+02)
( 3.60000000e+01,  2.77000000e+02)
( 3.70000000e+01,  2.93000000e+02)
( 3.80000000e+01,  3.08000000e+02)
( 3.90000000e+01,  3.22000000e+02)
( 4.00000000e+01,  3.38000000e+02)
( 4.10000000e+01,  3.53000000e+02)
( 4.20000000e+01,  3.67000000e+02)
( 4.30000000e+01,  3.84000000e+02)
( 4.40000000e+01,  4.00000000e+02)
( 4.50000000e+01,  4.10000000e+02)
( 4.60000000e+01,  4.23000000e+02)
( 4.70000000e+01,  4.35000000e+02)
( 4.80000000e+01,  4.51000000e+02)
( 4.90000000e+01,  4.65000000e+02)
( 5.00000000e+01,  4.80000000e+02)
( 5.10000000e+01,  4.96000000e+02)
( 5.20000000e+01,  5.12000000e+02)
( 5.30000000e+01,  5.28000000e+02)
( 5.40000000e+01,  5.43000000e+02)
( 5.50000000e+01,  5.59000000e+02)
( 5.60000000e+01,  5.75000000e+02)
( 5.70000000e+01,  5.90000000e+02)
( 5.80000000e+01,  6.05000000e+02)
( 5.90000000e+01,  6.19000000e+02)
( 6.00000000e+01,  6.34000000e+02)
( 6.10000000e+01,  6.50000000e+02)
( 6.20000000e+01,  6.66000000e+02)
( 6.30000000e+01,  6.81000000e+02)
( 6.40000000e+01,  6.97000000e+02)
( 6.50000000e+01,  7.12000000e+02)
( 6.60000000e+01,  7.21000000e+02)
( 6.70000000e+01,  7.35000000e+02)
( 6.80000000e+01,  7.49000000e+02)
( 6.90000000e+01,  7.62000000e+02)
( 7.00000000e+01,  7.79000000e+02)
( 7.10000000e+01,  7.93000000e+02)
( 7.20000000e+01,  8.09000000e+02)
( 7.30000000e+01,  8.23000000e+02)
( 7.40000000e+01,  8.40000000e+02)
( 7.50000000e+01,  8.55000000e+02)
( 7.60000000e+01,  8.71000000e+02)
( 7.70000000e+01,  8.86000000e+02)
( 7.80000000e+01,  8.99000000e+02)
( 7.90000000e+01,  9.13000000e+02)
( 8.00000000e+01,  9.30000000e+02)
( 8.10000000e+01,  9.44000000e+02)
( 8.20000000e+01,  9.61000000e+02)
( 8.30000000e+01,  9.76000000e+02)
( 8.40000000e+01,  9.93000000e+02)
( 8.50000000e+01,  1.00800000e+03)
( 8.60000000e+01,  1.02400000e+03)
( 8.70000000e+01,  1.04000000e+03)
( 8.80000000e+01,  1.05600000e+03)
( 8.90000000e+01,  1.06800000e+03)
( 9.00000000e+01,  1.08000000e+03)
( 9.10000000e+01,  1.09700000e+03)
( 9.20000000e+01,  1.11300000e+03)
( 9.30000000e+01,  1.12800000e+03)
( 9.40000000e+01,  1.14200000e+03)
( 9.50000000e+01,  1.15700000e+03)
( 9.60000000e+01,  1.17100000e+03)
( 9.70000000e+01,  1.18700000e+03)
( 9.80000000e+01,  1.20100000e+03)
( 9.90000000e+01,  1.21500000e+03)
( 1.00000000e+02,  1.22900000e+03)};

\addplot [dashed, thick, mark options={solid, thick}, mark=x, mark size=3, color=blue, label=line:trrom1, mark repeat=5, forget plot]
coordinates {
( 1.00000000e+00,  1.00000000e+00)
( 2.00000000e+00,  2.00000000e+00)
( 3.00000000e+00,  3.00000000e+00)
( 4.00000000e+00,  5.00000000e+00)
( 5.00000000e+00,  8.00000000e+00)
( 6.00000000e+00,  1.00000000e+01)
( 7.00000000e+00,  1.40000000e+01)
( 8.00000000e+00,  1.50000000e+01)
( 9.00000000e+00,  1.80000000e+01)
( 1.00000000e+01,  2.10000000e+01)
( 1.10000000e+01,  2.40000000e+01)
( 1.20000000e+01,  2.70000000e+01)
( 1.30000000e+01,  3.30000000e+01)
( 1.40000000e+01,  4.10000000e+01)
( 1.50000000e+01,  5.10000000e+01)
( 1.60000000e+01,  6.00000000e+01)
( 1.70000000e+01,  6.90000000e+01)
( 1.80000000e+01,  8.10000000e+01)
( 1.90000000e+01,  8.90000000e+01)
( 2.00000000e+01,  1.00000000e+02)
( 2.10000000e+01,  1.14000000e+02)
( 2.20000000e+01,  1.23000000e+02)
( 2.30000000e+01,  1.37000000e+02)
( 2.40000000e+01,  1.47000000e+02)
( 2.50000000e+01,  1.62000000e+02)
( 2.60000000e+01,  1.73000000e+02)
( 2.70000000e+01,  1.85000000e+02)
( 2.80000000e+01,  1.95000000e+02)
( 2.90000000e+01,  2.07000000e+02)
( 3.00000000e+01,  2.20000000e+02)
( 3.10000000e+01,  2.36000000e+02)
( 3.20000000e+01,  2.51000000e+02)
( 3.30000000e+01,  2.68000000e+02)
( 3.40000000e+01,  2.83000000e+02)
( 3.50000000e+01,  2.98000000e+02)
( 3.60000000e+01,  3.14000000e+02)
( 3.70000000e+01,  3.29000000e+02)
( 3.80000000e+01,  3.45000000e+02)
( 3.90000000e+01,  3.61000000e+02)
( 4.00000000e+01,  3.77000000e+02)
( 4.10000000e+01,  3.93000000e+02)
( 4.20000000e+01,  4.02000000e+02)
( 4.30000000e+01,  4.16000000e+02)
( 4.40000000e+01,  4.29000000e+02)
( 4.50000000e+01,  4.42000000e+02)
( 4.60000000e+01,  4.55000000e+02)
( 4.70000000e+01,  4.69000000e+02)
( 4.80000000e+01,  4.83000000e+02)
( 4.90000000e+01,  4.97000000e+02)
( 5.00000000e+01,  5.13000000e+02)
( 5.10000000e+01,  5.27000000e+02)
( 5.20000000e+01,  5.41000000e+02)
( 5.30000000e+01,  5.56000000e+02)
( 5.40000000e+01,  5.72000000e+02)
( 5.50000000e+01,  5.86000000e+02)
( 5.60000000e+01,  6.02000000e+02)
( 5.70000000e+01,  6.17000000e+02)
( 5.80000000e+01,  6.34000000e+02)
( 5.90000000e+01,  6.49000000e+02)
( 6.00000000e+01,  6.64000000e+02)
( 6.10000000e+01,  6.80000000e+02)
( 6.20000000e+01,  6.96000000e+02)
( 6.30000000e+01,  7.13000000e+02)
( 6.40000000e+01,  7.31000000e+02)
( 6.50000000e+01,  7.47000000e+02)
( 6.60000000e+01,  7.64000000e+02)
( 6.70000000e+01,  7.81000000e+02)
( 6.80000000e+01,  7.97000000e+02)
( 6.90000000e+01,  8.13000000e+02)
( 7.00000000e+01,  8.30000000e+02)
( 7.10000000e+01,  8.46000000e+02)
( 7.20000000e+01,  8.64000000e+02)
( 7.30000000e+01,  8.80000000e+02)
( 7.40000000e+01,  8.94000000e+02)
( 7.50000000e+01,  9.08000000e+02)
( 7.60000000e+01,  9.24000000e+02)
( 7.70000000e+01,  9.38000000e+02)
( 7.80000000e+01,  9.53000000e+02)
( 7.90000000e+01,  9.66000000e+02)
( 8.00000000e+01,  9.77000000e+02)
( 8.10000000e+01,  9.93000000e+02)
( 8.20000000e+01,  1.00700000e+03)
( 8.30000000e+01,  1.02400000e+03)
( 8.40000000e+01,  1.03900000e+03)
( 8.50000000e+01,  1.05500000e+03)
( 8.60000000e+01,  1.07000000e+03)
( 8.70000000e+01,  1.08600000e+03)
( 8.80000000e+01,  1.10200000e+03)
( 8.90000000e+01,  1.11700000e+03)
( 9.00000000e+01,  1.13200000e+03)
( 9.10000000e+01,  1.14700000e+03)
( 9.20000000e+01,  1.16200000e+03)
( 9.30000000e+01,  1.17800000e+03)
( 9.40000000e+01,  1.19400000e+03)
( 9.50000000e+01,  1.21000000e+03)
( 9.60000000e+01,  1.22600000e+03)
( 9.70000000e+01,  1.24100000e+03)
( 9.80000000e+01,  1.25700000e+03)
( 9.90000000e+01,  1.27400000e+03)
( 1.00000000e+02,  1.29000000e+03)};

\nextgroupplot[width=0.8\textwidth, grid=major, xmax=50, ylabel={$\Delta_k$}, xmin=0, ymode=log, xlabel={Major iteration}, height=0.4\textwidth]
\addplot [dash dot, thick, mark options={solid, thin}, mark=square*, mark size=1.5, color=magenta, label=line:rom10, mark repeat=5, forget plot]
coordinates {
( 0.00000000e+00,  5.19615242e-01)
( 1.00000000e+00,  5.19615242e-01)
( 2.00000000e+00,  7.79422863e-01)
( 3.00000000e+00,  1.16913430e+00)
( 4.00000000e+00,  1.75370144e+00)
( 5.00000000e+00,  2.63055216e+00)
( 6.00000000e+00,  3.94582825e+00)
( 7.00000000e+00,  5.91874237e+00)
( 8.00000000e+00,  5.91874237e+00)
( 9.00000000e+00,  2.95937118e+00)
( 1.00000000e+01,  2.95937118e+00)
( 1.10000000e+01,  4.43905678e+00)
( 1.20000000e+01,  6.65858517e+00)
( 1.30000000e+01,  9.98787775e+00)
( 1.40000000e+01,  9.98787775e+00)
( 1.50000000e+01,  4.99393887e+00)
( 1.60000000e+01,  4.99393887e+00)
( 1.70000000e+01,  4.99393887e+00)
( 1.80000000e+01,  2.49696944e+00)
( 1.90000000e+01,  2.49696944e+00)
( 2.00000000e+01,  2.49696944e+00)
( 2.10000000e+01,  2.49696944e+00)
( 2.20000000e+01,  2.49696944e+00)
( 2.30000000e+01,  2.49696944e+00)
( 2.40000000e+01,  2.49696944e+00)
( 2.50000000e+01,  2.49696944e+00)
( 2.60000000e+01,  1.24848472e+00)
( 2.70000000e+01,  1.87272708e+00)
( 2.80000000e+01,  1.87272708e+00)
( 2.90000000e+01,  1.87272708e+00)
( 3.00000000e+01,  2.80909062e+00)
( 3.10000000e+01,  2.80909062e+00)
( 3.20000000e+01,  2.80909062e+00)
( 3.30000000e+01,  2.80909062e+00)
( 3.40000000e+01,  2.80909062e+00)
( 3.50000000e+01,  2.80909062e+00)
( 3.60000000e+01,  2.80909062e+00)
( 3.70000000e+01,  2.80909062e+00)
( 3.80000000e+01,  2.80909062e+00)
( 3.90000000e+01,  2.80909062e+00)
( 4.00000000e+01,  2.80909062e+00)
( 4.10000000e+01,  1.40454531e+00)
( 4.20000000e+01,  2.10681796e+00)
( 4.30000000e+01,  1.05340898e+00)
( 4.40000000e+01,  5.54879139e-01)
( 4.50000000e+01,  8.32318709e-01)
( 4.60000000e+01,  1.24847806e+00)
( 4.70000000e+01,  8.99148793e-01)
( 4.80000000e+01,  5.21188680e-01)
( 4.90000000e+01,  7.81783020e-01)
( 5.00000000e+01,  1.17267453e+00)
( 5.10000000e+01,  8.98650535e-01)
( 5.20000000e+01,  6.11090108e-01)
( 5.30000000e+01,  9.16635162e-01)
( 5.40000000e+01,  1.37495274e+00)
( 5.50000000e+01,  1.33004896e+00)
( 5.60000000e+01,  1.13351959e+00)
( 5.70000000e+01,  9.51735215e-01)
( 5.80000000e+01,  6.51460903e-01)
( 5.90000000e+01,  9.77191355e-01)
( 6.00000000e+01,  1.46578703e+00)
( 6.10000000e+01,  2.19868055e+00)
( 6.20000000e+01,  1.09934027e+00)
( 6.30000000e+01,  1.64901041e+00)
( 6.40000000e+01,  1.64901041e+00)
( 6.50000000e+01,  1.64901041e+00)
( 6.60000000e+01,  1.47898229e+00)
( 6.70000000e+01,  1.32287013e+00)
( 6.80000000e+01,  9.85849162e-01)
( 6.90000000e+01,  6.48214557e-01)
( 7.00000000e+01,  2.38743106e-01)
( 7.10000000e+01,  3.58114660e-01)
( 7.20000000e+01,  5.37171990e-01)
( 7.30000000e+01,  8.05757984e-01)
( 7.40000000e+01,  4.00841996e-01)
( 7.50000000e+01,  6.01262994e-01)
( 7.60000000e+01,  9.01894491e-01)
( 7.70000000e+01,  4.23816599e-01)
( 7.80000000e+01,  1.18059093e-01)
( 7.90000000e+01,  1.77088640e-01)
( 8.00000000e+01,  2.65632960e-01)
( 8.10000000e+01,  3.98449441e-01)
( 8.20000000e+01,  5.97674161e-01)
( 8.30000000e+01,  8.96511241e-01)
( 8.40000000e+01,  1.34476686e+00)
( 8.50000000e+01,  1.03576104e+00)
( 8.60000000e+01,  6.58269490e-01)
( 8.70000000e+01,  2.21768567e-01)
( 8.80000000e+01,  3.32652851e-01)
( 8.90000000e+01,  4.98979277e-01)
( 9.00000000e+01,  7.48468915e-01)
( 9.10000000e+01,  2.84136169e-01)
( 9.20000000e+01,  4.26204253e-01)
( 9.30000000e+01,  6.39306380e-01)
( 9.40000000e+01,  9.58959570e-01)
( 9.50000000e+01,  6.08893778e-01)
( 9.60000000e+01,  9.13340668e-01)
( 9.70000000e+01,  5.37784934e-01)
( 9.80000000e+01,  1.44662495e-01)
( 9.90000000e+01,  2.16993742e-01)
( 1.00000000e+02,  3.25490613e-01)};

\addplot [dash dot, thick, mark options={solid, thin}, mark=square*, mark size=1.5, color=blue, label=line:rom11, mark repeat=5, forget plot]
coordinates {
( 0.00000000e+00,  5.19615242e+00)
( 1.00000000e+00,  5.19615242e+00)
( 2.00000000e+00,  7.79422863e+00)
( 3.00000000e+00,  1.16913430e+01)
( 4.00000000e+00,  5.84567148e+00)
( 5.00000000e+00,  5.84567148e+00)
( 6.00000000e+00,  5.84567148e+00)
( 7.00000000e+00,  5.84567148e+00)
( 8.00000000e+00,  5.84567148e+00)
( 9.00000000e+00,  5.84567148e+00)
( 1.00000000e+01,  5.84567148e+00)
( 1.10000000e+01,  8.76850721e+00)
( 1.20000000e+01,  8.76850721e+00)
( 1.30000000e+01,  8.76850721e+00)
( 1.40000000e+01,  4.38425361e+00)
( 1.50000000e+01,  4.38425361e+00)
( 1.60000000e+01,  4.38425361e+00)
( 1.70000000e+01,  2.19212680e+00)
( 1.80000000e+01,  3.28819020e+00)
( 1.90000000e+01,  3.28819020e+00)
( 2.00000000e+01,  3.28819020e+00)
( 2.10000000e+01,  3.28819020e+00)
( 2.20000000e+01,  1.64409510e+00)
( 2.30000000e+01,  1.64409510e+00)
( 2.40000000e+01,  1.64409510e+00)
( 2.50000000e+01,  1.64409510e+00)
( 2.60000000e+01,  2.46614265e+00)
( 2.70000000e+01,  2.46614265e+00)
( 2.80000000e+01,  2.46614265e+00)
( 2.90000000e+01,  2.46614265e+00)
( 3.00000000e+01,  2.46614265e+00)
( 3.10000000e+01,  2.46614265e+00)
( 3.20000000e+01,  2.46614265e+00)
( 3.30000000e+01,  2.46614265e+00)
( 3.40000000e+01,  2.46614265e+00)
( 3.50000000e+01,  2.46614265e+00)
( 3.60000000e+01,  2.46614265e+00)
( 3.70000000e+01,  2.46614265e+00)
( 3.80000000e+01,  2.46614265e+00)
( 3.90000000e+01,  1.23307133e+00)
( 4.00000000e+01,  1.84960699e+00)
( 4.10000000e+01,  1.83279165e+00)
( 4.20000000e+01,  1.81445632e+00)
( 4.30000000e+01,  1.81445632e+00)
( 4.40000000e+01,  9.07228159e-01)
( 4.50000000e+01,  1.36084224e+00)
( 4.60000000e+01,  1.10582447e+00)
( 4.70000000e+01,  6.90485732e-01)
( 4.80000000e+01,  1.03572860e+00)
( 4.90000000e+01,  5.67665057e-01)
( 5.00000000e+01,  8.51497586e-01)
( 5.10000000e+01,  1.27724638e+00)
( 5.20000000e+01,  8.86681495e-01)
( 5.30000000e+01,  4.76406745e-01)
( 5.40000000e+01,  7.14610117e-01)
( 5.50000000e+01,  1.07191518e+00)
( 5.60000000e+01,  1.60787276e+00)
( 5.70000000e+01,  8.03936382e-01)
( 5.80000000e+01,  1.20590457e+00)
( 5.90000000e+01,  1.09410125e+00)
( 6.00000000e+01,  7.85708414e-01)
( 6.10000000e+01,  1.17856262e+00)
( 6.20000000e+01,  1.76784393e+00)
( 6.30000000e+01,  1.76784393e+00)
( 6.40000000e+01,  1.76784393e+00)
( 6.50000000e+01,  1.71206505e+00)
( 6.60000000e+01,  1.63776545e+00)
( 6.70000000e+01,  8.18882727e-01)
( 6.80000000e+01,  1.22832409e+00)
( 6.90000000e+01,  9.98892155e-01)
( 7.00000000e+01,  6.22231713e-01)
( 7.10000000e+01,  2.04184565e-01)
( 7.20000000e+01,  3.06276847e-01)
( 7.30000000e+01,  4.59415271e-01)
( 7.40000000e+01,  6.89122906e-01)
( 7.50000000e+01,  2.95498791e-01)
( 7.60000000e+01,  4.43248186e-01)
( 7.70000000e+01,  6.64872280e-01)
( 7.80000000e+01,  9.97308419e-01)
( 7.90000000e+01,  5.38302403e-01)
( 8.00000000e+01,  2.01044040e-01)
( 8.10000000e+01,  3.01566061e-01)
( 8.20000000e+01,  4.52349091e-01)
( 8.30000000e+01,  6.78523636e-01)
( 8.40000000e+01,  1.01778545e+00)
( 8.50000000e+01,  6.36304076e-01)
( 8.60000000e+01,  2.54091812e-01)
( 8.70000000e+01,  3.81137718e-01)
( 8.80000000e+01,  5.71706578e-01)
( 8.90000000e+01,  8.57559866e-01)
( 9.00000000e+01,  4.28779933e-01)
( 9.10000000e+01,  6.43169900e-01)
( 9.20000000e+01,  2.62679817e-01)
( 9.30000000e+01,  3.94019725e-01)
( 9.40000000e+01,  5.91029588e-01)
( 9.50000000e+01,  8.86544382e-01)
( 9.60000000e+01,  1.32981657e+00)
( 9.70000000e+01,  6.64908286e-01)
( 9.80000000e+01,  2.85671587e-01)
( 9.90000000e+01,  5.34676643e-02)
( 1.00000000e+02,  8.02014965e-02)};

\addplot [dashed, thick, mark options={solid, thick}, mark=x, mark size=3, color=magenta, label=line:trrom0, mark repeat=5, forget plot]
coordinates {
( 0.00000000e+00,  6.97216689e-04)
( 1.00000000e+00,  6.97216689e-04)
( 2.00000000e+00,  1.04582503e-03)
( 3.00000000e+00,  1.56873755e-03)
( 4.00000000e+00,  2.35310632e-03)
( 5.00000000e+00,  3.52965949e-03)
( 6.00000000e+00,  5.29448923e-03)
( 7.00000000e+00,  7.94173385e-03)
( 8.00000000e+00,  1.19126008e-02)
( 9.00000000e+00,  1.78689012e-02)
( 1.00000000e+01,  1.78689012e-02)
( 1.10000000e+01,  1.78689012e-02)
( 1.20000000e+01,  2.68033517e-02)
( 1.30000000e+01,  2.68033517e-02)
( 1.40000000e+01,  4.02050276e-02)
( 1.50000000e+01,  4.02050276e-02)
( 1.60000000e+01,  4.02050276e-02)
( 1.70000000e+01,  4.02050276e-02)
( 1.80000000e+01,  2.01025138e-02)
( 1.90000000e+01,  3.01537707e-02)
( 2.00000000e+01,  3.01537707e-02)
( 2.10000000e+01,  1.50768853e-02)
( 2.20000000e+01,  1.50768853e-02)
( 2.30000000e+01,  1.50768853e-02)
( 2.40000000e+01,  7.53844267e-03)
( 2.50000000e+01,  7.53844267e-03)
( 2.60000000e+01,  7.53844267e-03)
( 2.70000000e+01,  7.53844267e-03)
( 2.80000000e+01,  7.53844267e-03)
( 2.90000000e+01,  7.53844267e-03)
( 3.00000000e+01,  7.53844267e-03)
( 3.10000000e+01,  7.53844267e-03)
( 3.20000000e+01,  7.53844267e-03)
( 3.30000000e+01,  7.53844267e-03)
( 3.40000000e+01,  7.53844267e-03)
( 3.50000000e+01,  7.53844267e-03)
( 3.60000000e+01,  7.53844267e-03)
( 3.70000000e+01,  7.53844267e-03)
( 3.80000000e+01,  7.53844267e-03)
( 3.90000000e+01,  7.53844267e-03)
( 4.00000000e+01,  7.53844267e-03)
( 4.10000000e+01,  7.53844267e-03)
( 4.20000000e+01,  7.53844267e-03)
( 4.30000000e+01,  7.53844267e-03)
( 4.40000000e+01,  3.76922134e-03)
( 4.50000000e+01,  5.65383201e-03)
( 4.60000000e+01,  5.65383201e-03)
( 4.70000000e+01,  5.65383201e-03)
( 4.80000000e+01,  5.65383201e-03)
( 4.90000000e+01,  5.65383201e-03)
( 5.00000000e+01,  5.65383201e-03)
( 5.10000000e+01,  5.65383201e-03)
( 5.20000000e+01,  5.65383201e-03)
( 5.30000000e+01,  5.65383201e-03)
( 5.40000000e+01,  5.65383201e-03)
( 5.50000000e+01,  5.65383201e-03)
( 5.60000000e+01,  5.65383201e-03)
( 5.70000000e+01,  5.65383201e-03)
( 5.80000000e+01,  5.65383201e-03)
( 5.90000000e+01,  5.65383201e-03)
( 6.00000000e+01,  5.65383201e-03)
( 6.10000000e+01,  5.65383201e-03)
( 6.20000000e+01,  5.65383201e-03)
( 6.30000000e+01,  5.65383201e-03)
( 6.40000000e+01,  5.65383201e-03)
( 6.50000000e+01,  2.82691600e-03)
( 6.60000000e+01,  2.82691600e-03)
( 6.70000000e+01,  2.82691600e-03)
( 6.80000000e+01,  2.82691600e-03)
( 6.90000000e+01,  2.82691600e-03)
( 7.00000000e+01,  2.82691600e-03)
( 7.10000000e+01,  2.82691600e-03)
( 7.20000000e+01,  2.82691600e-03)
( 7.30000000e+01,  2.82691600e-03)
( 7.40000000e+01,  2.82691600e-03)
( 7.50000000e+01,  2.82691600e-03)
( 7.60000000e+01,  2.82691600e-03)
( 7.70000000e+01,  2.82691600e-03)
( 7.80000000e+01,  2.82691600e-03)
( 7.90000000e+01,  2.82691600e-03)
( 8.00000000e+01,  2.82691600e-03)
( 8.10000000e+01,  2.82691600e-03)
( 8.20000000e+01,  2.82691600e-03)
( 8.30000000e+01,  2.82691600e-03)
( 8.40000000e+01,  2.82691600e-03)
( 8.50000000e+01,  2.82691600e-03)
( 8.60000000e+01,  2.82691600e-03)
( 8.70000000e+01,  2.82691600e-03)
( 8.80000000e+01,  2.82691600e-03)
( 8.90000000e+01,  1.41345800e-03)
( 9.00000000e+01,  2.12018700e-03)
( 9.10000000e+01,  2.12018700e-03)
( 9.20000000e+01,  2.12018700e-03)
( 9.30000000e+01,  2.12018700e-03)
( 9.40000000e+01,  2.12018700e-03)
( 9.50000000e+01,  2.12018700e-03)
( 9.60000000e+01,  2.12018700e-03)
( 9.70000000e+01,  2.12018700e-03)
( 9.80000000e+01,  2.12018700e-03)
( 9.90000000e+01,  2.12018700e-03)
( 1.00000000e+02,  2.12018700e-03)};

\addplot [dashed, thick, mark options={solid, thick}, mark=x, mark size=3, color=blue, label=line:trrom1, mark repeat=5, forget plot]
coordinates {
( 0.00000000e+00,  6.97216689e-03)
( 1.00000000e+00,  6.97216689e-03)
( 2.00000000e+00,  1.04582503e-02)
( 3.00000000e+00,  1.56873755e-02)
( 4.00000000e+00,  2.35310632e-02)
( 5.00000000e+00,  2.35310632e-02)
( 6.00000000e+00,  3.52965949e-02)
( 7.00000000e+00,  1.76482974e-02)
( 8.00000000e+00,  2.64724462e-02)
( 9.00000000e+00,  2.64724462e-02)
( 1.00000000e+01,  2.64724462e-02)
( 1.10000000e+01,  2.64724462e-02)
( 1.20000000e+01,  3.97086692e-02)
( 1.30000000e+01,  5.95630038e-02)
( 1.40000000e+01,  5.95630038e-02)
( 1.50000000e+01,  5.95630038e-02)
( 1.60000000e+01,  2.97815019e-02)
( 1.70000000e+01,  2.97815019e-02)
( 1.80000000e+01,  1.48907510e-02)
( 1.90000000e+01,  1.48907510e-02)
( 2.00000000e+01,  1.48907510e-02)
( 2.10000000e+01,  1.48907510e-02)
( 2.20000000e+01,  1.48907510e-02)
( 2.30000000e+01,  1.48907510e-02)
( 2.40000000e+01,  1.48907510e-02)
( 2.50000000e+01,  1.48907510e-02)
( 2.60000000e+01,  7.44537548e-03)
( 2.70000000e+01,  7.44537548e-03)
( 2.80000000e+01,  7.44537548e-03)
( 2.90000000e+01,  7.44537548e-03)
( 3.00000000e+01,  1.11680632e-02)
( 3.10000000e+01,  1.11680632e-02)
( 3.20000000e+01,  1.11680632e-02)
( 3.30000000e+01,  1.11680632e-02)
( 3.40000000e+01,  1.11680632e-02)
( 3.50000000e+01,  1.11680632e-02)
( 3.60000000e+01,  1.11680632e-02)
( 3.70000000e+01,  1.11680632e-02)
( 3.80000000e+01,  1.11680632e-02)
( 3.90000000e+01,  1.11680632e-02)
( 4.00000000e+01,  1.11680632e-02)
( 4.10000000e+01,  5.58403161e-03)
( 4.20000000e+01,  5.58403161e-03)
( 4.30000000e+01,  5.58403161e-03)
( 4.40000000e+01,  5.58403161e-03)
( 4.50000000e+01,  5.58403161e-03)
( 4.60000000e+01,  5.58403161e-03)
( 4.70000000e+01,  5.58403161e-03)
( 4.80000000e+01,  5.58403161e-03)
( 4.90000000e+01,  5.58403161e-03)
( 5.00000000e+01,  5.58403161e-03)
( 5.10000000e+01,  5.58403161e-03)
( 5.20000000e+01,  5.58403161e-03)
( 5.30000000e+01,  5.58403161e-03)
( 5.40000000e+01,  5.58403161e-03)
( 5.50000000e+01,  5.58403161e-03)
( 5.60000000e+01,  5.58403161e-03)
( 5.70000000e+01,  5.58403161e-03)
( 5.80000000e+01,  5.58403161e-03)
( 5.90000000e+01,  5.58403161e-03)
( 6.00000000e+01,  5.58403161e-03)
( 6.10000000e+01,  5.58403161e-03)
( 6.20000000e+01,  5.58403161e-03)
( 6.30000000e+01,  5.58403161e-03)
( 6.40000000e+01,  5.58403161e-03)
( 6.50000000e+01,  5.58403161e-03)
( 6.60000000e+01,  5.58403161e-03)
( 6.70000000e+01,  5.58403161e-03)
( 6.80000000e+01,  5.58403161e-03)
( 6.90000000e+01,  5.58403161e-03)
( 7.00000000e+01,  5.58403161e-03)
( 7.10000000e+01,  5.58403161e-03)
( 7.20000000e+01,  5.58403161e-03)
( 7.30000000e+01,  5.58403161e-03)
( 7.40000000e+01,  5.58403161e-03)
( 7.50000000e+01,  5.58403161e-03)
( 7.60000000e+01,  5.58403161e-03)
( 7.70000000e+01,  2.79201581e-03)
( 7.80000000e+01,  2.79201581e-03)
( 7.90000000e+01,  2.79201581e-03)
( 8.00000000e+01,  2.79201581e-03)
( 8.10000000e+01,  2.79201581e-03)
( 8.20000000e+01,  2.79201581e-03)
( 8.30000000e+01,  2.79201581e-03)
( 8.40000000e+01,  2.79201581e-03)
( 8.50000000e+01,  2.79201581e-03)
( 8.60000000e+01,  2.79201581e-03)
( 8.70000000e+01,  2.79201581e-03)
( 8.80000000e+01,  2.79201581e-03)
( 8.90000000e+01,  2.79201581e-03)
( 9.00000000e+01,  2.79201581e-03)
( 9.10000000e+01,  2.79201581e-03)
( 9.20000000e+01,  2.79201581e-03)
( 9.30000000e+01,  2.79201581e-03)
( 9.40000000e+01,  2.79201581e-03)
( 9.50000000e+01,  2.79201581e-03)
( 9.60000000e+01,  2.79201581e-03)
( 9.70000000e+01,  2.79201581e-03)
( 9.80000000e+01,  2.79201581e-03)
( 9.90000000e+01,  2.79201581e-03)
( 1.00000000e+02,  2.79201581e-03)};

\end{groupplot}\end{tikzpicture}

%% file: py/mbb0_nel180x60_rmin0p12_majit0_matinterp1.dat
HDM-MMA & $0.100$ & - & $21.19$ & $19$ & - & $19.0$ \\
 & $0.050$ & - & $20.71$ & $21$ & - & $21.0$ \\
 & $0.010$ & - & $20.15$ & $32$ & - & $32.0$ \\\hline
ROM-TR-RES & $0.100$ & $0.10$ & $20.78$ & $13$ & 27 & $13.3$ \\
 & $0.050$ & $0.10$ & $20.54$ & $14$ & 33 & $14.3$ \\
 & $0.010$ & $0.10$ & $20.15$ & $23$ & 123 & $24.2$ \\\hline
ROM-TR-DIST & $0.100$ & $0.10$ & $20.87$ & $12$ & 29 & $12.3$ \\
 & $0.050$ & $0.10$ & $20.52$ & $13$ & 35 & $13.3$ \\
 & $0.010$ & $0.10$ & $20.13$ & $20$ & 100 & $21.0$ \\

%% file: py/cbeam0_geom.tikz
\begin{tikzpicture}
\begin{axis}[
axis equal image,
axis x line*=bottom,
axis y line*=left,
width=0.8\textwidth,
xtick={0, 40},
ytick={0, 12.5, 25},
grid=both,
ymax=30,
xmax=45,
xmin=-5,
ymin=-5]
\addplot [black, solid, opacity=0.6, fill=lightgray, forget plot]
coordinates {
( 0.00000000e+00,  0.00000000e+00)
( 4.00000000e+01,  0.00000000e+00)
( 4.00000000e+01,  2.50000000e+01)
( 0.00000000e+00,  2.50000000e+01)
( 0.00000000e+00,  0.00000000e+00)};

\addplot [black, solid, fill=black, forget plot]
coordinates {
(-5.00000000e-01,  0.00000000e+00)
(-5.00000000e-01,  2.50000000e+01)
( 7.77156117e-16,  2.50000000e+01)
(-7.77156117e-16,  0.00000000e+00)
(-5.00000000e-01,  0.00000000e+00)};

\addplot []
graphics [xmin=0,xmax=40,ymin=0,ymax=25] { 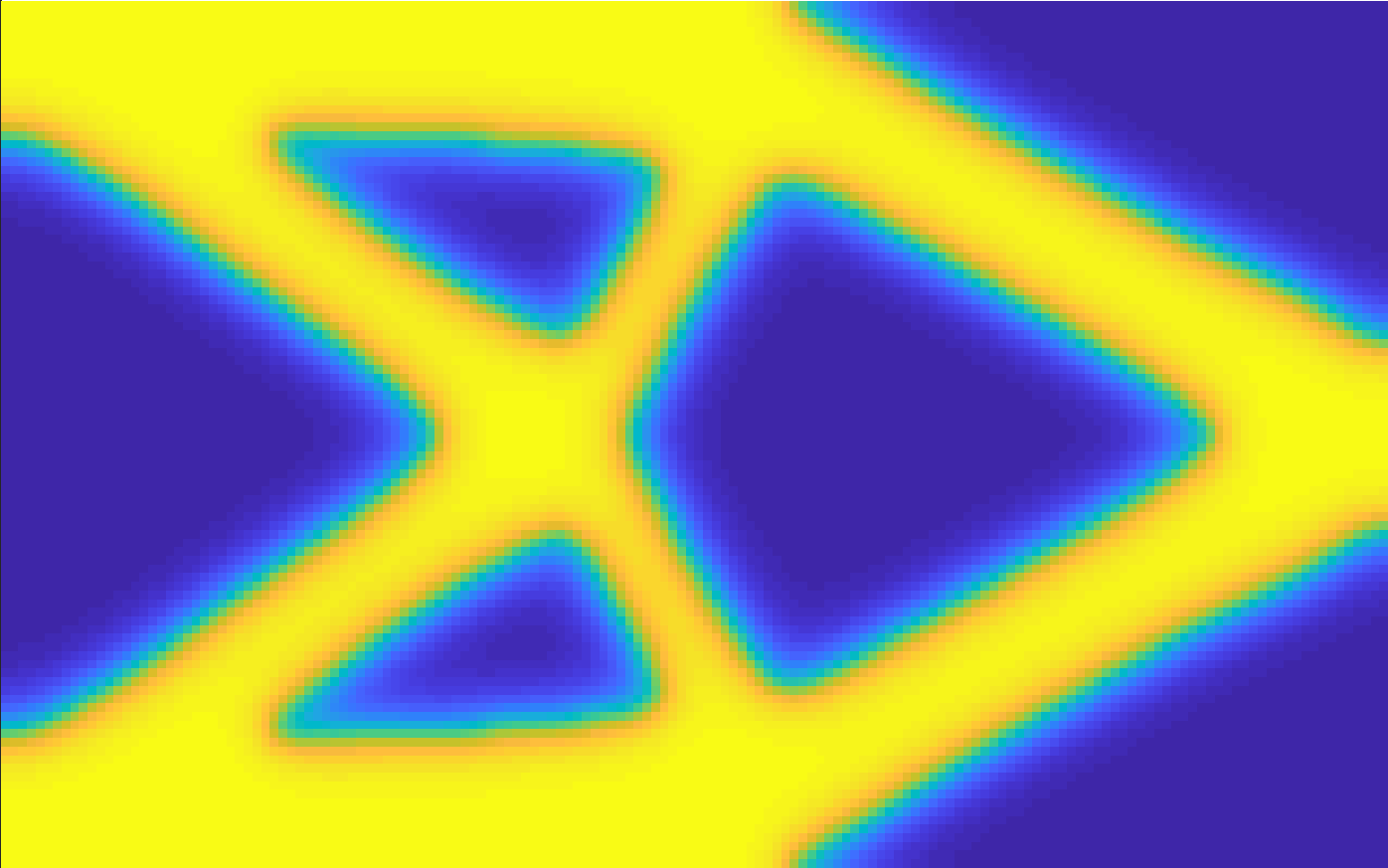};

\draw [-{Latex[width=2mm,length=4mm]}, thick] plot [] coordinates {(axis cs:40.0, 12.5) (axis cs:40.0, 6.5)
};

\node[right]    at    (axis cs:40, 9.5) {$F_\text{in}$};
\end{axis}
\end{tikzpicture}

%% file: py/cbeam0_nel160x100_rmin2_romset1_majit0_matinterp1.tikz
\begin{tikzpicture}
\begin{groupplot} [
group style={group size = 1 by 3, horizontal sep = 2.0cm, vertical sep = 1.0cm}]
\nextgroupplot[width=0.8\textwidth, grid=major, ymax=2.952804656223778, xmax=50, ylabel={$\displaystyle{\frac{|J(\psibold^{(k)})-J^*|}{|J^*|}}$}, xmin=0, ymode=log, height=0.4\textwidth]
\addplot [solid, thick, mark options={solid, thin}, mark=*, mark size=1.5, color=black, label=line:hdm, mark repeat=5]
coordinates {
( 0.00000000e+00,  2.95280466e+00)
( 1.00000000e+00,  1.95880807e+00)
( 2.00000000e+00,  1.19092021e+00)
( 3.00000000e+00,  8.66877773e-01)
( 4.00000000e+00,  7.21864240e-01)
( 5.00000000e+00,  6.43061607e-01)
( 6.00000000e+00,  5.77617057e-01)
( 7.00000000e+00,  5.10905101e-01)
( 8.00000000e+00,  4.44305805e-01)
( 9.00000000e+00,  3.80460217e-01)
( 1.00000000e+01,  3.20458571e-01)
( 1.10000000e+01,  2.70270637e-01)
( 1.20000000e+01,  2.26584489e-01)
( 1.30000000e+01,  1.85792289e-01)
( 1.40000000e+01,  1.48519252e-01)
( 1.50000000e+01,  1.16867537e-01)
( 1.60000000e+01,  9.37511369e-02)
( 1.70000000e+01,  7.61863238e-02)
( 1.80000000e+01,  6.24209806e-02)
( 1.90000000e+01,  4.89093275e-02)
( 2.00000000e+01,  3.51345620e-02)
( 2.10000000e+01,  2.40412500e-02)
( 2.20000000e+01,  1.78478790e-02)
( 2.30000000e+01,  1.44525667e-02)
( 2.40000000e+01,  1.22215163e-02)
( 2.50000000e+01,  1.07105542e-02)
( 2.60000000e+01,  9.69891789e-03)
( 2.70000000e+01,  9.01423260e-03)
( 2.80000000e+01,  8.56101281e-03)
( 2.90000000e+01,  8.24338909e-03)
( 3.00000000e+01,  8.01867258e-03)
( 3.10000000e+01,  7.84855242e-03)
( 3.20000000e+01,  7.70747937e-03)
( 3.30000000e+01,  7.57888360e-03)
( 3.40000000e+01,  7.46266678e-03)
( 3.50000000e+01,  7.35331389e-03)
( 3.60000000e+01,  7.24438907e-03)
( 3.70000000e+01,  7.14864444e-03)
( 3.80000000e+01,  7.07619921e-03)
( 3.90000000e+01,  7.01907319e-03)
( 4.00000000e+01,  6.97453009e-03)
( 4.10000000e+01,  6.94045033e-03)
( 4.20000000e+01,  6.90956010e-03)
( 4.30000000e+01,  6.88211375e-03)
( 4.40000000e+01,  6.85868122e-03)
( 4.50000000e+01,  6.83702915e-03)
( 4.60000000e+01,  6.81507290e-03)
( 4.70000000e+01,  6.79273136e-03)
( 4.80000000e+01,  6.76795910e-03)
( 4.90000000e+01,  6.74152043e-03)
( 5.00000000e+01,  6.71290022e-03)
( 5.10000000e+01,  6.68188896e-03)
( 5.20000000e+01,  6.64929447e-03)
( 5.30000000e+01,  6.61354133e-03)
( 5.40000000e+01,  6.57372136e-03)
( 5.50000000e+01,  6.52786344e-03)
( 5.60000000e+01,  6.47372842e-03)
( 5.70000000e+01,  6.41295655e-03)
( 5.80000000e+01,  6.35038900e-03)
( 5.90000000e+01,  6.28981864e-03)
( 6.00000000e+01,  6.23513469e-03)
( 6.10000000e+01,  6.18744850e-03)
( 6.20000000e+01,  6.14466642e-03)
( 6.30000000e+01,  6.10609266e-03)
( 6.40000000e+01,  6.07267363e-03)
( 6.50000000e+01,  6.04473170e-03)
( 6.60000000e+01,  6.01910365e-03)
( 6.70000000e+01,  5.99593827e-03)
( 6.80000000e+01,  5.97453921e-03)
( 6.90000000e+01,  5.95542969e-03)
( 7.00000000e+01,  5.93666533e-03)
( 7.10000000e+01,  5.91708335e-03)
( 7.20000000e+01,  5.89583866e-03)
( 7.30000000e+01,  5.87281527e-03)
( 7.40000000e+01,  5.84805615e-03)
( 7.50000000e+01,  5.82070031e-03)
( 7.60000000e+01,  5.79125751e-03)
( 7.70000000e+01,  5.75720695e-03)
( 7.80000000e+01,  5.72056995e-03)
( 7.90000000e+01,  5.68392439e-03)
( 8.00000000e+01,  5.65028543e-03)
( 8.10000000e+01,  5.61770923e-03)
( 8.20000000e+01,  5.58512907e-03)
( 8.30000000e+01,  5.55202597e-03)
( 8.40000000e+01,  5.52116449e-03)
( 8.50000000e+01,  5.49407417e-03)
( 8.60000000e+01,  5.46768682e-03)
( 8.70000000e+01,  5.44004856e-03)
( 8.80000000e+01,  5.41307111e-03)
( 8.90000000e+01,  5.38790646e-03)
( 9.00000000e+01,  5.36379313e-03)
( 9.10000000e+01,  5.33903498e-03)
( 9.20000000e+01,  5.31330228e-03)
( 9.30000000e+01,  5.28590221e-03)
( 9.40000000e+01,  5.25517387e-03)
( 9.50000000e+01,  5.22303468e-03)
( 9.60000000e+01,  5.19299386e-03)
( 9.70000000e+01,  5.16117804e-03)
( 9.80000000e+01,  5.12554158e-03)
( 9.90000000e+01,  5.08668162e-03)};\label{line:cbeam0_nel160x100_rmin2_romset1_majit0:hdm}

\addplot [dash dot, thick, mark options={solid, thin}, mark=square*, mark size=1.5, color=blue, label=line:rom11, mark repeat=5]
coordinates {
( 0.00000000e+00,  2.95280466e+00)
( 1.00000000e+00,  1.95880807e+00)
( 2.00000000e+00,  1.19092021e+00)
( 3.00000000e+00,  7.29064073e-01)
( 4.00000000e+00,  6.21586675e-01)
( 5.00000000e+00,  3.19756013e-01)
( 6.00000000e+00,  1.65650977e-01)
( 7.00000000e+00,  7.46193832e-02)
( 8.00000000e+00,  2.55168490e-02)
( 9.00000000e+00,  1.06625274e-02)
( 1.00000000e+01,  1.06625274e-02)
( 1.10000000e+01,  5.20759534e-03)
( 1.20000000e+01,  5.20759534e-03)
( 1.30000000e+01,  4.20286062e-03)
( 1.40000000e+01,  3.98399668e-03)
( 1.50000000e+01,  3.86947937e-03)
( 1.60000000e+01,  3.62890145e-03)
( 1.70000000e+01,  3.47497602e-03)
( 1.80000000e+01,  3.34842249e-03)
( 1.90000000e+01,  3.23277389e-03)
( 2.00000000e+01,  3.13825303e-03)
( 2.10000000e+01,  3.04564869e-03)
( 2.20000000e+01,  2.91217058e-03)
( 2.30000000e+01,  2.75941053e-03)
( 2.40000000e+01,  2.63550588e-03)
( 2.50000000e+01,  2.51195400e-03)
( 2.60000000e+01,  2.43632078e-03)
( 2.70000000e+01,  2.29437453e-03)
( 2.80000000e+01,  2.18470868e-03)
( 2.90000000e+01,  2.08317557e-03)
( 3.00000000e+01,  2.04138513e-03)
( 3.10000000e+01,  1.93698059e-03)
( 3.20000000e+01,  1.87463171e-03)
( 3.30000000e+01,  1.82142603e-03)
( 3.40000000e+01,  1.76978816e-03)
( 3.50000000e+01,  1.71968947e-03)
( 3.60000000e+01,  1.68418931e-03)
( 3.70000000e+01,  1.64361252e-03)
( 3.80000000e+01,  1.61114994e-03)
( 3.90000000e+01,  1.57754667e-03)
( 4.00000000e+01,  1.55250625e-03)
( 4.10000000e+01,  1.51591867e-03)
( 4.20000000e+01,  1.47427262e-03)
( 4.30000000e+01,  1.45060829e-03)
( 4.40000000e+01,  1.42673724e-03)
( 4.50000000e+01,  1.38433863e-03)
( 4.60000000e+01,  1.34165920e-03)
( 4.70000000e+01,  1.31211620e-03)
( 4.80000000e+01,  1.29326836e-03)
( 4.90000000e+01,  1.27079309e-03)
( 5.00000000e+01,  1.23432428e-03)
( 5.10000000e+01,  1.20339053e-03)
( 5.20000000e+01,  1.17384033e-03)
( 5.30000000e+01,  1.15184103e-03)
( 5.40000000e+01,  1.12109237e-03)
( 5.50000000e+01,  1.07382956e-03)
( 5.60000000e+01,  1.03562996e-03)
( 5.70000000e+01,  9.93540880e-04)
( 5.80000000e+01,  9.61683709e-04)
( 5.90000000e+01,  9.40089731e-04)
( 6.00000000e+01,  9.15463300e-04)
( 6.10000000e+01,  8.94902994e-04)
( 6.20000000e+01,  8.70804351e-04)
( 6.30000000e+01,  8.51440547e-04)
( 6.40000000e+01,  8.37733721e-04)
( 6.50000000e+01,  8.17419562e-04)
( 6.60000000e+01,  7.91540961e-04)
( 6.70000000e+01,  7.71625202e-04)
( 6.80000000e+01,  7.55313385e-04)
( 6.90000000e+01,  7.48351669e-04)
( 7.00000000e+01,  7.40160578e-04)
( 7.10000000e+01,  7.32003891e-04)
( 7.20000000e+01,  7.26012824e-04)
( 7.30000000e+01,  7.24111389e-04)
( 7.40000000e+01,  7.21369979e-04)
( 7.50000000e+01,  7.17303624e-04)
( 7.60000000e+01,  7.09270428e-04)
( 7.70000000e+01,  6.95546859e-04)
( 7.80000000e+01,  6.71545529e-04)
( 7.90000000e+01,  6.43982117e-04)
( 8.00000000e+01,  6.02508115e-04)
( 8.10000000e+01,  5.70928158e-04)
( 8.20000000e+01,  5.47156582e-04)
( 8.30000000e+01,  5.29559249e-04)
( 8.40000000e+01,  5.21339340e-04)
( 8.50000000e+01,  5.10247096e-04)
( 8.60000000e+01,  4.97539842e-04)
( 8.70000000e+01,  4.88415456e-04)
( 8.80000000e+01,  4.76657938e-04)
( 8.90000000e+01,  4.69518182e-04)
( 9.00000000e+01,  4.58665310e-04)
( 9.10000000e+01,  4.42933814e-04)
( 9.20000000e+01,  4.32188285e-04)
( 9.30000000e+01,  4.18084505e-04)
( 9.40000000e+01,  4.00391162e-04)
( 9.50000000e+01,  3.88043233e-04)
( 9.60000000e+01,  3.82462183e-04)
( 9.70000000e+01,  3.80545925e-04)
( 9.80000000e+01,  3.78759921e-04)
( 9.90000000e+01,  3.76591445e-04)
( 1.00000000e+02,  3.73665241e-04)};\label{line:cbeam0_nel160x100_rmin2_romset1_majit0:romA1}

\addplot [dash dot, thick, mark options={solid, thin}, mark=square*, mark size=1.5, color=red, label=line:rom12, mark repeat=5]
coordinates {
( 0.00000000e+00,  2.95280466e+00)
( 1.00000000e+00,  2.95280466e+00)
( 2.00000000e+00,  2.95280466e+00)
( 3.00000000e+00,  1.31962992e+00)
( 4.00000000e+00,  9.50371517e-01)
( 5.00000000e+00,  5.48246067e-01)
( 6.00000000e+00,  2.90785071e-01)
( 7.00000000e+00,  1.28240059e-01)
( 8.00000000e+00,  1.07491085e-01)
( 9.00000000e+00,  3.65636390e-02)
( 1.00000000e+01,  2.57966326e-02)
( 1.10000000e+01,  1.74852015e-02)
( 1.20000000e+01,  1.48647013e-02)
( 1.30000000e+01,  1.35011654e-02)
( 1.40000000e+01,  1.14784843e-02)
( 1.50000000e+01,  1.04636295e-02)
( 1.60000000e+01,  9.77445793e-03)
( 1.70000000e+01,  9.30198939e-03)
( 1.80000000e+01,  8.77954624e-03)
( 1.90000000e+01,  8.07421906e-03)
( 2.00000000e+01,  7.65102472e-03)
( 2.10000000e+01,  7.40760564e-03)
( 2.20000000e+01,  7.06198487e-03)
( 2.30000000e+01,  6.80134149e-03)
( 2.40000000e+01,  6.37306236e-03)
( 2.50000000e+01,  6.07484250e-03)
( 2.60000000e+01,  5.92641249e-03)
( 2.70000000e+01,  5.76980380e-03)
( 2.80000000e+01,  5.55178949e-03)
( 2.90000000e+01,  5.36407808e-03)
( 3.00000000e+01,  5.18447417e-03)
( 3.10000000e+01,  4.96801804e-03)
( 3.20000000e+01,  4.79509233e-03)
( 3.30000000e+01,  4.63220139e-03)
( 3.40000000e+01,  4.45001110e-03)
( 3.50000000e+01,  4.25155552e-03)
( 3.60000000e+01,  4.13133674e-03)
( 3.70000000e+01,  4.00803135e-03)
( 3.80000000e+01,  3.82953457e-03)
( 3.90000000e+01,  3.66819729e-03)
( 4.00000000e+01,  3.51824096e-03)
( 4.10000000e+01,  3.38746095e-03)
( 4.20000000e+01,  3.26938026e-03)
( 4.30000000e+01,  3.21772429e-03)
( 4.40000000e+01,  3.10439205e-03)
( 4.50000000e+01,  2.99320902e-03)
( 4.60000000e+01,  2.83897764e-03)
( 4.70000000e+01,  2.72593321e-03)
( 4.80000000e+01,  2.63268898e-03)
( 4.90000000e+01,  2.54361229e-03)
( 5.00000000e+01,  2.43944749e-03)
( 5.10000000e+01,  2.33761560e-03)
( 5.20000000e+01,  2.23914704e-03)
( 5.30000000e+01,  2.16526517e-03)
( 5.40000000e+01,  2.12468012e-03)
( 5.50000000e+01,  2.01979478e-03)
( 5.60000000e+01,  1.95360189e-03)
( 5.70000000e+01,  1.89419688e-03)
( 5.80000000e+01,  1.85052460e-03)
( 5.90000000e+01,  1.81325536e-03)
( 6.00000000e+01,  1.75959371e-03)
( 6.10000000e+01,  1.69593439e-03)
( 6.20000000e+01,  1.65261982e-03)
( 6.30000000e+01,  1.62604058e-03)
( 6.40000000e+01,  1.57503924e-03)
( 6.50000000e+01,  1.53139955e-03)
( 6.60000000e+01,  1.49552592e-03)
( 6.70000000e+01,  1.46490272e-03)
( 6.80000000e+01,  1.42691569e-03)
( 6.90000000e+01,  1.40705296e-03)
( 7.00000000e+01,  1.37293056e-03)
( 7.10000000e+01,  1.32988911e-03)
( 7.20000000e+01,  1.29528653e-03)
( 7.30000000e+01,  1.25790043e-03)
( 7.40000000e+01,  1.22758509e-03)
( 7.50000000e+01,  1.20245712e-03)
( 7.60000000e+01,  1.18275080e-03)
( 7.70000000e+01,  1.15625942e-03)
( 7.80000000e+01,  1.11604659e-03)
( 7.90000000e+01,  1.08674068e-03)
( 8.00000000e+01,  1.06178758e-03)
( 8.10000000e+01,  1.02677867e-03)
( 8.20000000e+01,  9.98523010e-04)
( 8.30000000e+01,  9.57238593e-04)
( 8.40000000e+01,  9.27535031e-04)
( 8.50000000e+01,  9.08112935e-04)
( 8.60000000e+01,  8.96937783e-04)
( 8.70000000e+01,  8.80787718e-04)
( 8.80000000e+01,  8.61477635e-04)
( 8.90000000e+01,  8.48660055e-04)
( 9.00000000e+01,  8.42353696e-04)
( 9.10000000e+01,  8.33953384e-04)
( 9.20000000e+01,  8.19980938e-04)
( 9.30000000e+01,  8.00736023e-04)
( 9.40000000e+01,  7.77136257e-04)
( 9.50000000e+01,  7.61817137e-04)
( 9.60000000e+01,  7.48158229e-04)
( 9.70000000e+01,  7.42859273e-04)
( 9.80000000e+01,  7.37084386e-04)
( 9.90000000e+01,  7.30000314e-04)
( 1.00000000e+02,  7.24559606e-04)};\label{line:cbeam0_nel160x100_rmin2_romset1_majit0:romA2}

\addplot [dashed, thick, mark options={solid, thick}, mark=x, mark size=3, color=blue, label=line:trrom1, mark repeat=5]
coordinates {
( 0.00000000e+00,  2.95280466e+00)
( 1.00000000e+00,  1.95880807e+00)
( 2.00000000e+00,  9.45407764e-01)
( 3.00000000e+00,  6.52039325e-01)
( 4.00000000e+00,  4.98785383e-01)
( 5.00000000e+00,  3.49064561e-01)
( 6.00000000e+00,  2.10956525e-01)
( 7.00000000e+00,  1.10420113e-01)
( 8.00000000e+00,  3.80236547e-02)
( 9.00000000e+00,  2.44963312e-02)
( 1.00000000e+01,  1.94009676e-02)
( 1.10000000e+01,  9.43874310e-03)
( 1.20000000e+01,  9.43874310e-03)
( 1.30000000e+01,  6.81837854e-03)
( 1.40000000e+01,  6.22691532e-03)
( 1.50000000e+01,  5.69700515e-03)
( 1.60000000e+01,  5.08476927e-03)
( 1.70000000e+01,  4.72869002e-03)
( 1.80000000e+01,  4.44700438e-03)
( 1.90000000e+01,  4.24880344e-03)
( 2.00000000e+01,  3.99661048e-03)
( 2.10000000e+01,  3.82097178e-03)
( 2.20000000e+01,  3.59961763e-03)
( 2.30000000e+01,  3.40064871e-03)
( 2.40000000e+01,  3.16716249e-03)
( 2.50000000e+01,  3.01742380e-03)
( 2.60000000e+01,  2.81111332e-03)
( 2.70000000e+01,  2.71322064e-03)
( 2.80000000e+01,  2.57687623e-03)
( 2.90000000e+01,  2.45263495e-03)
( 3.00000000e+01,  2.34305912e-03)
( 3.10000000e+01,  2.28027521e-03)
( 3.20000000e+01,  2.24188138e-03)
( 3.30000000e+01,  2.17416383e-03)
( 3.40000000e+01,  2.09448332e-03)
( 3.50000000e+01,  2.01591172e-03)
( 3.60000000e+01,  1.93950437e-03)
( 3.70000000e+01,  1.87942267e-03)
( 3.80000000e+01,  1.79046484e-03)
( 3.90000000e+01,  1.72072565e-03)
( 4.00000000e+01,  1.65328180e-03)
( 4.10000000e+01,  1.61733076e-03)
( 4.20000000e+01,  1.55555636e-03)
( 4.30000000e+01,  1.50829671e-03)
( 4.40000000e+01,  1.46788370e-03)
( 4.50000000e+01,  1.44743995e-03)
( 4.60000000e+01,  1.40832109e-03)
( 4.70000000e+01,  1.36483208e-03)
( 4.80000000e+01,  1.33175397e-03)
( 4.90000000e+01,  1.29875148e-03)
( 5.00000000e+01,  1.26984031e-03)
( 5.10000000e+01,  1.23748135e-03)
( 5.20000000e+01,  1.21061842e-03)
( 5.30000000e+01,  1.18325079e-03)
( 5.40000000e+01,  1.14186321e-03)
( 5.50000000e+01,  1.10642156e-03)
( 5.60000000e+01,  1.06967836e-03)
( 5.70000000e+01,  1.03895397e-03)
( 5.80000000e+01,  1.00075962e-03)
( 5.90000000e+01,  9.67253624e-04)
( 6.00000000e+01,  9.32939618e-04)
( 6.10000000e+01,  9.06280726e-04)
( 6.20000000e+01,  8.77694615e-04)
( 6.30000000e+01,  8.57351269e-04)
( 6.40000000e+01,  8.35736364e-04)
( 6.50000000e+01,  8.09551521e-04)
( 6.60000000e+01,  7.85386295e-04)
( 6.70000000e+01,  7.64505741e-04)
( 6.80000000e+01,  7.45871089e-04)
( 6.90000000e+01,  7.34674976e-04)
( 7.00000000e+01,  7.26591024e-04)
( 7.10000000e+01,  7.14241198e-04)
( 7.20000000e+01,  6.92011828e-04)
( 7.30000000e+01,  6.65043501e-04)
( 7.40000000e+01,  6.38913036e-04)
( 7.50000000e+01,  5.99616975e-04)
( 7.60000000e+01,  5.66922998e-04)
( 7.70000000e+01,  5.39169958e-04)
( 7.80000000e+01,  5.18808750e-04)
( 7.90000000e+01,  5.06706910e-04)
( 8.00000000e+01,  4.92072204e-04)
( 8.10000000e+01,  4.76610938e-04)
( 8.20000000e+01,  4.62710828e-04)
( 8.30000000e+01,  4.41793895e-04)
( 8.40000000e+01,  4.17127032e-04)
( 8.50000000e+01,  4.03041046e-04)
( 8.60000000e+01,  3.88707541e-04)
( 8.70000000e+01,  3.80263373e-04)
( 8.80000000e+01,  3.75285650e-04)
( 8.90000000e+01,  3.68391307e-04)
( 9.00000000e+01,  3.60660181e-04)
( 9.10000000e+01,  3.52733852e-04)
( 9.20000000e+01,  3.45255648e-04)
( 9.30000000e+01,  3.38773457e-04)
( 9.40000000e+01,  3.33199741e-04)
( 9.50000000e+01,  3.27385312e-04)
( 9.60000000e+01,  3.21851894e-04)
( 9.70000000e+01,  3.16615694e-04)
( 9.80000000e+01,  3.09959357e-04)
( 9.90000000e+01,  3.04473877e-04)
( 1.00000000e+02,  2.98399726e-04)};\label{line:cbeam0_nel160x100_rmin2_romset1_majit0:romB1}

\addplot [dashed, thick, mark options={solid, thick}, mark=x, mark size=3, color=red, label=line:trrom2, mark repeat=5]
coordinates {
( 0.00000000e+00,  2.95280466e+00)
( 1.00000000e+00,  2.95280466e+00)
( 2.00000000e+00,  1.59640649e+00)
( 3.00000000e+00,  1.59640649e+00)
( 4.00000000e+00,  9.64620107e-01)
( 5.00000000e+00,  7.50094259e-01)
( 6.00000000e+00,  3.90379091e-01)
( 7.00000000e+00,  1.67495810e-01)
( 8.00000000e+00,  6.76836432e-02)
( 9.00000000e+00,  2.66707081e-02)
( 1.00000000e+01,  2.66707081e-02)
( 1.10000000e+01,  1.21668183e-02)
( 1.20000000e+01,  9.05325379e-03)
( 1.30000000e+01,  9.05325379e-03)
( 1.40000000e+01,  5.92866084e-03)
( 1.50000000e+01,  5.92866084e-03)
( 1.60000000e+01,  4.06441263e-03)
( 1.70000000e+01,  4.06441263e-03)
( 1.80000000e+01,  3.69398994e-03)
( 1.90000000e+01,  3.55588798e-03)
( 2.00000000e+01,  3.33066205e-03)
( 2.10000000e+01,  3.17866293e-03)
( 2.20000000e+01,  3.05422701e-03)
( 2.30000000e+01,  2.83138631e-03)
( 2.40000000e+01,  2.68299356e-03)
( 2.50000000e+01,  2.54345156e-03)
( 2.60000000e+01,  2.46962785e-03)
( 2.70000000e+01,  2.36465567e-03)
( 2.80000000e+01,  2.28720653e-03)
( 2.90000000e+01,  2.24077830e-03)
( 3.00000000e+01,  2.19403568e-03)
( 3.10000000e+01,  2.12323465e-03)
( 3.20000000e+01,  2.04013156e-03)
( 3.30000000e+01,  1.95937720e-03)
( 3.40000000e+01,  1.89575709e-03)
( 3.50000000e+01,  1.83456006e-03)
( 3.60000000e+01,  1.77487048e-03)
( 3.70000000e+01,  1.71939548e-03)
( 3.80000000e+01,  1.66934941e-03)
( 3.90000000e+01,  1.63475062e-03)
( 4.00000000e+01,  1.59633698e-03)
( 4.10000000e+01,  1.55929843e-03)
( 4.20000000e+01,  1.50617790e-03)
( 4.30000000e+01,  1.47577418e-03)
( 4.40000000e+01,  1.45111271e-03)
( 4.50000000e+01,  1.41484737e-03)
( 4.60000000e+01,  1.36251899e-03)
( 4.70000000e+01,  1.30690732e-03)
( 4.80000000e+01,  1.26481099e-03)
( 4.90000000e+01,  1.22810738e-03)
( 5.00000000e+01,  1.18791808e-03)
( 5.10000000e+01,  1.13418641e-03)
( 5.20000000e+01,  1.07909376e-03)
( 5.30000000e+01,  1.03126093e-03)
( 5.40000000e+01,  9.86160539e-04)
( 5.50000000e+01,  9.57686709e-04)
( 5.60000000e+01,  9.35867989e-04)
( 5.70000000e+01,  8.98439621e-04)
( 5.80000000e+01,  8.72557372e-04)
( 5.90000000e+01,  8.56418621e-04)
( 6.00000000e+01,  8.36473894e-04)
( 6.10000000e+01,  8.15322320e-04)
( 6.20000000e+01,  7.97563021e-04)
( 6.30000000e+01,  7.79158115e-04)
( 6.40000000e+01,  7.58883304e-04)
( 6.50000000e+01,  7.41905002e-04)
( 6.60000000e+01,  7.33849523e-04)
( 6.70000000e+01,  7.28085424e-04)
( 6.80000000e+01,  7.20342763e-04)
( 6.90000000e+01,  7.04892146e-04)
( 7.00000000e+01,  6.81822089e-04)
( 7.10000000e+01,  6.53882833e-04)
( 7.20000000e+01,  6.21809840e-04)
( 7.30000000e+01,  5.89933706e-04)
( 7.40000000e+01,  5.63188070e-04)
( 7.50000000e+01,  5.37731741e-04)
( 7.60000000e+01,  5.20123411e-04)
( 7.70000000e+01,  5.04432345e-04)
( 7.80000000e+01,  4.90205623e-04)
( 7.90000000e+01,  4.78531322e-04)
( 8.00000000e+01,  4.65093627e-04)
( 8.10000000e+01,  4.50060179e-04)
( 8.20000000e+01,  4.31602341e-04)
( 8.30000000e+01,  4.07522008e-04)
( 8.40000000e+01,  3.95234165e-04)
( 8.50000000e+01,  3.82810541e-04)
( 8.60000000e+01,  3.76761749e-04)
( 8.70000000e+01,  3.72561474e-04)
( 8.80000000e+01,  3.63967232e-04)
( 8.90000000e+01,  3.51474635e-04)
( 9.00000000e+01,  3.42598107e-04)
( 9.10000000e+01,  3.33315319e-04)
( 9.20000000e+01,  3.26091537e-04)
( 9.30000000e+01,  3.19444045e-04)
( 9.40000000e+01,  3.10072914e-04)
( 9.50000000e+01,  3.00863827e-04)
( 9.60000000e+01,  2.98116692e-04)
( 9.70000000e+01,  2.89497475e-04)
( 9.80000000e+01,  2.81108742e-04)
( 9.90000000e+01,  2.73701330e-04)
( 1.00000000e+02,  2.63088247e-04)};\label{line:cbeam0_nel160x100_rmin2_romset1_majit0:romB2}

\nextgroupplot[width=0.8\textwidth, grid=major, xmax=50, ylabel={Cum. No. ROM solves}, xmin=0, ymode=log, height=0.4\textwidth]
\addplot [dash dot, thick, mark options={solid, thin}, mark=square*, mark size=1.5, color=blue, label=line:rom11, mark repeat=5, forget plot]
coordinates {
( 1.00000000e+00,  1.00000000e+00)
( 2.00000000e+00,  2.00000000e+00)
( 3.00000000e+00,  4.00000000e+00)
( 4.00000000e+00,  8.00000000e+00)
( 5.00000000e+00,  1.00000000e+01)
( 6.00000000e+00,  1.40000000e+01)
( 7.00000000e+00,  1.80000000e+01)
( 8.00000000e+00,  2.40000000e+01)
( 9.00000000e+00,  3.40000000e+01)
( 1.00000000e+01,  5.20000000e+01)
( 1.10000000e+01,  6.20000000e+01)
( 1.20000000e+01,  8.70000000e+01)
( 1.30000000e+01,  1.04000000e+02)
( 1.40000000e+01,  1.32000000e+02)
( 1.50000000e+01,  1.48000000e+02)
( 1.60000000e+01,  1.60000000e+02)
( 1.70000000e+01,  1.77000000e+02)
( 1.80000000e+01,  1.95000000e+02)
( 1.90000000e+01,  2.13000000e+02)
( 2.00000000e+01,  2.32000000e+02)
( 2.10000000e+01,  2.50000000e+02)
( 2.20000000e+01,  2.68000000e+02)
( 2.30000000e+01,  2.89000000e+02)
( 2.40000000e+01,  3.11000000e+02)
( 2.50000000e+01,  3.34000000e+02)
( 2.60000000e+01,  3.55000000e+02)
( 2.70000000e+01,  3.77000000e+02)
( 2.80000000e+01,  4.00000000e+02)
( 2.90000000e+01,  4.23000000e+02)
( 3.00000000e+01,  4.44000000e+02)
( 3.10000000e+01,  4.60000000e+02)
( 3.20000000e+01,  4.78000000e+02)
( 3.30000000e+01,  4.96000000e+02)
( 3.40000000e+01,  5.14000000e+02)
( 3.50000000e+01,  5.31000000e+02)
( 3.60000000e+01,  5.48000000e+02)
( 3.70000000e+01,  5.67000000e+02)
( 3.80000000e+01,  5.85000000e+02)
( 3.90000000e+01,  6.02000000e+02)
( 4.00000000e+01,  6.17000000e+02)
( 4.10000000e+01,  6.35000000e+02)
( 4.20000000e+01,  6.56000000e+02)
( 4.30000000e+01,  6.77000000e+02)
( 4.40000000e+01,  6.96000000e+02)
( 4.50000000e+01,  7.15000000e+02)
( 4.60000000e+01,  7.33000000e+02)
( 4.70000000e+01,  7.50000000e+02)
( 4.80000000e+01,  7.65000000e+02)
( 4.90000000e+01,  7.83000000e+02)
( 5.00000000e+01,  8.04000000e+02)
( 5.10000000e+01,  8.23000000e+02)
( 5.20000000e+01,  8.41000000e+02)
( 5.30000000e+01,  8.57000000e+02)
( 5.40000000e+01,  8.75000000e+02)
( 5.50000000e+01,  8.96000000e+02)
( 5.60000000e+01,  9.15000000e+02)
( 5.70000000e+01,  9.34000000e+02)
( 5.80000000e+01,  9.52000000e+02)
( 5.90000000e+01,  9.69000000e+02)
( 6.00000000e+01,  9.89000000e+02)
( 6.10000000e+01,  1.00700000e+03)
( 6.20000000e+01,  1.02800000e+03)
( 6.30000000e+01,  1.04700000e+03)
( 6.40000000e+01,  1.06400000e+03)
( 6.50000000e+01,  1.08300000e+03)
( 6.60000000e+01,  1.10400000e+03)
( 6.70000000e+01,  1.12300000e+03)
( 6.80000000e+01,  1.14100000e+03)
( 6.90000000e+01,  1.15600000e+03)
( 7.00000000e+01,  1.17400000e+03)
( 7.10000000e+01,  1.19600000e+03)
( 7.20000000e+01,  1.21300000e+03)
( 7.30000000e+01,  1.22500000e+03)
( 7.40000000e+01,  1.24000000e+03)
( 7.50000000e+01,  1.25600000e+03)
( 7.60000000e+01,  1.27400000e+03)
( 7.70000000e+01,  1.29300000e+03)
( 7.80000000e+01,  1.31300000e+03)
( 7.90000000e+01,  1.33400000e+03)
( 8.00000000e+01,  1.35300000e+03)
( 8.10000000e+01,  1.37200000e+03)
( 8.20000000e+01,  1.39100000e+03)
( 8.30000000e+01,  1.40900000e+03)
( 8.40000000e+01,  1.42500000e+03)
( 8.50000000e+01,  1.44400000e+03)
( 8.60000000e+01,  1.46500000e+03)
( 8.70000000e+01,  1.48200000e+03)
( 8.80000000e+01,  1.50200000e+03)
( 8.90000000e+01,  1.51800000e+03)
( 9.00000000e+01,  1.53700000e+03)
( 9.10000000e+01,  1.55700000e+03)
( 9.20000000e+01,  1.57400000e+03)
( 9.30000000e+01,  1.59300000e+03)
( 9.40000000e+01,  1.61500000e+03)
( 9.50000000e+01,  1.63400000e+03)
( 9.60000000e+01,  1.65100000e+03)
( 9.70000000e+01,  1.66300000e+03)
( 9.80000000e+01,  1.67900000e+03)
( 9.90000000e+01,  1.69800000e+03)
( 1.00000000e+02,  1.71900000e+03)};

\addplot [dash dot, thick, mark options={solid, thin}, mark=square*, mark size=1.5, color=red, label=line:rom12, mark repeat=5, forget plot]
coordinates {
( 1.00000000e+00,  4.90000000e+01)
( 2.00000000e+00,  5.30000000e+01)
( 3.00000000e+00,  5.50000000e+01)
( 4.00000000e+00,  5.80000000e+01)
( 5.00000000e+00,  6.20000000e+01)
( 6.00000000e+00,  6.60000000e+01)
( 7.00000000e+00,  7.00000000e+01)
( 8.00000000e+00,  7.70000000e+01)
( 9.00000000e+00,  8.20000000e+01)
( 1.00000000e+01,  9.00000000e+01)
( 1.10000000e+01,  1.01000000e+02)
( 1.20000000e+01,  1.16000000e+02)
( 1.30000000e+01,  1.32000000e+02)
( 1.40000000e+01,  1.42000000e+02)
( 1.50000000e+01,  1.56000000e+02)
( 1.60000000e+01,  1.73000000e+02)
( 1.70000000e+01,  1.90000000e+02)
( 1.80000000e+01,  2.14000000e+02)
( 1.90000000e+01,  2.36000000e+02)
( 2.00000000e+01,  2.58000000e+02)
( 2.10000000e+01,  2.83000000e+02)
( 2.20000000e+01,  3.05000000e+02)
( 2.30000000e+01,  3.26000000e+02)
( 2.40000000e+01,  3.41000000e+02)
( 2.50000000e+01,  3.59000000e+02)
( 2.60000000e+01,  3.77000000e+02)
( 2.70000000e+01,  3.95000000e+02)
( 2.80000000e+01,  4.12000000e+02)
( 2.90000000e+01,  4.30000000e+02)
( 3.00000000e+01,  4.49000000e+02)
( 3.10000000e+01,  4.67000000e+02)
( 3.20000000e+01,  4.85000000e+02)
( 3.30000000e+01,  5.05000000e+02)
( 3.40000000e+01,  5.24000000e+02)
( 3.50000000e+01,  5.43000000e+02)
( 3.60000000e+01,  5.63000000e+02)
( 3.70000000e+01,  5.82000000e+02)
( 3.80000000e+01,  6.00000000e+02)
( 3.90000000e+01,  6.20000000e+02)
( 4.00000000e+01,  6.40000000e+02)
( 4.10000000e+01,  6.60000000e+02)
( 4.20000000e+01,  6.79000000e+02)
( 4.30000000e+01,  6.99000000e+02)
( 4.40000000e+01,  7.13000000e+02)
( 4.50000000e+01,  7.31000000e+02)
( 4.60000000e+01,  7.53000000e+02)
( 4.70000000e+01,  7.74000000e+02)
( 4.80000000e+01,  7.95000000e+02)
( 4.90000000e+01,  8.15000000e+02)
( 5.00000000e+01,  8.35000000e+02)
( 5.10000000e+01,  8.56000000e+02)
( 5.20000000e+01,  8.79000000e+02)
( 5.30000000e+01,  9.01000000e+02)
( 5.40000000e+01,  9.21000000e+02)
( 5.50000000e+01,  9.37000000e+02)
( 5.60000000e+01,  9.55000000e+02)
( 5.70000000e+01,  9.73000000e+02)
( 5.80000000e+01,  9.90000000e+02)
( 5.90000000e+01,  1.00600000e+03)
( 6.00000000e+01,  1.02400000e+03)
( 6.10000000e+01,  1.04400000e+03)
( 6.20000000e+01,  1.06400000e+03)
( 6.30000000e+01,  1.08200000e+03)
( 6.40000000e+01,  1.09800000e+03)
( 6.50000000e+01,  1.11700000e+03)
( 6.60000000e+01,  1.13900000e+03)
( 6.70000000e+01,  1.15900000e+03)
( 6.80000000e+01,  1.17600000e+03)
( 6.90000000e+01,  1.19200000e+03)
( 7.00000000e+01,  1.21100000e+03)
( 7.10000000e+01,  1.23100000e+03)
( 7.20000000e+01,  1.25000000e+03)
( 7.30000000e+01,  1.26900000e+03)
( 7.40000000e+01,  1.28800000e+03)
( 7.50000000e+01,  1.30600000e+03)
( 7.60000000e+01,  1.32200000e+03)
( 7.70000000e+01,  1.34000000e+03)
( 7.80000000e+01,  1.36000000e+03)
( 7.90000000e+01,  1.37800000e+03)
( 8.00000000e+01,  1.39400000e+03)
( 8.10000000e+01,  1.41300000e+03)
( 8.20000000e+01,  1.43000000e+03)
( 8.30000000e+01,  1.45000000e+03)
( 8.40000000e+01,  1.47000000e+03)
( 8.50000000e+01,  1.48800000e+03)
( 8.60000000e+01,  1.50300000e+03)
( 8.70000000e+01,  1.52100000e+03)
( 8.80000000e+01,  1.54100000e+03)
( 8.90000000e+01,  1.55800000e+03)
( 9.00000000e+01,  1.57100000e+03)
( 9.10000000e+01,  1.58600000e+03)
( 9.20000000e+01,  1.60300000e+03)
( 9.30000000e+01,  1.62200000e+03)
( 9.40000000e+01,  1.64300000e+03)
( 9.50000000e+01,  1.66200000e+03)
( 9.60000000e+01,  1.67900000e+03)
( 9.70000000e+01,  1.69300000e+03)
( 9.80000000e+01,  1.71000000e+03)
( 9.90000000e+01,  1.73000000e+03)
( 1.00000000e+02,  1.75200000e+03)};

\addplot [dashed, thick, mark options={solid, thick}, mark=x, mark size=3, color=blue, label=line:trrom1, mark repeat=5, forget plot]
coordinates {
( 1.00000000e+00,  1.00000000e+00)
( 2.00000000e+00,  3.00000000e+00)
( 3.00000000e+00,  5.00000000e+00)
( 4.00000000e+00,  7.00000000e+00)
( 5.00000000e+00,  1.00000000e+01)
( 6.00000000e+00,  1.50000000e+01)
( 7.00000000e+00,  1.80000000e+01)
( 8.00000000e+00,  2.40000000e+01)
( 9.00000000e+00,  3.60000000e+01)
( 1.00000000e+01,  5.10000000e+01)
( 1.10000000e+01,  5.90000000e+01)
( 1.20000000e+01,  8.00000000e+01)
( 1.30000000e+01,  1.01000000e+02)
( 1.40000000e+01,  1.24000000e+02)
( 1.50000000e+01,  1.47000000e+02)
( 1.60000000e+01,  1.60000000e+02)
( 1.70000000e+01,  1.79000000e+02)
( 1.80000000e+01,  2.01000000e+02)
( 1.90000000e+01,  2.21000000e+02)
( 2.00000000e+01,  2.43000000e+02)
( 2.10000000e+01,  2.68000000e+02)
( 2.20000000e+01,  2.90000000e+02)
( 2.30000000e+01,  3.15000000e+02)
( 2.40000000e+01,  3.40000000e+02)
( 2.50000000e+01,  3.62000000e+02)
( 2.60000000e+01,  3.87000000e+02)
( 2.70000000e+01,  4.10000000e+02)
( 2.80000000e+01,  4.26000000e+02)
( 2.90000000e+01,  4.47000000e+02)
( 3.00000000e+01,  4.68000000e+02)
( 3.10000000e+01,  4.88000000e+02)
( 3.20000000e+01,  5.06000000e+02)
( 3.30000000e+01,  5.26000000e+02)
( 3.40000000e+01,  5.46000000e+02)
( 3.50000000e+01,  5.65000000e+02)
( 3.60000000e+01,  5.84000000e+02)
( 3.70000000e+01,  6.03000000e+02)
( 3.80000000e+01,  6.23000000e+02)
( 3.90000000e+01,  6.43000000e+02)
( 4.00000000e+01,  6.64000000e+02)
( 4.10000000e+01,  6.84000000e+02)
( 4.20000000e+01,  7.04000000e+02)
( 4.30000000e+01,  7.24000000e+02)
( 4.40000000e+01,  7.43000000e+02)
( 4.50000000e+01,  7.64000000e+02)
( 4.60000000e+01,  7.81000000e+02)
( 4.70000000e+01,  8.00000000e+02)
( 4.80000000e+01,  8.17000000e+02)
( 4.90000000e+01,  8.35000000e+02)
( 5.00000000e+01,  8.54000000e+02)
( 5.10000000e+01,  8.74000000e+02)
( 5.20000000e+01,  8.92000000e+02)
( 5.30000000e+01,  9.10000000e+02)
( 5.40000000e+01,  9.30000000e+02)
( 5.50000000e+01,  9.48000000e+02)
( 5.60000000e+01,  9.67000000e+02)
( 5.70000000e+01,  9.85000000e+02)
( 5.80000000e+01,  1.00400000e+03)
( 5.90000000e+01,  1.02200000e+03)
( 6.00000000e+01,  1.04200000e+03)
( 6.10000000e+01,  1.06200000e+03)
( 6.20000000e+01,  1.08200000e+03)
( 6.30000000e+01,  1.10200000e+03)
( 6.40000000e+01,  1.12100000e+03)
( 6.50000000e+01,  1.14100000e+03)
( 6.60000000e+01,  1.16100000e+03)
( 6.70000000e+01,  1.18100000e+03)
( 6.80000000e+01,  1.20100000e+03)
( 6.90000000e+01,  1.22200000e+03)
( 7.00000000e+01,  1.24300000e+03)
( 7.10000000e+01,  1.26400000e+03)
( 7.20000000e+01,  1.28400000e+03)
( 7.30000000e+01,  1.30400000e+03)
( 7.40000000e+01,  1.32200000e+03)
( 7.50000000e+01,  1.34100000e+03)
( 7.60000000e+01,  1.36000000e+03)
( 7.70000000e+01,  1.38100000e+03)
( 7.80000000e+01,  1.40200000e+03)
( 7.90000000e+01,  1.42200000e+03)
( 8.00000000e+01,  1.44400000e+03)
( 8.10000000e+01,  1.46400000e+03)
( 8.20000000e+01,  1.48400000e+03)
( 8.30000000e+01,  1.50400000e+03)
( 8.40000000e+01,  1.52500000e+03)
( 8.50000000e+01,  1.54400000e+03)
( 8.60000000e+01,  1.56500000e+03)
( 8.70000000e+01,  1.58800000e+03)
( 8.80000000e+01,  1.61100000e+03)
( 8.90000000e+01,  1.63100000e+03)
( 9.00000000e+01,  1.65200000e+03)
( 9.10000000e+01,  1.67200000e+03)
( 9.20000000e+01,  1.69100000e+03)
( 9.30000000e+01,  1.71000000e+03)
( 9.40000000e+01,  1.72800000e+03)
( 9.50000000e+01,  1.74700000e+03)
( 9.60000000e+01,  1.76700000e+03)
( 9.70000000e+01,  1.78500000e+03)
( 9.80000000e+01,  1.80500000e+03)
( 9.90000000e+01,  1.82300000e+03)
( 1.00000000e+02,  1.84200000e+03)};

\addplot [dashed, thick, mark options={solid, thick}, mark=x, mark size=3, color=red, label=line:trrom2, mark repeat=5, forget plot]
coordinates {
( 1.00000000e+00,  4.90000000e+01)
( 2.00000000e+00,  5.10000000e+01)
( 3.00000000e+00,  5.40000000e+01)
( 4.00000000e+00,  5.50000000e+01)
( 5.00000000e+00,  5.90000000e+01)
( 6.00000000e+00,  6.10000000e+01)
( 7.00000000e+00,  6.60000000e+01)
( 8.00000000e+00,  7.30000000e+01)
( 9.00000000e+00,  7.80000000e+01)
( 1.00000000e+01,  9.50000000e+01)
( 1.10000000e+01,  1.10000000e+02)
( 1.20000000e+01,  1.23000000e+02)
( 1.30000000e+01,  1.43000000e+02)
( 1.40000000e+01,  1.64000000e+02)
( 1.50000000e+01,  1.97000000e+02)
( 1.60000000e+01,  2.15000000e+02)
( 1.70000000e+01,  2.44000000e+02)
( 1.80000000e+01,  2.69000000e+02)
( 1.90000000e+01,  2.91000000e+02)
( 2.00000000e+01,  3.12000000e+02)
( 2.10000000e+01,  3.34000000e+02)
( 2.20000000e+01,  3.54000000e+02)
( 2.30000000e+01,  3.78000000e+02)
( 2.40000000e+01,  4.00000000e+02)
( 2.50000000e+01,  4.22000000e+02)
( 2.60000000e+01,  4.45000000e+02)
( 2.70000000e+01,  4.60000000e+02)
( 2.80000000e+01,  4.79000000e+02)
( 2.90000000e+01,  4.99000000e+02)
( 3.00000000e+01,  5.17000000e+02)
( 3.10000000e+01,  5.36000000e+02)
( 3.20000000e+01,  5.55000000e+02)
( 3.30000000e+01,  5.74000000e+02)
( 3.40000000e+01,  5.92000000e+02)
( 3.50000000e+01,  6.11000000e+02)
( 3.60000000e+01,  6.29000000e+02)
( 3.70000000e+01,  6.47000000e+02)
( 3.80000000e+01,  6.66000000e+02)
( 3.90000000e+01,  6.84000000e+02)
( 4.00000000e+01,  7.02000000e+02)
( 4.10000000e+01,  7.21000000e+02)
( 4.20000000e+01,  7.39000000e+02)
( 4.30000000e+01,  7.59000000e+02)
( 4.40000000e+01,  7.78000000e+02)
( 4.50000000e+01,  7.99000000e+02)
( 4.60000000e+01,  8.19000000e+02)
( 4.70000000e+01,  8.39000000e+02)
( 4.80000000e+01,  8.60000000e+02)
( 4.90000000e+01,  8.81000000e+02)
( 5.00000000e+01,  9.01000000e+02)
( 5.10000000e+01,  9.22000000e+02)
( 5.20000000e+01,  9.43000000e+02)
( 5.30000000e+01,  9.63000000e+02)
( 5.40000000e+01,  9.83000000e+02)
( 5.50000000e+01,  1.00200000e+03)
( 5.60000000e+01,  1.02200000e+03)
( 5.70000000e+01,  1.03800000e+03)
( 5.80000000e+01,  1.05800000e+03)
( 5.90000000e+01,  1.07600000e+03)
( 6.00000000e+01,  1.09500000e+03)
( 6.10000000e+01,  1.11400000e+03)
( 6.20000000e+01,  1.13200000e+03)
( 6.30000000e+01,  1.15000000e+03)
( 6.40000000e+01,  1.17000000e+03)
( 6.50000000e+01,  1.19100000e+03)
( 6.60000000e+01,  1.21000000e+03)
( 6.70000000e+01,  1.23000000e+03)
( 6.80000000e+01,  1.24900000e+03)
( 6.90000000e+01,  1.26800000e+03)
( 7.00000000e+01,  1.28800000e+03)
( 7.10000000e+01,  1.30700000e+03)
( 7.20000000e+01,  1.32600000e+03)
( 7.30000000e+01,  1.34400000e+03)
( 7.40000000e+01,  1.36300000e+03)
( 7.50000000e+01,  1.38400000e+03)
( 7.60000000e+01,  1.40400000e+03)
( 7.70000000e+01,  1.42500000e+03)
( 7.80000000e+01,  1.44600000e+03)
( 7.90000000e+01,  1.46600000e+03)
( 8.00000000e+01,  1.48500000e+03)
( 8.10000000e+01,  1.50600000e+03)
( 8.20000000e+01,  1.52400000e+03)
( 8.30000000e+01,  1.54500000e+03)
( 8.40000000e+01,  1.56600000e+03)
( 8.50000000e+01,  1.58800000e+03)
( 8.60000000e+01,  1.61000000e+03)
( 8.70000000e+01,  1.63400000e+03)
( 8.80000000e+01,  1.65600000e+03)
( 8.90000000e+01,  1.67800000e+03)
( 9.00000000e+01,  1.69800000e+03)
( 9.10000000e+01,  1.71900000e+03)
( 9.20000000e+01,  1.74100000e+03)
( 9.30000000e+01,  1.76100000e+03)
( 9.40000000e+01,  1.78300000e+03)
( 9.50000000e+01,  1.80500000e+03)
( 9.60000000e+01,  1.82500000e+03)
( 9.70000000e+01,  1.84000000e+03)
( 9.80000000e+01,  1.86000000e+03)
( 9.90000000e+01,  1.87800000e+03)
( 1.00000000e+02,  1.89700000e+03)};

\end{groupplot}\end{tikzpicture}

%% file: py/cbeam0_nel160x100_rmin2_majit0_matinterp1.dat
HDM-MMA & $0.100$ & - & $424.78$ & $17$ & - & $17.0$ \\
 & $0.050$ & - & $408.58$ & $20$ & - & $20.0$ \\
 & $0.010$ & - & $398.27$ & $27$ & - & $27.0$ \\
 & $0.001$ & - & $395.10$ & $468$ & - & $468.0$ \\\hline
ROM-TR-RES & $0.100$ & $0.10$ & $404.38$ & $9$ & 24 & $9.2$ \\
 & $0.050$ & $0.10$ & $404.38$ & $9$ & 24 & $9.2$ \\
 & $0.010$ & $0.10$ & $398.44$ & $12$ & 59 & $12.6$ \\
 & $0.001$ & $0.10$ & $395.08$ & $60$ & 1022 & $70.2$ \\\hline
ROM-TR-DIST & $0.100$ & $0.10$ & $404.78$ & $8$ & 18 & $8.2$ \\
 & $0.050$ & $0.10$ & $398.92$ & $9$ & 24 & $9.2$ \\
 & $0.010$ & $0.10$ & $396.77$ & $12$ & 62 & $12.6$ \\
 & $0.001$ & $0.10$ & $395.09$ & $58$ & 934 & $67.3$ \\

%% file: py/ssbeam0_geom.tikz
\begin{tikzpicture}
\begin{axis}[
axis equal image,
axis x line*=bottom,
axis y line*=left,
width=0.8\textwidth,
xtick={0, 15, 30},
ytick={0, 15},
grid=both,
ymax=17,
xmax=32,
xmin=-2,
ymin=-5]
\addplot [black, solid, opacity=0.6, fill=lightgray, forget plot]
coordinates {
( 0.00000000e+00,  0.00000000e+00)
( 3.00000000e+01,  0.00000000e+00)
( 3.00000000e+01,  1.50000000e+01)
( 0.00000000e+00,  1.50000000e+01)
( 0.00000000e+00,  0.00000000e+00)};

\draw [-{Latex[width=2mm,length=4mm]}, thick] plot [] coordinates {(axis cs:15.0, 0.0) (axis cs:15.0, -3.0)
};

\node[below]    at    (axis cs:15, -3) {$F_\text{in}$};
\addplot [black, solid, forget plot]
coordinates {
( 0.00000000e+00,  0.00000000e+00)
(-6.92820323e-01, -8.00000000e-01)
( 6.92820323e-01, -8.00000000e-01)
( 0.00000000e+00,  0.00000000e+00)};

\addplot [black, solid, forget plot]
coordinates {
(-6.92820323e-01, -8.00000000e-01)
( 6.92820323e-01, -8.00000000e-01)
( 6.92820323e-01, -9.60000000e-01)
(-6.92820323e-01, -9.60000000e-01)
(-6.92820323e-01, -8.00000000e-01)};

\addplot [black, solid, fill=black, forget plot]
coordinates {
(-6.92820323e-01, -8.00000000e-01)
( 6.92820323e-01, -8.00000000e-01)
( 6.92820323e-01, -9.60000000e-01)
(-6.92820323e-01, -9.60000000e-01)
(-6.92820323e-01, -8.00000000e-01)};

\addplot [black, solid, forget plot]
coordinates {
( 3.00000000e+01,  0.00000000e+00)
( 2.93071797e+01, -8.00000000e-01)
( 3.06928203e+01, -8.00000000e-01)
( 3.00000000e+01,  0.00000000e+00)};

\addplot [black, solid, forget plot]
coordinates {
( 3.02400000e+01, -1.04000000e+00)
( 3.02380296e+01, -1.00930948e+00)
( 3.02321508e+01, -9.79122900e-01)
( 3.02224600e+01, -9.49935919e-01)
( 3.02091165e+01, -9.22227788e-01)
( 3.01923393e+01, -8.96453473e-01)
( 3.01724038e+01, -8.73036188e-01)
( 3.01496376e+01, -8.52360444e-01)
( 3.01244142e+01, -8.34765737e-01)
( 3.00971480e+01, -8.20540970e-01)
( 3.00682866e+01, -8.09919715e-01)
( 3.00383040e+01, -8.03076372e-01)
( 3.00076924e+01, -8.00123308e-01)
( 2.99769545e+01, -8.01109013e-01)
( 2.99465950e+01, -8.06017301e-01)
( 2.99171124e+01, -8.14767579e-01)
( 2.98889908e+01, -8.27216166e-01)
( 2.98626920e+01, -8.43158659e-01)
( 2.98386478e+01, -8.62333281e-01)
( 2.98172530e+01, -8.84425185e-01)
( 2.97988589e+01, -9.09071624e-01)
( 2.97837675e+01, -9.35867903e-01)
( 2.97722266e+01, -9.64374028e-01)
( 2.97644258e+01, -9.94121929e-01)
( 2.97604931e+01, -1.02462315e+00)
( 2.97604931e+01, -1.05537685e+00)
( 2.97644258e+01, -1.08587807e+00)
( 2.97722266e+01, -1.11562597e+00)
( 2.97837675e+01, -1.14413210e+00)
( 2.97988589e+01, -1.17092838e+00)
( 2.98172530e+01, -1.19557481e+00)
( 2.98386478e+01, -1.21766672e+00)
( 2.98626920e+01, -1.23684134e+00)
( 2.98889908e+01, -1.25278383e+00)
( 2.99171124e+01, -1.26523242e+00)
( 2.99465950e+01, -1.27398270e+00)
( 2.99769545e+01, -1.27889099e+00)
( 3.00076924e+01, -1.27987669e+00)
( 3.00383040e+01, -1.27692363e+00)
( 3.00682866e+01, -1.27008028e+00)
( 3.00971480e+01, -1.25945903e+00)
( 3.01244142e+01, -1.24523426e+00)
( 3.01496376e+01, -1.22763956e+00)
( 3.01724038e+01, -1.20696381e+00)
( 3.01923393e+01, -1.18354653e+00)
( 3.02091165e+01, -1.15777221e+00)
( 3.02224600e+01, -1.13006408e+00)
( 3.02321508e+01, -1.10087710e+00)
( 3.02380296e+01, -1.07069052e+00)
( 3.02400000e+01, -1.04000000e+00)};

\addplot [black, solid, fill=black, forget plot]
coordinates {
( 2.93071797e+01, -1.28000000e+00)
( 3.06928203e+01, -1.28000000e+00)
( 3.06928203e+01, -1.44000000e+00)
( 2.93071797e+01, -1.44000000e+00)
( 2.93071797e+01, -1.28000000e+00)};

\addplot []
graphics [xmin=0,xmax=30,ymin=0,ymax=15] { 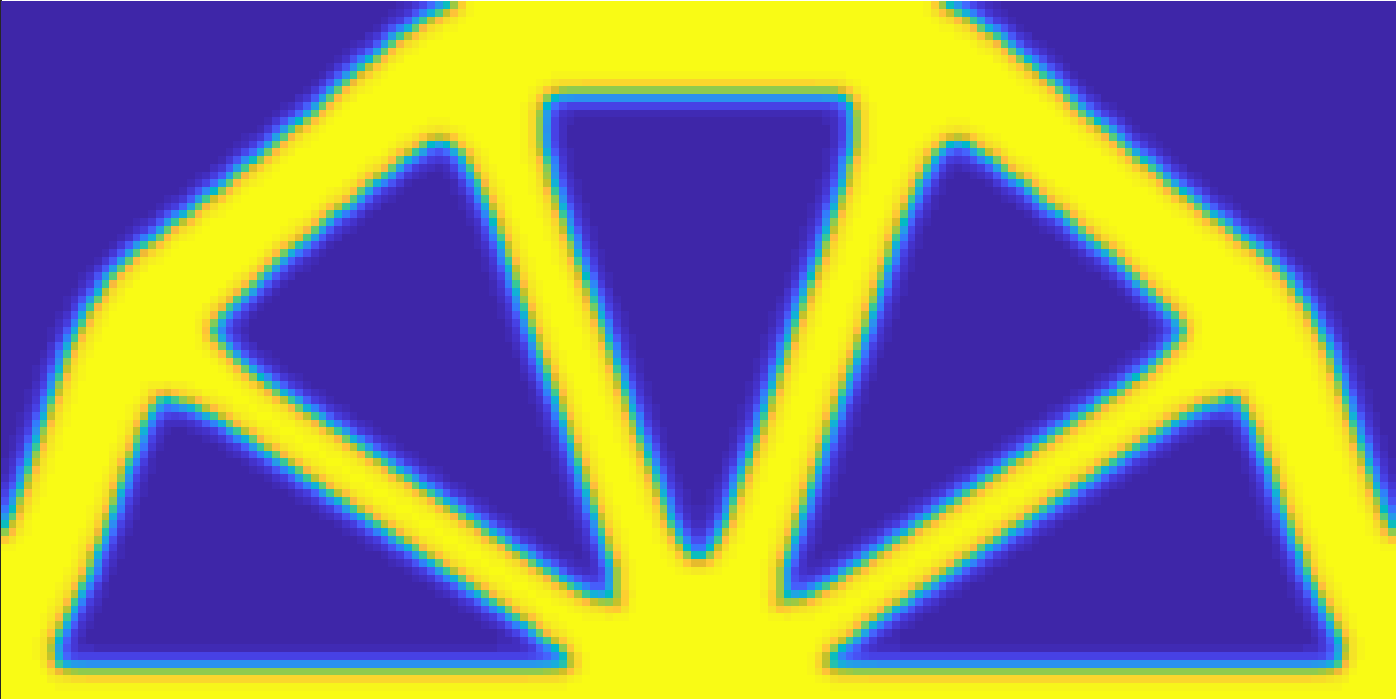};

\end{axis}
\end{tikzpicture}

%% file: py/ssbeam0_nel180x90_rmin0p5_romset1_majit0_matinterp1.tikz
\begin{tikzpicture}
\begin{groupplot} [
group style={group size = 1 by 3, horizontal sep = 2.0cm, vertical sep = 1.0cm}]
\nextgroupplot[width=0.8\textwidth, grid=major, ymax=7.131682862360981, xmax=50, ylabel={$\displaystyle{\frac{|J(\psibold^{(k)})-J^*|}{|J^*|}}$}, xmin=0, ymode=log, height=0.4\textwidth]
\addplot [solid, thick, mark options={solid, thin}, mark=*, mark size=1.5, color=black, label=line:hdm, mark repeat=5]
coordinates {
( 0.00000000e+00,  7.13168286e+00)
( 1.00000000e+00,  3.43731160e+00)
( 2.00000000e+00,  2.11150936e+00)
( 3.00000000e+00,  1.76015421e+00)
( 4.00000000e+00,  1.57905465e+00)
( 5.00000000e+00,  1.41674313e+00)
( 6.00000000e+00,  1.23468668e+00)
( 7.00000000e+00,  1.04343651e+00)
( 8.00000000e+00,  8.54201912e-01)
( 9.00000000e+00,  6.81428704e-01)
( 1.00000000e+01,  5.48353672e-01)
( 1.10000000e+01,  4.45503696e-01)
( 1.20000000e+01,  3.56227079e-01)
( 1.30000000e+01,  2.79300409e-01)
( 1.40000000e+01,  2.11389404e-01)
( 1.50000000e+01,  1.58015722e-01)
( 1.60000000e+01,  1.13374560e-01)
( 1.70000000e+01,  8.12638243e-02)
( 1.80000000e+01,  6.09522438e-02)
( 1.90000000e+01,  4.85207433e-02)
( 2.00000000e+01,  3.99659708e-02)
( 2.10000000e+01,  3.30224581e-02)
( 2.20000000e+01,  2.81003741e-02)
( 2.30000000e+01,  2.45701069e-02)
( 2.40000000e+01,  2.17945977e-02)
( 2.50000000e+01,  1.89176607e-02)
( 2.60000000e+01,  1.60733234e-02)
( 2.70000000e+01,  1.39913534e-02)
( 2.80000000e+01,  1.27471729e-02)
( 2.90000000e+01,  1.19603277e-02)
( 3.00000000e+01,  1.13132287e-02)
( 3.10000000e+01,  1.08230623e-02)
( 3.20000000e+01,  1.05026591e-02)
( 3.30000000e+01,  1.02567888e-02)
( 3.40000000e+01,  1.00718711e-02)
( 3.50000000e+01,  9.91742403e-03)
( 3.60000000e+01,  9.76390338e-03)
( 3.70000000e+01,  9.60801858e-03)
( 3.80000000e+01,  9.43491006e-03)
( 3.90000000e+01,  9.22376360e-03)
( 4.00000000e+01,  9.00805806e-03)
( 4.10000000e+01,  8.81005924e-03)
( 4.20000000e+01,  8.64157444e-03)
( 4.30000000e+01,  8.47525939e-03)
( 4.40000000e+01,  8.28442159e-03)
( 4.50000000e+01,  8.01123558e-03)
( 4.60000000e+01,  7.80567801e-03)
( 4.70000000e+01,  7.60549944e-03)
( 4.80000000e+01,  7.38486637e-03)
( 4.90000000e+01,  7.17400480e-03)
( 5.00000000e+01,  6.98625310e-03)
( 5.10000000e+01,  6.82278594e-03)
( 5.20000000e+01,  6.64793597e-03)
( 5.30000000e+01,  6.46615358e-03)
( 5.40000000e+01,  6.27854436e-03)
( 5.50000000e+01,  6.11068464e-03)
( 5.60000000e+01,  5.93790962e-03)
( 5.70000000e+01,  5.74408357e-03)
( 5.80000000e+01,  5.52926969e-03)
( 5.90000000e+01,  5.30273691e-03)
( 6.00000000e+01,  5.04351515e-03)
( 6.10000000e+01,  4.74019583e-03)
( 6.20000000e+01,  4.39406474e-03)
( 6.30000000e+01,  3.98105582e-03)
( 6.40000000e+01,  3.56370767e-03)
( 6.50000000e+01,  3.24758358e-03)
( 6.60000000e+01,  3.01121988e-03)
( 6.70000000e+01,  2.82177734e-03)
( 6.80000000e+01,  2.69486997e-03)
( 6.90000000e+01,  2.59688364e-03)
( 7.00000000e+01,  2.53182535e-03)
( 7.10000000e+01,  2.46644760e-03)
( 7.20000000e+01,  2.41302247e-03)
( 7.30000000e+01,  2.37131310e-03)
( 7.40000000e+01,  2.33679123e-03)
( 7.50000000e+01,  2.30276408e-03)
( 7.60000000e+01,  2.26738145e-03)
( 7.70000000e+01,  2.23268871e-03)
( 7.80000000e+01,  2.20006559e-03)
( 7.90000000e+01,  2.16894485e-03)
( 8.00000000e+01,  2.14914236e-03)
( 8.10000000e+01,  2.13414581e-03)
( 8.20000000e+01,  2.12025768e-03)
( 8.30000000e+01,  2.10715746e-03)
( 8.40000000e+01,  2.09589875e-03)
( 8.50000000e+01,  2.08503381e-03)
( 8.60000000e+01,  2.07407007e-03)
( 8.70000000e+01,  2.06214753e-03)
( 8.80000000e+01,  2.04798461e-03)
( 8.90000000e+01,  2.03279074e-03)
( 9.00000000e+01,  2.01891190e-03)
( 9.10000000e+01,  2.00503399e-03)
( 9.20000000e+01,  1.99160835e-03)
( 9.30000000e+01,  1.97866195e-03)
( 9.40000000e+01,  1.96612913e-03)
( 9.50000000e+01,  1.95674706e-03)
( 9.60000000e+01,  1.94977990e-03)
( 9.70000000e+01,  1.94447981e-03)
( 9.80000000e+01,  1.94023328e-03)
( 9.90000000e+01,  1.93671775e-03)};\label{line:ssbeam0_nel180x90_rmin0p5_romset1_majit0:hdm}

\addplot [dash dot, thick, mark options={solid, thin}, mark=square*, mark size=1.5, color=blue, label=line:rom11, mark repeat=5]
coordinates {
( 0.00000000e+00,  7.13168286e+00)
( 1.00000000e+00,  3.43731160e+00)
( 2.00000000e+00,  2.11150936e+00)
( 3.00000000e+00,  1.60749779e+00)
( 4.00000000e+00,  1.26851776e+00)
( 5.00000000e+00,  9.51972913e-01)
( 6.00000000e+00,  6.05728999e-01)
( 7.00000000e+00,  3.51613577e-01)
( 8.00000000e+00,  1.76719370e-01)
( 9.00000000e+00,  8.22654622e-02)
( 1.00000000e+01,  4.02060983e-02)
( 1.10000000e+01,  2.21880367e-02)
( 1.20000000e+01,  2.21880367e-02)
( 1.30000000e+01,  8.42712687e-03)
( 1.40000000e+01,  3.00674474e-03)
( 1.50000000e+01,  2.28785945e-03)
( 1.60000000e+01,  1.24137695e-03)
( 1.70000000e+01,  1.08825775e-03)
( 1.80000000e+01,  9.52055248e-04)
( 1.90000000e+01,  8.30568844e-04)
( 2.00000000e+01,  8.30568844e-04)
( 2.10000000e+01,  7.89395045e-04)
( 2.20000000e+01,  7.89395045e-04)
( 2.30000000e+01,  7.68714892e-04)
( 2.40000000e+01,  7.47954232e-04)
( 2.50000000e+01,  7.40250312e-04)
( 2.60000000e+01,  7.36781241e-04)
( 2.70000000e+01,  7.34137065e-04)
( 2.80000000e+01,  7.30231571e-04)
( 2.90000000e+01,  7.23764144e-04)
( 3.00000000e+01,  7.15039735e-04)
( 3.10000000e+01,  7.04490975e-04)
( 3.20000000e+01,  6.99782200e-04)
( 3.30000000e+01,  6.94418659e-04)
( 3.40000000e+01,  6.87549286e-04)
( 3.50000000e+01,  6.75986372e-04)
( 3.60000000e+01,  6.60622891e-04)
( 3.70000000e+01,  6.36124732e-04)
( 3.80000000e+01,  6.10816391e-04)
( 3.90000000e+01,  5.80501748e-04)
( 4.00000000e+01,  5.71405817e-04)
( 4.10000000e+01,  5.61899191e-04)
( 4.20000000e+01,  5.47853249e-04)
( 4.30000000e+01,  5.42413763e-04)
( 4.40000000e+01,  5.41373770e-04)
( 4.50000000e+01,  5.40722303e-04)
( 4.60000000e+01,  5.39731639e-04)
( 4.70000000e+01,  5.38321286e-04)
( 4.80000000e+01,  5.36600577e-04)
( 4.90000000e+01,  5.33952033e-04)
( 5.00000000e+01,  5.30756672e-04)
( 5.10000000e+01,  5.27607658e-04)
( 5.20000000e+01,  5.26578418e-04)
( 5.30000000e+01,  5.25604445e-04)
( 5.40000000e+01,  5.24327134e-04)
( 5.50000000e+01,  5.22468772e-04)
( 5.60000000e+01,  5.20138553e-04)
( 5.70000000e+01,  5.17707447e-04)
( 5.80000000e+01,  5.17025018e-04)
( 5.90000000e+01,  5.16332233e-04)
( 6.00000000e+01,  5.15547901e-04)
( 6.10000000e+01,  5.14525572e-04)
( 6.20000000e+01,  5.13256737e-04)
( 6.30000000e+01,  5.13256737e-04)
( 6.40000000e+01,  5.12750662e-04)
( 6.50000000e+01,  5.12161373e-04)
( 6.60000000e+01,  5.11334571e-04)
( 6.70000000e+01,  5.10250538e-04)
( 6.80000000e+01,  5.09933134e-04)
( 6.90000000e+01,  5.09699469e-04)
( 7.00000000e+01,  5.09358645e-04)
( 7.10000000e+01,  5.08828470e-04)
( 7.20000000e+01,  5.07965682e-04)
( 7.30000000e+01,  5.06991156e-04)
( 7.40000000e+01,  5.06563525e-04)
( 7.50000000e+01,  5.05918068e-04)
( 7.60000000e+01,  5.05364572e-04)
( 7.70000000e+01,  5.04763746e-04)
( 7.80000000e+01,  5.03947046e-04)
( 7.90000000e+01,  5.02949010e-04)
( 8.00000000e+01,  5.01152088e-04)
( 8.10000000e+01,  5.00080356e-04)
( 8.20000000e+01,  4.98692192e-04)
( 8.30000000e+01,  4.97431865e-04)
( 8.40000000e+01,  4.96166519e-04)
( 8.50000000e+01,  4.94287166e-04)
( 8.60000000e+01,  4.90936765e-04)
( 8.70000000e+01,  4.85640956e-04)
( 8.80000000e+01,  4.85640956e-04)
( 8.90000000e+01,  4.81831531e-04)
( 9.00000000e+01,  4.76310949e-04)
( 9.10000000e+01,  4.71047845e-04)
( 9.20000000e+01,  4.66526269e-04)
( 9.30000000e+01,  4.62347409e-04)
( 9.40000000e+01,  4.57336473e-04)
( 9.50000000e+01,  4.55714550e-04)
( 9.60000000e+01,  4.54294278e-04)
( 9.70000000e+01,  4.52044497e-04)
( 9.80000000e+01,  4.48281863e-04)
( 9.90000000e+01,  4.43356718e-04)
( 1.00000000e+02,  4.36591710e-04)};\label{line:ssbeam0_nel180x90_rmin0p5_romset1_majit0:romA1}

\addplot [dash dot, thick, mark options={solid, thin}, mark=square*, mark size=1.5, color=red, label=line:rom12, mark repeat=5]
coordinates {
( 0.00000000e+00,  7.13168286e+00)
( 1.00000000e+00,  7.13168286e+00)
( 2.00000000e+00,  7.13168286e+00)
( 3.00000000e+00,  3.29016443e+00)
( 4.00000000e+00,  2.07635063e+00)
( 5.00000000e+00,  1.20585326e+00)
( 6.00000000e+00,  5.24163795e-01)
( 7.00000000e+00,  2.15844081e-01)
( 8.00000000e+00,  2.15844081e-01)
( 9.00000000e+00,  1.20973969e-01)
( 1.00000000e+01,  7.52845955e-02)
( 1.10000000e+01,  5.62159404e-02)
( 1.20000000e+01,  4.65619862e-02)
( 1.30000000e+01,  4.25332534e-02)
( 1.40000000e+01,  3.93372991e-02)
( 1.50000000e+01,  3.78458816e-02)
( 1.60000000e+01,  3.73475390e-02)
( 1.70000000e+01,  3.68044444e-02)
( 1.80000000e+01,  3.68044444e-02)
( 1.90000000e+01,  3.65062949e-02)
( 2.00000000e+01,  3.64289518e-02)
( 2.10000000e+01,  3.63501752e-02)
( 2.20000000e+01,  3.63021467e-02)
( 2.30000000e+01,  3.62710286e-02)
( 2.40000000e+01,  3.62600946e-02)
( 2.50000000e+01,  3.62468181e-02)
( 2.60000000e+01,  3.62291385e-02)
( 2.70000000e+01,  3.61962373e-02)
( 2.80000000e+01,  3.61503726e-02)
( 2.90000000e+01,  3.60709689e-02)
( 3.00000000e+01,  3.59440212e-02)
( 3.10000000e+01,  3.57627666e-02)
( 3.20000000e+01,  3.56290818e-02)
( 3.30000000e+01,  3.54809847e-02)
( 3.40000000e+01,  3.53426783e-02)
( 3.50000000e+01,  3.52021986e-02)
( 3.60000000e+01,  3.50736465e-02)
( 3.70000000e+01,  3.49746681e-02)
( 3.80000000e+01,  3.48917999e-02)
( 3.90000000e+01,  3.47967440e-02)
( 4.00000000e+01,  3.46847028e-02)
( 4.10000000e+01,  3.45824600e-02)
( 4.20000000e+01,  3.45266204e-02)
( 4.30000000e+01,  3.44612001e-02)
( 4.40000000e+01,  3.43973337e-02)
( 4.50000000e+01,  3.43222738e-02)
( 4.60000000e+01,  3.42781724e-02)
( 4.70000000e+01,  3.42163000e-02)
( 4.80000000e+01,  3.41476537e-02)
( 4.90000000e+01,  3.40958438e-02)
( 5.00000000e+01,  3.40676418e-02)
( 5.10000000e+01,  3.40335810e-02)
( 5.20000000e+01,  3.39819848e-02)
( 5.30000000e+01,  3.39465024e-02)
( 5.40000000e+01,  3.39291998e-02)
( 5.50000000e+01,  3.39084950e-02)
( 5.60000000e+01,  3.38696604e-02)
( 5.70000000e+01,  3.38131619e-02)
( 5.80000000e+01,  3.37148036e-02)
( 5.90000000e+01,  3.36232118e-02)
( 6.00000000e+01,  3.35461253e-02)
( 6.10000000e+01,  3.34995135e-02)
( 6.20000000e+01,  3.34136157e-02)
( 6.30000000e+01,  3.33733965e-02)
( 6.40000000e+01,  3.33242889e-02)
( 6.50000000e+01,  3.32746805e-02)
( 6.60000000e+01,  3.32562243e-02)
( 6.70000000e+01,  3.32358347e-02)
( 6.80000000e+01,  3.32135866e-02)
( 6.90000000e+01,  3.31805231e-02)
( 7.00000000e+01,  3.31278096e-02)
( 7.10000000e+01,  3.30879644e-02)
( 7.20000000e+01,  3.30315460e-02)
( 7.30000000e+01,  3.29737836e-02)
( 7.40000000e+01,  3.29319038e-02)
( 7.50000000e+01,  3.28960036e-02)
( 7.60000000e+01,  3.28798235e-02)
( 7.70000000e+01,  3.28597672e-02)
( 7.80000000e+01,  3.28257290e-02)
( 7.90000000e+01,  3.27794801e-02)
( 8.00000000e+01,  3.27500518e-02)
( 8.10000000e+01,  3.27102126e-02)
( 8.20000000e+01,  3.26887011e-02)
( 8.30000000e+01,  3.26610624e-02)
( 8.40000000e+01,  3.26275209e-02)
( 8.50000000e+01,  3.25857801e-02)
( 8.60000000e+01,  3.25584538e-02)
( 8.70000000e+01,  3.25468492e-02)
( 8.80000000e+01,  3.25311539e-02)
( 8.90000000e+01,  3.25096724e-02)
( 9.00000000e+01,  3.24700340e-02)
( 9.10000000e+01,  3.24172307e-02)
( 9.20000000e+01,  3.23799679e-02)
( 9.30000000e+01,  3.23664862e-02)
( 9.40000000e+01,  3.23480855e-02)
( 9.50000000e+01,  3.23182556e-02)
( 9.60000000e+01,  3.22803553e-02)
( 9.70000000e+01,  3.22292308e-02)
( 9.80000000e+01,  3.21959279e-02)
( 9.90000000e+01,  3.21823141e-02)
( 1.00000000e+02,  3.21680793e-02)};\label{line:ssbeam0_nel180x90_rmin0p5_romset1_majit0:romA2}

\addplot [dashed, thick, mark options={solid, thick}, mark=x, mark size=3, color=blue, label=line:trrom1, mark repeat=5]
coordinates {
( 0.00000000e+00,  7.13168286e+00)
( 1.00000000e+00,  3.43731160e+00)
( 2.00000000e+00,  2.11150936e+00)
( 3.00000000e+00,  1.76291248e+00)
( 4.00000000e+00,  1.41830829e+00)
( 5.00000000e+00,  1.08625861e+00)
( 6.00000000e+00,  7.18803366e-01)
( 7.00000000e+00,  4.95366816e-01)
( 8.00000000e+00,  2.90354704e-01)
( 9.00000000e+00,  1.27871468e-01)
( 1.00000000e+01,  5.30939581e-02)
( 1.10000000e+01,  2.73045722e-02)
( 1.20000000e+01,  1.37713489e-02)
( 1.30000000e+01,  9.16670927e-03)
( 1.40000000e+01,  4.45301357e-03)
( 1.50000000e+01,  2.47581137e-03)
( 1.60000000e+01,  2.04572429e-03)
( 1.70000000e+01,  1.30919614e-03)
( 1.80000000e+01,  9.69375126e-04)
( 1.90000000e+01,  5.98340928e-04)
( 2.00000000e+01,  4.46100315e-04)
( 2.10000000e+01,  3.26412554e-04)
( 2.20000000e+01,  2.94911511e-04)
( 2.30000000e+01,  2.51948097e-04)
( 2.40000000e+01,  2.40101116e-04)
( 2.50000000e+01,  2.26695733e-04)
( 2.60000000e+01,  2.22213888e-04)
( 2.70000000e+01,  2.17553230e-04)
( 2.80000000e+01,  2.16250764e-04)
( 2.90000000e+01,  2.16250764e-04)
( 3.00000000e+01,  2.15592061e-04)
( 3.10000000e+01,  2.15337694e-04)
( 3.20000000e+01,  2.15093109e-04)
( 3.30000000e+01,  2.14982489e-04)
( 3.40000000e+01,  2.14723337e-04)
( 3.50000000e+01,  2.14605695e-04)
( 3.60000000e+01,  2.14533958e-04)
( 3.70000000e+01,  2.14441036e-04)
( 3.80000000e+01,  2.14354368e-04)
( 3.90000000e+01,  2.14289628e-04)
( 4.00000000e+01,  2.14241945e-04)
( 4.10000000e+01,  2.14163111e-04)
( 4.20000000e+01,  2.14089392e-04)
( 4.30000000e+01,  2.13998889e-04)
( 4.40000000e+01,  2.13926217e-04)
( 4.50000000e+01,  2.13846093e-04)
( 4.60000000e+01,  2.13747744e-04)
( 4.70000000e+01,  2.13661881e-04)
( 4.80000000e+01,  2.13573056e-04)
( 4.90000000e+01,  2.13481752e-04)
( 5.00000000e+01,  2.13389993e-04)
( 5.10000000e+01,  2.13304327e-04)
( 5.20000000e+01,  2.13211113e-04)
( 5.30000000e+01,  2.13129657e-04)
( 5.40000000e+01,  2.13040817e-04)
( 5.50000000e+01,  2.12973855e-04)
( 5.60000000e+01,  2.12913151e-04)
( 5.70000000e+01,  2.12864624e-04)
( 5.80000000e+01,  2.12824670e-04)
( 5.90000000e+01,  2.12801553e-04)
( 6.00000000e+01,  2.12770884e-04)
( 6.10000000e+01,  2.12763614e-04)
( 6.20000000e+01,  2.12740215e-04)
( 6.30000000e+01,  2.12734048e-04)
( 6.40000000e+01,  2.12730531e-04)
( 6.50000000e+01,  2.12725961e-04)
( 6.60000000e+01,  2.12721443e-04)
( 6.70000000e+01,  2.12716999e-04)
( 6.80000000e+01,  2.12715455e-04)
( 6.90000000e+01,  2.12714779e-04)
( 7.00000000e+01,  2.12713250e-04)
( 7.10000000e+01,  2.12710282e-04)
( 7.20000000e+01,  2.12709245e-04)
( 7.30000000e+01,  2.12706506e-04)
( 7.40000000e+01,  2.12704652e-04)
( 7.50000000e+01,  2.12703865e-04)
( 7.60000000e+01,  2.12701391e-04)
( 7.70000000e+01,  2.12699917e-04)};\label{line:ssbeam0_nel180x90_rmin0p5_romset1_majit0:romB1}

\addplot [dashed, thick, mark options={solid, thick}, mark=x, mark size=3, color=red, label=line:trrom2, mark repeat=5]
coordinates {
( 0.00000000e+00,  7.13168286e+00)
( 1.00000000e+00,  7.13168286e+00)
( 2.00000000e+00,  3.43731160e+00)
( 3.00000000e+00,  1.98128437e+00)
( 4.00000000e+00,  1.98128437e+00)
( 5.00000000e+00,  1.61998571e+00)
( 6.00000000e+00,  1.19729409e+00)
( 7.00000000e+00,  7.38354113e-01)
( 8.00000000e+00,  3.16186502e-01)
( 9.00000000e+00,  1.67755193e-01)
( 1.00000000e+01,  6.95218667e-02)
( 1.10000000e+01,  3.11087185e-02)
( 1.20000000e+01,  1.81928400e-02)
( 1.30000000e+01,  8.76206243e-03)
( 1.40000000e+01,  8.76206243e-03)
( 1.50000000e+01,  3.65474168e-03)
( 1.60000000e+01,  2.72585099e-03)
( 1.70000000e+01,  1.55403278e-03)
( 1.80000000e+01,  1.22648560e-03)
( 1.90000000e+01,  1.04146526e-03)
( 2.00000000e+01,  8.87187153e-04)
( 2.10000000e+01,  8.50113533e-04)
( 2.20000000e+01,  8.01233958e-04)
( 2.30000000e+01,  7.76568074e-04)
( 2.40000000e+01,  7.60900111e-04)
( 2.50000000e+01,  7.37351785e-04)
( 2.60000000e+01,  7.22962750e-04)
( 2.70000000e+01,  7.09549055e-04)
( 2.80000000e+01,  7.01238023e-04)
( 2.90000000e+01,  6.93798849e-04)
( 3.00000000e+01,  6.89860285e-04)
( 3.10000000e+01,  6.84742936e-04)
( 3.20000000e+01,  6.82515883e-04)
( 3.30000000e+01,  6.81024211e-04)
( 3.40000000e+01,  6.79698369e-04)
( 3.50000000e+01,  6.78381344e-04)
( 3.60000000e+01,  6.77321754e-04)
( 3.70000000e+01,  6.76238605e-04)
( 3.80000000e+01,  6.74619045e-04)
( 3.90000000e+01,  6.73321144e-04)
( 4.00000000e+01,  6.71456891e-04)
( 4.10000000e+01,  6.69692612e-04)
( 4.20000000e+01,  6.68432750e-04)
( 4.30000000e+01,  6.67498504e-04)
( 4.40000000e+01,  6.66398878e-04)
( 4.50000000e+01,  6.65144433e-04)
( 4.60000000e+01,  6.63867042e-04)
( 4.70000000e+01,  6.62237101e-04)
( 4.80000000e+01,  6.60655439e-04)
( 4.90000000e+01,  6.59207145e-04)
( 5.00000000e+01,  6.57444184e-04)
( 5.10000000e+01,  6.55879706e-04)
( 5.20000000e+01,  6.54221082e-04)
( 5.30000000e+01,  6.53136364e-04)
( 5.40000000e+01,  6.52342346e-04)
( 5.50000000e+01,  6.52008175e-04)
( 5.60000000e+01,  6.51723206e-04)
( 5.70000000e+01,  6.51644382e-04)
( 5.80000000e+01,  6.51644382e-04)
( 5.90000000e+01,  6.51593073e-04)
( 6.00000000e+01,  6.51569241e-04)
( 6.10000000e+01,  6.51569241e-04)
( 6.20000000e+01,  6.51552503e-04)
( 6.30000000e+01,  6.51547393e-04)
( 6.40000000e+01,  6.51545206e-04)
( 6.50000000e+01,  6.51544502e-04)
( 6.60000000e+01,  6.51543659e-04)
( 6.70000000e+01,  6.51542083e-04)
( 6.80000000e+01,  6.51540933e-04)};\label{line:ssbeam0_nel180x90_rmin0p5_romset1_majit0:romB2}

\nextgroupplot[width=0.8\textwidth, grid=major, xmax=50, ylabel={Cum. No. ROM solves}, xmin=0, ymode=log, height=0.4\textwidth]
\addplot [dash dot, thick, mark options={solid, thin}, mark=square*, mark size=1.5, color=blue, label=line:rom11, mark repeat=5, forget plot]
coordinates {
( 1.00000000e+00,  1.00000000e+00)
( 2.00000000e+00,  2.00000000e+00)
( 3.00000000e+00,  4.00000000e+00)
( 4.00000000e+00,  6.00000000e+00)
( 5.00000000e+00,  8.00000000e+00)
( 6.00000000e+00,  1.10000000e+01)
( 7.00000000e+00,  1.40000000e+01)
( 8.00000000e+00,  1.70000000e+01)
( 9.00000000e+00,  2.10000000e+01)
( 1.00000000e+01,  2.80000000e+01)
( 1.10000000e+01,  3.70000000e+01)
( 1.20000000e+01,  4.90000000e+01)
( 1.30000000e+01,  5.40000000e+01)
( 1.40000000e+01,  6.40000000e+01)
( 1.50000000e+01,  7.60000000e+01)
( 1.60000000e+01,  8.40000000e+01)
( 1.70000000e+01,  9.60000000e+01)
( 1.80000000e+01,  1.05000000e+02)
( 1.90000000e+01,  1.12000000e+02)
( 2.00000000e+01,  1.24000000e+02)
( 2.10000000e+01,  1.35000000e+02)
( 2.20000000e+01,  1.48000000e+02)
( 2.30000000e+01,  1.59000000e+02)
( 2.40000000e+01,  1.72000000e+02)
( 2.50000000e+01,  1.81000000e+02)
( 2.60000000e+01,  1.85000000e+02)
( 2.70000000e+01,  1.90000000e+02)
( 2.80000000e+01,  1.97000000e+02)
( 2.90000000e+01,  2.06000000e+02)
( 3.00000000e+01,  2.16000000e+02)
( 3.10000000e+01,  2.27000000e+02)
( 3.20000000e+01,  2.32000000e+02)
( 3.30000000e+01,  2.39000000e+02)
( 3.40000000e+01,  2.47000000e+02)
( 3.50000000e+01,  2.57000000e+02)
( 3.60000000e+01,  2.68000000e+02)
( 3.70000000e+01,  2.81000000e+02)
( 3.80000000e+01,  2.95000000e+02)
( 3.90000000e+01,  3.06000000e+02)
( 4.00000000e+01,  3.14000000e+02)
( 4.10000000e+01,  3.24000000e+02)
( 4.20000000e+01,  3.37000000e+02)
( 4.30000000e+01,  3.44000000e+02)
( 4.40000000e+01,  3.46000000e+02)
( 4.50000000e+01,  3.48000000e+02)
( 4.60000000e+01,  3.51000000e+02)
( 4.70000000e+01,  3.56000000e+02)
( 4.80000000e+01,  3.62000000e+02)
( 4.90000000e+01,  3.70000000e+02)
( 5.00000000e+01,  3.80000000e+02)
( 5.10000000e+01,  3.92000000e+02)
( 5.20000000e+01,  3.96000000e+02)
( 5.30000000e+01,  4.02000000e+02)
( 5.40000000e+01,  4.10000000e+02)
( 5.50000000e+01,  4.20000000e+02)
( 5.60000000e+01,  4.32000000e+02)
( 5.70000000e+01,  4.45000000e+02)
( 5.80000000e+01,  4.49000000e+02)
( 5.90000000e+01,  4.56000000e+02)
( 6.00000000e+01,  4.65000000e+02)
( 6.10000000e+01,  4.76000000e+02)
( 6.20000000e+01,  4.89000000e+02)
( 6.30000000e+01,  5.04000000e+02)
( 6.40000000e+01,  5.12000000e+02)
( 6.50000000e+01,  5.22000000e+02)
( 6.60000000e+01,  5.34000000e+02)
( 6.70000000e+01,  5.48000000e+02)
( 6.80000000e+01,  5.52000000e+02)
( 6.90000000e+01,  5.57000000e+02)
( 7.00000000e+01,  5.64000000e+02)
( 7.10000000e+01,  5.73000000e+02)
( 7.20000000e+01,  5.84000000e+02)
( 7.30000000e+01,  5.96000000e+02)
( 7.40000000e+01,  6.10000000e+02)
( 7.50000000e+01,  6.13000000e+02)
( 7.60000000e+01,  6.18000000e+02)
( 7.70000000e+01,  6.26000000e+02)
( 7.80000000e+01,  6.37000000e+02)
( 7.90000000e+01,  6.49000000e+02)
( 8.00000000e+01,  6.63000000e+02)
( 8.10000000e+01,  6.77000000e+02)
( 8.20000000e+01,  6.81000000e+02)
( 8.30000000e+01,  6.88000000e+02)
( 8.40000000e+01,  6.97000000e+02)
( 8.50000000e+01,  7.08000000e+02)
( 8.60000000e+01,  7.21000000e+02)
( 8.70000000e+01,  7.35000000e+02)
( 8.80000000e+01,  7.50000000e+02)
( 8.90000000e+01,  7.60000000e+02)
( 9.00000000e+01,  7.72000000e+02)
( 9.10000000e+01,  7.86000000e+02)
( 9.20000000e+01,  7.92000000e+02)
( 9.30000000e+01,  8.02000000e+02)
( 9.40000000e+01,  8.14000000e+02)
( 9.50000000e+01,  8.19000000e+02)
( 9.60000000e+01,  8.25000000e+02)
( 9.70000000e+01,  8.33000000e+02)
( 9.80000000e+01,  8.43000000e+02)
( 9.90000000e+01,  8.54000000e+02)
( 1.00000000e+02,  8.67000000e+02)};

\addplot [dash dot, thick, mark options={solid, thin}, mark=square*, mark size=1.5, color=red, label=line:rom12, mark repeat=5, forget plot]
coordinates {
( 1.00000000e+00,  4.90000000e+01)
( 2.00000000e+00,  5.20000000e+01)
( 3.00000000e+00,  5.40000000e+01)
( 4.00000000e+00,  5.60000000e+01)
( 5.00000000e+00,  5.90000000e+01)
( 6.00000000e+00,  6.10000000e+01)
( 7.00000000e+00,  6.70000000e+01)
( 8.00000000e+00,  8.80000000e+01)
( 9.00000000e+00,  9.20000000e+01)
( 1.00000000e+01,  9.70000000e+01)
( 1.10000000e+01,  1.05000000e+02)
( 1.20000000e+01,  1.15000000e+02)
( 1.30000000e+01,  1.30000000e+02)
( 1.40000000e+01,  1.39000000e+02)
( 1.50000000e+01,  1.52000000e+02)
( 1.60000000e+01,  1.67000000e+02)
( 1.70000000e+01,  1.82000000e+02)
( 1.80000000e+01,  1.99000000e+02)
( 1.90000000e+01,  2.14000000e+02)
( 2.00000000e+01,  2.29000000e+02)
( 2.10000000e+01,  2.41000000e+02)
( 2.20000000e+01,  2.52000000e+02)
( 2.30000000e+01,  2.61000000e+02)
( 2.40000000e+01,  2.66000000e+02)
( 2.50000000e+01,  2.72000000e+02)
( 2.60000000e+01,  2.79000000e+02)
( 2.70000000e+01,  2.88000000e+02)
( 2.80000000e+01,  2.98000000e+02)
( 2.90000000e+01,  3.10000000e+02)
( 3.00000000e+01,  3.23000000e+02)
( 3.10000000e+01,  3.37000000e+02)
( 3.20000000e+01,  3.50000000e+02)
( 3.30000000e+01,  3.63000000e+02)
( 3.40000000e+01,  3.76000000e+02)
( 3.50000000e+01,  3.89000000e+02)
( 3.60000000e+01,  4.02000000e+02)
( 3.70000000e+01,  4.16000000e+02)
( 3.80000000e+01,  4.30000000e+02)
( 3.90000000e+01,  4.44000000e+02)
( 4.00000000e+01,  4.58000000e+02)
( 4.10000000e+01,  4.72000000e+02)
( 4.20000000e+01,  4.86000000e+02)
( 4.30000000e+01,  5.00000000e+02)
( 4.40000000e+01,  5.13000000e+02)
( 4.50000000e+01,  5.25000000e+02)
( 4.60000000e+01,  5.35000000e+02)
( 4.70000000e+01,  5.48000000e+02)
( 4.80000000e+01,  5.62000000e+02)
( 4.90000000e+01,  5.74000000e+02)
( 5.00000000e+01,  5.84000000e+02)
( 5.10000000e+01,  5.96000000e+02)
( 5.20000000e+01,  6.10000000e+02)
( 5.30000000e+01,  6.21000000e+02)
( 5.40000000e+01,  6.29000000e+02)
( 5.50000000e+01,  6.38000000e+02)
( 5.60000000e+01,  6.49000000e+02)
( 5.70000000e+01,  6.61000000e+02)
( 5.80000000e+01,  6.75000000e+02)
( 5.90000000e+01,  6.90000000e+02)
( 6.00000000e+01,  7.05000000e+02)
( 6.10000000e+01,  7.20000000e+02)
( 6.20000000e+01,  7.30000000e+02)
( 6.30000000e+01,  7.39000000e+02)
( 6.40000000e+01,  7.51000000e+02)
( 6.50000000e+01,  7.65000000e+02)
( 6.60000000e+01,  7.77000000e+02)
( 6.70000000e+01,  7.86000000e+02)
( 6.80000000e+01,  7.98000000e+02)
( 6.90000000e+01,  8.11000000e+02)
( 7.00000000e+01,  8.25000000e+02)
( 7.10000000e+01,  8.36000000e+02)
( 7.20000000e+01,  8.49000000e+02)
( 7.30000000e+01,  8.64000000e+02)
( 7.40000000e+01,  8.78000000e+02)
( 7.50000000e+01,  8.90000000e+02)
( 7.60000000e+01,  8.99000000e+02)
( 7.70000000e+01,  9.10000000e+02)
( 7.80000000e+01,  9.23000000e+02)
( 7.90000000e+01,  9.37000000e+02)
( 8.00000000e+01,  9.48000000e+02)
( 8.10000000e+01,  9.62000000e+02)
( 8.20000000e+01,  9.73000000e+02)
( 8.30000000e+01,  9.86000000e+02)
( 8.40000000e+01,  1.00100000e+03)
( 8.50000000e+01,  1.01400000e+03)
( 8.60000000e+01,  1.02500000e+03)
( 8.70000000e+01,  1.03200000e+03)
( 8.80000000e+01,  1.04100000e+03)
( 8.90000000e+01,  1.05100000e+03)
( 9.00000000e+01,  1.06300000e+03)
( 9.10000000e+01,  1.07600000e+03)
( 9.20000000e+01,  1.08600000e+03)
( 9.30000000e+01,  1.09200000e+03)
( 9.40000000e+01,  1.10000000e+03)
( 9.50000000e+01,  1.11000000e+03)
( 9.60000000e+01,  1.12100000e+03)
( 9.70000000e+01,  1.13400000e+03)
( 9.80000000e+01,  1.14400000e+03)
( 9.90000000e+01,  1.15100000e+03)
( 1.00000000e+02,  1.16000000e+03)};

\addplot [dashed, thick, mark options={solid, thick}, mark=x, mark size=3, color=blue, label=line:trrom1, mark repeat=5, forget plot]
coordinates {
( 1.00000000e+00,  1.00000000e+00)
( 2.00000000e+00,  2.00000000e+00)
( 3.00000000e+00,  3.00000000e+00)
( 4.00000000e+00,  5.00000000e+00)
( 5.00000000e+00,  8.00000000e+00)
( 6.00000000e+00,  1.00000000e+01)
( 7.00000000e+00,  1.40000000e+01)
( 8.00000000e+00,  1.50000000e+01)
( 9.00000000e+00,  1.80000000e+01)
( 1.00000000e+01,  2.40000000e+01)
( 1.10000000e+01,  3.00000000e+01)
( 1.20000000e+01,  3.60000000e+01)
( 1.30000000e+01,  4.90000000e+01)
( 1.40000000e+01,  5.60000000e+01)
( 1.50000000e+01,  6.50000000e+01)
( 1.60000000e+01,  8.00000000e+01)
( 1.70000000e+01,  9.00000000e+01)
( 1.80000000e+01,  1.01000000e+02)
( 1.90000000e+01,  1.16000000e+02)
( 2.00000000e+01,  1.31000000e+02)
( 2.10000000e+01,  1.43000000e+02)
( 2.20000000e+01,  1.57000000e+02)
( 2.30000000e+01,  1.68000000e+02)
( 2.40000000e+01,  1.83000000e+02)
( 2.50000000e+01,  1.94000000e+02)
( 2.60000000e+01,  2.09000000e+02)
( 2.70000000e+01,  2.18000000e+02)
( 2.80000000e+01,  2.31000000e+02)
( 2.90000000e+01,  2.46000000e+02)
( 3.00000000e+01,  2.59000000e+02)
( 3.10000000e+01,  2.74000000e+02)
( 3.20000000e+01,  2.87000000e+02)
( 3.30000000e+01,  3.02000000e+02)
( 3.40000000e+01,  3.13000000e+02)
( 3.50000000e+01,  3.29000000e+02)
( 3.60000000e+01,  3.43000000e+02)
( 3.70000000e+01,  3.55000000e+02)
( 3.80000000e+01,  3.71000000e+02)
( 3.90000000e+01,  3.85000000e+02)
( 4.00000000e+01,  3.99000000e+02)
( 4.10000000e+01,  4.14000000e+02)
( 4.20000000e+01,  4.29000000e+02)
( 4.30000000e+01,  4.45000000e+02)
( 4.40000000e+01,  4.60000000e+02)
( 4.50000000e+01,  4.76000000e+02)
( 4.60000000e+01,  4.92000000e+02)
( 4.70000000e+01,  5.07000000e+02)
( 4.80000000e+01,  5.23000000e+02)
( 4.90000000e+01,  5.38000000e+02)
( 5.00000000e+01,  5.54000000e+02)
( 5.10000000e+01,  5.69000000e+02)
( 5.20000000e+01,  5.85000000e+02)
( 5.30000000e+01,  6.00000000e+02)
( 5.40000000e+01,  6.16000000e+02)
( 5.50000000e+01,  6.31000000e+02)
( 5.60000000e+01,  6.47000000e+02)
( 5.70000000e+01,  6.63000000e+02)
( 5.80000000e+01,  6.78000000e+02)
( 5.90000000e+01,  6.94000000e+02)
( 6.00000000e+01,  7.09000000e+02)
( 6.10000000e+01,  7.25000000e+02)
( 6.20000000e+01,  7.37000000e+02)
( 6.30000000e+01,  7.48000000e+02)
( 6.40000000e+01,  7.58000000e+02)
( 6.50000000e+01,  7.71000000e+02)
( 6.60000000e+01,  7.86000000e+02)
( 6.70000000e+01,  8.01000000e+02)
( 6.80000000e+01,  8.12000000e+02)
( 6.90000000e+01,  8.20000000e+02)
( 7.00000000e+01,  8.31000000e+02)
( 7.10000000e+01,  8.46000000e+02)
( 7.20000000e+01,  8.58000000e+02)
( 7.30000000e+01,  8.74000000e+02)
( 7.40000000e+01,  8.90000000e+02)
( 7.50000000e+01,  9.03000000e+02)
( 7.60000000e+01,  9.20000000e+02)
( 7.70000000e+01,  9.37000000e+02)};

\addplot [dashed, thick, mark options={solid, thick}, mark=x, mark size=3, color=red, label=line:trrom2, mark repeat=5, forget plot]
coordinates {
( 1.00000000e+00,  4.90000000e+01)
( 2.00000000e+00,  5.00000000e+01)
( 3.00000000e+00,  5.20000000e+01)
( 4.00000000e+00,  5.50000000e+01)
( 5.00000000e+00,  5.60000000e+01)
( 6.00000000e+00,  5.90000000e+01)
( 7.00000000e+00,  6.10000000e+01)
( 8.00000000e+00,  6.50000000e+01)
( 9.00000000e+00,  6.70000000e+01)
( 1.00000000e+01,  7.20000000e+01)
( 1.10000000e+01,  7.80000000e+01)
( 1.20000000e+01,  8.80000000e+01)
( 1.30000000e+01,  9.30000000e+01)
( 1.40000000e+01,  1.08000000e+02)
( 1.50000000e+01,  1.17000000e+02)
( 1.60000000e+01,  1.30000000e+02)
( 1.70000000e+01,  1.39000000e+02)
( 1.80000000e+01,  1.51000000e+02)
( 1.90000000e+01,  1.65000000e+02)
( 2.00000000e+01,  1.77000000e+02)
( 2.10000000e+01,  1.92000000e+02)
( 2.20000000e+01,  2.03000000e+02)
( 2.30000000e+01,  2.18000000e+02)
( 2.40000000e+01,  2.31000000e+02)
( 2.50000000e+01,  2.41000000e+02)
( 2.60000000e+01,  2.53000000e+02)
( 2.70000000e+01,  2.66000000e+02)
( 2.80000000e+01,  2.78000000e+02)
( 2.90000000e+01,  2.92000000e+02)
( 3.00000000e+01,  3.08000000e+02)
( 3.10000000e+01,  3.19000000e+02)
( 3.20000000e+01,  3.34000000e+02)
( 3.30000000e+01,  3.50000000e+02)
( 3.40000000e+01,  3.64000000e+02)
( 3.50000000e+01,  3.76000000e+02)
( 3.60000000e+01,  3.91000000e+02)
( 3.70000000e+01,  4.04000000e+02)
( 3.80000000e+01,  4.19000000e+02)
( 3.90000000e+01,  4.32000000e+02)
( 4.00000000e+01,  4.46000000e+02)
( 4.10000000e+01,  4.60000000e+02)
( 4.20000000e+01,  4.74000000e+02)
( 4.30000000e+01,  4.87000000e+02)
( 4.40000000e+01,  5.02000000e+02)
( 4.50000000e+01,  5.18000000e+02)
( 4.60000000e+01,  5.33000000e+02)
( 4.70000000e+01,  5.48000000e+02)
( 4.80000000e+01,  5.63000000e+02)
( 4.90000000e+01,  5.77000000e+02)
( 5.00000000e+01,  5.92000000e+02)
( 5.10000000e+01,  6.06000000e+02)
( 5.20000000e+01,  6.21000000e+02)
( 5.30000000e+01,  6.36000000e+02)
( 5.40000000e+01,  6.51000000e+02)
( 5.50000000e+01,  6.68000000e+02)
( 5.60000000e+01,  6.81000000e+02)
( 5.70000000e+01,  6.97000000e+02)
( 5.80000000e+01,  7.12000000e+02)
( 5.90000000e+01,  7.28000000e+02)
( 6.00000000e+01,  7.43000000e+02)
( 6.10000000e+01,  7.59000000e+02)
( 6.20000000e+01,  7.75000000e+02)
( 6.30000000e+01,  7.89000000e+02)
( 6.40000000e+01,  8.00000000e+02)
( 6.50000000e+01,  8.08000000e+02)
( 6.60000000e+01,  8.19000000e+02)
( 6.70000000e+01,  8.34000000e+02)
( 6.80000000e+01,  8.50000000e+02)};

\end{groupplot}\end{tikzpicture}

%% file: py/ssbeam0_nel180x90_rmin0p5_majit0_matinterp1.dat
HDM-MMA & $0.100$ & - & $163.30$ & $18$ & - & $18.0$ \\
 & $0.050$ & - & $160.07$ & $20$ & - & $20.0$ \\
 & $0.010$ & - & $155.42$ & $36$ & - & $36.0$ \\
 & $0.001$ & - & $154.07$ & $207$ & - & $207.0$ \\\hline
ROM-TR-RES & $0.100$ & $0.10$ & $158.12$ & $11$ & 24 & $11.2$ \\
 & $0.050$ & $0.10$ & $156.04$ & $12$ & 30 & $12.3$ \\
 & $0.010$ & $0.10$ & $154.60$ & $14$ & 49 & $14.5$ \\
 & $0.001$ & $0.10$ & $154.01$ & $19$ & 101 & $20.0$ \\\hline
ROM-TR-DIST & $0.100$ & $0.10$ & $160.11$ & $10$ & 21 & $10.2$ \\
 & $0.050$ & $0.10$ & $157.33$ & $11$ & 28 & $11.3$ \\
 & $0.010$ & $0.10$ & $154.38$ & $14$ & 54 & $14.5$ \\
 & $0.001$ & $0.10$ & $154.05$ & $19$ & 105 & $20.1$ \\

%% file: conclusion.tex
\section{Conclusion}
\label{sec:concl}
The contributions of this work are twofold:
\begin{inparaenum}[(1)]
 \item the introduction and analysis of a generalized trust-region method
   for nonlinear optimization problems with convex constraints and
 \item a class of efficient, globally convergent topology optimization
   methods based on projection-based ROMs and generalized
   trust regions.
\end{inparaenum}
We consider two different types of trust regions for the reduced
topology optimization problem: the traditional ball (two-norm) in
the design space and the sublevel sets of the norm of the HDM residual
evaluated at the ROM solution. We provide detailed analysis to show
the trust-region method that uses the ROM as the trust-region model
and sublevel sets of the HDM residual norm as the trust-region
constraint is globally convergent. To ensure the
ROM satisfies the accuracy constraints of the trust-region theory,
we sample the HDM solution and adjoint at each trust-region center,
which guarantees the ROM objective function and gradient match those
of the HDM at the trust-region center.
This leads to a ROM that is constructed
on-the-fly and specialized for the solutions obtained on the optimization path. 
We apply the ROM-accelerated methods to three benchmark topology optimization problems; the methods converge to the optimal design up to an order of magnitude
faster than the standard MMA method that does not incorporate a ROM. 
The two trust-region constraints perform similarly in our numerical
experiments, which favors the simpler distance-based constraint;
for the case where the ROM solves are very
efficient ($\Ocal(10^3)$ speedup), we expect to see additional benefits
from the residual-based trust-region by combining with more sophisticated
subproblem solvers.
For cases where the ROM is less efficient, e.g., nonlinear problems
requiring hyperreduction, the cost of solving the trust-region subproblem
with the nonquadratic ROM objective function could become a bottleneck.
In such situations, approximating the ROM objective with a quadratic function
and using trust-region solvers designed to rapidly find a point that satisfies a
fraction of the Cauchy decrease condition, e.g., truncated conjugate gradient,
will likely be beneficial.

Having demonstrated the potential of the ROM-based trust-region method
for topology optimization on several benchmark problems, there are a
number of important avenues for future research. While the ROM evaluations
in this work are much cheaper than HDM evaluations, the cost still scale with
the large dimension ($N_e$) due to the Helmholtz filter. The ROMs
could be further accelerated by constructing a separate reduced-order
model for the Helmholtz system, which would also require an adaptation
strategy to ensure global convergence. In addition, the cost could be
further reduced by relaxing the requirement that the value and
gradient of the reduced objective function be exact at trust-region
centers. This would require the development of a more general error-aware
trust-region theory based on the work in
\cite{kouri_trust-region_2013, kouri_inexact_2014, zahr_phd_2016}.
Another interesting avenue of research is to extend the reduced topology
optimization approach to more complex topology optimization problems including
non-compliance optimization, nonlinear and multiphysics problems, and problems
with solution-dependent and non-convex constraints.


%% file: adjoint.tex
\section{Derivation of the objective function gradient by adjoint method}
\label{sec:adjoint}
As $J(\dparam) = j(\dstvcL^\star(\dparam), \ddens^\star(\dparam))$, the gradient can be
expanded as
\begin{equation} \label{eqn:grad-chain0}
 \nabla J(\psibold) \coloneqq
 \pder{J}{\psibold}(\psibold)^T =
 \pder{\ubm^\star}{\psibold}(\psibold)^T
 \pder{j}{\ubm}(\ubm^\star(\psibold),\rhobold^\star(\psibold))^T +
 \pder{\rhobold^\star}{\psibold}(\psibold)^T
 \pder{j}{\rhobold}(\ubm^\star(\psibold),\rhobold^\star(\psibold))^T.
\end{equation}
The gradient $\nabla J$ requires matrix-vector products of two sensitivities,
$\ds{\pder{\ubm^\star}{\psibold}}$ and
$\ds{\pder{\rhobold^\star}{\psibold}}$, with the partial derivatives of
$\func{j}{\Rbb^{N_\ubm}\times\Rbb^{N_e}}{\Rbb}$, which can be derived
analytically from the expression for $j(\ubm, \rhobold)$.

We first consider the evaluation of $\ds{\pder{\rhobold^\star}{\psibold}}^T \vbm$ for an arbitrary $\vbm \in \Rbb^{N_e}$.
First, from the definition of $\func{\rho_e^\star}{\Rbb^{N_v}}{\Rbb}$ in \eqref{eqn:cdens-star},  
\begin{equation} \label{eqn:dens-adj0}
 \pder{\rhobold^\star}{\psibold}^T\vbm =
 \sum_{e=1}^{N_e}\pder{\rho_e^\star}{\psibold}^Tv_e =
 \pder{\phibold^\star}{\psibold}^T
 \sum_{e=1}^{N_e} \frac{v_e}{N_v^e} \Qbm_e\onebold =
 \pder{\phibold^\star}{\psibold}^T \qbm(\vbm),
\end{equation}
where $\qbm(\vbm) = \sum_{e=1}^{N_e} \frac{v_e}{N_v^e} \Qbm_e\onebold$.
Next, from the definition of $\phibold^\star$, the Helmholtz
residual is zero irrespective of variations in $\psibold \in \Rbb^{N_e}$,
which in turn implies the total derivative of the Helmholtz residual
with respect to $\psibold$ is zero, i.e.,
\begin{equation} \label{eqn:helm-sens0}
 \oder{\Rbm_\phibold(\phibold^\star(\psibold),\psibold)}{\psibold} = 
 \zerobold.
\end{equation}
Expanding the total derivative in \eqref{eqn:helm-sens0} via the chain rule
and dropping arguments, we have
\begin{equation*} 
 \pder{\Rbm_\phibold}{\phibold}\pder{\phibold^\star}{\psibold} +
 \pder{\Rbm_\phibold}{\psibold} = \zerobold,
\end{equation*}
which can be rearranged to obtain the following expression for $\ds{\pder{\phibold^\star}{\psibold}^T} \qbm(\vbm)$: 
\begin{equation*} 
 \pder{\phibold^\star}{\psibold}^T\qbm(\vbm) =
 -\pder{\Rbm_\phibold}{\psibold}^T\pder{\Rbm_\phibold}{\phibold}^{-T}\qbm(\vbm).
\end{equation*}
From the definition of the Helmholtz residual
\eqref{eqn:helm-res0} and \eqref{eqn:helm-res1}, we have
\begin{equation*} 
  \pder{\Rbm_\phibold}{\phibold}(\phibold;\psibold) = \Hbm,
  \qquad
 \pder{\Rbm_\phibold}{\psi_e}(\phibold;\psibold) =
 -\Qbm_e \bbm_e, \quad e = 1,\dots,N_e
\end{equation*}
where summation is \emph{not} implied over the repeated index, which yields
\begin{equation*} 
 \pder{\phibold^\star}{\psi_e}^T\qbm(\vbm) =
 -\bbm_e^T\Qbm_e^T\Hbm^{-T}\qbm(\vbm) = -\bbm_e^T\Qbm_e^T\mubold(\vbm),
 \qquad e = 1,\dots,N_e,
\end{equation*}
where $\mubold(\vbm) \in \Rbb^{N_v}$ is the adjoint variable that satisfies~\eqref{eqn:helm-adj}, $\Hbm^T\mubold(\vbm) = \qbm(\vbm)$.



We next consider the evaluation of $\ds{\pder{\ubm^\star}{\psibold}}^T \vbm$ for an arbitrary $\vbm \in \Rbb^{N_\ubm}$.
From the definition of $\ubm^\star$, the elasticity
residual is zero irrespective of variations in $\psibold \in \Rbb^{N_e}$,
which in turn implies the total derivative of the elasticity residual
with respect to $\psibold$ is zero, i.e.,
\begin{equation} \label{eqn:linelast-sens0}
 \oder{\rbm_\ubm(\ubm^\star(\psibold),\rhobold^\star(\psibold))}{\psibold} = 
 \zerobold.
\end{equation}
Expanding the total derivative in \eqref{eqn:linelast-sens0} via the chain
rule and dropping arguments, we have
\begin{equation} \label{eqn:linelast-sens1}
 \pder{\rbm_\ubm}{\ubm}\pder{\ubm^\star}{\psibold} +
 \pder{\rbm_\ubm}{\rhobold}\pder{\rhobold^\star}{\psibold} = \zerobold,
\end{equation}
which can be rearranged to obtain the following expression for the product of $\ds{\pder{\ubm^\star}{\psibold}^T}$ with $\ds{\pder{j}{\ubm}^T}$
\begin{equation} \label{eqn:linelast-adj0}
 \pder{\ubm^\star}{\psibold}^T\pder{j}{\ubm}^T =
 -\pder{\rhobold^\star}{\psibold}^T\pder{\rbm_\ubm}{\rhobold}^T
  \pder{\rbm_\ubm}{\ubm}^{-T}\pder{j}{\ubm}^T =
 -\pder{\rhobold^\star}{\psibold}^T\pder{\rbm_\ubm}{\rhobold}^T
  \lambdabold^\star(\psibold);
\end{equation}
here $\lambdabold^\star(\psibold)$ satisfies the linear elasticity adjoint equation \eqref{eqn:linelast-adj1}, $\Kbm(\rhobold^\star(\psibold))^T\lambdabold^\star(\psibold) =  \pder{j}{\ubm}(\ubm(\psibold),\psibold)^T$, 
where we have appealed to the definition of the elasticity residual
\eqref{eqn:linelastU} to obtain $\pder{\rbm_\ubm}{\ubm}(\ubm;\rhobold) = \Kbm(\rhobold)$.



Combining \eqref{eqn:grad-chain0}, \eqref{eqn:dens-adj0}, \eqref{eqn:linelast-adj0}, the gradient $\nabla J$ can be written compactly as \eqref{eq:pap_sensitivity}, $\nabla J = \pder{\rhobold^\star}{\psibold}^T \left(\pder{j}{\rhobold}^T-\pder{\rbm_\ubm}{\rhobold}^T \lambdabold^\star\right)$.

%% file: tr_proofs.tex
\section{Proof of global convergence of error-aware trust-region method}
\label{sec:tr-proofs}
The proofs in this section directly follow those setforth in
\cite{kouri_trust-region_2013}.

\begin{proof}[Proof of Theorem~\ref{thm:fcd}]
Let 
\begin{equation}
 \delta_k \coloneqq \kappa_{\nabla\vartheta}^{-1}\Delta_k, \qquad
 \Dcal_k \coloneqq \{\xbm\in\Ccal\mid\norm{\xbm-\xbm_k}_2\leq \delta_k\}.
\end{equation}
We will show that (i) $\Dcal_k \subset \Bcal_k$, (ii) there exists $\xbm \in \Dcal_k$ that satisfies the fraction of (generalized) Cauchy decrease condition, and hence (iii) there exists $\xbm \in \Bcal_k$ that satisfies the condition.


First we show that $\Dcal_k\subset\Bcal_k$. Take any $\xbm\in\Dcal_k$,
which can be written as $\xbm = \xbm_k+\pbm_k$ for
$\norm{\pbm_k}_2\leq\delta_k$. Then
\begin{equation}
 \vartheta_k(\xbm) = \vartheta_k(\xbm_k+\pbm_k) =
 \vartheta_k(\xbm_k) + \nabla\vartheta_k(\xibold)^T\pbm_k \leq
 \left|\nabla\vartheta_k(\xibold)^T\pbm_k\right| \leq
 \kappa_{\nabla\vartheta}\delta_k \leq \Delta_k
\end{equation}
for some $\xibold = t \xbm + (1-t) \xbm_k$, $t\in[0,1]$; here, the second equality follows from the continuous differentiability of $\vartheta$ and the mean value theorem, the
first inequality follows from
Assumption~\ref{assume:trcon}.\ref{assume:trcon:val} ($\vartheta_k(\xbm_k)=0$),
the second inequality follows from the Cauchy-Schwartz inequality,
and the last inequality follows from the definition of $\delta_k$. Since
$\xbm\in\Ccal$ and $\vartheta_k(\xbm)\leq\Delta_k$, we have $\xbm\in\Bcal_k$,
which establishes the inclusion $\Dcal_k\subset\Bcal_k$.

Next we recall a standard result from trust-region theory
(Theorem~12.2.2 of \cite{conn_trust-region_2000}) that establishes
the existence of a point $\xbm\in\Dcal_k$ such that
\begin{equation}
m_k(\xbm_k) - m_k(\xbm) \geq
\kappa \chi(\xbm_k)
\min\left[\frac{\chi(\xbm_k)}{\beta_k}, \delta_k, 1\right]
\end{equation}
for some $\kappa\in(0,1)$. By $\Dcal_k\subset\Bcal_k$ and
the definition of $\delta_k$, there exists $\xbm\in\Bcal_k$
such that
\begin{equation}
m_k(\xbm_k) - m_k(\xbm) \geq
\kappa \chi(\xbm_k)
\min\left[\frac{\chi(\xbm_k)}{\beta_k}, \kappa_{\nabla\vartheta}^{-1}\Delta_k, 1\right],
\end{equation}
which completes the proof.
\end{proof}

\begin{lemma} \label{lemma:delta-zero}
Suppose Assumptions~\ref{assume:con}.\ref{assume:con:c2}-\ref{assume:con}.\ref{assume:con:qual},~\ref{assume:obj}.\ref{assume:obj:c2}-\ref{assume:obj}.\ref{assume:obj:bnd},~\ref{assume:model}.\ref{assume:model:c2}-\ref{assume:model}.\ref{assume:model:hess},~\ref{assume:trcon}.\ref{assume:trcon:c2}-\ref{assume:trcon}.\ref{assume:trcon:val} hold and there
exists $\epsilon>0$ and $K > 0$ such that $\chi(\xbm_k) \geq \epsilon$ for all $k > K$.
Then the sequence of trust-region radii $\{\Delta_k\}$
produced by Algorithm~\ref{alg:etr} satisfies
\begin{equation*}
 \sum_{k=1}^\infty \Delta_k < \infty.
\end{equation*}
\begin{proof}

  We first consider the case with a finite number of successful iterations. 
In this case, there exists $K' > 0$ such that all iterations $k > K'$ are unsuccessful. Then,
\begin{equation*}
\sum_{k=1}^\infty \Delta_k = \sum_{k=1}^{K'} \Delta_k +
                       \sum_{k=K'+1}^\infty \Delta_k
= C + \sum_{k=K'+1}^\infty \Delta_k,
\end{equation*}
where $\displaystyle{C = \sum_{k=1}^{K'} \Delta_k < \infty}$.  Since
iterations $k > K'$ are unsuccessful,
$\Delta_{k+1} \leq \gamma_2\Delta_k$ and
$\displaystyle{\sum_{k={K'}+1}^\infty \Delta_k}$ is bounded above by a
geometric series, implying the infinite sum is finite. Therefore, the result
holds if there are a finite number of successful iterations.

We now consider the case with an infinite sequence of successful iterations $\{k_i\}$.  In this case, for all $i$ such that $k_i > K$,
\begin{equation*}
 \begin{aligned}
  F(\xbm_{k_i}) - F(\xbm_{k_{i+1}}) &\geq
  F(\xbm_{k_i}) - F(\xbm_{k_{i}+1})  =
  F(\xbm_{k_i}) - F(\hat\xbm_{k_i})
  \geq \eta_1(m_{k_i}(\xbm_{k_i}) - m_{k_i}(\hat\xbm_{k_i})) \\
  &\geq \eta_1\kappa\chi(\xbm_{k_i})
      \min\left[\frac{\chi(\xbm_{k_i})}{\beta_k},~
                 \kappa_{\nabla\vartheta}^{-1}\Delta_{k_i},1\right] 
  \geq \eta_1\kappa\epsilon
      \min\left[\frac{\epsilon}{\beta},~
                 \kappa_{\nabla\vartheta}^{-1}\Delta_{k_i},1\right],
 \end{aligned}
\end{equation*}
for some constant $\kappa\in(0,1)$.  (Note that the subscript for the second $\xbm$ in the first and second expressions are $k_{i+1}$ and $k_i + 1$, respectively.)  
The first inequality holds because
the sequence $\{F(\xbm_k)\}$ is non-increasing owing to the step acceptance
condition in Algorithm~\ref{alg:etr}, and the first equality holds because
iteration $k_i$ is successful (by construction). The remaining inequalities
follow from the step acceptance condition in Algorithm~\ref{alg:etr}, the
the fraction of Cauchy decrease (\ref{eqn:fcd}), and the assumption that
$\chi(\xbm_{k_i}) \geq \epsilon$ for all $k_i > K$. Summing over all $i > I$ for $k_I = K$,
\begin{equation*}
\eta_1\kappa\epsilon\sum_{i \geq I}
\min\left[
\frac{\epsilon}{\beta},
~\kappa_{\nabla\vartheta}^{-1}\Delta_{k_i},1\right]
\leq \sum_{i \geq I} ( F(\xbm_{k_i}) - F(\xbm_{k_{i}+1}) )
= F(\xbm_{k_I}) - \lim_{i \rightarrow \infty} F(\xbm_{k_i}) < \infty,
\end{equation*}
where the equality follows from the telescoping series, and the finiteness of the limit follows from $F$ being bounded below. 
Since $\epsilon/\beta$ and $1$ are bounded away from zero, the inequality above
implies that $\sum_{i=1}^\infty \Delta_{k_i} < \infty$.

Let $\Scal \subset \Nbb$ be the ordered set of indices of successful
iterations. For every $k \notin \Scal$, $\Delta_k \leq
\gamma_2^{k-j(k)}\Delta_{j(k)}$, where $j(k) \in \Scal$ is the largest index
such that $j(k) < k$, i.e. $j(k)$ represents the last successful iteration
before the unsuccessful iteration $k$. Summing over all $k \notin \Scal$,
\begin{equation*}
\sum_{k \notin \Scal} \Delta_k \leq
\sum_{k \notin \Scal} \gamma_2^{k-j(k)}\Delta_{j(k)} =
\sum_{i = 1}^\infty \sum_{j(i) < k < j(i+1)}\gamma_2^{k-j(i)}\Delta_{j(i)}
\leq \frac{1}{1 - \gamma_2}\sum_{i=1}^\infty \Delta_{j(i)} =
\frac{1}{1 - \gamma_2}\sum_{k \in \Scal} \Delta_k
< \infty.
\end{equation*}
Then,
\begin{equation*}
\sum_{k=1}^\infty \Delta_k = \sum_{k \in \Scal} \Delta_k +
                      \sum_{k \notin \Scal} \Delta_k
\leq \left(1 + \frac{1}{1-\gamma_2}\right) \sum_{k \in \Scal} \Delta_k
< \infty.
\end{equation*}
This proves the desired result.
\end{proof}
\end{lemma}

\begin{lemma} \label{lemma:rho-one}
  Suppose Assumptions~\ref{assume:con}.\ref{assume:con:c2}-\ref{assume:con}.\ref{assume:con:qual},~\ref{assume:obj}.\ref{assume:obj:c2}-\ref{assume:obj}.\ref{assume:obj:bnd},~\ref{assume:model}.\ref{assume:model:c2}-\ref{assume:model}.\ref{assume:model:hess},~\ref{assume:trcon}.\ref{assume:trcon:c2}-\ref{assume:trcon}.\ref{assume:trcon:err} hold and there exists $\epsilon>0$ and $K > 0$ such that
$\chi(\xbm_k)\geq \epsilon$ for all $k > K$. Then the ratios
  $\{\varrho_k\}$ produced by Algorithm~\ref{alg:etr} converge to one.
  
\begin{proof}
We appeal to the asymptotic error bound on the approximation model in
Assumption~\ref{assume:trcon}.\ref{assume:trcon:err} and the fact
that the candidate step lies within the trust region (i.e.,
$\hat\xbm_k \in \Bcal_k$) to obtain
\begin{equation}
 |F(\hat\xbm_k)-m_k(\hat\xbm_k)| \leq
 \zeta\vartheta_k(\hat\xbm_k)^\nu \leq \zeta\Delta_k^\nu.
\end{equation}
From Theorem~\ref{thm:fcd} and the convergence criteria on the trust-region
subproblem in Algorithm~\ref{alg:etr}, we have
\begin{equation*}
 m_k(\xbm_k) - m_k(\hat\xbm_k) \geq
 \kappa\chi(\xbm_k)
 \min\left[\frac{\chi(\xbm_k)}{\beta_k},
            \kappa_{\nabla\vartheta}^{-1}\Delta_k,1\right]
\end{equation*}
for constant $\kappa\in(0,1)$.
Then, there exists $K' > K$ such that, for all $k > K$,
\begin{equation*}
 m_k(\xbm_k) - m_k(\hat\xbm_k) \geq
  \kappa_{\nabla\vartheta}^{-1}\kappa\epsilon\Delta_k,
\end{equation*}
because $\Delta_k \to 0$ by Lemma~\ref{lemma:delta-zero},
$\chi(\xbm_k) \geq \epsilon$ by the assumption, and $\beta_k\leq\beta$.
Combining
the above inequalities yields
\begin{equation*}
 \left|\varrho_k - 1\right| =
 \left|\frac{F(\xbm_k)-F(\hat\xbm_k)+m_k(\hat\xbm_k)-m_k(\xbm_k)}
            {m_k(\xbm_k) - m_k(\hat\xbm_k)}\right| =
 \left|\frac{F(\hat\xbm_k)-m_k(\hat\xbm_k)}
            {m_k(\xbm_k) - m_k(\hat\xbm_k)}\right| \leq
 \frac{\zeta\Delta_k^{\nu}}{\kappa_{\nabla\vartheta}^{-1}
                            \kappa\epsilon\Delta_k} =
 \frac{\zeta}{\kappa_{\nabla\vartheta}^{-1}\kappa\epsilon}
              \Delta_k^{\nu-1}
\end{equation*}
for all $k > K'$.
Therefore, $\varrho_k \rightarrow 1$ since $\Delta_k \rightarrow 0$
(Lemma~\ref{lemma:delta-zero}) and $\nu > 1$.
\end{proof}
\end{lemma}

\begin{proof}[Proof of Theorem~\ref{thm:globconv}]
  We prove by contradiction.  Suppose there exists $\epsilon > 0$ such that
$\chi(\xbm_k) \geq \epsilon$ for all $k > K$ for some $K > 0$. 
By Lemma~\ref{lemma:rho-one}, there exists $K'' \geq K$ such that, for all
$k > K''$, $\varrho_k$ is sufficiently close to
$1$ and the corresponding step is successful. From
Algorithm~\ref{alg:etr}, this implies
$\Delta_{K''} \leq \Delta_k \leq \Delta_\text{max}$. This result contradicts
Lemma~\ref{lemma:delta-zero} and hence
$ \liminf_{k \rightarrow \infty}~\chi(\xbm_k) = 0.$
\end{proof}


\begin{proof}[Proof of Theorem~\ref{thm:globconv2}]
To prove the result we need to verify the objective function $F$
(\ref{eqn:obj-topopt}), constraints $\Ccal$ (\ref{eqn:feaset-topopt}),
approximation model $m_k$ (\ref{eqn:model-topopt}), and trust-region
constraint $\vartheta_k$ (\ref{eqn:trcon-topopt}) satisfy
Assumptions~\ref{assume:con}-\ref{assume:trcon}.  The proofs are provided in~\ref{sec:assume_proofs}.
\end{proof}

\section{Verification of Assumptions~\ref{assume:con}--\ref{assume:trcon} for topology optimization}
\label{sec:assume_proofs}
The following lemmas verify Assumptions~\ref{assume:con}--\ref{assume:trcon} for ROM-accelerated topology optimization introduced in~Section~\ref{sec:tr:etr-topopt}.

\begin{lemma}[Verification of Assumption~\ref{assume:con} for topology optimization.]
  \label{lem:assume:con}
  The constraints of the topology optimization problem~\eqref{eqn:topopt0} satisfies Assumption~\ref{assume:con}.

  \begin{proof}
    We wish to show the following: (C1) $\Ccal = \cap_{i=1}^m \Ccal_i$, where $\Ccal_i=\{\xbm\in\Rbb^N\mid c_i(\xbm)\geq 0\}$ and each $\func{c_i}{\Rbb^N}{\Rbb}$ is twice continuously differentiable on $\Rbb^N$; (C2) $\Ccal$ is nonempty, closed, convex; (C3) a first-order constraint qualification holds at any critical point of \eqref{eqn:optim-gen}.
    The condition (C1) is satisfied because the constraint set \eqref{eqn:feaset-topopt} is the intersection of linear constraints, which comprise simple bounds and a single general linear constraint.
    The condition (C2) is satisfied because (i) $\Ccal$ is convex and closed because it is the intersection of linear and non-strict inequality constraints and (ii) $\Ccal$ is nonempty because  $\zerobold \in \Ccal$ since $V\geq 0$.  Finally, the condition (C3) is satisfied since all constraints are linear functions, a first-order constraint qualification holds (Lemma 12.7 of \cite{nocedal_numerical_2006}).
\end{proof}
\end{lemma}

\begin{lemma}[Verification of Assumption~\ref{assume:obj} for topology optimization.]
  \label{lem:assume:obj}
   Suppose the objective function $j: \Rbb^{N_\ubm} \times P \to \Rbb$ is twice continuously differentiable.  The objective function of the topology optimization problem~\eqref{eqn:topopt0} satisfies Assumption~\ref{assume:obj}.
  \begin{proof}
    \label{proof:assume1}
  We wish to show that the objective function $J : \Ccal \to \Rbb$, for $\Ccal$ defined by \eqref{eqn:feaset-topopt}, (C1) is twice continuous differentiable, (C2) is bounded from below, and (C3) has bounded second derivative.

  We first consider (C1). Given $J(\psibold) = j(\ubm^\ast(\rhobold^\star(\psibold)), \rhobold^\star(\psibold))$, we need to show each of the following is twice continuously differentiable: (i) $\ubm^\ast: P \to \Rbb^{N_\ubm}$; (ii) $\rhobold^\star: \Psi \to P$; (iii) $j: \Rbb^{N_\ubm} \times P \to \Rbb$.  To analyze the differentiability of (i) $\ubm^\ast: P \to \Rbb^{N_\ubm}$, we differentiate the linear elasticity equation $\Kbm(\rhobold) \ubm^\ast(\rhobold) = \fbm$ with respect to $\rhobold$ twice to obtain
  \begin{equation*}
    \pder{^2 \ubm^*}{\rhobold_p \partial \rhobold_q}
    =
    \Kbm^{-1} \left[
     - \pder{^2 \Kbm}{\rhobold_p \partial \rhobold_q}
     + \pder{\Kbm}{\rhobold_p} \Kbm^{-1} \pder{\Kbm}{\rhobold_q}
     + \pder{\Kbm}{\rhobold_q} \Kbm^{-1} \pder{\Kbm}{\rhobold_p}
     \right] \Kbm^{-1} \fbm
  \end{equation*}
  for $p,q = 1,\dots,N_e$, where all terms are evaluated at $\rhobold$.  Terms in the right hand side are continuous because the stiffness matrix $\Kbm(\rhobold)$ is nonsingular for all $\rhobold \in P$, and $\Kbm(\rhobold)$ is polynomial in $\rhobold$.  To analyze the differentiability of (ii), $\rhobold^\star: \Rbb^{N_e} \to P$, we differentiate the Helmholtz equation $\Hbm \phibold^\star(\psibold) = \bbm(\psibold)$ with respect to $\psibold$ twice to obtain
  \begin{equation*}
    \pder{^2 \phibold^\star}{\psibold_p \partial \psibold_q}
    = \Hbm^{-1} \pder{^2 \bbm}{\psibold_p \partial \psibold_q} = 0;
  \end{equation*}
  note that the second derivative vanishes because $\bbm(\psibold)$ is linear in $\psibold$.  The twice differentiability of (iii) $j: \Rbb^{N_\ubm} \times P \to \Rbb$ follows from the assumption of the lemma.

  We next consider (C2).  It suffices to show that $J: \Ccal \to \Rbb$ is continuous and $\Ccal$ is compact.  The function $J$ is continuous because it is twice continuously differentiable as proven for (C1).  The domain $\Ccal$ is compact by Lemma~\ref{lem:assume:con}.


  We finally note that (C3) follows from (C1) and the fact $\Ccal$ is compact by Lemma~\ref{lem:assume:con}.
\end{proof}
\end{lemma}

\begin{lemma}[Verification of Assumption~\ref{assume:model} for topology optimization.]
  \label{lem:assume:model}
   Suppose the objective function $j: \Rbb^{N_\ubm} \times P \to \Rbb$ is twice continuously differentiable.  The approximation model introduced in Section~\ref{sec:tr:etr-topopt} for the topology optimization problem~\eqref{eqn:topopt0} satisfies Assumption~\ref{assume:model}.
\begin{proof}
  We wish to show that the ROM approximation of the objective function $J_k : \Rbb^{N_e} \to \Rbb$ (C1) is twice continuously differentiable, (C2) satisfies $J_k(\psibold^{(k)}) = J(\psibold^{(k)})$ where $\psibold^{(k)}$ is a trust-region center, (C3) satisfies $\nabla J_k(\psibold^{(k)}) = \nabla J(\psibold^{(k)})$ where $\psibold^{(k)}$ is a trust-region center, and (C4) yields $\beta_k \coloneqq 1 + \max_{\psibold \in \Ccal} \| \nabla^2 J_k(\psibold) \|$ that is uniformly bounded from the above.

  We first note that the conditions (C2) and (C3) are satisfied as a consequence of Theorems~\ref{thm:rom_output_error} and \ref{thm:rom_sensitivity_error}, respectively, and our choice of reduced basis which includes the state at the trust-region center.

  We next consider condition (C1).  The verification of this condition is similar to the verification of condition (C1) of Lemma~\ref{lem:assume:obj}.  Given $J_k(\psibold) = j(\ubm_k^\ast(\rhobold^\star(\psibold)), \rhobold^\star(\psibold))$,  we need to show each of the following is twice continuously differentiable: (i) $\ubm_k^\ast: P \to \Rbb^{N_\ubm}$; (ii) $\rhobold^\star: \Rbb^{N_e} \to P $; (iii) $j: \Rbb^{N_\ubm} \times P \to \Rbb$.  We have shown (ii) and (iii) are satisfied in the verification of Lemma~\ref{lem:assume:obj}.  To analyze the differentiability of $\ubm_k^\ast: P \to \Rbb^{N_\ubm}$, we differentiate $\hat \Kbm_k(\rhobold) \hat \ubm_k^\ast(\rhobold) = \hat \fbm$ twice and appeal to $\ubm_k^\ast(\rhobold) = \Phibold_k \hat \ubm_k^\ast(\rhobold)$ to obtain
  \begin{equation}
    \pder{^2 \ubm_k^\ast}{\rhobold_p \partial \rhobold_q}
    =
    \Phibold_k \pder{^2 \hat \ubm_k^\ast}{\rhobold_p \partial \rhobold_q}
    =
    \Phibold_k \hat \Kbm^{-1}_k \left[
     - \pder{^2 \hat \Kbm_k}{\rhobold_p \partial \rhobold_q}
     + \pder{\hat \Kbm_k}{\rhobold_p} \hat \Kbm_k^{-1} \pder{\hat \Kbm_k}{\rhobold_q}
     + \pder{\hat \Kbm_k}{\rhobold_q} \hat \Kbm_k^{-1} \pder{\hat \Kbm_k}{\rhobold_p}
     \right] \hat \Kbm_k^{-1} \hat \fbm
    \label{eq:pder_uk_pq}
  \end{equation}
  for $p,q = 1,\dots,N_e$, where all terms are evaluated at $\rhobold$.  Because $\hat \Kbm(\rhobold) \coloneqq \Phibold_k^T \Kbm(\rhobold) \Phibold_k$, $\Kbm(\rhobold)$ is symmetric positive definite for all $\rhobold \in P$, and $\Phibold_k \in \Rbb^{N_\ubm \times k}$ has orthonormal columns, the maximum singular value of $\hat \Kbm(\rhobold)^{-1}$ satisfies $\sigma_{\rm max}(\hat \Kbm(\rhobold)^{-1}) \leq \sigma_{\rm max}(\Kbm(\rhobold)^{-1}) < \infty$.  In addition, $\hat \Kbm(\rhobold)$ is polynomial in $\rhobold$ and hence is smooth in $\rhobold$. Consequently, $\pder{^2 \ubm_k^*}{\rhobold_p \partial \rhobold_q}$ is continuous.  Because (ii) and (iii) are also twice continuously differentiable as shown for condition (C1) of Lemma~\ref{lem:assume:obj}, it follows that $J_k(\psibold)$ is twice continuously differentiable for any $k$.

  We finally consider condition (C4). It suffices to show that $\{ \| \nabla^2 J_k(\psibold) \| \}_{k \in \Nbb,\psibold \in \Ccal}$ is uniformly bounded from the above.  The boundedness follows from the twice continuous differentiability of $J_k(\psibold)$ (i.e., (C1)) and the compactness of $\Ccal$.  To see that this bound is uniform, we first note that $J_k(\psibold) = j(\ubm_k^\ast(\rhobold^\star(\psibold)), \rhobold^\star(\psibold))$ and the only term that depends on $k$ is $\ubm_k^\ast$.  We next note that $\hat \Kbm(\rhobold) \coloneqq \Phibold_k^T \Kbm(\rhobold) \Phibold_k$, $\Kbm(\rhobold)$ is symmetric positive definite for all $\rhobold \in P$, and $\Phibold_k \in \Rbb^{N_\ubm \times k}$ has orthonormal columns.  It follows that $\pder{^2 \ubm_k^*}{\rhobold_p \partial \rhobold_q}$ in \eqref{eq:pder_uk_pq} is bounded independent of $k$ because the range of singular values of $\hat \Kbm_k^{-1}$, $\hat \Kbm_k$, $\pder{\hat \Kbm_k}{\rhobold_q}$, and $\pder{^2 \hat \Kbm_k}{\rhobold_p \partial \rhobold_q}$ are bounded by the range of singular values of the respective full-order model entities, $\Kbm^{-1}$, $\Kbm$, $\pder{\Kbm}{\rhobold_q}$, and $\pder{^2 \Kbm}{\rhobold_p \partial \rhobold_q}$.
\end{proof}
\end{lemma}

\begin{lemma}[Verification of Assumption~\ref{assume:trcon} for topology optimization.]
  \label{lem:assume:trcon}
  The trust-region constraint introduced in Section~\ref{sec:tr:etr-topopt} for the topology optimization problem~\eqref{eqn:topopt0} satisfies Assumption~\ref{assume:trcon}.

\begin{proof}
  We wish to show that, for  $\vartheta_k(\psibold) \coloneqq \| \rbm(\ubm_{k}^\ast(\rhobold^\star(\psibold));\rhobold^\star(\psibold))\|^{1-\epsilon}$ and $\epsilon > 0$ but $\epsilon \ll 1$, (C1) $\vartheta_k(\cdot)$ is twice continuously differentiable, (C2) $\max_{\psibold \in \Ccal} \| \nabla \vartheta_k(\psibold) \|$ is uniformly bounded from the above, (C3) $\vartheta_k(\psibold^{(k)}) = 0$ for the trust-region center $\psibold^{(k)}$, and (C4) there exist $\zeta > 0$ and $\nu > 1$ independent of $k$ such that $|J(\psibold) - J_k(\psibold)| \leq \zeta \vartheta_k(\psibold)^\nu$ for all $k \in \Nbb$.

  We first consider (C1). Given (C1) of Lemma~\ref{lem:assume:model} is satisfied, we only need to show that $\rbm: \Rbb^{N_\ubm} \times P \to \Rbb^{N_\ubm}$ is twice continuously differentiable.  We recall that $\rbm(\ubm, \rhobold) = \Kbm(\rhobold) \ubm - \fbm$.  We observe that $\rbm(\ubm,\rhobold)$ is twice continuously differentiable in $\ubm$ and $\rhobold$ because $\rbm(\ubm,\rhobold)$ is linear in $\ubm$ and $\Kbm(\rhobold)$ is polynomial in $\rhobold$. 

  We next note that (C2) is satisfied because $\rbm(\ubm_{k}^\ast(\rhobold^\star(\psibold));\rhobold^\star(\psibold))$ is twice continuously differentiable in $\psibold$ and $\Ccal$ is compact.  The uniform boundedness follows from the same argument as Lemma~\ref{lem:assume:model} (C4).

  We then note that (C3) is satisfied because $\ubm_k^\star(\psibold^{(k)}) = \ubm(\psibold^{(k)})$ for any trust-region center $\psibold^{(k)}$ as a consequence of Theorem~\ref{thm:rom_primal_error} and our choice of reduced basis which includes the state at the trust-region center.

  We finally consider (C4). We first recall Theorem~\ref{thm:rom_output_error_simple}: for any $\psibold \in \Bcal \subset \Psi$, 
   \begin{align*}
    | J(\psibold) - J_k(\psibold) |
    \leq \| \lambdabold^\star(\psibold) \|_2 \| \rbm(\ubm_k^\star(\psibold); \rhobold^\star(\psibold)) \|_2 + \sigma_{\rm max}(\Bbm(\psibold)) \| \rbm(\ubm_k^\star(\psibold); \rhobold^\star(\psibold)) \|_2^2,
   \end{align*}
   where $\sigma_{\rm max}(\Bbm(\psibold))$ is the maximum singular value of 
   \begin{equation}
     \label{eq:tr_proof_Bbm}
    \Bbm(\psibold) \coloneqq \Kbm(\rhobold^\star(\psibold))^{-1} \left[ \int_{\theta = 0}^1 \pder{^2j}{\ubm^2}(\theta \ubm^\star(\psibold) + (1- \theta) \ubm^\star_k(\psibold)) \theta d\theta \right] \Kbm^{-1}(\rhobold^\star(\psibold)).
   \end{equation}
   We immediately observe that the desired inequality, $|J(\psibold) - J_k(\psibold)| \leq \zeta (\| \rbm(\ubm_{k}^\star(\psibold);\rhobold^\star(\psibold))\|^{1-\epsilon}) ^\nu$, is satisfied for $\zeta = \max \left\{ \max_{\psibold \in \Bcal_k} \| \lambdabold^\star(\psibold) \|_2 , \max_{\psibold \in \Bcal_k} \sigma_{\rm max}(\Bbm(\psibold)) \Delta_{\rm max} \right\}$ and $\nu = 1/(1-\epsilon) > 1$.

   It remains to show that $\zeta$ is bounded independent of $k$.  To this end, we first note that $\max_{\psibold \in \Bcal_k} \| \lambdabold^\star(\psibold) \|_2$ is bounded because the adjoint problem \eqref{eqn:linelast-adj1} is well-posed for all $\psibold \in \Psi$ and $\Bcal_k \subset \Psi$ is compact.
   We second note that $\max_{\psibold \in \Bcal_k} \sigma_{\rm max}(\Bbm(\psibold)) \Delta_{\rm max}$ is bounded independent of $k$ because all terms of $\Bbm(\psibold)$ in \eqref{eq:tr_proof_Bbm} are bounded: $\| \Kbm(\rhobold^\star(\psibold))^{-1} \|_2$ is bounded because $\Kbm(\rhobold)$ is nonsingular for all $\rhobold \in \Bcal_k \subset P$ and $\Bcal_k$ is compact;  the term in the square bracket is bounded because (i) $j$ is twice continuously differentiable by assumption, (ii) $\ubm^\star(\psibold)$ and $\ubm^\star_k(\psibold)$ are continuous in $\psibold \in \Bcal_k \subset P$, and (iii) $\psibold$ belongs to a compact set $\Bcal_k$.

   We note that for $\epsilon = 0$ the condition (C4) is not satisfied but conditions (C1)--(C3) are satisfied.
\end{proof}
\end{lemma}

%% file: rom_proofs.tex
\section{Proofs of error estimates}
We prove Theorems~\ref{thm:rom_primal_error}--\ref{thm:rom_output_error_simple}.  For notational brevity, we suppress the arguments $\psibold$ and $\rhobold^\star(\psibold)$ for all functions and forms throughout the proofs in this section.
\label{sec:rom_proofs}
\begin{proof}[Proof of Theorem~\ref{thm:rom_primal_error}]
  By the definition of the residual,
  \begin{equation}
    \Kbm (\ubm^\star - \ubm_k^\star) = \fbm - \Kbm \ubm_k^\star = -\rbm(\ubm_k^\star).
    \label{eq:rom_proof_u_error}
  \end{equation}
  It follows that
  $\Kbm^{1/2} (\ubm^\star - \ubm_k^\star) = - \Kbm^{-1/2} \rbm(\ubm_k^\star)$,
  and hence
  \begin{equation*}
    \| \ubm^\star - \ubm_k^\star\|_{\Kbm}
    = \| \Kbm^{-1/2} \rbm(\ubm_k^\star) \|_2
    \leq \sigma_{\mathrm{min}}(\Kbm)^{-1/2} \| \rbm(\ubm_k^\star) \|_2,    
  \end{equation*}
  which is the desired result.
\end{proof}
\begin{proof}[Proof of Theorem~\ref{thm:rom_adjoint_error}]
  We appeal to the definition of the adjoint residual and the mean-value theorem to obtain
  \begin{equation}
    \Kbm (\lambdabold^\star - \lambdabold_k^\star)
    =
    \pder{j}{\ubm}(\ubm^\star)^T - \pder{j}{\ubm}(\ubm^\star_k)^T - \rbm^{\mathrm{adj}}(\lambdabold^\star_k)
    =
    \overline{\pder{^2j}{\ubm^2}}(\ubm^\star,\ubm^\star_k) (\ubm^\star - \ubm^\star_k)
    - \rbm^{\mathrm{adj}}(\lambdabold^\star_k),
    \label{eq:rom_proof_lambda_error}
  \end{equation}
  where $\overline{\pder{^2j}{\ubm^2}}(\ubm^\star,\ubm^\star_k)  \in \Rbb^{N_\ubm \times N_\ubm}$ is the mean-value linearized Hessian given by
  \begin{equation}
    \overline{\pder{^2j}{\ubm^2}}(\ubm^\star,\ubm^\star_k) 
    \coloneqq
    \int_{\theta = 0}^1 \pder{^2j}{\ubm^2}(\theta \ubm^\star + (1- \theta) \ubm^\star_k) d\theta.
    \label{eq:rom_proof_j2mv}
  \end{equation}
  It follows that
  \begin{align*}
    \Kbm^{1/2} (\lambdabold^\star - \lambdabold_k^\star)
    &=
    \Kbm^{-1/2} \overline{\pder{^2j}{\ubm^2}}(\ubm^\star,\ubm^\star_k) \Kbm^{-1} \Kbm (\ubm^\star - \ubm^\star_k) - \Kbm^{-1/2} \rbm^{\mathrm{adj}}(\lambdabold^\star_k)
    \\
    &=
    -\Kbm^{-1/2} \overline{\pder{^2j}{\ubm^2}}(\ubm^\star,\ubm^\star_k) \Kbm^{-1} \rbm(\ubm^\star_k) - \Kbm^{-1/2} \rbm^{\mathrm{adj}}(\lambdabold^\star_k)
    =
    -\Abm \rbm(\ubm^\star_k) - \Kbm^{-1/2} \rbm^{\mathrm{adj}}(\lambdabold^\star_k),
  \end{align*}
  where $\Abm \coloneqq \Kbm^{-1/2} \overline{\pder{^2j}{\ubm^2}}(\ubm^\star,\ubm^\star_k) \Kbm^{-1}$. 
  Hence
  \begin{equation*}
    \| \lambdabold^\star - \lambdabold_k^\star \|_{\Kbm}
    \leq
      \sigma_{\mathrm{max}}(\Abm) \| \rbm(\ubm^\star_k) \|_2 + 
      \sigma_{\mathrm{min}}(\Kbm)^{-1/2} \| \rbm^{\mathrm{adj}}(\lambdabold^\star_k) \|_2,
  \end{equation*}
  which is the desired relationship.
\end{proof}
\begin{proof}[Proof of Theorem~\ref{thm:rom_output_error}]
  We first note that
  \begin{align}
    J - J_k
    &= j(\ubm^\star) - j(\ubm^\star_k)
    = \int_{\theta = 0}^1 \pder{j}{\ubm}(\theta \ubm^\star + (1-\theta) \ubm^\star_k) (\ubm^\star - \ubm^\star_k) d\theta
    \notag \\
    &=
    - \int_{\theta = 0}^1 (\ubm^\star - \ubm^\star_k)^T \pder{^2j}{\ubm^2}(\theta \ubm^\star + (1-\theta) \ubm^\star_k)(\ubm^\star - \ubm^\star_k) \theta d\theta
    +
    \pder{j}{\ubm}(\ubm^\star) (\ubm^\star - \ubm^\star_k)
    \notag \\
    &=- (\ubm^\star - \ubm^\star_k)^T \left[ \int_{\theta = 0}^1  \pder{^2j}{\ubm^2}(\theta \ubm^\star + (1-\theta) \ubm^\star_k) \theta d\theta \right](\ubm^\star - \ubm^\star_k)
    +
    (\Kbm\lambdabold^\star)^T (\ubm^\star - \ubm^\star_k),
    \label{eq:rom_proof_jmjk}
  \end{align}
  where the first equality follows from the definition of $J$ and $J_k$, the second equality follows from mean-value theorem, the third equality follows from integration by parts, and the last equality follows from the adjoint equation $\Kbm \lambdabold^\star = \pder{j}{\ubm}^T$.    We then appeal to Galerkin orthogonality (i.e., $\vbm_k^T \Kbm (\ubm^\star - \ubm^\star_k) = 0$ $\forall \vbm_k \in \text{Img}(\Phibold_k)$) to obtain
  \begin{equation*}
    J - J_k
    =-  (\ubm^\star - \ubm^\star_k)^T \left[ \int_{\theta = 0}^1 \pder{^2j}{\ubm^2}(\theta \ubm^\star + (1-\theta) \ubm^\star_k) \theta d\theta \right] (\ubm^\star - \ubm^\star_k)
    +
    (\lambdabold^\star - \lambdabold^\star_k) {}^T \Kbm (\ubm^\star - \ubm^\star_k),
  \end{equation*}
  It follows that
  \begin{align*}
    &J - J_k
    \\
    &=
    - (\ubm^\star - \ubm^\star_k)^T \Kbm \Kbm^{-1} \left[ \int_{\theta = 0}^1 \pder{^2j}{\ubm^2}(\theta \ubm^\star + (1-\theta) \ubm^\star_k) \theta d\theta \right] \Kbm^{-1} \Kbm (\ubm^\star - \ubm^\star_k)
    + (\lambdabold^\star - \lambdabold^\star_k)^T \Kbm \Kbm^{-1} \Kbm (\ubm^\star - \ubm^\star_k)
    \\
    &=
    - \rbm(\ubm^\star_k)^T \Kbm^{-1} \left[ \int_{\theta = 0}^1 \pder{^2j}{\ubm^2}(\theta \ubm^\star + (1-\theta) \ubm^\star_k) \theta d\theta \right] \Kbm^{-1} \rbm(\ubm^\star_k)
    + \rbm^{\mathrm{adj}}(\lambdabold^\star_k)^T \Kbm^{-1} \rbm(\ubm^\star_k)
    \\
    &=
    - \rbm(\ubm^\star_k)^T \Bbm \rbm(\ubm^\star_k)
    + \rbm^{\mathrm{adj}}(\lambdabold^\star_k)^T \Kbm^{-1} \rbm(\ubm^\star_k)    
  \end{align*}
  where
  \begin{equation}
    \label{eq:rom_proof_Bbm}
    \Bbm \coloneqq \Kbm^{-1} \left[ \int_{\theta = 0}^1 \pder{^2j}{\ubm^2}(\theta \ubm^\star + (1-\theta) \ubm^\star_k) \theta d\theta \right] \Kbm^{-1}.
  \end{equation}
  
  We hence conclude that
  \begin{equation*}
    | J- J_k |
    \leq
    \sigma_{\rm max}(\Bbm) \| \rbm(\ubm^\star_k) \|^2_2
    + \sigma_{\rm min}(\Kbm)^{-1} \|  \rbm^{\mathrm{adj}}(\lambdabold^\star_k) \|_2 \| \rbm(\ubm^\star_k) \|_2,
  \end{equation*}
  which is the desired relationship.
\end{proof}
\begin{proof}[Proof of Theorem~\ref{thm:rom_sensitivity_error}]
  For clarity, we use tensor notation, with implied sum on repeated indices. We first note that, by \eqref{eq:pap_sensitivity},
  \begin{align*}
    (\nabla J - \nabla J_k)_p &=
    \left( \pder{j}{\rhobold_q}(\ubm^\star) - \lambdabold^\star_i \pder{\rbm_i}{\rhobold_q}(\ubm^\star)
    - \pder{j}{\rhobold_q}(\ubm^\star_k) + \lambdabold^\star_{k,i} \pder{\rbm_i}{\rhobold_q}(\ubm^\star_k)
    \right) \pder{\rhobold^\star_q}{\psibold_p}
    \\
    &=
    \left( \pder{j}{\rhobold_q}(\ubm^\star) - \pder{j}{\rhobold_q}(\ubm^\star_k)
    - \lambdabold^\star_i \pder{\rbm_i}{\rhobold_q}(\ubm^\star)
    + \lambdabold^\star_{k,i} \pder{\rbm_i}{\rhobold_q}(\ubm^\star)
    - \lambdabold^\star_{k,i} \pder{\rbm_i}{\rhobold_q}(\ubm^\star)
    + \lambdabold^\star_{k,i} \pder{\rbm_i}{\rhobold_q}(\ubm^\star_k)     
    \right) \pder{\rhobold^\star_q}{\psibold_p}
    \\
    &=
    \left(
    \overline{\pder{^2j}{\rhobold_q \partial \ubm_s}}(\ubm^\star,\ubm^\star_k)(\ubm^\star - \ubm^\star_k)_s
    - (\lambdabold^\star - \lambdabold^\star_k)_i \pder{\rbm_i}{\rhobold_q}(\ubm^\star)
    - \lambdabold^\star_{k,i} \overline{\pder{^2\rbm_i}{\rhobold_q \partial \ubm_s}}(\ubm^\star,\ubm^\star_k)(\ubm^\star - \ubm^\star_k)_s
    \right) \pder{\rhobold^\star_q}{\psibold_p}
    \\
    &=
    \left[
    \left( \overline{\pder{^2j}{\rhobold_q \partial \ubm_s}}(\ubm^\star,\ubm^\star_k) - \lambdabold^\star_{k,i} \overline{\pder{^2\rbm_i}{\rhobold_q \partial \ubm_s}}(\ubm^\star,\ubm^\star_k) \right) (\ubm^\star - \ubm^\star_k)_s
    - (\lambdabold^\star - \lambdabold^\star_k)_i \pder{\rbm_i}{\rhobold_q}(\ubm^\star)
    \right] \pder{\rhobold^\star_q}{\psibold_p},
  \end{align*}
  where the bared quantities are the mean-value linearizations given by
  $\overline{\pder{^2j}{\rhobold_q \partial \ubm_s}}(\ubm^\star,\ubm^\star_k)
    \coloneqq
    \int_{\theta = 0}^1 \pder{^2j}{\rhobold_q \partial \ubm_s}(\theta \ubm^\star + (1-\theta) \ubm^\star_k) d\theta$
    and 
    $\overline{\pder{^2\rbm}{\rhobold_q \partial \ubm_s}}(\ubm^\star,\ubm^\star_k)
    \coloneqq
    \int_{\theta = 0}^1 \pder{^2\rbm}{\rhobold_q \partial \ubm_s}(\theta \ubm^\star + (1-\theta) \ubm^\star_k) d\theta$.
  We now substitute~\eqref{eq:rom_proof_lambda_error} and \eqref{eq:rom_proof_u_error} to obtain
  \begin{align*}
    &(\nabla J - \nabla J_k)_p
    \\
    &
    =\left[
      \left( \overline{\pder{^2j}{\rhobold_q \partial \ubm_s}}(\ubm^\star,\ubm^\star_k) - \lambdabold^\star_{k,i} \overline{\pder{^2\rbm_i}{\rhobold_q \partial \ubm_s}}(\ubm^\star,\ubm^\star_k) \right) (\ubm^\star - \ubm^\star_k)_s
      \right.
      \\
      & \qquad \left.
      - (\Kbm^{-1})_{im} \left(
      \overline{\pder{^2j}{\ubm_m \partial \ubm_s}}(\ubm^\star,\ubm^\star_k) (\ubm^\star - \ubm^\star_k)_s
      - \rbm^{\mathrm{adj}}(\lambdabold^\star_k)_m \right)
      \pder{\rbm_i}{\rhobold_q}(\ubm^\star)
      \right] \pder{\rhobold^\star_q}{\psibold_p}
    \\
    &=\left[
      \left( \overline{\pder{^2j}{\rhobold_q \partial \ubm_s}}(\ubm^\star,\ubm^\star_k)
      - \lambdabold^\star_{k,i} \overline{\pder{^2\rbm_i}{\rhobold_q \partial \ubm_s}}(\ubm^\star,\ubm^\star_k)
      - \pder{\rbm_i}{\rhobold_q}(\ubm^\star) (\Kbm^{-1})_{il} \overline{\pder{^2j}{\ubm_l \partial \ubm_s}}(\ubm^\star,\ubm^\star_k) 
      \right)
      (-\Kbm^{-1}_{sm}  \rbm(\ubm^\star_k)_m) 
      \right.
      \\
      & \qquad \left.
      + \pder{\rbm_i}{\rhobold_q}(\ubm^\star) (\Kbm^{-1})_{im} 
       \rbm^{\mathrm{adj}}(\lambdabold^\star_k)_m
      \right] \pder{\rhobold^\star_q}{\psibold_p}
  \end{align*}
  The expression simplifies to
  \begin{equation*}
    (\nabla J - \nabla J_k)_p
    =
    \Cbm_{pm} \rbm(\ubm^\star_k)_m
    + \Dbm_{pm} \rbm^{\mathrm{adj}}(\lambdabold^\star_k)_m,
  \end{equation*}
  where the entries of matrices $\Cbm \in \Rbb^{N_e \times N_\ubm}$ and $\Dbm \in \Rbb^{N_e \times N_\ubm}$ are given by
  \begin{align*}
    \Cbm_{pm}
    &\coloneqq
    - \pder{\rhobold^\star_q}{\psibold_p}  \left( \overline{\pder{^2j}{\rhobold_q \partial \ubm_s}}(\ubm^\star,\ubm^\star_k)
    - \lambdabold^\star_{k,i} \overline{\pder{^2\rbm_i}{\rhobold_q \partial \ubm_s}}(\ubm^\star,\ubm^\star_k)
    - \pder{\rbm_i}{\rhobold_q}(\ubm^\star) (\Kbm^{-1})_{il} \overline{\pder{^2j}{\ubm_l \partial \ubm_s}}(\ubm^\star,\ubm^\star_k) 
    \right)
    \Kbm^{-1}_{sm},\\
    \Dbm_{pm}
    &\coloneqq
    \pder{\rhobold^\star_q}{\psibold_p}
    \pder{\rbm_i}{\rhobold_q}(\ubm^\star) (\Kbm^{-1})_{im}  .
  \end{align*}
  It follows that
  \begin{equation*}
    \| \nabla J - \nabla J_k \|_2
    =
    \sigma_{\rm max}(\Cbm_{pm}) \| \rbm(\ubm^\star_k) \|_2
    + \sigma_{\rm max}(\Dbm_{pm}) \| \rbm^{\mathrm{adj}}(\lambdabold^\star_k) \|_2,
  \end{equation*}
  which is the desired relationship.
\end{proof}

\begin{proof}[Proof of Theorem~\ref{thm:rom_output_error_simple}]
  We begin with \eqref{eq:rom_proof_jmjk} and appeal to \eqref{eq:rom_proof_u_error} to obtain
  \begin{align*}
   |J - J_k|
   &=
   \left| \lambdabold^\star {}^T \Kbm (\ubm^\star - \ubm^\star_k) - (\ubm^\star - \ubm^\star_k)^T \left[ \int_{\theta = 0}^1  \pder{^2j}{\ubm^2}(\theta \ubm^\star + (1-\theta) \ubm^\star_k) \theta d\theta \right]  (\ubm^\star - \ubm^\star_k)
   \right|
   \\
   &= \left| -\lambdabold^\star {}^T \rbm(\ubm_k^\star)
   - \rbm(\ubm^\star_k)^T \Bbm \rbm(\ubm^\star_k) \right|
   \\
   &\leq \| \lambdabold^\star \|_2 \| \rbm(\ubm_k^\star) \|_2
   +
   \sigma_{\rm max}(\Bbm) \| \rbm(\ubm^\star_k) \|_2^2,
  \end{align*}
  where $\Bbm$ is given by \eqref{eq:rom_proof_Bbm}.
\end{proof}

%% file: pap.bbl
\begin{thebibliography}{10}

\bibitem{aage_giga-voxel_2017}
Niels Aage, Erik Andreassen, Boyan~S. Lazarov, and Ole Sigmund.
\newblock Giga-voxel computational morphogenesis for structural design.
\newblock {\em Nature}, 550(7674):84--86, October 2017.

\bibitem{agarwal_trust-region_2013}
Anshul Agarwal and Lorenz~T. Biegler.
\newblock A trust-region framework for constrained optimization using reduced
  order modeling.
\newblock {\em Optimization and Engineering}, 14(1):3--35, March 2013.

\bibitem{alexandrov_approximation_2001}
Natalia~M. Alexandrov, Robert~Michael Lewis, Clyde~R. Gumbert, Lawrence~L.
  Green, and Perry~A. Newman.
\newblock Approximation and {Model} {Management} in {Aerodynamic}
  {Optimization} with {Variable}-{Fidelity} {Models}.
\newblock {\em Journal of Aircraft}, 38(6):1093--1101, November 2001.

\bibitem{amir_multigrid-cg_2014}
Oded Amir, Niels Aage, and Boyan~S. Lazarov.
\newblock On multigrid-{CG} for efficient topology optimization.
\newblock {\em Structural and Multidisciplinary Optimization}, 49(5):815--829,
  May 2014.

\bibitem{amir_efficient_2010}
Oded Amir, Mathias Stolpe, and Ole Sigmund.
\newblock Efficient use of iterative solvers in nested topology optimization.
\newblock {\em Structural and Multidisciplinary Optimization}, 42(1):55--72,
  July 2010.

\bibitem{amsallem_fast_2015}
David Amsallem, Matthew~J. Zahr, and Kyle Washabaugh.
\newblock Fast local reduced basis updates for the efficient reduction of
  nonlinear systems with hyper-reduction.
\newblock {\em Advances in Computational Mathematics}, 41(5):1187--1230,
  October 2015.

\bibitem{andreassen_efficient_2011}
Erik Andreassen, Anders Clausen, Mattias Schevenels, Boyan~S. Lazarov, and Ole
  Sigmund.
\newblock Efficient topology optimization in {MATLAB} using 88 lines of code.
\newblock {\em Structural and Multidisciplinary Optimization}, 43(1):1--16,
  January 2011.

\bibitem{arian_trust-region_2000}
E.~Arian, M.~Fahl, and E.~W. Sachs.
\newblock Trust-{Region} {Proper} {Orthogonal} {Decomposition} for {Flow}
  {Control}.
\newblock Technical Report ICASE-2000-25, Institue for Computer Applications in
  Science and Engineering, May 2000.

\bibitem{arian_managing_2002}
E.~Arian, M.~Fahl, and E.W. Sachs.
\newblock Managing {POD} models by optimization methods.
\newblock In {\em Proceedings of the 41st {IEEE} {Conference} on {Decision} and
  {Control}, 2002.}, volume~3, pages 3300--3305, Las Vegas, NV, USA, 2002.
  IEEE.

\bibitem{Barrault_2004_EIM}
M.~Barrault, Y.~Maday, N.~C. Nguyen, and A.~T. Patera.
\newblock An ``empirical interpolation'' method: application to efficient
  reduced-basis discretization of partial differential equations.
\newblock {\em C. R. Acad. Sci. Paris, Ser. I}, 339:667--–672, 2004.

\bibitem{bourdin_filters_2001}
Blaise Bourdin.
\newblock Filters in topology optimization.
\newblock {\em International Journal for Numerical Methods in Engineering},
  50(9):2143--2158, March 2001.

\bibitem{brand_fast_2006}
Matthew Brand.
\newblock Fast low-rank modifications of the thin singular value decomposition.
\newblock {\em Linear Algebra and its Applications}, 415(1):20--30, May 2006.

\bibitem{carter_global_1991}
Richard~G. Carter.
\newblock On the {Global} {Convergence} of {Trust} {Region} {Algorithms}
  {Using} {Inexact} {Gradient} {Information}.
\newblock {\em SIAM Journal on Numerical Analysis}, 28(1):251--265, February
  1991.

\bibitem{choi_accelerating_2019}
Youngsoo Choi, Geoffrey Oxberry, Daniel White, and Trenton Kirchdoerfer.
\newblock Accelerating design optimization using reduced order models.
\newblock {\em arXiv:1909.11320 [cs, math]}, September 2019.

\bibitem{conn_trust-region_2000}
A.~R. Conn, Nicholas I.~M. Gould, and Ph~L. Toint.
\newblock {\em {T}rust-{R}egion {M}ethods}.
\newblock {MPS}-{SIAM} series on optimization. Society for Industrial and
  Applied Mathematics, Philadelphia, PA, 2000.

\bibitem{evgrafov_large-scale_2008}
Anton Evgrafov, Cory~J. Rupp, Kurt Maute, and Martin~L. Dunn.
\newblock Large-scale parallel topology optimization using a dual-primal
  substructuring solver.
\newblock {\em Structural and Multidisciplinary Optimization}, 36(4):329--345,
  October 2008.

\bibitem{gogu_improving_2015}
Christian Gogu.
\newblock Improving the efficiency of large scale topology optimization through
  on-the-fly reduced order model construction.
\newblock {\em International Journal for Numerical Methods in Engineering},
  101(4):281--304, January 2015.

\bibitem{heinkenschloss_matrix-free_2014}
Matthias Heinkenschloss and Denis Ridzal.
\newblock A {Matrix}-{Free} {Trust}-{Region} {SQP} {Method} for {Equality}
  {Constrained} {Optimization}.
\newblock {\em SIAM Journal on Optimization}, 24(3):1507--1541, January 2014.

\bibitem{heinkenschloss_analysis_2002}
Matthias Heinkenschloss and Luis~N. Vicente.
\newblock Analysis of {Inexact} {Trust}-{Region} {SQP} {Algorithms}.
\newblock {\em SIAM Journal on Optimization}, 12(2):283--302, January 2002.

\bibitem{Hesthaven_2016_RB_Book}
Jan~S. Hesthaven, Gianluigi Rozza, and Benjamin Stamm.
\newblock {\em Certified reduced basis methods for parametrized partial
  differential equations}.
\newblock Springer, 2016.

\bibitem{kirsch_structural_2001}
U.~Kirsch and P.Y. Papalambros.
\newblock Structural reanalysis for topological modifications - a unified
  approach.
\newblock {\em Structural and Multidisciplinary Optimization}, 21(5):333--344,
  July 2001.

\bibitem{kouri_trust-region_2013}
D.~P. Kouri, M.~Heinkenschloss, D.~Ridzal, and B.~G. van Bloemen~Waanders.
\newblock A {Trust}-{Region} {Algorithm} with {Adaptive} {Stochastic}
  {Collocation} for {PDE} {Optimization} under {Uncertainty}.
\newblock {\em SIAM Journal on Scientific Computing}, 35(4):A1847--A1879,
  January 2013.

\bibitem{kouri_inexact_2014}
D.~P. Kouri, M.~Heinkenschloss, D.~Ridzal, and B.~G. van Bloemen~Waanders.
\newblock Inexact {Objective} {Function} {Evaluations} in a {Trust}-{Region}
  {Algorithm} for {PDE}-{Constrained} {Optimization} under {Uncertainty}.
\newblock {\em SIAM Journal on Scientific Computing}, 36(6):A3011--A3029,
  January 2014.

\bibitem{lazarov_filters_2011}
B.~S. Lazarov and O.~Sigmund.
\newblock Filters in topology optimization based on {Helmholtz}-type
  differential equations.
\newblock {\em International Journal for Numerical Methods in Engineering},
  86(6):765--781, May 2011.

\bibitem{Maday_2002_RB_Noncoercive}
Y.~Maday, A.~T. Patera, and D.~V. Rovas.
\newblock A blackbox reduced-basis output bound method for noncoercive linear
  problems.
\newblock In {\em Nonlinear Partial Differential Equations and their
  Applications - Coll{\`{e}}ge de France Seminar Volume {XIV}}, pages 533--569.
  Elsevier, 2002.

\bibitem{nguyen_three-dimensional_2019}
Chuong Nguyen, Xiaoying Zhuang, Ludovic Chamoin, Hung Nguyen-Xuan, Xianzhong
  Zhao, and Timon Rabczuk.
\newblock Three-dimensional topology optimization of auxetic metamaterial using
  isogeometric analysis and model order reduction.
\newblock {\em arXiv:1908.11449 [cs]}, August 2019.
\newblock arXiv: 1908.11449.

\bibitem{nguyen_computational_2010}
Tam~H. Nguyen, Glaucio~H. Paulino, Junho Song, and Chau~H. Le.
\newblock A computational paradigm for multiresolution topology optimization
  ({MTOP}).
\newblock {\em Structural and Multidisciplinary Optimization}, 41(4):525--539,
  April 2010.

\bibitem{nocedal_numerical_2006}
Jorge Nocedal and Stephen~J. Wright.
\newblock {\em Numerical optimization}.
\newblock Springer series in operations research. Springer, New York, 2nd ed
  edition, 2006.
\newblock OCLC: ocm68629100.

\bibitem{oxberry_limited-memory_2017}
Geoffrey~M. Oxberry, Tanya Kostova-Vassilevska, William Arrighi, and Kyle
  Chand.
\newblock Limited-memory adaptive snapshot selection for proper orthogonal
  decomposition.
\newblock {\em International Journal for Numerical Methods in Engineering},
  109(2):198--217, January 2017.

\bibitem{qian_certified_2017}
Elizabeth Qian, Martin Grepl, Karen Veroy, and Karen Willcox.
\newblock A {Certified} {Trust} {Region} {Reduced} {Basis} {Approach} to
  {PDE}-{Constrained} {Optimization}.
\newblock {\em SIAM Journal on Scientific Computing}, 39(5):S434--S460, January
  2017.

\bibitem{rojas-labanda_efficient_2016}
Susana Rojas-Labanda and Mathias Stolpe.
\newblock An efficient second-order {SQP} method for structural topology
  optimization.
\newblock {\em Structural and Multidisciplinary Optimization},
  53(6):1315--1333, June 2016.

\bibitem{Rozza_2008_RB_Review}
G.~Rozza, D.~B.~P. Huynh, and A.~T. Patera.
\newblock Reduced basis approximation and a posteriori error estimation for
  affinely parametrized elliptic coercive partial differential equations ---
  {Application} to transport and continuum mechanics.
\newblock {\em Archives of Computational Methods in Engineering},
  15(3):229--275, 2008.

\bibitem{sachs_adaptive_2014}
Ekkehard~W. Sachs, Marina Schneider, and Matthias Schu.
\newblock Adaptive {Trust}-{Region} {POD} {Methods} in {PIDE}-{Constrained}
  {Optimization}.
\newblock In G{\"u}nter Leugering, Peter Benner, Sebastian Engell, Andreas
  Griewank, Helmut Harbrecht, Michael Hinze, Rolf Rannacher, and Stefan
  Ulbrich, editors, {\em Trends in {PDE} {Constrained} {Optimization}},
  International {Series} of {Numerical} {Mathematics}, pages 327--342. Springer
  International Publishing, Cham, 2014.

\bibitem{senne_approximate_2019}
Thadeu~A. Senne, Francisco A.~M. Gomes, and Sandra~A. Santos.
\newblock On the approximate reanalysis technique in topology optimization.
\newblock {\em Optimization and Engineering}, 20(1):251--275, March 2019.

\bibitem{sigmund_99_2001}
O.~Sigmund.
\newblock A 99 line topology optimization code written in {Matlab}.
\newblock {\em Structural and Multidisciplinary Optimization}, 21(2):120--127,
  April 2001.

\bibitem{Sigmund_2013_TO_Review}
Ole Sigmund and Kurt Maute.
\newblock Topology optimization approaches.
\newblock {\em Structural and Multidisciplinary Optimization},
  48(6):1031--1055, aug 2013.

\bibitem{sirovich_turbulence_1987}
Lawrence Sirovich.
\newblock Turbulence and the dynamics of coherent structures. {I}. {Coherent}
  structures.
\newblock {\em Quarterly of Applied Mathematics}, 45(3):561--571, October 1987.

\bibitem{sun_efficient_2018}
Youhong Sun, Xuqi Zhao, Yongping Yu, and Shaopeng Zheng.
\newblock An {Efficient} {Reanalysis} {Method} for {Topological} {Optimization}
  of {Vibrating} {Continuum} {Structures} for {Simple} and {Multiple}
  {Eigenfrequencies}.
\newblock {\em Mathematical Problems in Engineering}, 2018:1--10, November
  2018.

\bibitem{svanberg_method_1987}
Krister Svanberg.
\newblock The method of moving asymptotes---a new method for structural
  optimization.
\newblock {\em International Journal for Numerical Methods in Engineering},
  24(2):359--373, February 1987.

\bibitem{wang_large-scale_2007}
Shun Wang, Eric~de Sturler, and Glaucio~H. Paulino.
\newblock Large-scale topology optimization using preconditioned {Krylov}
  subspace methods with recycling.
\newblock {\em International Journal for Numerical Methods in Engineering},
  69(12):2441--2468, March 2007.

\bibitem{washabaugh_use_2016}
Kyle~M. Washabaugh, Matthew~J. Zahr, and Charbel Farhat.
\newblock On the {Use} of {Discrete} {Nonlinear} {Reduced}-{Order} {Models} for
  the {Prediction} of {Steady}-{State} {Flows} {Past} {Parametrically}
  {Deformed} {Complex} {Geometries}.
\newblock In {\em 54th {AIAA} {Aerospace} {Sciences} {Meeting}}, San Diego,
  California, USA, January 2016. American Institute of Aeronautics and
  Astronautics.

\bibitem{xiao_fly_2020}
Manyu Xiao, Dongcheng Lu, Piotr Breitkopf, Balaji Raghavan, Subhrajit Dutta,
  and Weihong Zhang.
\newblock On-the-fly model reduction for large-scale structural topology
  optimization using principal components analysis.
\newblock {\em Structural and Multidisciplinary Optimization}, February 2020.

\bibitem{yoon_structural_2010}
Gil~Ho Yoon.
\newblock Structural topology optimization for frequency response problem using
  model reduction schemes.
\newblock {\em Computer Methods in Applied Mechanics and Engineering},
  199(25-28):1744--1763, May 2010.

\bibitem{yue_accelerating_2013}
Yao Yue and Karl Meerbergen.
\newblock Accelerating {Optimization} of {Parametric} {Linear} {Systems} by
  {Model} {Order} {Reduction}.
\newblock {\em SIAM Journal on Optimization}, 23(2):1344--1370, January 2013.

\bibitem{zahr_phd_2016}
Matthew~J. Zahr.
\newblock {\em Adaptive Model Reduction to Accelerate Optimization Problems
  Governed by Partial Differential Equations}.
\newblock PhD thesis, Stanford University, August 2016.

\bibitem{zahr_progressive_2015}
Matthew~J. Zahr and Charbel Farhat.
\newblock Progressive construction of a parametric reduced-order model for
  {PDE}-constrained optimization.
\newblock {\em International Journal for Numerical Methods in Engineering},
  102(5):1111--1135, May 2015.

\bibitem{zegard_bridging_2016}
Tom{\'a}s Zegard and Glaucio~H. Paulino.
\newblock Bridging topology optimization and additive manufacturing.
\newblock {\em Structural and Multidisciplinary Optimization}, 53(1):175--192,
  January 2016.

\bibitem{zheng_approximate_2017}
Shaopeng Zheng, Xuqi Zhao, Yongping Yu, and Youhong Sun.
\newblock The approximate reanalysis method for topology optimization under
  harmonic force excitations with multiple frequencies.
\newblock {\em Structural and Multidisciplinary Optimization},
  56(5):1185--1196, November 2017.

\end{thebibliography}
